\documentclass[10pt, reqno]{amsart}
\usepackage{amssymb}
\usepackage{graphicx}
\usepackage{hyperref}
\usepackage{xcolor} 
\usepackage{tensor}
\usepackage{cancel}

\usepackage{fullpage} 

\usepackage{enumitem}

\definecolor{green}{rgb}{0,0.8,0} 

\newtheorem{theorem}{Theorem}[section]
\newtheorem{corollary}[theorem]{Corollary}
\newtheorem{lemma}[theorem]{Lemma}
\newtheorem{proposition}[theorem]{Proposition}
\theoremstyle{definition}
\newtheorem{definition}[theorem]{Definition}

\theoremstyle{remark}
\newtheorem{remark}[theorem]{Remark}
\numberwithin{equation}{section}
\newcommand{\relphantom}[1]{\mathrel{\phantom{#1}}}

\makeatletter
\newcommand{\nrm}{\@ifstar{\nrmb}{\nrmi}}
\newcommand{\nrmi}[1]{\Vert{#1}\Vert}
\newcommand{\nrmb}[1]{\left\Vert{#1}\right\Vert}
\newcommand{\abs}{\@ifstar{\absb}{\absi}}
\newcommand{\absi}[1]{\vert{#1}\vert}
\newcommand{\absb}[1]{\left\vert{#1}\right\vert}
\newcommand{\brk}{\@ifstar{\brkb}{\brki}}
\newcommand{\brki}[1]{\langle{#1}\rangle}
\newcommand{\brkb}[1]{\left\langle{#1}\right\rangle}
\newcommand{\set}{\@ifstar{\setb}{\seti}}
\newcommand{\seti}[1]{\{#1\}}
\newcommand{\setb}[1]{\left\{ #1\right\}}
\makeatother

\newcommand{\nnrm}[1]{{\vert\kern-0.25ex\vert\kern-0.25ex\vert #1 
    \vert\kern-0.25ex\vert\kern-0.25ex\vert}}

\newcommand{\td}[1]{\widetilde{#1}}
\newcommand{\br}[1]{\overline{#1}}

\newcommand{\wh}[1]{\widehat{#1}}

\newcommand{\VERT}[1]{{\left\vert\kern-0.25ex\left\vert\kern-0.25ex\left\vert #1 
    \right\vert\kern-0.25ex\right\vert\kern-0.25ex\right\vert}}

\DeclareMathOperator{\dist}{dist}

\DeclareMathOperator{\supp}{supp}

\newcommand{\aeq}{\simeq}
\newcommand{\aleq}{\lesssim}
\newcommand{\ageq}{\gtrsim}

\newcommand{\lap}{\Delta}

\newcommand{\ud}{\mathrm{d}}
\newcommand{\rd}{\partial}
\newcommand{\nb}{\nabla}

\newcommand{\bb}{\Big}

\newcommand{\0}{\emptyset}

\newcommand{\peq}{\relphantom{=}}			
\newcommand{\pleq}{\relphantom{\leq}}			
\newcommand{\pgeq}{\relphantom{\geq}}			

\newcommand{\alp}{\alpha}
\newcommand{\bt}{\beta}
\newcommand{\gmm}{\gamma}

\newcommand{\dlt}{\delta}

\newcommand{\eps}{\epsilon}

\newcommand{\lmb}{\lambda}

\newcommand{\sgm}{\sigma}

\newcommand{\bfb}{{\bf b}}

\newcommand{\bfe}{{\bf e}}

\newcommand{\bfp}{{\bf p}}

\newcommand{\bfu}{{\bf u}}

\newcommand{\bfA}{{\bf A}}
\newcommand{\bfB}{{\bf B}}

\newcommand{\bfL}{{\bf L}}

\newcommand{\bfalp}{\boldsymbol{\alpha}}
\newcommand{\bfbt}{\boldsymbol{\beta}}
\newcommand{\bfgmm}{\boldsymbol{\gamma}}

\newcommand{\bbC}{\mathbb C}

\newcommand{\bbR}{\mathbb R}

\newcommand{\bbZ}{\mathbb Z}

\newcommand{\calB}{\mathcal B}

\newcommand{\calI}{\mathcal I}

\newcommand{\calL}{\mathcal L}
\newcommand{\calM}{\mathcal M}

\newcommand{\calP}{\mathcal P}
\newcommand{\calQ}{\mathcal Q}

\newcommand{\calS}{\mathcal S}

\newcommand{\calX}{\mathcal X}

\newcommand{\frkF}{\mathfrak F}

\newcommand{\frkR}{\mathfrak R}

\newcommand{\frka}{\mathfrak a}
\newcommand{\frkb}{\mathfrak b}
\newcommand{\frkc}{\mathfrak c}

\newcommand{\frke}{\mathfrak e}

\newcommand{\frkq}{\mathfrak q}
\newcommand{\frkr}{\mathfrak r}
\newcommand{\frks}{\mathfrak s}

\newcommand{\ackn}[1]{
\addtocontents{toc}{\protect\setcounter{tocdepth}{1}}
\subsection*{Acknowledgements} {#1}
\addtocontents{toc}{\protect\setcounter{tocdepth}{2}} }

\DeclareMathOperator{\Op}{Op}

\newcommand{\nonlin}{\calP}			
\newcommand{\lin}{\calL}			
\newcommand{\linmain}{\calP}			
\newcommand{\linlower}{\calQ}			
\newcommand{\prin}{\bfp}			
\newcommand{\diag}{P}			
\newcommand{\dprin}{p}			
\newcommand{\symm}{\mathfrak{S}}			
\newcommand{\asymm}{\mathfrak{A}}		
\newcommand{\covec}{\mathfrak{U}}			
\newcommand{\rem}{\mathfrak{R}}			
\newcommand{\comm}{\mathfrak{C}}	

\newcommand{\bgB}{\mathring{\bfB}}			

\newcommand{\dlta}{\dlt_{1}}		

\newcommand{\dfrm}[1]{{}^{(#1)} \pi}	

\newcommand{\tb}{\tilde{b}}

\newcommand{\df}{\mathrm{df}}

\newcommand{\out}{\mathrm{out}}
\newcommand{\intr}{\mathrm{in}}
\newcommand{\med}{\mathrm{med}}

\newcommand{\chf}{\mathbf{1}}

\makeatletter
\newcommand{\doublewidetilde}[1]{{%
  \mathpalette\double@widetilde{#1}%
}}
\newcommand{\double@widetilde}[2]{%
  \sbox\z@{$\m@th#1\widetilde{#2}$}%
  \ht\z@=.9\ht\z@
  \widetilde{\box\z@}%
}
\makeatother

\newcommand{\pfstep}[1]{\smallskip \noindent {\it #1.}}

\begin{document}

\title[]{Wellposedness of the electron MHD without resistivity \\ for large perturbations of the uniform magnetic field}
\author{In-Jee Jeong}%
\address{Department of Mathematical Sciences and Research Institute of Mathematics, Seoul National University, Seoul, Korea 08826}%
\email{injee\_j@snu.ac.kr}%

\author{Sung-Jin Oh}%
\address{Department of Mathematics, UC Berkeley, Berkeley, CA, USA 94720 and School of Mathematics, KIAS, Seoul, Korea 02455}%
\email{sjoh@math.berkeley.edu}%



\date{\today}%

\begin{abstract}
We prove the local wellposedness of the Cauchy problems for the electron magnetohydrodynamics equations (E-MHD) without resistivity for possibly large perturbations of nonzero uniform magnetic fields. While the local wellposedness problem for (E-MHD) has been extensively studied in the presence of resistivity (which provides dissipative effects), this seems to be the first such result without resistivity. (E-MHD) is a fluid description of plasma in small scales where the motion of electrons relative to ions is significant. Mathematically, it is a quasilinear dispersive equation with nondegenerate but nonelliptic second-order principal term. Our result significantly improves upon the straightforward adaptation of the classical work of Kenig--Ponce--Rolvung--Vega on the quasilinear ultrahyperbolic Schr\"odinger equations, as the regularity and decay assumptions on the initial data are greatly weakened to the level analogous to the recent work of Marzuola--Metcalfe--Tataru in the case of elliptic principal term.

A key ingredient of our proof is a simple observation about the relationship between the size of a symbol and the operator norm of its quantization as a pseudodifferential operator when restricted to high frequencies. This allows us to localize the (non-classical) pseudodifferential renormalization operator considered by Kenig--Ponce--Rolvung--Vega, and produce instead a classical pseudodifferential renormalization operator. We furthermore incorporate the function space framework of Marzuola--Metcalfe--Tataru to the present case of nonelliptic principal term.

\end{abstract}
\maketitle


\section{Introduction}
In the $(3+1)$-dimensional spacetime $\bbR_{t} \times \bbR^{3}_{x}$, the \emph{electron magnetohydrodynamics} system is given by
\begin{equation} \label{eq:e-mhd} \tag{E-MHD}
	\left\{
	\begin{aligned}
		& \rd_{t} \bfB + \nb \times ((\nb \times \bfB) \times \bfB) = 0, \\
		& \nb \cdot \bfB = 0,
	\end{aligned}
	\right.
\end{equation}
where $\bfB : \bbR_{t} \times  {\bbR^{3}_{x}} \to \bbR^{3}$ is a time-dependent vector field (magnetic field). The equation we consider is \emph{without resistivity}, in the sense that it does not have the dissipative term $\mu \lap \bfB$ on the  {RHS} of the first equation, and  as a result, the energy $\frac12\int \abs{\bfB(t)}^{2} \, \ud x$ is conserved. This equation is a fluid description of plasma in small scales, where the motion of electrons relative to ions is significant -- which is typical in low plasma density and/or high temperature regimes, as argued in the work of Lighthill \cite{Light} where the model was introduced. It is also the simpler small-scale limit of Hall-MHD, which is another extended MHD model that takes the relative motion of electrons and ions into account (see  {equation \eqref{eq:hall-mhd}} below). It is by now a well-established fact that \eqref{eq:e-mhd} reproduces many essential features of the Hall-MHD system and (more accurate) two-fluid plasma models (see, e.g., \cite{Pecseli}).

The quadratic nonlinearity $\nb \times ((\nb \times \bfB) \times \bfB)$ is referred to as the \textit{Hall current term} after the discovery of Hall \cite{Hall}, and makes \eqref{eq:e-mhd} a \textit{quasilinear dispersive} equation. While the Hall current term is omitted in the usual ideal MHD model, it is known to be directly responsible for various situations including plasma confinement, collisionless magnetic reconnection featuring $X$-configuration (\cite{BSD,SDRD,HoGr,DKM}), Hall-effect thrusters (\cite{YBMH,SRF,RaFi}), planetary magnetospheres including that of the Earth's (\cite{DGCT,LSDP}) and magnetic field dynamics in neutron stars (\cite{GoRe,Jones,GoCu,WoHoLy}). In all of these situations, the effect of magnetic resistivity is considered negligible with respect to that of the Hall current term. Turbulence of electron-MHD system is a very active area of study in the plasma physics literature as well (\cite{WaHo,GaSm,ChLa}). 

On the mathematical side, the illuminating work of Jang--Masmoudi \cite{JM} (see also \cite{ADF,PWX}) formally derived the Hall-MHD equations (among other equations)  from a two-species kinetic system consisting of electrons and ions, which clarifies that the Hall term is indeed a two-species effect. We shall review the existing mathematical literature on \eqref{eq:e-mhd} in more detail below.

It is expected that a local wellposedness theory for \eqref{eq:e-mhd} and related systems would be not only useful in rigorous derivation of MHD models but also provide a ground for a systematic study of various aforementioned physical phenomena involving the Hall effect. In the current work, we obtain the first local wellposedness result for \eqref{eq:e-mhd} without resistivity, when the initial magnetic field is a (possibly large) perturbation of a nonzero uniform magnetic field, which we take to be $\bfe_{3}$ (i.e., the uniform magnetic field of unit strength in the $x^{3}$-direction). This indeed seems to be the usual setup for \eqref{eq:e-mhd} in the physics literature, including the original work of Lighthill (\cite{Light}). Furthermore, such an assumption cannot be completely avoided: in stark contrast to the resistive case, or even the usual ideal MHD model, the Cauchy problem for \eqref{eq:e-mhd} and Hall-MHD is, in general, strongly illposed for arbitrarily small and smooth (even analytic) data (\cite{JO1}). Analyzing the linearized equations near steady states with a degeneracy where $\bfB$ vanishes, one sees that illposedness arises from \textit{degenerate dispersion}, wherein the frequency of wave packets traveling towards degeneracy grows indefinitely in an arbitrarily short interval of time.

Our local wellposedness result is phrased in terms of the translation-invariant space introduced by Marzulola--Metcalfe--Tataru \cite{MMT1, MMT2, MMT3} in the context of quasilinear Schr\"odinger equations, but adapted to \eqref{eq:e-mhd}. To define this space, we first fix a partition $\calI_{k}$ of $\bbR$ (for each nonnegative integer $k$) into intervals of length $2^{k}$, and a smooth partition of unity $\set{\chi_{I}}_{I \in \calI_{k}}$ such that $\supp \chi_{I} \subseteq 2 I$ (i.e., the interval with the same center as $I$ but twice the length). Given $s \in \bbR$ and $u : \bbR^{3} \to \bbR$, define
\begin{equation*}
	\nrm{u}_{\ell^{1}_{\calI} H^{s}}^{2} := \sum_{k \ge 0} \left( 2^{s k} \sum_{I \in \calI_{k}} \nrm{   \chi_{I}(x^{3}) P_{k} u}_{L^{2}} \right)^{2},
\end{equation*}
where $\set{P_{k}}$ are the inhomogeneous Littlewood--Paley projections (see Section~\ref{subsec:notation}). As in \cite{MMT1, MMT2, MMT3}, note that any translations of this norm are equivalent. But unlike in those papers, the variable $x^{3}$ plays a distinguished role in the definition of our norm; this is because of the \emph{conic directionality} of the group velocities for the linearization of \eqref{eq:e-mhd} around $\bfe_{3}$, to be discussed in more detail below. As usual, we extend this norm to vector-valued functions by summing up the norm of all components.

Our main theorem may now be stated as follows.
\begin{theorem}[Main theorem, simple version] \label{thm:main-simple}
Let $s > \frac{7}{2}$ and consider a vector field $\bfB_{0} : \bbR^{3} \to \bbR^{3}$ satisfying $\nb \cdot \bfB_{0} = 0$. Assume  furthermore that $\bfB_{0}$ satisfies the following properties:
\begin{enumerate}
\item {\bf Nondegeneracy.} $\bfB_{0}(x) \neq 0$ at every point $x \in \bbR^{3}$,
\item {\bf Asymptotic uniformity.} $\nrm{\bfB_{0} - \bfe_{3}}_{\ell^{1}_{\calI} H^{s}} < + \infty$,
\item {\bf Nontrapping.} Every nonconstant solution $(X, \Xi)(t)$ to the Hamiltonian system associated with $\dprin_{\bfB_{0}}(x, \xi) = \bfB_{0}(x) \cdot \xi \abs{\xi}$ escapes to $x^{3} = \pm \infty$, i.e., $X^{3}(t) \to \infty$ or $X^{3}(t) \to - \infty$ as $t \to \infty$.
\end{enumerate}
Then the Cauchy problem for \eqref{eq:e-mhd} with $\bfB(t=0) = \bfB_{0}$ is locally wellposed. 
\end{theorem}  
The solution in Theorem \ref{thm:main-simple} belongs to the space $\bfB - \bfe_{3} \in C([0,T); \ell^{1}_{\calI}H^{s})$ for some $T>0$. To quantify the lifespan $T$ guaranteed by our theorem, we need more quantitative information on $\bfB_{0}$. We refer the reader to Theorem~\ref{thm:main} below for a more precise formulation. We shall see later in Corollary \ref{cor:main} that the nontrapping assumption is satisfied for small perturbations of $\bfe_{3}$, which gives the following: \begin{corollary}[Small data local wellposedness]
	For a given $s>\frac{7}{2}$, there exist positive constants $\varepsilon, T$ depending only on $s$ such that for any $\bfB_{0} : \bbR^{3} \to \bbR^{3}$ satisfying $\nb \cdot \bfB_{0} = 0$ and $\nrm{\bfB_{0} - \bfe_{3}}_{\ell^{1}_{\calI} H^{s}} < \varepsilon$, there exists a unique local-in-time solution to \eqref{eq:e-mhd} on the time interval $[0,T]$ with initial data $\bfB_{0}$.
\end{corollary}

The motivation for the assumptions in Theorem~\ref{thm:main-simple} is as follows. Assumption~(1) is a basic requirement that avoids, in particular, degenerate dispersion, which may lead to illposedness \cite{JO1}. To motivate Assumption~(2), as well as the distinguished role played by $x^{3}$ in the definition of $\ell^{1}_{\calI} H^{s}$, let us consider the linearization \eqref{eq:e-mhd} around $\bfB = \bfe_{3}$, which is the constant-coefficient system
\begin{equation*}
	\rd_{t} b + \bfe_{3} \cdot \nb (\nb \times b) = 0, \quad \nb \cdot b = 0. 
\end{equation*}
Using the $L^{2}$-projections $\Pi_{\pm}(D) = $ (where $D = i^{-1} \rd_{x}$ and $m(D)$ is the Fourier multiplier with symbol $m(\xi)$) to decompose $b = b_{+} + b_{-}$ with $b_{\pm} = \Pi_{\pm}(D) b$, this system is diagonalized and we arrive at the equations
\begin{equation*}
	\rd_{t} b_{\pm} \pm \bfe_{3} \cdot D \abs{D} b_{\pm} = 0,
\end{equation*}
which is a dispersive equation with dispersion relation $\pm p(\xi) = \pm \bfe_{3} \cdot \xi \abs{\xi}$. Hence, the group velocity takes the form
\begin{equation*}
	v_{\pm}(\xi) := \pm \nb_{\xi} p(\xi) = \pm \begin{pmatrix} \frac{\xi_{1} \xi_{3}}{\abs{\xi}} \\ \frac{\xi_{2} \xi_{3}}{\abs{\xi}} \\ \frac{\xi_{1}^{2} + \xi_{2}^{3} + 2 \xi_{3}^{2}}{\abs{\xi}} \end{pmatrix}.
\end{equation*}
Observe that the set of all possible group velocities form a cone around $\bfe_{3}$ with aperture less than $\pi$ -- we call this the \emph{conic directionality} of \eqref{eq:e-mhd} (see also Lemma~\ref{lem:cone-dir}). In particular, the $x^{3}$-component of the group velocity is always comparable to $\abs{\xi}$, which is why the above definition of $\ell^{1}_{\calI} H^{s}$ -- involving localization and summation in only the $x^{3}$ variable -- is the suitable adaptation of the $l^{1} H^{s}$ space introduced in \cite{MMT1}. As in \cite{MMT1}, the extra $\ell^{1}$ summability controls the frequency evolution of wave packets in the far-away region -- see also Lemmas~\ref{lem:freq-evol} and \ref{lem:nontrapping-cor} below.

If we consider instead the linearization of \eqref{eq:e-mhd} around an arbitrary magnetic field $\bgB$, then the diagonalization of the principal symbol leads to the scalar symbols $\pm p_{\bgB}(t, x, \xi) := \pm \bgB(t, x) \cdot \xi \abs{\xi}$, and the Hamiltonian flows associated to $\pm p_{\bgB}$ are the adequate generalization of the group velocities $v_{\pm}(\xi)$ in the constant coefficient case. Assumption (3) ensures that, at least for a short time, there are no trapped bicharacteristics (i.e., solution to the Hamiltonian flow that stays in a bounded region for all times) with $\bgB = \bfB$, which would be an obstruction for the \emph{local smoothing effect} of \eqref{eq:e-mhd}. This effect is key to our proof, and will be discussed further in Section~\ref{subsec:ideas} below.

The system \eqref{eq:e-mhd} may be compared with \emph{quasilinear Schr\"odinger equations}, which are equations of the form
\begin{equation} \label{eq:ultrahyp}
i \rd_{t} u + \sum_{j, k = 1}^{d} \rd_{j} (g^{j k}(u, \br{u}, \nb_{x} u, \nb_{x} \br{u}) \rd_{k} u) = F(u, \br{u}, \nb_{x} u, \nb_{x} \br{u})
\end{equation}
where $g$ and $F$ are smooth functions of their variables, with $g$ a real-valued symmetric matrix. In view of the nonellipticity of the principal symbol, \eqref{eq:e-mhd} most closely resembles the \emph{ultrahyperbolic} case, i.e., when $g$ is nondegenerate but \emph{not} positive definite. 

Local wellposedness of the Cauchy problem for quasilinear ultrahyperbolic Schr\"odinger equations was first obtained in the landmark papers of Kenig--Ponce--Rolvung--Vega \cite{KPRV1, KPRV2}. However, Theorem~\ref{thm:main-simple} significantly improves upon the straightforward adaptation of \cite{KPRV1, KPRV2}, as the regularity and decay assumptions on the initial data are greatly weakened to the level analogous to the recent work of Marzuola--Metcalfe--Tataru \cite{MMT3} for the case when $g$ is positive definite. Our approach to proving Theorem~\ref{thm:main-simple} builds upon these works, but also differs from both approaches in some key aspects; see Section~\ref{subsec:ideas} below.

\begin{remark}[Extension to the Hall-MHD system]
	In the absence of magnetic resistivity, the (incompressible) Hall-MHD system reads: 
	\begin{equation} \label{eq:hall-mhd} \tag{Hall-MHD}
		\left\{
		\begin{aligned}
			&\rd_{t} \bfB + \nb \times ((\nb \times \bfB) \times \bfB)= \nb \times (\bfu \times \bfB) ,\\
			&\rd_{t} \bfu + \bfu \cdot \nb \bfu + \nb \bfp  - \nu \lap \bfu = (\nb \times \bfB) \times \bfB,   \\
			&\nb \cdot \bfu = \nb \cdot \bfB = 0.
		\end{aligned}
		\right.
	\end{equation} Here, $\bfu, p$ are the fluid velocity field and the pressure, respectively, and $\nu\ge0$ is the kinematic viscosity. As a simple extension of Theorem \ref{thm:main-simple}, one can obtain the following result for \eqref{eq:hall-mhd}: Under the same assumptions as in Theorem \ref{thm:main-simple} for $\bfB_{0}$, if the initial velocity field $\bfu_0 : \bbR^{3} \to \bbR^{3}$ satisfies $\nb\cdot \bfu_0 = 0$ and $\bfu_0 \in \ell^{1}_{\calI} H^{s+1}$, then we have local wellposedness for \eqref{eq:hall-mhd} with initial data $(\bfB_{0}, \bfu_{0})$. We shall defer the precise statement and proof to our forthcoming work  in \cite{JO2}. 	
\end{remark}

\begin{remark} [Extension to quasilinear ultrahyperbolic Schr\"odinger equations \cite{PiTa}]\label{rem:ultrahyp}
In view of the analogies with \eqref{eq:ultrahyp}, a natural question is whether a result analogous to Theorem~\ref{thm:main-simple} can be also proved for \eqref{eq:ultrahyp} building on the ideas put forth in this paper. This is a nontrivial problem, as the conic directionality of $\dprin_{\bgB}$ significantly simplifies the analysis of \eqref{eq:e-mhd} compared to \eqref{eq:ultrahyp}. Nevertheless, the answer to this question has been answered in \emph{affirmative}; see the recent work of Pineau--Taylor \cite{PiTa}. See also Remark~\ref{rem:ultrahyp-long} below.
\end{remark}

\begin{remark}[Possible relaxations of assumptions]
Let us discuss some ways the assumptions in Theorem~\ref{thm:main-simple} may be relaxed. Some kind of nondegeneracy and nontrapping assumptions are always needed, but in view of the directionality of the bicharacteristics associated with $\dprin_{\bfB_{0}}$, we may replace $\nrm{\bfB_{0} - \bfe_{3}}_{\ell^{1}_{\calI} H^{s}} < + \infty$ by an assumption that does not require decay in the orthogonal direction to $x^{3}$. Moreover, we may let $\bfB_{0}$ tend to different nondegenerate configurations as $x^{3} \to \pm \infty$. Some of these extensions will be explored in \cite{JO2}.
\end{remark}

The remainder of the introduction is structured as follows. In Section~\ref{subsec:ideas}, we describe the main ideas of our proof of Theorem~\ref{thm:main-simple}. In Section~\ref{subsec:discussions}, we discuss the previous literature on related problems. Then in Section~\ref{subsec:outline}, an outline of the rest of the paper is provided. 

\subsection{Main ideas} \label{subsec:ideas} We now give a discussion of the main ideas of this paper.
\subsubsection{Basic setup and local smoothing estimate}
A useful way to understand the nonlinear term $\nb \times ((\nb \times \bfB) \times \bfB)$ in \eqref{eq:e-mhd} is to use paraproduct decomposition. To wit, we write \eqref{eq:e-mhd} as
\begin{equation*}
	\rd_{t} \bfB - \nb \times (T_{\bfB} \times (\nb \times \bfB)) +  \nb \times (T_{(\nb \times \bfB)} \times \bfB) = (\cdots),
\end{equation*}
where $T_{\bfA} \times \bfB$ is the \emph{(low-high) paraproduct} defined (using Littlewood--Paley projections) as 
\begin{equation*}
T_{\bfA}  \times \bfB := \sum_{k} P_{<k-10} \bfA \times P_{k} \bfB.
\end{equation*}
The point of this decomposition is that the terms we omitted on the RHS, called \emph{residual} or \emph{high-high interaction}, are smoother (more precisely, see Proposition~\ref{prop:paralin-err}). Moreover, $\nb \times (T_{(\nb \times \bfB)} \times \bfB)$ on the LHS is less dangerous than the other term, since at most one derivative falls on the high frequency factor (namely, $\bfB$ on the right). As in the previous discussion of the linearized \eqref{eq:e-mhd} around $\bfe_{3}$, the first two terms on the LHS can be diagonalized (to the leading order) by introducing $b_{\pm} := \Pi_{\pm}(D) b$ where $b = \bfB - \bfe_{3}$. We arrive at the system
\begin{equation*}
\bfL_{\bfB}^{\sharp} \begin{pmatrix} b_{+} \\ b_{-} \end{pmatrix} := \rd_{t} \begin{pmatrix} b_{+} \\ b_{-} \end{pmatrix} +\begin{pmatrix} \diag_{\bfB}^{(2) \sharp} & 0 \\ 0 & - \diag_{\bfB}^{(2) \sharp} \end{pmatrix} \begin{pmatrix} b_{+} \\ b_{-} \end{pmatrix} 
+ \begin{pmatrix} 0 &  \symm^{(1) \sharp}_{\bfB}   \\ - \symm^{(1) \sharp}_{\bfB}  & 0 \end{pmatrix} \begin{pmatrix} b_{+} \\ b_{-} \end{pmatrix}
 = (\cdots)
\end{equation*}
where $\diag_{\bfB}^{(2) \sharp}$ is a diagonal anti-symmetric paradifferential (see Section~\ref{subsec:pseudo-diff}) second-order operator and $\symm_{\bfB}^{(1) \sharp}$ is a matrix-valued symmetric paradifferential first-order operator given by
\begin{align*}
	\diag_{\bfB}^{(2) \sharp} b &:= \sum_{k} \left[ \frac{1}{2} \bfB^{\alp}_{<k-10} P_{k} \rd_{\alp} \abs{\nb} b + \frac{1}{2} \abs{\nb} \rd_{\alp} P_{k} \bfB^{\alp}_{<k-10}\right], \\
	\symm_{\bfB}^{(1) \sharp} b &:= \sum_{k} \left[ \frac{1}{2} \left(\abs{\nb} (\rd^{\alp} P_{<k-10} \bfB_{\bt}) + (\rd_{\bt} P_{<k-10} \bfB^{\alp}) \abs{\nb} \right) P_{k} b^{\bt}
+ \frac{1}{2} [\abs{\nb}, P_{<k-10} \bfB \cdot \nb] P_{k} b^{\alp}\right],
\end{align*}
and we have moved to the RHS and omitted less important terms for this heuristic discussion. See Sections~\ref{sec:reformulation} and \ref{sec:lwp-pf}, as well as the proof of Proposition~\ref{prop:para-diag-tdb}, for details.
Terms in $(\cdots)$ possibly loses $1$ derivative of $b$, so we need to prove an estimate that gains $1$ derivative for this system. This is precisely the role of the \emph{local smoothing estimate}
\begin{equation} \label{eq:ideas-led}
\begin{aligned}
	\nrm{\brk{D}^{s} \vec{b}}_{X^{0}} &:= \nrm{\brk{D}^{s+\frac{1}{2}} \vec{b}}_{LE} + \nrm{\brk{D}^{s} \vec{b}}_{L^{\infty} L^{2}[0, T]}  \\
	&\aleq \nrm{\brk{D}^{s}  \vec{b}(t=0)}_{L^{2}} + \nrm{\brk{D}^{s} \vec{g}^{(1)}}_{L^{1} L^{2}} + \nrm{\brk{D}^{s-\frac{1}{2}} \vec{g}^{(2)}}_{LE^{\ast}},
\end{aligned}\end{equation}
where $s > 0$, $\vec{b} = (b_{+}, b_{-})^{\top}$, $\bfL_{\bfB}^{\sharp} \vec{b} = \vec{g}^{(1)} + \vec{g}^{(2)}$ and $LE$ and $LE^{\ast}$ are norms defined in \eqref{eq:LE-def}--\eqref{eq:LE*-def} below. Comparing the terms involving $LE$ and $LE^{\ast}$, one may note that this estimate indeed controls $1$ more derivative of $\vec{b}$ compared to $\bfL_{\bfB}^{\sharp} \vec{b}$. Heuristically, this order-$1$ local smoothing phenomenon for \eqref{eq:e-mhd} is expected based on the nondegeneracy, asymptotic uniformity and nontrapping properties of $\bfB$ (see Theorem~\ref{thm:main-simple}); see also the discussion below. 

For the ensuing discussion, let us also introduce the (non-paradifferential) linear operator
\begin{equation*}
	\bfL_{\bfB} \begin{pmatrix} b_{+} \\ b_{-} \end{pmatrix} := \rd_{t} \begin{pmatrix} b_{+} \\ b_{-} \end{pmatrix} +\begin{pmatrix} \diag_{\bfB}^{(2)} & 0 \\ 0 & - \diag_{\bfB}^{(2)} \end{pmatrix} \begin{pmatrix} b_{+} \\ b_{-} \end{pmatrix} 
+ \begin{pmatrix} 0 &  \symm^{(1)}_{\bfB}   \\ - \symm^{(1)}_{\bfB}  & 0 \end{pmatrix} \begin{pmatrix} b_{+} \\ b_{-} \end{pmatrix}
\end{equation*}
where
\begin{align*}
	\diag_{\bfB}^{(2)} b &:= \frac{1}{2} \bfB^{\alp} \rd_{\alp} \abs{\nb} b + \frac{1}{2} \abs{\nb} \rd_{\alp} \bfB^{\alp}, \\
	\symm_{\bfB}^{(1)} b &:= \frac{1}{2} \left(\abs{\nb} (\rd^{\alp} \bfB_{\bt}) + (\rd_{\bt} \bfB^{\alp}) \abs{\nb} \right) b^{\bt}
+ \frac{1}{2} [\abs{\nb}, \bfB \cdot \nb] b^{\alp},
\end{align*}
which is a part of the linearization of \eqref{eq:e-mhd} around $\bfB$ (see Section~\ref{sec:reformulation}).

\subsubsection{Previous approaches}
Having set up a more concrete setting for a heuristic discussion, let us now review the previous approaches to local wellposedness of quasilinear dispersive equations. Indeed, in the pioneering papers of Kenig--Ponce--Vega \cite{KPV1, KPV2, KPV3} and Kenig--Ponce--Rolvung--Vega \cite{KPRV1, KPRV2}, which proved local wellposedness of quasilinear Schr\"odinger equations with elliptic and ultrahyperbolic principal terms, respectively, the main idea was to establish local smoothing estimates for suitable class of variable coefficient linear dispersive equations (see also \cite{LP, Chi, HaOz} for earlier works). Translated to our context\footnote{In particular, $b_{+}$ and $b_{-}$ in our context are analogous to $u$ and $\bar{u}$, respectively, in \cite{KPRV1}. The symbol for the renormalization operator -- referred to as an ``integrating factor'' --  is denoted by $K$.}, the class of equations they considered consists of non-paradifferential operators $\bfL_{\bfB_{0}}$ around a time-independent magnetic field $\bfB_{0}$ -- we note that $\bfB_{0}$ may be taken to be the initial magnetic field (hence our notation), since we expect $\bfB$ to stay close to the initial data (in an adequate sense) on the time interval under consideration. The proof of local smoothing estimate for $\bfL_{\bfB_{0}}$ according to \cite{KPRV1} (the ultrahyperbolic case, which is closer to \eqref{eq:e-mhd}) would proceed as follows:

\smallskip
\noindent {\it Step~1: Local smoothing assuming boundedness of energy.} Establish \eqref{eq:ideas-led} for $\bfL_{\bfB_{0}} \vec{b} = \vec{g}^{(1)} + \vec{g}^{(2)}$ assuming
\begin{equation} \label{eq:ideas-en}
\nrm{\vec{b}}_{L^{\infty} H^{s}[0, T]} \aleq \nrm{\vec{b}(t=0)}_{H^{s}} + \nrm{\vec{g}^{(1)}}_{L^{1} H^{s}} + \nrm{\brk{D}^{s-\frac{1}{2}} \vec{g}^{(2)}}_{LE^{\ast}},
\end{equation}

\smallskip
\noindent {\it Step~2: Boundedness of energy.} Establish \eqref{eq:ideas-en} for $\bfL_{\bfB_{0}} \vec{b} = \vec{g}^{(1)} + \vec{g}^{(2)}$.

\smallskip
The first step involves a \emph{positive commutator argument}, based on the geometric observation that every bicharacteristic $(X(t), \Xi(t))$ associated with $\bfB_{0} \cdot \xi \abs{\xi}$ initially in a slab of the form $\set{\alp < x^{3} < \bt}$ exits the slab at time at most $O(\frac{\bt - \alp}{\abs{\Xi(0)}})$ (during which $\abs{\Xi(t)}$ stays comparable to $\Xi(0)$) under suitable nondegeneracy, asymptotic uniformity (i.e., decay) and nontrapping assumptions (as in Theorem~\ref{thm:main-simple}). For the second step, we start by writing out the equations for $\tb_{\pm} := \brk{D}^{s} b_{\pm}$:
\begin{align*}
	\rd_{t} \begin{pmatrix} \tb_{+} \\ \tb_{-} \end{pmatrix} + \left[ \begin{pmatrix} \diag_{\bfB_{0}}^{(2)} & 0 \\ 0 & - \diag_{\bfB_{0}}^{(2)} \end{pmatrix} 
	+ \begin{pmatrix} \comm^{s (1)}_{\bfB_{0}} & 0 \\ 0 & - \comm^{s (1)}_{\bfB_{0}} \end{pmatrix} 
+ \begin{pmatrix} 0 &  \symm^{(1)}_{\bfB_{0}}   \\ - \symm^{(1)}_{\bfB_{0}}  & 0 \end{pmatrix} \right] \begin{pmatrix} \tb_{+} \\ \tb_{-} \end{pmatrix}
= \brk{D}^{s} \bfL_{\bfB} \begin{pmatrix} b_{+} \\ b_{-} \end{pmatrix} + (\cdots),
\end{align*}
where $\comm^{s (1)}_{\bfB_{0}} := [\brk{D}^{s}, \diag_{\bfB_{0}}^{(2)}] \brk{D}^{-s}$ is a diagonal \emph{symmetric} first-order operator arises from commuting $\brk{D}^{s}$ with $P_{\bfB_{0}}^{(2)}$. Then we control the time derivative of the $H^{s}$-energy $\nrm{\tb_{+}}_{L^{2}}^{2} + \nrm{\tb_{-}}_{L^{2}}^{2}$ by multiplying the above equation on the left by $(\tb_{+}, \tb_{-})$ and integrating on $\bbR^{3}$. The key difficulty here is that {\it $\comm^{s(1)}_{\bfB_{0}}$ contributes a term involving one derivative of $\tb_{\pm}$ without smallness} (if $\bfB_{0} - \bfe_{3}$ is large), and hence which cannot be directly  handled with local smoothing in Step~1 without circular reasoning.

The idea of \cite{KPRV1} (a formal version of this idea goes back to \cite{CrKaSt, Doi, HaOz}) was to remove the problematic term $\pm \comm^{s(1)}_{\bfB_{0}} \tb_{\pm}$ by working with a conjugation $\Op(O) \tb_{\pm}$, where $\Op(O)$ is a suitable pseudodifferential operator (see Section~\ref{subsec:pseudo-diff} for the notation). A short formal computation tells us that the symbol of $O$ needs to satisfy
\begin{equation} \label{eq:ideas-O-eq}
\mathrm{H}_{\bfB_{0}} O(x, \xi) = \frkc^{s (1)}_{\bfB_{0}}(x, \xi) O(x, \xi),
\end{equation}
where $\mathrm{H}_{\bfB_{0}}$ is the vector field on $\bbR^{3}_{x} \times \bbR^{3}_{x}$ (geometrically, $T^{\ast} \bbR^{3}$) associated with the Hamiltonian flow of $\bfB_{0}(x) \cdot \xi \abs{\xi}$ and $\frkc^{s (1)}_{\bfB_{0}}(x, \xi)$ is the principal symbol of $\comm^{s (1)}_{\bfB_{0}}$. The good news is that such a symbol $O(x, \xi)$ is uniformly bounded under the assumptions on $\bfB_{0}$ in Theorem~\ref{thm:main-simple}. The bad news, however, is that the derivatives of $O(x, \xi)$ behave badly. In fact, $O(x, \xi)$ does \emph{not} belong to the classical symbol class $S^{0}$, but rather only obeys
\begin{equation*}
	\abs{\rd_{x}^{\bfalp} \rd_{\xi}^{\bfbt} O(x, \xi)} \aleq_{\bfalp, \bfbt} \brk{x}^{\abs{\bfbt}} \abs{\xi}^{-\abs{\bfbt}},
\end{equation*}
which is a nonstandard symbol class considered by Craig--Kappeler--Strauss \cite{CrKaSt}. In particular, it is known that the above assumptions are \emph{not} sufficient to guarantee the boundedness of the operator $\Op(O)$ in $L^{2}$. Kenig--Ponce--Vega--Rolvung \cite{KPRV1} show, nevertheless, that $O(x, \xi)$, after suitable modifications in the construction, belongs to a refined (but still nonstandard) class of symbols for which $L^{2}$-boundedness of $\Op(O)$ and a suitable version of symbolic calculus hold. The operator $\Op(O)$ is then used to carry out the proof of boundedness of energy, thereby completing the proof of local smoothing. The argument in \cite{KPRV1} requires strong regularity and decay assumptions on $\bfB_{0}$ (in particular, compared to Theorem~\ref{thm:main-simple}).

Recently, Marzuola--Metcalfe--Tataru \cite{MMT1, MMT2, MMT3} revisited the local wellposedness problem for quasilinear Schr\"odinger equations and greatly improved the hypotheses on regularity and decay of initial data. Some of the ideas introduced in \cite{MMT1, MMT2, MMT3} are: working with a paralinearization instead of the usual linearization of the equation (i.e., $\bfL_{\bfB}^{\sharp}$ instead of $\bfL_{\bfB_{0}}$), introduction of a function space framework that encodes the necessary decay in a translation-invariant manner (see $\ell^{1}_{\calI} H^{s}$) and a new proof of local smoothing estimate that, in particular, does not involve the renormalization operators $\Op(O)$. However, the method of \cite{MMT3} (the large data case) cannot be directly applied to \eqref{eq:e-mhd}, as it uses the ellipticity of the principal term in an important manner.

\subsubsection{New ideas in this work}
In this work, we introduce a new approach that builds upon -- but is distinct from -- both \cite{KPRV1} and \cite{MMT1}. A key new ingredient in our proof is a simple observation, which is of interest on its own, that we dub the \emph{high frequency Calder\'on--Vaillancourt trick} (see Lemma~\ref{lem:hf-L2} for details): 
\begin{quote}
{\it For $a \in S^{0}$ (classical symbol class), the $L^{2} \to L^{2}$-operator norm of $Op(a)$ restricted to inputs supported in high frequency (i.e., $\nrm{Op(a) P_{>k}(D)}_{L^{2} \to L^{2}}$) is bounded by $\nrm{a}_{L^{\infty}_{x, \xi}}$, provided that the frequency threshold $k$ is large enough (depending on higher order bounds for $a$). }
\end{quote}
This trick is useful since, in the problem of local wellposedness, a low frequency part of the solution is easy to deal with using the energy estimate and the smallness of the time interval. 

Our approach for proving local smoothing for $\bfL_{\bfB_{0}}$ follows the basic roadmap of \cite{KPRV1} sketched above, but using the high frequency Calder\'on--Vaillancourt trick we get enough quantitative precision to perform a \emph{physical space localization} of the renormalization operator so that (i)~it lies in the \emph{classical} symbol class $S^{0}$ instead of an nonstandard class as in \cite{KPRV1}, and (ii)~weaker regularity and decay assumptions for $\bfB_{0}$ as in Theorem~\ref{thm:main-simple} are sufficient. Concretely, our localized renormalization operator $O_{\pm}$ for $\tb_{\pm}$ takes the form
\begin{align*}
	O_{\pm}(x, \xi) = \begin{cases} O(x, \xi) & \abs{x^{3}} \aleq R, \\ \exp(\pm C_{O}) & x^{3} \ageq R \\ \exp(\mp C_{O}) & - x^{3} \ageq R,\end{cases},
\end{align*}
where $C_{O} > 0$ is some constant (independent of $R$); $R$ is chosen sufficiently large so that $\bfB_{0}$ is small in $\set{x^{3} \ageq R}$, which we quantify by $\eps$; $O$ solves \eqref{eq:ideas-O-eq}; and $O_{\pm}(x, \xi)$ is chosen to be (suitably) monotonic along bicharacteristics in the transition regions $\abs{x^{3}} \aeq R$, which is crucial for the localization to be admissible (which, in turn, is why we let $O_{\pm}$ depend on $\pm$). The error resulting from this localization is bounded by $C \eps \nrm{O_{\pm} \tb_{\pm}}_{X^{0}}$, which we must absorb using the local smoothing estimate in Step~1.

Since $O_{\pm}(x, \xi)$ is constant in $\abs{x^{3}} \ageq R$, we have $O_{\pm}(x, \xi) \in S^{0}$. But the symbol bounds depend on $R$, so the operator norm of $\Op(O_{\pm})$ will still depend on $R$\footnote{To make matters worse, in reality the derivatives of $O_{\pm}$ also depend on the maximal length $L$ of bicharacteristics inside $\set{\abs{x^{3}} \aleq R}$ (see Definition~\ref{def:nontrapping-class}). But this issue is resolved similarly to the $R$-dependence issue discussed here via the high frequency Calder\'on--Vaillancourt trick.}. At this point, one may worry about running into a circular reasoning in the choice of $\eps$ and $R$ unless $R$ is quantitatively related to $\eps$ (e.g., there is a strong decay assumption on $\bfB_{0}$ analogous to \cite{KPRV1}). Our key observation, however, is that $\nrm{O_{\pm}}_{L^{\infty}}$ is bounded \emph{independently} of $\eps$ and $R$ (essentially since $O$ is) and hence, by the high frequency Calder\'on--Vaillancourt trick (more precisely, its strengthening in Proposition~\ref{prop:psdo-ests}), $\nrm{\Op(O_{\pm}) P_{>k_{(1)}}}_{X^{0} \to X^{0}}$ is also bounded \emph{independently} of $R$ for $k_{(1)}$ sufficiently large. As remarked above, the contribution of the low frequencies $P_{<k_{(1)}}$ can be treated via simpler energy estimates by choosing the time $T$ small\footnote{The frequency threshold $k_{(1)}$ depends on $R$, but it is compensated by choosing $T$ small at the end of the argument.}. This observation breaks the possible circularity and allows us to prove boundedness of energy with only a translation-invariant decay assumption on $\bfB_{0}$. Moreover, as we avoid introducing any nonstandard symbol classes, our renormalization argument is significantly simpler than \cite{KPRV1}. See Section~\ref{subsec:paralin-bdd} for the detailed argument\footnote{Although, pedantically, Section~\ref{subsec:paralin-bdd} considers the paralinearized system, which is more difficult than the simplified problem discussed here.}.

In fact, we are able to carry out our argument (essentially) in the context of \emph{paralinearized operator} $\bfL_{\bfB}^{\sharp}$ (see Section~\ref{sec:mag-lin} for the precise setting) and in the function space framework analogous to \cite{MMT1, MMT2, MMT3} (see Section~\ref{subsec:ftn-sp} and Section~\ref{sec:multi-ests}). As a result, we optimize the regularity and decay assumptions on the initial data to the same level as \cite{MMT3}.

\begin{remark} [Extension to quasilinear ultrahyperbolic Schr\"odinger equations \cite{PiTa}]\label{rem:ultrahyp-long}
The conic directionality of \eqref{eq:e-mhd} simplifies our argument, as evinced by the form of $O_{\pm}(x, \xi)$ above (in particular, it suffices to localize the symbol to the slab $\{ |x^{3}| \aleq R \}$ independent of $\xi$). In a recent work \cite{PiTa} of Pineau--Taylor, the strategy of physical space localization via high frequency Calder\'on--Vaillancourt trick has been extended to the quasilinear ultrahyperbolic Schr\"odinger equations \eqref{eq:ultrahyp}, where the conic directionality no longer holds and further new ideas are required.
\end{remark}

\subsection{Discussions} \label{subsec:discussions} 


\noindent \textbf{Whistler waves versus Alfv\'en waves}. We make a comparison between Whistler and Alfv\'en waves, where the latter is the basic wave for the usual ideal MHD system. The ideal MHD system linearized around $(0,\bfe_{3})$ reads \begin{equation*}
	\begin{split}
		\rd_{t} \bfb + \bfe_{3} \cdot \nb \bfu = 0 , \qquad \rd_{t} \bfu = \bfe_{3} \cdot \nb \bfb  , 
	\end{split}
\end{equation*} and this can be written as \begin{equation*}
	\begin{split}
		\rd_{tt} \bfb + (\bfe_{3}\cdot\nb)^{2} \bfb = 0,
	\end{split}
\end{equation*} which is nothing but the (1+1)-dimensional wave equation in $(t,x^{3})$. Therefore, we see that the Alfv\'en waves travel along $\bfB$ with speed $\abs{\bfB}$, while Whistler waves travels within a cone around $\bfB$ with speed proportional to $\abs{\bfB} \abs{\xi}$. It is this property of Whistler waves that is responsible for various aforementioned phenomena observed in plasmas. 
		
		\medskip 
		
		\noindent \textbf{Mathematical literature on Hall and electron MHD.} Regarding Hall and electron MHD equations with magnetic resistivity, local well-posedness and temporal decay was obtained in pioneering works \cite{CDL} and \cite{CSch}, respectively. Most of the subsequent mathematical literature dealt with the resistive case as well (\cite{CL,Dai20,DaiLiu,CWo1,Ya,Ze,WaZh1,WaZh2,DKL,Dai21,Dai212}), and the simpler case of $2\frac12$-dimensional solutions have been widely studied ( \cite{BaeKang,RaYa,Ya,CWo2,Du1,Dai3}), not only by mathematicians but also by physicists. This means that one considers magnetic fields of the form \begin{equation}\label{eq:twoandhalfd}
			\begin{split}
				\bfB = \nb \times (\psi \, \bfe_{3}) + \phi \, \bfe_{3}, 
			\end{split}
		\end{equation} where $\psi,\phi$ are scalar valued functions of $t, x^{1}, x^{2}$. In this case, \eqref{eq:e-mhd} reduces to \begin{equation}\label{eq:e-mhd-two-and-half}
		\left\{
		\begin{aligned}
			\rd_{t}\psi + \nb^\perp \phi \cdot \nb\psi = 0,  & \\
			\rd_{t} \phi - \nb^\perp \psi \cdot \nb \lap \psi = 0. & 
		\end{aligned}
		\right.
		\end{equation} (Here, $\nb^\perp = (\rd_{x^{1}}, \rd_{x^{2}})$ and $\nb^\perp = (-\rd_{x^{2}}, \rd_{x^{1}})$. In the presence of the velocity field, a similar ansatz can be given for it as well.) Recently, Dai \cite{Dai3} added a $\lap\phi$ term to the RHS of the $\phi$ equation in \eqref{eq:e-mhd-two-and-half} and obtained global wellposedness of smooth solutions near the uniform magnetic field $\bfB = \bfe_{1}$. In the same paper, it is conjectured that global regularity persists when \eqref{eq:e-mhd-two-and-half} has a resistivity term only in the $\psi$ equation. Moreover, we note that dyadic models for Hall- and electron-MHD systems have been studied in \cite{Dai-IMRN} and \cite{DaiFr}. 
		
		In the irresistive case, Chae--Weng \cite{CWe} presented smooth axisymmetric solutions of Hall- and electron-MHD equations which blow up in finite time (see also \cite{DRG,FHN,JKL}).\footnote{This was done without having \textit{uniqueness} corresponding to the initial data, which is still an open problem.} However, for general smooth initial data, strong illposedness statements including \textit{nonexistence} of smooth solutions to Hall- and electron-MHD were proved in \cite{JO1}, by considering degenerate magnetic fields under the $2\frac12$-dimensional setup. That is, for certain $C^{\infty}$-smooth and decaying initial data to \eqref{eq:e-mhd-two-and-half}, there cannot be no associated smooth solution in any arbitrarily short interval of time. Similar illposedness statements for smooth data compactly supported in $\bbR^{3}$ for \eqref{eq:e-mhd} and \eqref{eq:hall-mhd} will be handled in our forthcoming work \cite{JO3}.

\medskip 
		
		\noindent \textbf{Degenerate quasilinear dispersive PDE.} 
		The nondegeneracy condition of the principal term was assumed in the aforementioned works on quasilinear dispersive equations. (This is the condition (1) in the statement of Theorem \ref{thm:main-simple} in our case.)  When the principal term has a point of degeneracy, analysis of the bicharacteristic curves near such a point shows a very rapid growth of the frequency (for instance, see \cite{GHGM,JO4,Ak,AmWr,ASWY}) in general. An argument involving generalized energy identities and testing against wavepackets introduced in \cite{JO1} provides a robust way of establishing norm growth of the solutions to quasilinear degenerate dispersive PDE corresponding to the behavior of bicharacteristics. This argument was then generalized to give strong illposedness for various families of degenerate quasilinear dispersive PDE in \cite{JO4,CJO}, including degenerate versions of KdV and Schr\"odinger equations.  For some of these equations, local wellposedness (not in Sobolev or H\"older-type spaces, but rather in a space adapted to the degeneracy) for certain classes of degenerate initial data has been obtained in \cite{GHGM,HGM} based on the Lagrangian approach.  Motivations for studying degenerate dispersive equations can be found for instance in \cite{GHGM,RoHy,Ros05,RoZi15,RoSc}.

\subsection{Structure of the paper} \label{subsec:outline}
In {\bf Section~2}, we collect some preliminaries for the whole paper, including notation and conventions (Section~\ref{subsec:notation}), a precise formulation of the function space framework (Section~\ref{subsec:ftn-sp}), review of pseudodifferential and paradifferential operators (Sections~\ref{subsec:pseudo-diff}), and the high frequency Calder\'on--Vaillancourt trick (Section~\ref{subsec:cv}). In {\bf Section~\ref{sec:reformulation}}, we record some core computations of this paper, namely a reformulation of \eqref{eq:e-mhd} as an equation for a perturbation $b$ of a background magnetic field and also the introduction of the variables $b_{\pm}$, with respect to which the principal term in the linearized E-MHD is diagonalized. In {\bf Section~\ref{sec:prin-symbol}}, we study the properties of the Hamiltonian flow associated with the principal symbol of the linearized system after diagonalization. Moreover, a quantitative formulation of the assumptions in Theorem~\ref{thm:main-simple} is given in this section (Definition~\ref{def:nontrapping-class}). In {\bf Section~\ref{sec:lwp-pf}}, we give a precise version of the main theorem (Theorem~\ref{thm:main}) and reduce its proof to a local smoothing estimate for the paralinearized system (Proposition~\ref{prop:paralin-full}) and nonlinear estimates for the remainder (Proposition~\ref{prop:paralin-err}). Then, after proving some relevant linear and nonlinear estimates in our function space framework in {\bf Section~\ref{sec:multi-ests}}, we prove Propositions~\ref{prop:paralin-full} and \ref{prop:paralin-err} in {\bf Section~\ref{sec:mag-lin}}, which is the technical heart of our paper.
\ackn{
I.-J.~Jeong was supported by the Samsung Science and Technology Foundation under Project Number SSTF-BA2002-04. S.-J.~Oh was partially supported by a Sloan Research Fellowship and a National Science Foundation CAREER Grant under NSF-DMS-1945615.}

\section{Preliminaries} \label{sec:prelim}

\subsection{Notation and conventions} \label{subsec:notation}
We use the following notation in the paper.
\begin{itemize}
\item $\bbZ_{\geq 0}$ denotes the set of nonnegative integers.
\item $\brk{\cdot} = (1 + (\cdot)^{2})^{\frac{1}{2}}$. $(\cdot)_{+} = \max\set{\cdot, 0}$.
\item $\chf_{E}$: characteristic function of $E$
\item $\chi_{>1}(\cdot), \chi_{>R}(\cdot)$: $\chi_{>1}(\cdot)$ is a smooth even function such that $\chi_{>1}(s) = 1$ for $\abs{s} > 1$, $\chi_{>1}(s) = 0$ for $\abs{s} < \frac{1}{2}$ and nondecreasing on $(0, \infty)$. $\chi_{>R}(\cdot) = \chi_{> 1}(\tfrac{\cdot}{R})$
\item $\chi_{<1}(\cdot), \chi_{<R}(\cdot)$: $\chi_{<1}(\cdot)$ is a smooth even function such that $\chi_{<1}(s) = 1$ for $\abs{s} < 1$, $\chi_{<1}(s) = 0$ for $\abs{s} > 2$ and nonincreasing on $(0, \infty)$. $\chi_{<R}(\cdot) = \chi_{< 1}(\tfrac{\cdot}{R})$
\end{itemize}
\subsubsection*{Vector calculus and index notation}
\begin{itemize}
\item We use the vector calculus notation
\begin{equation*}
	\nb f = \begin{pmatrix}
	\rd_{1} f \\ \rd_{2} f \\ \rd_{3} f
\end{pmatrix}, \quad 
	\nb^{\top} f = \begin{pmatrix}
	\rd_{1} f & \rd_{2} f & \rd_{3} f 
\end{pmatrix}, \quad
	\nb \times \bfu = \begin{pmatrix}
	\rd_{2} \bfu_{3} - \rd_{3} \bfu^{2} \\ - \rd_{1} \bfu^{3} + \rd_{3} \bfu^{1} \\ \rd_{1} \bfu^{2} - \rd_{2} \bfu^{1}
	\end{pmatrix}.
\end{equation*}
\item We also employ the index notation, where the plain greek letters $\alp, \bt, \gmm, \ldots$ are used for vector indices. As usual, repeated upper and lower indices are always summed (unless otherwise specified), and we raise and lower indices using the Kronecker delta $\tensor{\dlt}{^{\alp}_{\bt}}$ (which is $1$ if $\alp = \bt$ and $0$ otherwise). For computation, it will sometimes be useful to express a usual vector calculus operation in the index notation. For instance, the curl is expressed as
\begin{equation*}
	(\nb \times \bfu)^{\alp} = (\nb \times)^{\alp}_{\bt} \bfu^{\bt} = \tensor{\eps}{^{\alp \gmm}_{\bt}} \rd_{\gmm} \bfu^{\bt},
\end{equation*}
where $\tensor{\eps}{_{\alp \bt \gmm}}$ is the Levi-Civita symbol (totally antisymmetric with $\eps_{1 2 3} = 1$).
\item We use the boldfaced greek letters $\bfalp, \bfbt, \bfgmm, \ldots$ for multi-indices.
\end{itemize}

\subsubsection*{Fourier analysis and Littlewood--Paley theory}
\begin{itemize}
\item Given a function $m:\bbR^{d}\to\bbR$, we write $m(D)$ for the Fourier multiplier operator with multiplier $m(\xi)$; $D = i^{-1} \nb$. For $\sgm\in\bbR$, we write $\abs{\nb}^{\sgm} = (-\lap)^{\frac{\sgm}{2}}$.
\item In this paper, we consider the inhomogeneous Littlewood--Paley projections with respect to the spatial frequency. Fix some $\phi_{0}:\bbR_+\to\bbR$ be a non-negative, decreasing, smooth function supported on $[0,2]$, which satisfies $\phi_{0}(\xi)=1$ for $|\xi|\le1$. For $k\ge1$, set $\phi_{k}(\xi) = \phi_{0}(2^{-k}\xi) - \phi_{0}(2^{-k+1}\xi)$, and then we define $P_{k}$ to be the Fourier multiplier operator with multiplier $\phi_{k}(|\xi|)$ for each $k\ge0$. We also write  $P_{<k} = \sum_{k' < k} P_{k'}$. $P_{[k, k')} = P_{<k'} - P_{<k}$. Since we use the inhomogeneous projections, it is assumed that the summation with respect to $k, k', \cdots$ is only for nonnegative integers. 

\end{itemize}

\subsection{Function spaces} \label{subsec:ftn-sp}
Let $J$ be a time interval. Recall that $\calI_{\ell}$ refers to a partition of $\bbR$ into intervals of size $2^{\ell}$, and that $\set{\chi_{I}}_{I \in \calI_{\ell}}$ is a smooth partition of unity such that $\supp \chi_{I} \subseteq 2I$. Note that for $I, I' \in \calI_{\ell}$, either $I \cap I' \neq \0$ or $\dist(I, I') \ageq 2^{\ell}$.

The local smoothing (or local energy decay) norms are:
\begin{equation} \label{eq:LE-def}
\nrm{b}_{LE[J]} = \sup_{\ell \in \bbZ_{\geq 0}} \sup_{I \in \calI_{\ell}} 2^{-\frac{\ell}{2}} \nrm{u}_{L^{2} L^{2}(J \times \set{x^{3} \in I})}
\end{equation}

\begin{equation} \label{eq:LE*-def}
\nrm{g}_{LE^{\ast}[J]} = \inf_{g = \sum_{\ell \in \bbZ_{\geq 0}} g_{\ell}} \sum_{\ell \in \bbZ_{\geq 0}} \sum_{I \in \calI_{\ell}} 2^{\frac{\ell}{2}} \nrm{g_{\ell}}_{L^{2} L^{2}(J \times \set{x^{3} \in I})}
\end{equation}

Then for each $k \in \bbZ_{\geq 0}$, we define the dyadic norms (i.e., a family of norms indexed by $k \in \bbZ_{\geq 0}$) by
\begin{equation} \label{eq:Xk-def}
\nrm{b}_{X_{k}[J]} = 2^{\frac{k}{2}} \nrm{b}_{LE[J]} + \nrm{b}_{L^{\infty} L^{2}[J]}
\end{equation}
\begin{equation}\label{eq:Yk-def}
\nrm{g}_{Y_{k}[J]} = \inf_{g = g_{1} + g_{2}} \left( 2^{-\frac{k}{2}} \nrm{g_{1}}_{LE^{\ast}[J]} + \nrm{g_{2}}_{L^{1} L^{2}[J]} \right).
\end{equation}
Following \cite{MMT1}, given any dyadic norm $\calX_{k}$ and $r \in [1, \infty]$, we define 
\begin{align}
\nrm{b}_{\ell^{r}_{\calI} \calX_{k}[J]}
&= \left(\sum_{I \in \calI_{k}} \nrm{\chi_{I}(x^{3}) b}_{\calX_{k}[J]}^{r} \right)^{\frac{1}{r}},  \label{eq:ellXk-def}
\end{align}
and for $r = \infty$, we replace the summation in $I \in \calI_{k}$ by the supremum as usual. 

These dyadic spaces are slowly varying in the following sense:
\begin{lemma}[Slow variance of $X_{k}$ and $Y_{k}$] \label{lem:XY-slow-var}
For any $k, k' \in \bbZ_{\geq 0}$ and $1 \leq r \leq \infty$, we have
\begin{align} 
\nrm{b}_{X_{k}}
&\aleq 2^{\frac{1}{2} \abs{k-k'}} \nrm{b}_{X_{k'}},  \label{eq:X-slow-var} \\
\nrm{b}_{\ell^{r}_{\calI} X_{k}}
&\aleq 2^{\frac{3}{2}\abs{k-k'}} \nrm{b}_{\ell^{r}_{\calI} X_{k'}},  \label{eq:X-ell-slow-var} \\
\nrm{g}_{Y_{k}}
&\aleq 2^{\frac{1}{2} \abs{k-k'}} \nrm{g}_{Y_{k'}},  \label{eq:Y-slow-var} \\
\nrm{g}_{\ell^{r}_{\calI} Y_{k}}
&\aleq 2^{\frac{3}{2}\abs{k-k'}} \nrm{g}_{\ell^{r}_{\calI} Y_{k'}}.  \label{eq:Y-ell-slow-var} 
\end{align}
\end{lemma}
Moreover, the following embedding properties hold:
\begin{lemma} \label{lem:XY-ell}
For any $k \in \bbZ_{\geq 0}$, we have
\begin{align} 
	& \nrm{b}_{\ell^{\infty}_{\calI} X_{k}} \aleq \nrm{b}_{X_{k}} \aleq \nrm{b}_{\ell^{1}_{\calI} X_{k}}, \label{eq:X-ell} \\
	& \nrm{g}_{\ell^{\infty}_{\calI} Y_{k}} \aleq \nrm{g}_{Y_{k}} \aleq \nrm{g}_{\ell^{1}_{\calI} Y_{k}}. \label{eq:Y-ell}
\end{align}
\end{lemma}
We omit the proofs of Lemmas~\ref{lem:XY-slow-var} and \ref{lem:XY-ell}, which follow rather immediately from the definitions.

For $s \in \bbR$, we define the summed-up spaces by
\begin{align*}
\nrm{b}_{X^{s}[J]}^{2} &= \sum_{k} (2^{sk} \nrm{P_{k} b}_{X_{k}[J]})^{2}, \\
\nrm{g}_{Y^{s}[J]}^{2} &= \sum_{k} (2^{sk} \nrm{P_{k} g}_{Y_{k}[J]})^{2}, \\
\nrm{b}_{\ell^{r}_{\calI} X^{s}[J]}^{2} &= \sum_{k} (2^{sk} \nrm{P_{k} b}_{\ell^{r}_{\calI} X_{k}[J]})^{2}, \\
\nrm{g}_{\ell^{r}_{\calI} Y^{s}[J]}^{2} &= \sum_{k} (2^{sk} \nrm{P_{k} g}_{\ell^{r}_{\calI} Y_{k}[J]})^{2},
\end{align*}
and so on. For $s > \frac{7}{2}$ and $r \in [1, \infty]$, observe that the following string of embeddings hold:
\begin{align*}
	\ell^{r} X^{s}[J] \hookrightarrow \ell^{r} L^{\infty} H^{s}[J] \hookrightarrow \ell^{r} C_{t} C^{2, s-\frac{7}{2}}_{x}[J] \hookrightarrow C_{t} C^{2, s-\frac{7}{2}}_{x}[J],
\end{align*}
where the third embedding is Morrey's inequality. {Observe also that the following bound holds for any $s \in \bbR$ and $r \in [1, \infty]$:
\begin{equation} \label{eq:Xs-Hs+1/2}
	\nrm{b}_{\ell^{r}_{\calI} X^{s}[J]} \aleq \nrm{b}_{\ell^{r}_{\calI} L^{\infty} H^{s}[J]} + \abs{J}^{\frac{1}{2}} \nrm{b}_{\ell^{r}_{\calI} L^{\infty} H^{s+\frac{1}{2}}[J]}.
\end{equation}
We omit the proofs of these straightforward bounds.}

\subsection{Pseudo- and paradifferential operators} \label{subsec:pseudo-diff}
In this subsection, we review the basic theory of \emph{pseudo-} and \emph{paradifferential operators} that will be used in this paper.

\subsubsection{Classical symbols and quantizations}
Let $\frka(x, \xi)$ be a complex-valued smooth function on $\bbR^{d} \times \bbR^{d}$. We say that $\frka$ is a (scalar-valued) \emph{classical symbol of order $m$} $(m \in \bbR)$, and write $\frka \in S^{m}$, if
\begin{equation} \label{eq:symb-cl}
	[\frka]_{S^{m}; N} := \sum_{\bfalp, \bfbt : \abs{\bfalp} + \abs{\bfbt} \leq N} \sup_{x, \xi} \abs{\brk{\xi}^{\abs{\bfbt}  {- m} } \rd_{x}^{\bfalp} \rd_{\xi}^{\bfbt} \frka(x, \xi)} < + \infty \hbox{ for every } N \geq 0.
\end{equation}  {We also write $S^{-\infty} = \cap_{m \in \bbR} S^{m}$.}
Given a symbol $\frka(x, \xi) \in S^{m}$, the corresponding \emph{left} (resp.~\emph{right}) \emph{quantization} is formally\footnote{The formal expressions of the form $\iint (\cdots) e^{i \xi \cdot (x-y)} \, \ud y \ud \xi$ in this section may be made precise by multiplying the integrand by $m_{\leq 0}(\eps x) m_{\leq 0}(\eps \xi)$ and taking $\eps \to 0$, thanks to the regularity bounds on the integrand originating from the classical symbol bounds. We omit the standard details.} defined as
\begin{align*}
	\Op^{(l)}(\frka(x, \xi)) u  = \frka(x, D) u &= \iint \frka(x, \xi) u(y) e^{i \xi \cdot (x - y)} \, \ud y \frac{\ud \xi}{(2 \pi)^{d}}, \\
	\bb(\hbox{resp.}~ \Op^{(r)}(\frka(x, \xi)) u  = \frka(D, x) u &= \iint \frka(y, \xi) u(y) e^{i \xi \cdot (x - y)} \, \ud y \frac{\ud \xi}{(2 \pi)^{d}}\bb)
\end{align*}
for $u \in \calS(\bbR^{d})$. Since we will mainly use left quantization, and often simply write $\Op(\frka(x, \xi)) = \Op^{(l)}(\frka(x, \xi))$. Note that if $a(x, \xi)$ splits into $a(x, \xi) = c(x) m(\xi)$, then $c(x)$ is multiplied to the left (resp.~right) of $m(D)$ under left (resp.~right) quantization, i.e.,
\begin{equation} \label{eq:psdo-op-product}
	\Op(\frka(x, \xi)) = \Op^{(l)}(\frka(x, \xi)) = c(x) m(D), \quad
	\Op^{(r)}(\frka(x, \xi) = m(D) c(x).
\end{equation}

The following basic $L^{2}$ boundedness and G{\aa}rding's inequality are standard:
\begin{proposition} \label{prop:psdo-L2-bdd-garding}
Let $\frka \in S^{0}$.
\begin{enumerate}
\item ($L^{2}$-boundedness) We have
\begin{equation*}
	\nrm{\frka(x, D)}_{L^{2} \to L^{2}} \aleq [\frka]_{S^{0}; 10 d}.
\end{equation*}
\item (G{\aa}rding's inequality) If, in addition, $\frka$ is real-valued and $\frka \geq 0$, then
\begin{equation*}
	\brk{\frka(x, D) u, u} \geq - C [\frka]_{S^{0}; 20 d} \nrm{u}_{H^{-\frac{1}{2}}}^{2}.
\end{equation*}
\end{enumerate}
The same result holds for $\frka(D, x)$. 
\end{proposition}
For a proof, see \cite[Thm.~0.5.C]{Tay} and \cite[Prop.~0.7.A]{Tay}.
However, in this paper (especially in Section~\ref{sec:mag-lin}), we will need a more precise understanding of the dependence of $\nrm{\frka(x, D)}_{L^{2} \to L^{2}}$ on the symbol bounds. For this purpose, we will rely on the Calder\'on--Vaillancourt theorem, which is a more powerful $L^{2}$-boundedness theorem; see Section~\ref{subsec:cv} below.

\subsubsection{Paradifferential operators}
To work in the low-regularity translation-invariant setting analogous to Marzuola--Metcalfe--Tataru \cite{MMT1, MMT2, MMT3}, we also employ \emph{paradifferential operators}. Since we also need to consider the composition and commutation of such operators with classical pseudodifferential operators, it will be beneficial to review some of their theory. A more systematic account of paradifferential operators can be found in, e.g., \cite{Tay}. 

Adopting the notation in \cite{Tay}, we define the symbol class $\calB S^{m}_{1, 1}$ as
\begin{equation*}
	\calB S^{m}_{1, 1} := \set{\frka : [\frka]_{S^{m}_{1, 1}; N} < + \infty \hbox{ for all } N \geq 0, \, \supp \wh{\frka}(\eta, \xi) \subseteq \set{\abs{\eta} < \tfrac{1}{8} \abs{\xi}}},
\end{equation*}
where (formally) $\wh{\frka}(\eta, \xi) = \int \frka(x, \xi) e^{- i \eta \cdot x} \, \ud x$ and 
\begin{equation*}
	[\frka]_{S^{m}_{1, 1}; N} := \sum_{\bfalp, \bfbt : \abs{\bfalp} + \abs{\bfbt} \leq N} \sup_{x, \xi}\abs{\brk{\xi}^{-\abs{\bfalp} + \abs{\bfbt} -  {m} } \rd_{x}^{\bfalp}\rd_{\xi}^{\bfbt}\frka(x, \xi)}.
\end{equation*}
An important example of a paradifferential operator is the \emph{paraproduct}, defined as follows: given $g \in L^{\infty}(\bbR^{d})$, the paraproduct of $g$ and $u$ is
\begin{equation*}
	T_{g} u = \sum_{k} P_{<k-10} g P_{k} u,
\end{equation*}
or equivalently, $T_{g} = \Op(\sum_{k} P_{<k-10} g(x) P_{k}(\xi))$. Clearly, $T_{g} \in \calB S^{0}_{1, 1}$ with
\begin{equation*}
	\bb[ \sum_{k} P_{<k-10} g(x) P_{k}(\xi) \bb]_{S^{0}_{1, 1}; N} \aleq_{N} \nrm{g}_{L^{\infty}}.
\end{equation*}
Another important way that paradifferential operators arise is by the following decomposition. Given $\frka \in S^{m}$, consider the decomposition $\frka = \frka^{\sharp} + \frka^{\flat}$, where
\begin{gather*}
	\frka^{\sharp}(x, \xi) = \sum_{k} P_{<k-10} \frka(x, \xi) P_{k}(\xi), \quad
	\frka^{\flat}(x, \xi) = \sum_{k} P_{\geq k-10} \frka(x, \xi) P_{k}(\xi).
\end{gather*}
where $P_{< k-10}$ and $P_{\geq k-10}$ acts on the $x$-variable. It is easy to verify that (for any $N, m_{0} \geq 0$)
\begin{align}
\frka^{\sharp} &\in \calB S^{m}_{1, 1} \hbox{ with } [\frka^{\sharp}]_{S^{m}_{1, 1}; N} \leq [\frka]_{S^{m}; N}, \label{eq:symb-sharp} \\
\frka^{\flat} &\in S^{-\infty} \hbox{ with } [\frka^{\flat}]_{S^{m - m_{0}}_{1, 1}; N} \aleq_{m, m_{0}, N} [\frka]_{S^{m}; N + m_{0}}. \label{eq:symb-flat}
\end{align}
We have the following basic $L^{p}$ boundedness theorem for paradifferential operators.
\begin{proposition} \label{prop:para-L2-bdd}
Let $\frka \in \calB S^{0}_{1, 1}$. Then for every  $1 < p < \infty$, we have
\begin{equation*}
	\nrm{\frka(x, D)}_{L^{p} \to L^{p}} \aleq [\frka]_{S^{0}_{1,1}; 10 d}.
\end{equation*}
The same result holds for $\frka(D, x)$. 
\end{proposition}
For a proof, see \cite[Prop.~3.4.F]{Tay}.

\subsubsection{Symbolic calculus for pseudodifferential operators}
We now recall the symbolic calculus for pseudodifferential operators. In our argument, we need to consider composition (and commutation) of classical pseudodifferential operators and paradifferential operators. Here, we record precise formulas for the remainders.

We begin with composition. Let $\frka(x, \xi)$ and $\frkb(x, \xi)$ be symbols that are either in $S^{m}$ or $\calB S^{m}_{1, 1}$ (with possibly different orders $m$). Then the composition $\frka(x, D) \frkb(x, D)$ is again a left quantization of a classical symbol, i.e.,
\begin{equation*}
	\frka(x, D) \frkb(x, D) = (\frka \circ \frkb)(x, D)
\end{equation*}
where $(\frka \circ \frkb)(x, \xi)$ is formally given by
\begin{equation*}
	(\frka \circ \frkb)(x, \xi) = e^{-i \xi \cdot x} \frka(x, D) \frkb(x, D) e^{i \xi \cdot (\cdot)}
	=  \iint \frka(x, \eta) \frkb(y, \xi) e^{i(\eta - \xi) \cdot (x-y)} \, \ud y \frac{\ud \eta}{(2 \pi)^{d}}.
\end{equation*}
By Taylor expansion, we have \begin{equation*}
	\begin{split}
		\frka(x, \eta) 
		&= \frka(x, \xi) +\int_{0}^{1} (1-s) (\eta-\xi)_{a}  (\rd_{\xi_{a}} \frka)(x, s \eta + (1-s) \xi)  \, \ud s
	\end{split}
\end{equation*} and 
\begin{align*}
	\frka(x, \eta) 
	&= \frka(x, \xi) + (\eta - \xi)_{a} (\rd_{\xi_{a}} \frka)(x, \xi) + \int_{0}^{1} (1-s) (\eta-\xi)_{a} (\eta-\xi)_{b} (\rd_{\xi_{a}} \rd_{\xi_{b}}\frka)(x, s \eta + (1-s) \xi)  \, \ud s.
\end{align*}
Using $(\eta - \xi)_{a} e^{i (\eta - \xi) \cdot (x-y)} = - i^{-1} \rd_{y^{a}} e^{i(\eta-\xi) \cdot (x-y)}$, integrating $\rd_{y^{a}}$ by parts and noting that $\int e^{i(\eta-\xi) \cdot (x-y)} \frac{\ud \eta}{(2 \pi)^{d}} = \dlt_{0}(x-y)$, we obtain the formal formulae \begin{equation} \label{eq:symb-comp0}
	\begin{aligned}
		& (\frka \circ \frkb)(x, \xi) - \frka(x, \xi) \frkb(x, \xi) \\
		& = - \iint \int_{0}^{1} (1-s) (\rd_{\xi_{a}} \frka)(x, s \eta + (1-s) \xi) (\rd_{x^{a}} \frkb)(y, \xi) e^{i (\eta - \xi) \cdot (x - y)}\, \ud s \ud y \frac{\ud \eta}{(2 \pi)^{d}} 
	\end{aligned}
\end{equation} and 
\begin{equation} \label{eq:symb-comp}
\begin{aligned}
	& (\frka \circ \frkb)(x, \xi) - \frka(x, \xi) \frkb(x, \xi) - i^{-1} \rd_{\xi_{a}} \frka (x, \xi) \rd_{x^{a}} \frkb(x, \xi) \\
	&= - \iint \int_{0}^{1} (1-s) (\rd_{\xi_{a}} \rd_{\xi_{b}} \frka)(x, s \eta + (1-s) \xi) (\rd_{x^{a}} \rd_{x^{b}} \frkb)(y, \xi) e^{i (\eta - \xi) \cdot (x - y)}\, \ud s \ud y \frac{\ud \eta}{(2 \pi)^{d}}.
\end{aligned}
\end{equation}
Based on these, the following result can be established. 

\begin{proposition}[Composition of pseudodifferential operators] \label{prop:psdo-comp}
Let $\frka \in S^{m}$.
\begin{enumerate}
\item Let $\frkb \in S^{m'}$. Then
\begin{align*}
(\frka \circ \frkb)(x, \xi) & \in S^{m+m'}, \\
(\frka \circ \frkb)(x, \xi) - \frka(x, \xi) \frkb(x, \xi) & \in S^{m+m'-1}, \\
(\frka \circ \frkb)(x, \xi) - \frka(x, \xi) \frkb(x, \xi) - i^{-1} \rd_{\xi_{a}} \frka (x, \xi) \rd_{x^{a}} \frkb(x, \xi) &\in S^{m+m'-2},
\end{align*}
with
\begin{align*}
[(\frka \circ \frkb)(x, \xi)]_{S^{m+m'}; N} &\aleq_{m, m', N} [\frka]_{S^{m}; N+10d} [\frkb]_{S^{m'}; N+10d}, \\ 
[(\frka \circ \frkb)(x, \xi) - \frka(x, \xi) \frkb(x, \xi)]_{S^{m+m'-1}; N} &\aleq_{m, m', N} [\frka]_{S^{m}; N+10d} [\frkb]_{S^{m'}; N+10d}, \\ 
[(\frka \circ \frkb)(x, \xi) - \frka(x, \xi) \frkb(x, \xi) - i^{-1} \rd_{\xi_{a}} \frka (x, \xi) \rd_{x^{a}} \frkb(x, \xi)]_{S^{m+m'-2}; N} &\aleq_{m, m', N} [\frka]_{S^{m}; N+10d} [\frkb]_{S^{m'}; N+10d}.
\end{align*}
\item Let $\frkb \in \calB S^{m'}_{1, 1}$ and
\begin{equation} \label{eq:paradiff-supp}
\supp \wh{\frkb}(\eta, \xi) \subseteq \set{\abs{\eta} < \tfrac{1}{16} \abs{\xi}}.
\end{equation}
Then
\begin{align*}
(\frkb \circ \frka)(x, \xi) - \frkb(x, \xi) \frka(x, \xi) &\in \calB S^{m+m'-1}_{1, 1} + S^{-\infty}.
\end{align*}
More precisely, the expression decomposes into $\frkr^{(m+m'-1) \sharp} \in \calB S^{m+m'-1}_{1, 1}$ and $\frkr^{(m+m'-1) \flat} \in S^{-\infty}$, where
\begin{align*}
[\frkr^{(m+m'-1) \sharp}]_{S^{m+m'-1}_{1, 1}; N} &\aleq_{m, m', N} [\frka]_{S^{m}; N+10d} [\frkb]_{S^{m'}_{1, 1}; N+10d}, \\ 
[\frkr^{(m+m'-1) \flat}]_{S^{m+m'-1+m_{0}}; N} &\aleq_{m, m', m_{0}, N} [\frka]_{S^{m}; N+m_{0}+10d} [\frkb]_{S^{m'}_{1, 1}; N+10d},
\end{align*}
Moreover, we have
\begin{align*}
(\frkb \circ \frka)(x, \xi) - \frkb(x, \xi) \frka(x, \xi) - i^{-1} \rd_{\xi_{a}} \frkb (x, \xi) \rd_{x^{a}} \frka(x, \xi) & \in \calB S^{m+m'-2}_{1, 1} + S^{-\infty}.
\end{align*}
More precisely, the expression decomposes into $\frkr^{(m+m'-2) \sharp} \in \calB S^{m+m'-2}_{1, 1}$ and $\frkr^{(m+m'-2) \flat} \in S^{-\infty}$, where
\begin{align*}
[\frkr^{(m+m'-2) \sharp}]_{S^{m+m'-2}_{1, 1}; N} 
&\aleq_{m, m', N} [\frka]_{S^{m}; N+10d} [\frkb]_{S^{m'}_{1, 1}; N+10d}, \\ 
[\frkr^{(m+m'-2) \flat}]_{S^{m+m'-1+m_{0}}; N} 
&\aleq_{m, m', m_{0}, N} [\frka]_{S^{m}; N+m_{0}+10d} [\frkb]_{S^{m'}_{1, 1}; N+10d}.
\end{align*}
\item If $\rd_{x}^{\bfalp} \frkb  \in \calB S^{m'}_{1,1}$ for $\abs{\bfalp} \leq 1$ and \eqref{eq:paradiff-supp} holds, then
\begin{align*}
(\frka \circ \frkb)(x, \xi) - \frka(x, \xi) \frkb(x, \xi) &\in \calB S^{m+m'-1}_{1,1} + S^{-\infty}. 
\end{align*}
More precisely, the expression decomposes into $\frkr^{(m+m'-1) \sharp} \in \calB S^{m+m'-1}_{1, 1}$ and $\frkr^{(m+m'-1) \flat} \in S^{-\infty}$, where
\begin{align*}
[\frkr^{(m+m'-1) \sharp}]_{S^{m+m'-1}_{1, 1}; N} &\aleq_{m, m', N} [\frka]_{S^{m}; N+10d} \sum_{\bfalp: \abs{\bfalp} = 1} [\rd_{x}^{\bfalp}\frkb]_{S^{m'}_{1, 1}; N+10d}, \\ 
[\frkr^{(m+m'-1) \flat}]_{S^{m+m'-1+m_{0}}; N} &\aleq_{m, m', m_{0}, N} [\frka]_{S^{m}; N+m_{0}+10d} \sum_{\bfalp: \abs{\bfalp} = 1} [\rd_{x}^{\bfalp}\frkb]_{S^{m'}_{1, 1}; N+10d},
\end{align*}
Moreover, if $\rd_{x}^{\bfalp} \frkb  \in \calB S^{m'}_{1,1}$ for $\abs{\bfalp} \leq 2$ and \eqref{eq:paradiff-supp} holds, then
\begin{align*}
(\frka \circ \frkb)(x, \xi) - \frka(x, \xi) \frkb(x, \xi) - i^{-1} \rd_{\xi_{a}} \frka (x, \xi) \rd_{x^{a}} \frkb(x, \xi) &\in \calB S^{m+m'-2}_{1, 1} + S^{-\infty},
\end{align*}
More precisely, the expression decomposes into $\frkr^{(m+m'-2) \sharp} \in \calB S^{m+m'-2}_{1, 1}$ and $\frkr^{(m+m'-2) \flat} \in S^{-\infty}$, where
\begin{align*}
[\frkr^{(m+m'-2) \sharp}]_{S^{m+m'-2}_{1, 1}; N} 
&\aleq_{m, m', N} [\frka]_{S^{m}; N+10d} \sum_{\bfalp: \abs{\bfalp} = 2} [\rd_{x}^{\bfalp} \frkb]_{S^{m'}_{1, 1}; N+10d}, \\ 
[\frkr^{(m+m'-2) \flat}]_{S^{m+m'-1+m_{0}}; N} 
&\aleq_{m, m', m_{0}, N} [\frka]_{S^{m}; N+m_{0}+10d} \sum_{\bfalp: \abs{\bfalp} = 2} [\rd_{x}^{\bfalp}\frkb]_{S^{m'}_{1, 1}; N+10d}.
\end{align*}
\end{enumerate}
\end{proposition}
Statement~(1) is standard (see, e.g. \cite[Prop.~0.3.C]{Tay}), while Statements~(2) and (3) follow essentially from \cite[Proof of Thm.~3.4.A]{Tay}. These may also be easily obtained by studying the explicit remainder \eqref{eq:symb-comp} using stationary phase. 

Proposition \ref{prop:psdo-comp} has immediate consequences for commutators and Poisson brackets. Given two symbols $\frka$ and $\frkb$, recall that the commutator is $[\frkb,\frka] = \frkb \circ \frka - \frka \circ \frkb$ and the Poisson bracket is defined by the symbol $\{ \frka,\frkb \} = ( \rd_{\xi_{a}} \frka (x, \xi) \rd_{x^{a}} \frkb(x, \xi) - \rd_{\xi_{a}} \frkb (x, \xi) \rd_{x^{a}} \frka(x, \xi) )$. Then, as a consequence of Proposition \ref{prop:psdo-comp} (1), we have the following statement. \begin{corollary}\label{cor:psdo-comp0}
	Let $\frka \in S^{m}$ and $\frkb \in S^{m'}$ be classical symbols. Then, we have that \begin{equation*}
		\begin{split}
			[\frka,\frkb], \{ \frka,\frkb \}  \in S^{m+m'-1}
		\end{split}
	\end{equation*} and \begin{equation*}
	\begin{split}
		[\frka,\frkb] - \frac{1}{i} \{ \frka,\frkb \}  \in S^{m+m'-2}. 
	\end{split}
\end{equation*}
\end{corollary}

For later use, it will be convenient to record the following specific case of Proposition \ref{prop:psdo-comp} (2)--(3), when one of the symbols takes the form of a paraproduct.
\begin{corollary}\label{cor:psdo-comp}
	Let $\frka$ be a classical symbol and $T_{g} = \sum_{k} P_{<k-10}g P_{k}$. Then, we have  \begin{equation}\label{eq:psdo-comp-bound1}
		\begin{split}
			\nrm{ [\frka, T_{g}] }_{L^{2} \to L^{2}} \aleq [\frka]_{S^{1};10d} \nrm{ g }_{W^{1,\infty}}
		\end{split}
	\end{equation} and \begin{equation}\label{eq:psdo-comp-bound2}
		\begin{split}
			\nrm{ [\frka, T_{g}] - \frac{1}{i} \Op \{ \frka, T_{g} \} }_{L^{2} \to L^{2}} \aleq [\frka]_{S^{2};10d} \nrm{ g }_{W^{2,\infty}}. 
		\end{split}
	\end{equation}
\end{corollary}

Next, we discuss the adjoint operation. Given a classical symbol $\frka(x, \xi)$, it is not difficult to verify that
\begin{equation*}
	\frka(x, D)^{\ast} = \frka^{\dagger}(D, x).
\end{equation*}
This simple relation motivates the question of computing the symbol of a right quantization $\frka(D, x)$ as a left-quantized pseudodifferential operator. We have
\begin{equation*}
\frka(D, x) = \frka^{r \to \ell}(x, D),
\end{equation*}
where $\frka^{r \to \ell}(x, \xi)$ is formally given by
\begin{align*}
	\frka^{r \to \ell}(x, \xi) 
	= e^{-i \xi \cdot x} \frka(D, x) e^{i \xi \cdot (\cdot)}
	= \iint \frka(y, \eta) e^{i (\eta-\xi) \cdot (x-y)} \, \ud y \frac{\ud \eta}{(2 \pi)^{d}}.\end{align*}
Using the expansion $\frka(y, \eta) = \frka(y, \xi) + \int_{0}^{1} (\eta-\xi)^{\alp} (\rd_{\xi_{\alp}} \frka)(y, s \eta + (1-s) \xi) \, \ud s$
and proceeding as in the derivation of \eqref{eq:symb-comp}, we obtain the formal formula
\begin{equation} \label{eq:symb-r2l}
\begin{aligned}
	& \frka^{r \to \ell}(x, \xi) - \frka(x, \xi)
	= \iint \int_{0}^{1} i^{-1} (\rd_{x^{a}} \rd_{\xi_{a}} \frka)(y, s \eta + (1-s) \xi) e^{i (\eta-\xi) \cdot (x-y)}\, \ud s \ud y \frac{\ud \eta}{(2 \pi)^{d}}.
\end{aligned}
\end{equation}

\begin{proposition}[Right to left quantization] \label{prop:psdo-adj}
The following statements hold.
\begin{enumerate}
\item If $\frka \in S^{m}$, then $\frka^{r \to \ell} - \frka \in S^{m-1}$, with
\begin{align*}
	[\frka^{r \to \ell} - \frka]_{S^{m-1}; N} \aleq_{m, N} [\frka]_{S^{m}; N+10d}.
\end{align*}
\item If $\frkb, \rd_{x} \frkb \in \calB S^{m}_{1, 1}$ with \eqref{eq:paradiff-supp}, then $\frkb^{r \to \ell} - \frkb \in \calB S^{m-1}_{1, 1}$, with
\begin{align*}
	[\frkb^{r \to \ell} - \frkb]_{S^{m-1}_{1, 1}; N} \aleq_{m, N} \sum_{\bfalp : \abs{\bfalp} = 1}  [\rd_{x}^{\bfalp} \frkb]_{S^{m}_{1, 1}; N+10d}.
\end{align*}
\end{enumerate}
\end{proposition}
Statement~(1) is standard; see, e.g., \cite[Prop.~0.3.B]{Tay}. Statement~(2) again follows essentially from \cite[Proof of Thm.~3.4.A]{Tay}. It may also be easily obtained by studying the explicit remainder \eqref{eq:symb-r2l} using stationary phase.

Although we will not need them in the sequel, we note that analogous formulas where the left and right quantizations are swapped can be derived in a similar manner.

\subsection{$L^{2}$-boundedness, positivity and high frequency Calder\'on--Vaillancourt trick} \label{subsec:cv}
We recall the Calder\'on--Vaillancourt theorem, which is a key $L^{2}$-boundedness theorem for pseudodifferential operators.

\begin{theorem}[Calder\'on--Vaillancourt] \label{thm:cv}
Let $\frka(x,\xi) \in C^{\infty}(\bbR^{d} \times \bbR^{d})$ satisfy 
\begin{equation} \label{eq:symb-cv}
	\abs{\nb_{x}^{(n')} \nb_{\xi}^{(n)} \frka(x, \xi)} \leq 1 \quad \hbox{ for } \abs{n} + \abs{n'} \leq 10 d.
\end{equation}
Then $\frka(x, D)$ is bounded from $L^{2}(\bbR^{d})$ to $L^{2}(\bbR^{d})$, with the norm bounded by a universal constant.
\end{theorem}
An important difference of this statement from Proposition \ref{prop:psdo-L2-bdd-garding} (1) is that $\frka$ is not required to belong to the classical symbol class $S^{m}$. We refer to \cite{Stein} for a proof.

Consider now a smooth matrix-valued symbol $\frka$ obeying\footnote{One may of course formulate \eqref{eq:symb-AB} and Lemma~\ref{lem:hf-L2} setting $A = B = 1$ without losing any generality. However, we choose the current formulation as we often work with symbols that behave in a regular way when differentiated in $x$ or in $\xi$ as in \eqref{eq:symb-AB}.}
\begin{equation} \label{eq:symb-AB}
	\abs{\rd_{x}^{\bfalp} \rd_{\xi}^{\bfbt} \frka(x, \xi)}
	\leq c_{\bfalp \bfbt} \lmb^{-\abs{\bt}}, \quad
	\supp \frka \subseteq \set{(x, \xi) : \frac{1}{100} \lmb< \abs{\xi} < 100 \lmb}, 
\end{equation}
for some constants $c_{\bfalp \bfbt}$. 

Of course, qualitatively, $\frka \in S^{0}$, so the standard $L^{2}$ boundedness theorem and sharp G{\aa}rding's inequality apply. However, such results do not give precise constants in relation to $\lmb$, and hence are insufficient for our paper. Nevertheless, it turns out that the operator norm of $\Op(\frka)$ (in, say, $L^{2}$) can be directly bounded by the size $c_{00}$ of the symbol $\frka$, as long as we restrict the input to sufficiently high frequencies (which may depend on $c_{\bfalp \bfbt}$ with $(\bfalp, \bfbt) \neq (0, 0)$). The proof is by a simple scaling argument combined with Theorem~\ref{thm:cv}.

\begin{lemma}[High-frequency Calder\'on--Vaillancourt] \label{lem:hf-L2}
Let $\frka$ be a smooth symbol satisfying \eqref{eq:symb-AB} for $\abs{\bfalp} + \abs{\bfbt} \leq 100d$. If
\begin{equation} \label{eq:cv-lmb}
	\lmb \geq \max_{1 \leq \abs{\bfalp} + \abs{\bfbt} \leq 10 d} \left(\frac{c_{\bfalp \bfbt}}{c_{00}}\right)^{\frac{2}{\abs{\bfalp} + \abs{\bfbt}}},
\end{equation}
then, for some universal constant $C > 0$,
\begin{equation} \label{eq:cv-AB}
	\nrm{\frka(x, D)}_{L^{2} \to L^{2}} \leq C c_{00}.
\end{equation}
\end{lemma}
\begin{proof}
Let us consider a rescaling $\tilde{\frka}$ of the symbol $\frka$ of the form
\begin{equation*}
	\tilde{\frka}(x, \xi) = c_{00}^{-1} \frka(\lmb_{0}^{-1} x, \lmb_{0} \xi).
\end{equation*}
If we take $\lmb_{0} = \lmb^{\frac{1}{2}}$, then
\begin{equation*}
	\abs{\rd_{x}^{\alp} \rd_{\xi}^{\bfbt} \tilde{\frka}(x, \xi)}
	\leq \frac{c_{\bfalp \bfbt}}{c_{00}} \left(\frac{1}{\lmb}\right)^{\frac{\abs{\bfalp} + \abs{\bfbt}}{2}}.
\end{equation*}
Thus, if \eqref{eq:cv-lmb} holds, then $\tilde{\frka}$ obeys the hypothesis of Theorem~\ref{thm:cv}. Since $\nrm{\tilde{\frka}(x, D)}_{L^{2} \to L^{2}} = c_{00}^{-1} \nrm{\frka(x, D)}_{L^{2} \to L^{2}}$, we see that $\nrm{\frka(x, D)}_{L^{2} \to L^{2}} \leq C c_{00}$ as desired.  \qedhere
\end{proof}

\section{Reformulation of \eqref{eq:e-mhd}} \label{sec:reformulation}
\subsection{Reformulation of \eqref{eq:e-mhd} for perturbations of $\bgB$}
Given a background magnetic field $\bgB$ with $\nb \cdot \bgB = 0$, which may not necessarily be a solution to \eqref{eq:e-mhd}, consider the decomposition of a solution $\bfB$ to \eqref{eq:e-mhd} into $\bgB$ and perturbation $b$:
\begin{equation*}
	\bfB = \bgB + b.
\end{equation*}
In terms of $b$, \eqref{eq:e-mhd} may be reformulated in the following way:
\begin{equation} \label{eq:e-mhd-b}
\left\{
\begin{aligned}
& \rd_{t} b + \nb \times ((\nb \times b) \times (\bgB + b)) + \nb \times ((\nb \times \bgB) \times b)= - \rd_{t} \bgB - \nonlin(\bgB) \\
& \nb \cdot b = 0,
\end{aligned}	
\right.
\end{equation}
where $\nonlin(\bgB)$ is the nonlinearity of \eqref{eq:e-mhd}:
\begin{equation*}
\nonlin(\bgB) := \nb \times ((\nb \times \bgB) \times \bgB)  \label{eq:nonlin-B}
\end{equation*}
We call \eqref{eq:e-mhd-b} the \emph{perturbation equation around the background magnetic field $\bgB$}. Note that the {RHS} vanishes if $\bgB$ solves \eqref{eq:e-mhd}. In this case, keeping only the linear part in $b$, we arrive at the \emph{linearized \eqref{eq:e-mhd} around the background solution $\bgB$}:
\begin{equation} \label{eq:e-mhd-lin}
\left\{
\begin{aligned}
& \rd_{t} b + \nb \times ((\nb \times b) \times \bgB) + \nb \times ((\nb \times \bgB) \times b)= 0, \\
& \nb \cdot b = 0.
\end{aligned}	
\right.
\end{equation}
Motivated by \eqref{eq:e-mhd-lin}, we introduce the following shorthand for the spatial part of the linearized equation:
\begin{align} 
\lin_{\bgB} b=& \nb \times ((\nb \times b) \times \bgB) + \nb \times ((\nb \times \bgB) \times b). \label{eq:lin-B}
\end{align}
We also introduce the decomposition $\lin_{\bgB} = \linmain_{\bgB} + \linlower_{\bgB}$, where
\begin{equation} \label{eq:lin-main-lower}
\linmain_{\bgB} b= \nb \times ((\nb \times b) \times \bgB), \quad 
\linlower_{\bgB} b=  \nb \times ((\nb \times \bgB) \times b) .
\end{equation}
Note that $\linmain_{\bgB}$ contains the principal term, while $\linlower_{\bgB}$ is lower order.

\subsection{Diagonalization of the principal symbol} \label{subsec:symb-diag}
Replacing $\nb$ by the symbol $i \xi$ in the term $\linmain_{\bgB} b$ in \eqref{eq:lin-main-lower}, we see that the principal symbol of the linearized \eqref{eq:e-mhd} around $\bgB$ (i.e., \eqref{eq:e-mhd-lin}) takes the form
\begin{equation*}
	\tensor{(\prin_{\bgB})}{^{\alp}_{\bt}}(x, \xi) = - \tensor{\eps}{^{\gmm \alp}_{\bt}} \bgB^{\dlt} \xi_{\dlt} \xi_{\gmm} .
\end{equation*}
Our goal in this subsection is to diagonalize $\prin_{\bgB}$. In what follows, we use the shorthand $\prin = \prin_{\bgB}$. Note that
\begin{align*}
	\tensor{\prin}{^{\alp}_{\gmm}} \tensor{\prin}{^{\gmm}_{\bt}}
	= & \tensor{\eps}{^{\dlt \alp}_{\gmm}}\tensor{\eps}{^{\dlt' \gmm}_{\bt}} (\bgB^{\dlt} \xi_{\dlt})^{2} \xi_{\dlt} \xi_{\dlt'} \\
	= & (-\dlt^{\dlt \dlt'} \dlt^{\alp}_{\bt} + \dlt^{\dlt}_{\bt} \dlt^{\alp \dlt'}) (\bgB^{\dlt} \xi_{\dlt})^{2} \xi_{\dlt} \xi_{\dlt'} \\
	= & - (\bgB^{\dlt} \xi_{\dlt})^{2} \abs{\xi}^{2} \dlt^{\alp}_{\bt} + (\bgB^{\dlt} \xi_{\dlt})^{2} \xi^{\alp} \xi_{\bt},
\end{align*}
so that
\begin{equation*}
	(\prin \mp i \bgB^{\alp} \xi_{\alp} \abs{\xi} I)(\prin \pm i \bgB^{\alp} \xi_{\alp} \abs{\xi} I) = (\bgB^{\alp} \xi_{\alp})^{2} \xi \otimes \xi.
\end{equation*}
By simple linear algebra, we see that if we define
\begin{equation*}
	\tensor{(\Pi_{0})}{^{\alp}_{\bt}} = \frac{\xi^{\alp} \xi_{\bt}}{\abs{\xi}^{2}}, \quad
	\tensor{(\Pi_{\pm})}{^{\alp}_{\bt}} = \frac{1}{2} \left(\dlt^{\alp}_{\bt} \mp \frac{\tensor{\eps}{^{\gmm \alp}_{\bt}} \xi_{\gmm}}{i\abs{\xi}} - \frac{\xi^{\alp} \xi_{\bt}}{\abs{\xi}^{2}}  \right),
\end{equation*}
then 
\begin{equation*}
	\Pi_{s}^{2} = \Pi_{s}, \quad \Pi_{s} \Pi_{s'} = 0, \quad s, s' \in \set{-, 0, +}, 
\end{equation*}
and $\Pi_{-} + \Pi_{0} + \Pi_{+} = I$. Moreover,
\begin{align*}
	\prin (x, \xi) =& i \dprin (x, \xi) \Pi_{+}(\xi) - i \dprin (x, \xi) \Pi_{-}(\xi), \\
	\dprin(x, \xi) =& \dprin_{\bgB}(x, \xi) :=  \bgB^{\alp} \xi_{\alp} \abs{\xi},
\end{align*}
which is the desired diagonalization of $\prin_{\bgB}$. We will refer to $p(x, \xi) = p_{\bgB}(x, \xi)$ as the \emph{scalar principal symbol} of \eqref{eq:e-mhd-lin}.

\begin{remark} \label{rem:diag}
The matrices $\Pi_{+}(\xi)$, $\Pi_{0}(\xi)$, $\Pi_{-}(\xi)$ are nothing but eigenspace projections of the cross product $\xi \times$ viewed as an operator on $\bbC^{3} (= \bbR^{3} \otimes_{\bbR} \bbC)$. Accordingly, the corresponding multipliers $\Pi_{+}(D)$, $\Pi_{0}(D)$ and $\Pi_{-}(D)$ are $L^{2}(\bbR^{3}; \bbC^{3})$ projections that diagonalize the curl operator $\nb \times$. Since 
\begin{equation*}
\tensor{(\Pi_{s})}{^{\alp}_{\bt}}(\xi) = \overline{\tensor{(\Pi_{s})}{^{\alp}_{\bt}}(-\xi)}, \qquad s \in \set{-, 0, +},
\end{equation*}
it follows that $\Pi_{s}(D) : L^{2}(\bbR^{3}; \bbR^{3}) \to L^{2}(\bbR^{3}; \bbR^{3})$.
\end{remark}

\subsection{Linearized equation after diagonalization of the principal part} \label{subsec:lin-diag}
Motivated by the computation in Section~\ref{subsec:symb-diag}, we return to the linearized \eqref{eq:e-mhd} equation \eqref{eq:e-mhd-lin},
\begin{equation*}
\left\{
\begin{aligned}
	&\rd_{t} b + \lin_{\bgB} b = g, \\
	&\nb \cdot b = 0,
\end{aligned}
\right.
\end{equation*}
and compute the equation satisfied by the diagonalized variables $b_{\pm} = \Pi_{\pm}(D) b$.

\begin{proposition} \label{prop:lin-diag}
The system \eqref{eq:e-mhd-lin} is equivalent to the following system for the diagonalized variables $b_{\pm} = \Pi_{\pm}(D) b$:
 \begin{equation} \label{eq:e-mhd-lin-diag}
\begin{alignedat}{2}
	&  \rd_{t} b_{\pm} \pm \diag_{\bgB}^{(2)} b_{\pm} \pm \symm^{(1)}_{\bgB} b_{\mp} \pm \asymm^{(1)}_{\bgB} b_{\pm} \pm \nb \left(  \covec^{(0)}_{\bgB} (b_{\pm} + b_{\mp})\right) & \\
	& \quad + (\nb \times \bgB) \cdot \nb b_{\pm} + \rem^{(0)}_{\bgB; \pm} (b_{\pm} + b_{\mp}) 
	& = \Pi_{\pm}(D) g,
\end{alignedat}
\end{equation}
or in the form of a system,
\begin{equation} \label{eq:e-mhd-lin-diag-sys}
\begin{alignedat}{2}
	 & \rd_{t} \begin{pmatrix} b_{+} \\ b_{-} \end{pmatrix} 
	 + \begin{pmatrix} \diag_{\bgB}^{(2)} & 0 \\ 0 & - \diag_{\bgB}^{(2)} \end{pmatrix} \begin{pmatrix} b_{+} \\ b_{-} \end{pmatrix} 
	 + \begin{pmatrix} 0 & \symm^{(1)}_{\bgB}  \\ - \symm^{(1)}_{\bgB}  & 0 \end{pmatrix} \begin{pmatrix} b_{+} \\ b_{-} \end{pmatrix}  & \\
	& + \begin{pmatrix} \asymm^{(1)}_{\bgB} & 0 \\ 0 & - \asymm^{(1)}_{\bgB} \end{pmatrix} \begin{pmatrix} b_{+} \\ b_{-} \end{pmatrix}  
	 + (\nb \times \bgB) \cdot \nb \begin{pmatrix} b_{+} \\ b_{-} \end{pmatrix} & \\
	& + \nb \begin{pmatrix}  \covec^{(0)}_{\bgB} & \covec^{(0)}_{\bgB} \\ - \covec^{(0)}_{\bgB} & - \covec^{(0)}_{\bgB}  \end{pmatrix}  \begin{pmatrix}  b_{+} \\ b_{-}  \end{pmatrix}
	+ \begin{pmatrix}  \rem^{(0)}_{\nb \times \bgB; +} & \rem^{(0)}_{\nb \times \bgB; +} \\ \rem^{(0)}_{\nb \times \bgB; -} & \rem^{(0)}_{\nb \times \bgB; -}  \end{pmatrix} \begin{pmatrix}  b_{+} \\ b_{-}  \end{pmatrix}  
	& =  \begin{pmatrix} \Pi_{+}(D) g \\ \Pi_{-}(D) g \end{pmatrix},
\end{alignedat}
\end{equation}
where $\diag^{(2)}_{\bgB}$ is the scalar 2nd order operator of the form
\begin{equation*}
	\diag^{(2)}_{\bgB} = \frac{1}{2} \left( (\bgB \cdot \nb) \abs{\nb} + \abs{\nb} (\bgB \cdot \nb) \right),
\end{equation*}
$\symm^{(1)}_{\bgB}$ is a matrix-valued symmetric 1st order operator of the form
\begin{equation*}
\tensor{(\symm^{(1)}_{\bgB})}{^{\alp}_{\bt}} = \frac{1}{2} \left(\abs{\nb} (\rd^{\alp} \bgB_{\bt}) + (\rd_{\bt} \bgB^{\alp}) \abs{\nb} \right) 
+ \frac{1}{2} [\abs{\nb}, \bgB \cdot \nb] \dlt^{\alp}_{\bt},
\end{equation*}
$\asymm^{(1)}_{\bgB}$ is a matrix-valued anti-symmetric 1st order operator of the form
\begin{equation*}
\tensor{(\asymm^{(1)}_{\bgB})}{^{\alp}_{\bt}} = \frac{1}{2} \left(\abs{\nb} (\rd^{\alp} \bgB_{\bt}) - (\rd_{\bt} \bgB^{\alp}) \abs{\nb} \right) ,
\end{equation*}
$\covec^{(0)}_{\bgB}$ is a covector-valued 0th order operator of the form
\begin{align*}
(\covec^{(0)}_{\bgB})_{\bt} =& \frac{1}{2} \frac{\rd_{\gmm}}{\abs{\nb}}  (\rd^{\gmm} \bgB_{\bt} + \rd_{\bt} \bgB^{\gmm}) ,
\end{align*}
and $\rem^{(0)}_{\nb \times \bgB; \pm}$ is a matrix-valued 0th order operator of the form
\begin{equation*}
\tensor{(\rem^{(0)}_{\nb \times \bgB; \pm})}{^{\alp}_{\bt}}  = [ \tensor{\Pi_{\pm}(D)}{^{\alp}_{\bt}}, (\nb \times \bgB) \cdot \nb]   - \tensor{\Pi_{\pm}(D)}{^{\alp}_{\gmm}}  (\rd_{\bt} (\nb \times \bgB)^{\gmm}) .
\end{equation*}
\end{proposition}
We also make the obvious observation that each $b_{\pm}$ satisfies the divergence-free condition 
\begin{equation}\label{eq:div-free-diag}
\nb \cdot b_{\pm} = 0.
\end{equation}
\begin{proof}
Note that
\begin{equation*}
	\tensor{\Pi_{\pm}(D)}{^{\alp}_{\bt}} = \frac{1}{2} \dlt^{\alp}_{\bt} \pm \frac{1}{2} \abs{\nb}^{-1} (\nb \times)^{\alp}_{\bt} + \frac{1}{2} \abs{\nb}^{-2} \nb^{\alp} \nb_{\bt} .
\end{equation*}
Recall the definition of $\lin_{\bgB}$ in \eqref{eq:lin-B}:
\begin{equation*}
	\lin_{\bgB} b = \underbrace{\nb \times ((\nb \times b) \times \bgB)}_{\linmain_{\bgB} b} + \underbrace{\nb \times ((\nb \times \bgB) \times b)}_{\linlower_{\bgB} b}
\end{equation*}
We first compute
\begin{align*}
	\left(\Pi_{\pm}(D) \linmain_{\bgB} b\right)^{\alp}
	 & = \left( \Pi_{\pm}(D)\left( \nb \times ((\nb \times b) \times \bgB)\right)\right)^{\alp} \\
	& = \left(\frac{1}{2} \dlt^{\alp}_{\bt} \pm \frac{1}{2} \abs{\nb} (\nb \times)^{\alp}_{\bt}\right) \left(\nb \times ((\nb \times b) \times \bgB) \right)^{\bt} \\
	&=  \frac{1}{2} \left( \bgB \cdot \nb (\nb \times)^{\alp}_{\bt} b^{\bt}  - (\rd_{\gmm} \bgB^{\alp}) (\nb \times)^{\gmm}_{\bt} b^{\bt}  \right) \\
	&\peq \pm \frac{1}{2} \abs{\nb}^{-1}(-\lap) \left( (\bgB \cdot \nb) b^{\alp} - \bgB_{\bt} \rd^{\alp} b^{\bt} \right)  \\
	&\peq \pm \frac{1}{2} \abs{\nb}^{-1} \rd^{\alp} \rd_{\gmm} \left( (\bgB \cdot \nb) b^{\gmm} - \bgB_{\bt} \rd^{\gmm} b^{\bt} \right)   \\
	&= \frac{1}{2} \left( \bgB \cdot \nb (\nb \times)^{\alp}_{\bt} b^{\bt} \pm \abs{\nb} (\bgB \cdot \nb) \dlt^{\alp}_{\bt} b^{\bt}\right)  \\
	&\peq + \frac{1}{2} \left( - (\rd_{\gmm} \bgB^{\alp}) (\nb \times)^{\gmm}_{\bt} b^{\bt} \pm \abs{\nb} (\rd^{\alp} \bgB_{\gmm}) \dlt^{\gmm}_{\bt} b^{\bt} \right)  \\
	&\peq \pm \frac{1}{2} \rd^{\alp} \abs{\nb}^{-1} \left( \lap (\bgB_{\bt} b^{\bt}) + (\rd_{\bt} \bgB^{\gmm} )\rd_{\gmm} b^{\bt} - \rd_{\gmm} (\bgB_{\bt} \rd^{\gmm} b^{\bt}) \right) \\
	&= \frac{1}{2} \left( \bgB \cdot \nb \abs{\nb} (\pm b_{\pm} \mp b_{\mp})^{\alp} \pm \abs{\nb} (\bgB \cdot \nb) (b_{\pm} + b_{\mp})^{\alp}\right)  \\
	&\peq + \frac{1}{2} \left( - (\rd_{\bt} \bgB^{\alp}) \abs{\nb} (\pm b_{\pm} \mp b_{\mp})^{\bt} \pm \abs{\nb} (\rd^{\alp} \bgB_{\bt}) (b_{\pm} + b_{\mp})^{\bt} \right)  \\
	&\peq \pm \rd^{\alp} \left( \frac{1}{2} \frac{\rd_{\gmm}}{\abs{\nb}}\left(  (\rd^{\gmm} \bgB_{\bt} + \rd_{\bt} \bgB^{\gmm}) (b_{\pm} + b_{\mp})^{\bt} \right) \right)
\end{align*}
After rearranging terms, we obtain
\begin{align*}
	\left(\Pi_{\pm}(D) \linmain_{\bgB} b \right)^{\alp}
	&= \pm \frac{1}{2} (\bgB \cdot \nb \abs{\nb} + \abs{\nb} (\bgB \cdot \nb)) b_{\pm}^{\alp} \\
	&\peq \pm \frac{1}{2} [\abs{\nb}, \bgB \cdot \nb] b_{\mp}^{\alp} \\
	&\peq \pm \frac{1}{2} \left(\abs{\nb} (\rd^{\alp} \bgB_{\bt}) - (\rd_{\bt} \bgB^{\alp}) \abs{\nb} \right) b_{\pm}^{\bt} \\
	&\peq \pm \frac{1}{2} \left(\abs{\nb} (\rd^{\alp} \bgB_{\bt}) + (\rd_{\bt} \bgB^{\alp}) \abs{\nb} \right) b_{\mp}^{\bt} \\
	&\peq \pm \rd^{\alp} \left( \frac{1}{2} \frac{\rd_{\gmm}}{\abs{\nb}}\left(  (\rd^{\gmm} \bgB_{\bt} + \rd_{\bt} \bgB^{\gmm}) (b_{\pm} + b_{\mp})^{\bt} \right) \right)
\end{align*}
On the other hand,
\begin{align*}
\left( \Pi_{\pm}(D) \linlower_{\bgB} b \right)^{\alp}
& =\left( \Pi_{\pm}(D)\left( \nb \times ((\nb \times \bgB) \times b)\right)\right)^{\alp} \\
& = \tensor{\Pi_{\pm}(D)}{^{\alp}_{\bt}} \left((\nb \times \bgB) \cdot \nb b^{\bt} - (\rd_{\gmm} (\nb \times \bgB)^{\alp}) b^{\gmm}\right) \\
& = (\nb \times \bgB) \cdot \nb b_{\pm}^{\alp} \\
& \peq + \tensor{[\Pi_{\pm}(D)}{^{\alp}_{\bt}}, (\nb \times \bgB) \cdot \nb] (b_{\pm} + b_{\mp})^{\bt} \\
& \peq - \tensor{\Pi_{\pm}(D)}{^{\alp}_{\gmm}}  (\rd_{\bt} (\nb \times \bgB)^{\gmm}) (b_{\pm} + b_{\mp})^{\bt},
\end{align*}
which completes the proof. \qedhere
\end{proof}

\section{Analysis of bicharacteristics for linearized equations} \label{sec:prin-symbol}
In this section, we assume that $\bfB_{0}$ is a smooth, time-independent and divergence-free vector field on $\bbR^{3}$. In particular, it need not be a solution to \eqref{eq:e-mhd}.

\subsection{Basic properties of the bicharacteristic flow} \label{subsec:ham-flow}
The Hamiltonian flow equation associated to the scalar symbol $\dprin_{\bfB_{0}}(x, \xi)$ is
\begin{equation} \label{eq:h-flow}
	\left\{
	\begin{aligned}
		\dot{X}^{\alp} =& \rd_{\xi_{\alp}} \dprin_{\bfB_{0}} (X, \Xi)= \bfB_{0}^{\bt} (X) \abs{\Xi} \left( \dlt^{\alp}_{\bt} + \frac{\Xi_{\bt} \Xi_{\alp}}{\abs{\Xi}^{2}} \right), \\
		\dot{\Xi}_{\alp} =& - \rd_{x^{\alp}} \dprin_{\bfB_{0}} (X, \Xi) = - \rd_{\alp} \bfB_{0}^{\bt}(X) \Xi_{\bt} \abs{\Xi},
	\end{aligned}
	\right.
\end{equation}
where $\alp = 1, 2, 3$. An integral curve $(X, \Xi)(t)$ of this flow is called a \emph{bicharacteristic}; its projection $X(t)$ to $M$ is called a \emph{projected characteristic curve}.

We begin by identifying the possible set of directions of $X(t)$ in relation to $\bfB_{0}(X(t))$.
\begin{lemma} [Cone of directions around $\bfB_{0}$] \label{lem:cone-dir}
	Along any bicharacteristic $(X, \Xi)(t)$ associated with $\pm \dprin_{\bfB_{0}}$, we have
	\begin{equation} \label{eq:cone-dir}
		\abs{\angle(\dot{X}(t), \pm \bfB_{0}(X(t)))} \leq \tan^{-1} \frac{1}{2 \sqrt{2}}.
	\end{equation}
	In particular, we have
	\begin{equation} \label{eq:speed-lower}
		\abs{\bfB_{0}(X(t))} \abs{\Xi(t)} \leq \abs{\dot{X}(t)}.
	\end{equation}
	Moreover, if $\abs{\bfB_{0} - \bfe_{3}} < \frac{1}{2}$, then
	\begin{equation} \label{eq:speed3-lower}
		2 \max\set{\abs{\dot{X}^{1}(t)}, \abs{\dot{X}^{2}(t)}} \leq \abs{\Xi(t)} \leq  {12} \abs{\dot{X}^{3}(t)}. 
	\end{equation}
\end{lemma}
\begin{proof}
	Without loss of generality, we fix the sign $\pm \dprin_{\bfB_{0}} = + \dprin_{\bfB_{0}}$ and consider $t = 0$. Let $(x, \xi) = (X, \Xi)(0)$. Since the case $\bfB_{0}(x) = 0$ is trivial, we may assume that $\bfB_{0}(x) \neq 0$. Rotate the (local) coordinate axes so that $\bfB_{0}(x)$ is aligned with $\bfe_{3} = \rd_{x^{3}}$; by a further rotation about the (new) $x^{3}$-axis, we may also set $\xi_{1} = 0$. Clearly, $\abs{\angle(\dot{X}(0), \bfB_{0}(x))} < \frac{\pi}{2}$ and 
	\begin{equation*}
		\tan \angle(\dot{X}(0), \bfB_{0}(x)) = \frac{\dot{X}^{2}(0)}{\dot{X}^{3}(0)} = \frac{\xi_{2} \xi_{3}}{\xi_{2}^{2} + 2 \xi_{3}^{2}}.
	\end{equation*}
	By symmetry, it suffices to compute the maximum of the last expression. By elementary calculus, we have
	\begin{equation*}
		\max \frac{\xi_{2} \xi_{3}}{\xi_{2}^{2} + 2 \xi_{3}^{2}}
		= \max \frac{(\xi_{3}/\xi_{2})}{1+ 2 (\xi_{3}/\xi_{2})^{2}}
		= \frac{1}{2 \sqrt{2}},
	\end{equation*}
	where the maximum occurs at $\xi_{3}/\xi_{2} = \frac{1}{\sqrt{2}}$. Next, \eqref{eq:speed-lower} follows by the computation
	\begin{equation*}
		\abs{\bfB_{0}(x)}^{2}\abs{\xi}^{2}
		\leq \abs{\bfB_{0}(x)}^{2}\abs{\xi}^{2} \frac{(\xi_{2}^{2} + 2 \xi_{3}^{2})^{2} + \xi_{2}^{2} \xi_{3}^{2}}{\abs{\xi}^{2}}  = \abs{\dot{X}(0)}^{2}. 
	\end{equation*}
	Finally, if $\abs{\bfB_{0}(x) - \bfe_{3}} < \frac{1}{2}$, then 
	\begin{equation*}
		\abs{\xi} \leq 2 (1 + \bfB_{0}^{3}(x)) \abs{\xi} \frac{\abs{\xi}^{2}+\xi_{3}^{2}}{\abs{\xi}^{2}}  {\le 6\bfB_{0}^{3}(x)\abs{\xi} \frac{\abs{\xi}^{2}+\xi_{3}^{2}}{\abs{\xi}^{2}} \le 12 \dot{X}^{3}(0)},
	\end{equation*}
	which proves the second inequality of \eqref{eq:speed3-lower}. The first inequality of \eqref{eq:speed3-lower} is straightforward to prove. \qedhere
\end{proof}

Next, we derive an evolution equation for the magnitude of the canonical momentum $\Xi$.
\begin{lemma} [Formula for $\frac{\ud}{\ud t} \abs{\Xi}$] \label{lem:freq-evol}
	Along any bicharacteristic $(X, \Xi)(t)$ associated with $\pm \dprin_{\bfB_{0}}$,
	\begin{equation} \label{eq:freq-evol}
		\frac{\ud}{\ud t} \abs{\Xi} = \mp \dfrm{\bfB_{0}}^{\alp \bt}(X) \Xi_{\alp} \Xi_{\bt},
	\end{equation}
	where $\dfrm{\bfB_{0}}^{\alp \bt} = \frac{1}{2} (\rd^{\alp} \bfB_{0}^{\bt} + \rd^{\bt} \bfB_{0}^{\alp})$.
\end{lemma}
\begin{proof}
	For simplicity, we fix the sign $\pm \dprin_{\bfB_{0}} = + \dprin_{\bfB_{0}}$. Using the Hamiltonian flow equation, we compute
	\begin{align*}
		\frac{\ud}{\ud t} \abs{\Xi}^{2}
		= 2 \dlt^{\alp \bt} \Xi_{\alp} \dot{\Xi}_{\bt} 
		=- 2 \dlt^{\alp \bt} \xi_{\alp} \rd_{\bt} \bfB_{0}^{\gmm}(X) \Xi_{\gmm} \abs{\Xi}
		=- (\nb^{\alp} \bfB_{0}^{\gmm} + \nb^{\gmm} \bfB_{0}^{\alp})(X) \Xi_{\alp} \Xi_{\gmm} \abs{\Xi}.
	\end{align*}
	Recalling the definition $\dfrm{\bfB_{0}}^{\alp \bt} = \frac{1}{2}(\nb^{\alp} \bfB_{0}^{\bt} + \nb^{\bt} \bfB_{0}^{\alp})$, \eqref{eq:freq-evol} follows.
\end{proof}

\subsection{Quantification of the assumptions}
We are now ready to make the qualitative hypotheses in Theorem~\ref{thm:main-simple} quantitative -- these will be the precise parameters on which the lower bound on the lifespan $T$ depends. The following definition is central to the precise formulation of our local wellposedness result.
\begin{definition}[Quantification of asymptotic uniformity, nondegeneracy and nontrapping] \label{def:nontrapping-class}
	Let $s > \frac{7}{2}$ and $R \ge 1$. Furthermore, let $\eps > 0$ satisfy $\eps<\min\{ \eps_{0}, \alp_{1}M \}$ where $\eps_{0},\alp_{1}>0$ are some absolute constants that will be fixed later. We say that $\bfB_{0} : \bbR^{3} \to \bbR^{3}$ belongs to the space $\calB^{s}_{\eps} (M, \mu, A, R, L)$ if the following holds:
	\begin{enumerate}
		\item {\bf Size bound.}
		\begin{equation} \label{eq:size-M}
			\nrm{\bfB_{0} - \bfe_{3}}_{\ell^{1}_{\calI} H^{s}} < M,
		\end{equation}
		\item {\bf Nondegeneracy.}
		\begin{equation} \label{eq:nondegen-sgm}
			\inf_{\bbR^{3}} \abs{\bfB_{0}} > \mu,
		\end{equation}
		\item {\bf Takeuchi--Mizohota-type condition.} For every nontrivial bicharacteristic $(X, \Xi)(t)$ associated with $\dprin_{\bfB_{0}}$, we have
		\begin{equation} \label{eq:mizohata-A}
			\int_{-\infty}^{\infty} \abs{\nb \bfB_{0}(X(t))} \abs{\Xi(t)}\, \ud t < A,
		\end{equation}
		
		\item {\bf $\eps$-asymptotic uniformity.}
		\begin{equation} \label{eq:asymp-unif-R}
			\begin{aligned}
				\nrm{\chi_{>R}(\abs{x^{3}}) (\bfB_{0} - \bfe_{3})}_{\ell^{1}_{\calI} H^{s}} < \eps,
			\end{aligned}
		\end{equation}
		\item {\bf $R$-nontrapping.} For every nontrivial bicharacteristic $(X, \Xi)(t)$ associated with $\dprin_{\bfB_{0}}$, we have
		\begin{equation} \label{eq:nontrapping-L}
			\int_{-\infty}^{\infty} \chf_{\set{-2 R < x^{3} < 2 R}}(X(t)) \abs{\dot{X}(t)} \, \ud t < L.
		\end{equation}
	\end{enumerate}
\end{definition}
\begin{remark}[On the definition of $\calB^{s}_{\eps}(M, \mu, A, R, L)$]
	Note that $M$, $\mu$ and $A$ are independent of the choice of $\eps$, whereas $R$ and $L$ depend on $\eps$. Observe also that $\mu$, $A$, $R$ and $L$ are \emph{not} simply controlled in terms of the size bound $M$ in \eqref{eq:size-M} unless $M$ is sufficiently small. For this reason, we interpret these parameters as quantifying various properties of $\bfB_{0}$ in the large data regime. In our precise formulation of the main theorem, we will fix $\eps$ depending on $M$, $\mu$ and $A$, and the time of existence $T$ will depend only on the (resulting) parameters $s$, $M$, $\mu$, $A$, $R$, $L$ for the initial data $\bfB_{0}$.
\end{remark}

Up to a choice of a small parameter $\eps$, Definition~\ref{def:nontrapping-class} makes the qualitative assumptions for $\bfB_{0}$ in Theorem~\ref{thm:main-simple} quantitative, as the following lemma and corollary show.
\begin{lemma}[Qualitative assumptions imply quantitative assumptions] \label{lem:nontrapping-id}
	Let $\bfB_{0} \in \bfe_{3} + \ell^{1}_{\calI} H^{s}$ satisfy \eqref{eq:size-M} for some $M$, as well as the hypotheses of Theorem~\ref{thm:main-simple}. Then the following statements hold.
	\begin{enumerate}
		\item There exists $\mu > 0$ such that \eqref{eq:nondegen-sgm} holds.
		\item Given any $\eps > 0$ there exist $R$ and $L$ such that \eqref{eq:asymp-unif-R} and \eqref{eq:nontrapping-L} hold. 
		\item There exists a constant $\eps_{0} = \eps_{0}(s) > 0$ such that if \eqref{eq:asymp-unif-R} and \eqref{eq:nontrapping-L} hold with $\eps < \eps_{0}$, then
		\begin{gather}
			\nrm{\chi_{>R}(\abs{x^{3}})(\bfB_{0} - \bfe_{3})}_{L^{1}_{x^{3}} L^{\infty}_{x^{1}, x^{2}} \cap L^{\infty}} + \nrm{\chi_{>R}(\abs{x^{3}})\nb \bfB_{0}}_{L^{1}_{x^{3}} L^{\infty}_{x^{1}, x^{2}}} + \nrm{\chi_{>R}(\abs{x^{3}}) \nb^{2} \bfB_{0}}_{L^{\infty}} < \frac{1}{100}, \label{eq:asymp-unif-1/2}\\
			\int_{-\infty}^{\infty} \abs{\nb \bfB_{0}(X(t))} \abs{\Xi(t)} \, \ud t \aleq M(1+ \mu^{-1} L). \label{eq:mizohata-est}
		\end{gather}
	\end{enumerate}
\end{lemma}
\begin{corollary} \label{cor:nontrapping-id}
	Let $\bfB_{0} \in \bfe_{3} + \ell^{1}_{\calI} H^{s}$ satisfy \eqref{eq:size-M} for some $M$, as well as the hypotheses of Theorem~\ref{thm:main-simple}. Then the following statements hold.
	\begin{enumerate}
		\item There exist $\mu, A > 0$ such that \eqref{eq:nondegen-sgm} and \eqref{eq:mizohata-A} hold. 
		\item Given any $\eps > 0$, there exist $R, L > 0$  such that $\bfB_{0} \in \calB^{s}_{\eps}(M, \mu, A, R, L)$.
	\end{enumerate}
\end{corollary}
We omit the proof of the corollary, which is obvious.
\begin{proof}[Proof of Lemma~\ref{lem:nontrapping-id}]
	Statements~(1) and (2) follow from $\limsup_{R \to \infty} \nrm{\chi_{>R}(x^{3}) (\bfB_{0} - \bfe_{3})}_{\ell^{1} H^{s}} = 0$ (in view of \eqref{eq:size-M}), as well as the nondegeneracy and the nontrapping assumptions. For Statement~(3), we choose $\eps_{0}$ sufficiently small compared to the constant in the Sobolev embedding $H^{s-2} \hookrightarrow L^{\infty}$ so that \eqref{eq:asymp-unif-1/2} holds. To establish \eqref{eq:mizohata-est}, note that
	\begin{align*}
		{\int_{-\infty}^{\infty} \chf_{\set{-R \le x^{3} \le R}}(X(t))} \abs{\nb \bfB_{0}(X(t))} \abs{\Xi(t)} \, \ud t \aleq \mu^{-1} M  {\int_{-\infty}^{\infty} \chf_{\set{-R \le x^{3} \le R}}(X(t))} \abs{\dot{X}(t)} \, \ud t \aleq \mu^{-1} M L.
	\end{align*}
	where we used \eqref{eq:speed-lower} in the first inequality. On the other hand, in $\set{x^{3} >  {R}}$, we use \eqref{eq:speed3-lower} (which is possible thanks to \eqref{eq:asymp-unif-1/2}) to estimate
	\begin{align*}
		{\int_{-\infty}^{\infty} \chf_{\set{R < x^{3} }}(X(t))} \abs{\nb \bfB_{0}(X(t))} \abs{\Xi(t)} \, \ud t \aleq \sum_{I \in \calI: I \cap (R, \infty) \neq \0} \int_{I} \nrm{\nb \bfB_{0}(x^{1}, x^{2}, x^{3})}_{L^{\infty}_{x^{1}, x^{2}}} \, \ud x^{3} \aleq M.
	\end{align*}
	The region $\set{x^{3} <  {-R}}$ is treated similarly, which completes the proof of \eqref{eq:mizohata-est}.
	\qedhere
\end{proof}
We also record some consequences of $\bfB_{0} \in \calB^{s}_{\eps}(M, \mu, A, R, L)$.
\begin{lemma} \label{lem:nontrapping-cor}
	Let $s > \frac{7}{2}$, $M > 0$, $\mu > 0$, $\eps > 0$, $R > 0$, $L > 0$ and $\bfB_{0} \in \calB^{s}_{\eps}(M, \mu, A, R, L)$. Provided that $\eps < \eps_{0}$ (where $\eps_{0}$ is as in Lemma~\ref{lem:nontrapping-id}.(3)), we have
	\begin{gather}
		\sup_{t, t' \in \bbR} \frac{\abs{\Xi(t)}}{\abs{\Xi(t')}} \aleq A, \label{eq:nontrapping-freq-bnd} \\
		\frac{1}{12} \abs{\Xi(t)} \leq \abs{\dot{X}^{3}(t)} \leq \abs{\dot{X}(t)}  \leq 2 \abs{\Xi(t)} \qquad \hbox{ if } X^{3}(t) > R \hbox{ or } X^{3}(t) < -R, \label{eq:nontrapping-v}
	\end{gather}
	\begin{align}
		\int_{0}^{t} \abs{\dot{X}(t')} \, \ud t' &\aleq \begin{cases}
			X^{3}(t) - X^{3}(0) & \hbox{ if } X^{3}(0) > R \hbox{ or } X^{3}(t) < - R, \\
			L + (\abs{X^{3}(t)} - R)_{+} + (\abs{X^{3}(0)} - R)_{+}  & \hbox{ otherwise},
		\end{cases} \notag \\
		&\aleq L + (\abs{X^{3}(0)} - R)_{+} + (\abs{X^{3}(t)} - R)_{+}. \label{eq:nontrapping-length-bnd}
	\end{align}
\end{lemma}
\begin{proof}
	Estimate \eqref{eq:nontrapping-freq-bnd} follows from Lemma~\ref{lem:freq-evol} and \eqref{eq:mizohata-A}. Estimate \eqref{eq:nontrapping-v} follows from Lemma~\ref{lem:cone-dir} and \eqref{eq:asymp-unif-1/2}. To establish \eqref{eq:nontrapping-length-bnd}, note that $X(t')$ for $0 \leq t' \leq t$ does not intersect $\set{-R < x^{3} < R}$ if $X^{3}(0) > R$ or $X^{3}(t) < -R$. Now \eqref{eq:nontrapping-length-bnd} follows from \eqref{eq:nontrapping-v} and \eqref{eq:nontrapping-L}. \qedhere
\end{proof}
Importantly, the class $\calB^{s}_{\eps}(M, \mu, A, R, L)$ is stable under small perturbations of $\bfB_{0}$ in $\ell^{1}_{\calI} H^{s}$.
\begin{proposition}[$\calB^{s}_{\eps}(M, \mu, A, R, L)$ is stable] \label{prop:nontrapping-stable}
	Let $s > \frac{7}{2}$, $M > 0$, $\mu > 0$, $A > 0$, $\eps > 0$, $R \ge 1$, $L > 0$ and $\br{\bfB}_{0} \in \calB^{s}_{\eps}(M, \mu, A, R, L)$. Then the following holds:
	\begin{enumerate}
		\item $\calB^{s}_{\eps}(M, \mu, A, R, L)$ is an open subset of $\calB^{s}(M)$. 
		\item Moreover, there exist $c(s, M, \mu, A), C(s, M, \mu) > 0$ and if $\eps < \frac{1}{2} \eps_{0}$ (where $\eps_{0}$ is as in Lemma~\ref{lem:nontrapping-id}) and $\bfB_{0} : \bbR^{3} \to \bbR^{3}$ satisfies
		\begin{equation*}
			\nrm{\bfB_{0} - \br{\bfB}_{0}}_{\ell^{1}_{\calI} H^{s}} < c(s, M, \mu, A, \eps) e^{- C(s, M, \mu, A) L}
		\end{equation*}
		then $\bfB_{0} \in \calB^{s}_{2 \eps}(2 M, \frac{1}{2} \mu, 2 {\max}\set{A, M}, R,  2 L)$.
	\end{enumerate}
\end{proposition}
We defer the proof of this proposition until the end of Section~\ref{subsec:lin-h-flow}, as we need to develop some tools for controlling the variation of the bicharacteristics when perturbing $\bfB_{0}$. 

We also formulate and prove a lemma that will be needed in the proof of Theorem~\ref{thm:main} (in particular, persistence of regularity).
\begin{lemma} \label{lem:nontrap-unif}
	For $s > \frac{7}{2}$, $0 < \eps < \eps_{0}$ (where $\eps_{0}$ is as in Lemma~\ref{lem:nontrapping-id}.(3)) and $M, \mu, A, R, L > 0$, suppose that $\bfB \in C_{t}([0, T]; \ell^{1}_{\calI} H^{s})$ and $\bfB(t) \in \calB^{s}_{\eps}(M, \mu, A, R, L)$ for all $t \in [0, T]$. Then for any $\eps' > 0$, there exist $R', L' > 0$, which are independent of $t$, such that $\bfB(t) \in \calB^{s}_{\eps'}(M, \mu, A, R', L')$ for all $t \in [0, T]$.
\end{lemma}
\begin{proof}
	The case $\eps' \geq \eps$ is trivial, so it suffices to assume that $\eps' < \eps$. The existence of a $t$-independent $R'$ such that \eqref{eq:asymp-unif-R} holds (with $\eps'$ and $R'$) for all $t \in [0, T]$ follows from the continuity of the function $\nrm{\chi_{>R}(x^{3}) \bfB(t)}_{\ell^{1}_{\calI} H^{s}}$ in $R$ and $t$. Then, in view of Lemma~\ref{lem:nontrapping-cor}, it follows that $L' = L + C (R' - R)_{+}$ for some $C > 0$ works. \qedhere
\end{proof}

\subsection{Variation of bicharacteristics} \label{subsec:lin-h-flow}
Here we study the linearization of \eqref{eq:h-flow} with forcing terms $F, G$:
\begin{equation} \label{eq:lin-h-flow}
	\begin{aligned}
		\frac{\ud}{\ud t}
		\begin{pmatrix}
			X' \\ \Xi'
		\end{pmatrix}
		= 
		\begin{pmatrix}
			\nb_{x}^{\top} \nb_{\xi} \dprin_{\bfB_{0}} & \nb_{\xi}^{\top} \nb_{\xi} \dprin_{\bfB_{0}} \\
			- \nb_{x}^{\top} \nb_{x} \dprin_{\bfB_{0}} & - \nb_{\xi}^{\top} \nb_{x} \dprin_{\bfB_{0}}
		\end{pmatrix}_{(x, \xi) = (X, \Xi)}
		\begin{pmatrix}
			X' \\ \Xi'
		\end{pmatrix}
		+ \begin{pmatrix}
			F \\ G
		\end{pmatrix}
	\end{aligned}
\end{equation}
Given a bicharacteristic $(X, \Xi)$, we need to analyze \eqref{eq:lin-h-flow} to control the nearby bicharacteristics. The following simple Gr\"onwall estimate for \eqref{eq:lin-h-flow} will be sufficient for our purpose.

\begin{lemma} \label{lem:lin-bichar}
	Let $\bfB_{0} \in \calB_{\eps}^{s}(M, \mu, A, R, L)$ for $s > \frac{7}{2}$. We have
	\begin{equation} \label{eq:lin-bichar}
		\begin{aligned}
			\abs{X'(t)} + \frac{1}{\abs{\Xi(0)}} \abs{\Xi'(t)} 
			&\aleq \left( \abs{X'(0)} + \frac{1}{\abs{\Xi(0)}} \abs{\Xi'(0)} 
			+ \int_{0}^{t} \left( \abs{F(t')} + \frac{1}{\abs{\Xi(0)}} \abs{G(t')} \right) \, \ud t' \right)  \\
			&\phantom{\aleq} \times \exp \left( C \mu^{-1} M A \int_{0}^{t} \abs{\dot{X}(t')} \, \ud t' \right).
	\end{aligned}\end{equation}
\end{lemma}

\begin{proof}
	In view of Lemma~\ref{lem:freq-evol} and \eqref{eq:mizohata-A}, note that 
	\begin{equation} \label{eq:h-flow-freq-control}
		\sup_{t, t' \in \bbR} \frac{\abs{\Xi(t)}}{\abs{\Xi(t')}} \aleq A.
	\end{equation}
	Let $\lmb = \abs{\Xi(0)}$. We rescale \eqref{eq:lin-h-flow} in the following way:
	\begin{align}\label{eq:lin-h-flow-rescale}
		\frac{1}{\lmb} \frac{\ud}{\ud t}
		\begin{pmatrix}
			X' \\ \frac{1}{\lmb} \Xi'
		\end{pmatrix}
		= 
		\begin{pmatrix}
			\frac{1}{\lmb} \nb_{x}^{\top} \nb_{\xi} \dprin_{\bfB_{0}} & \nb_{\xi}^{\top} \nb_{\xi} \dprin_{\bfB_{0}} \\
			- \frac{1}{\lmb^{2}} \nb_{x}^{\top} \nb_{x} \dprin_{\bfB_{0}} & - \frac{1}{\lmb} \nb_{\xi}^{\top} \nb_{x} \dprin_{\bfB_{0}}
		\end{pmatrix}_{(x, \xi) = (X, \Xi)}
		\begin{pmatrix}
			X' \\ \frac{1}{\lmb}\Xi'
		\end{pmatrix}
		+ \frac{1}{\lmb} \begin{pmatrix}
			F \\ \frac{1}{\lmb} G
		\end{pmatrix}.
	\end{align}
	In view of \eqref{eq:h-flow-freq-control}, the coefficient matrix of this ODE is bounded by $C \lmb^{-1} M A \abs{\Xi(t)}$. It follows by Gr\"onwall's inequality that
	\begin{align*}
		\abs{X'(t)} + \frac{1}{\abs{\Xi(0)}} \abs{\Xi'(t)} 
		&\aleq \left( \abs{X'(0)} + \frac{1}{\abs{\Xi(0)}} \abs{\Xi'(0)} 
		+ \int_{0}^{t} \left( \abs{F(t')} + \frac{1}{\abs{\Xi(0)}} \abs{G(t')} \right) \, \ud t' \right)  \\
		&\phantom{\aleq} \times \exp \left( C M A \int_{0}^{t} \abs{\Xi(t')} \, \ud t' \right) .
	\end{align*}
	Using $\abs{\Xi(t')} \leq \mu^{-1} \abs{\dot{X}(t')}$, the desired inequality follows. \qedhere
\end{proof}

Using the preceding lemma, we obtain the following result concerning the variation of bicharacteristics under the change of the initial conditions.
\begin{proposition} \label{prop:d-bichar}
	Let $\bfB_{0} \in \calB^{s}_{\eps}(M, \mu, A, R, L)$ for $s > \frac{7}{2}$. Given $(x, \xi) \in T^{\ast} \bbR^{3}$, denote by $(X(t; x, \xi), \Xi(t; x, \xi))$ the bicharacteristic (i.e., the solution to \eqref{eq:h-flow}) with the initial conditions 
	\begin{equation*}
		(X(0; x, \xi), \Xi(0; x, \xi)) = (x, \xi).
	\end{equation*}
	Then the following statements hold.
	\begin{enumerate}
		\item We have
		\begin{equation} \label{eq:d-bichar-1}
			\begin{aligned}
				& \abs{\rd_{x} X(t; x, \xi)}
				+ \abs{\xi} \abs{\rd_{\xi} X(t; x, \xi)}
				+ \abs{\xi}^{-1} \abs{\rd_{x} \Xi(t; x, \xi)}
				+ \abs{\rd_{\xi} \Xi(t; x, \xi)} \\
				&\aleq \exp \left( C \mu^{-1} M A (L + (\abs{X^{3}(t)} - R)_{+} + (\abs{X^{3}(0)} - R)_{+}) \right) .
			\end{aligned}
		\end{equation}
		\item Assume, in addition, that $\bfB_{0} = P_{<k_{0}} \bfB_{0}$. Then we have
		\begin{equation} \label{eq:d-bichar-high}
			\begin{aligned}
				& \abs{\xi}^{\abs{\bfbt}} \abs{\rd_{x}^{\bfalp} \rd_{\xi}^{\bfbt} X(t; x, \xi)}
				+ \abs{\xi}^{\abs{\bfbt}-1} \abs{\rd_{x}^{\bfalp} \rd_{\xi}^{\bfbt} \Xi(t; x, \xi)} \\
				&\aleq_{\bfalp, \bfbt} 2^{(\abs{\bfalp}-1)_{+}  k_{0}} A^{(\abs{\bfbt}-1)_{+}} \exp \left( C_{\bfalp, \bfbt} \mu^{-1} M A (L + (\abs{X^{3}(t)} - R)_{+} + (\abs{X^{3}(0)} - R)_{+}) \right) .
			\end{aligned}
		\end{equation}
	\end{enumerate}
	
\end{proposition}

\begin{proof}
	Observe that $(\rd_{x^{j}} X(t; x, \xi), \rd_{x^{j}} \Xi(t; x, \xi))$ solves \eqref{eq:lin-h-flow} with 
	\begin{equation*}
		(X'(0), \Xi'(0), F, G) = (\bfe_{j}, 0, 0, 0).
	\end{equation*}
	Similarly, $(\rd_{\xi_{j}} X(t; x, \xi), \rd_{\xi_{j}} \Xi(t; x, \xi))$ solves \eqref{eq:lin-h-flow} with 
	\begin{equation*}
		(X'(0), \Xi'(0), F, G) = (0, \bfe_{j}, 0, 0). 
	\end{equation*}
	Hence \eqref{eq:d-bichar-1} follows from Lemma~\ref{lem:lin-bichar} and \eqref{eq:nontrapping-length-bnd}. Note that \eqref{eq:d-bichar-1} is simply \eqref{eq:d-bichar-high} with $\abs{\bfalp} + \abs{\bfbt} = 1$. We summarize the proof of \eqref{eq:d-bichar-high} in the case $\abs{\bfalp} + \abs{\bfbt} = 2$; the general case follows from induction. Note that $(\rd_{x}^{\bfalp} \rd_{\xi}^{\bfbt} X, \rd_{x}^{\bfalp} \rd_{\xi}^{\bfbt} \Xi)$ solves \eqref{eq:lin-h-flow} with
	\begin{equation*}
		(\rd_{x}^{\bfalp} \rd_{\xi}^{\bfbt} X(0), \rd_{x}^{\bfalp} \rd_{\xi}^{\bfbt} \Xi(0), F, G) = (0,0, F_{\bfalp, \bfbt}, G_{\bfalp, \bfbt}),
	\end{equation*} for some $F_{\bfalp, \bfbt}, G_{\bfalp, \bfbt}$, where the initial conditions vanish since $\rd_{x}^{\bfalp} \rd_{\xi}^{\bfbt} x = \rd_{x}^{\bfalp} \rd_{\xi}^{\bfbt} \xi = 0$ if $\abs{\bfalp} + \abs{\bfbt} > 1$. Moreover, as in the proof of Lemma~\ref{lem:lin-bichar}, we have, for $\abs{\bfalp'} + \abs{\bfbt'} \geq 2$,
	\begin{align}\label{eq:d-bichar-ind}
		\abs{\xi}^{\abs{\bfbt'}-2} \abs{\rd_{x}^{\bfalp'} \rd_{\xi}^{\bfbt'}  p_{(\bfB_{0})_{<k_{0}}} (X, \Xi)} 
		\aleq_{\bfalp', \bfbt'} 2^{(\abs{\bfalp'}-2)_{+} k_{0}} A^{(\abs{\bfbt'}-2)_{+}} A^{2} M .
	\end{align}
	When $|\bfalp|=2$ and $|\bfbt|=0$, $|F_{\bfalp,\bfbt}|$ and $|\xi|^{-1}|G_{\bfalp,\bfbt}|$ are bounded by \begin{equation*}
		\begin{split}
			|F_{\bfalp,\bfbt}|  & \lesssim |\nb_{x}^{2} \nb_{\xi} p(X,\Xi)| |\nb_{x}X|^{2} +  |\xi||\nb_{x} \nb_{\xi}^{2} p(X,\Xi)| |\nb_{x}X| ( |\xi|^{-1} |\nb_{x}\Xi|) +  |\xi|^{2}|\nb_{\xi}^{3} p(X,\Xi)|( |\xi|^{-1} |\nb_{x}\Xi|)^{2} 
		\end{split}
	\end{equation*} and \begin{equation*}
		\begin{split}
			|\xi|^{-1}|G_{\bfalp,\bfbt}|  & \lesssim |\xi|^{-1}|\nb_{x}^{3}  p(X,\Xi)| |\nb_{x}X|^{2} +  |\nb_{x}^2 \nb_{\xi}  p(X,\Xi)| |\nb_{x}X| ( |\xi|^{-1} |\nb_{x}\Xi|) +  |\xi| |\nb_{\xi}^{2} \nb_{x} p(X,\Xi)| ( |\xi|^{-1} |\nb_{x}\Xi|)^{2} 
		\end{split}
	\end{equation*} where $p=p_{(\bfB_{0})_{<k_{0}}} $. Then, we simply apply \eqref{eq:d-bichar-1} to bound $|\nb_xX|, |\xi|^{-1}|\nb_x\Xi|$ and \eqref{eq:d-bichar-ind} to bound the derivatives of $p$. This gives \begin{equation*}
		\begin{split}
			\int_{0}^{t} \left( |F_{\bfalp,\bfbt} |+|\xi|^{-1}|G_{\bfalp,\bfbt}| \right) \, \ud t' \lesssim_{\bfalp,\bfbt} 2^{k_{0}} \mu^{-1} M A (L + (\abs{X^{3}(t)} - R)_{+} + (\abs{X^{3}(0)} - R)_{+}),
		\end{split}
	\end{equation*} which gives \eqref{eq:d-bichar-high} after applying Lemma \ref{lem:lin-bichar} in this case. 
	The other cases of $|\bfalp|=|\bfbt|=1$ and $|\bfalp|=0,|\bfbt|=2$ can be handled similarly. 
\end{proof}

We are also ready to give a proof of Proposition~\ref{prop:nontrapping-stable}.
\begin{proof}[Proof of Proposition~\ref{prop:nontrapping-stable}]
	We need to establish the five properties listed in Definition \ref{def:nontrapping-class} for $\bfB_{0}$, under the assumption $\nrm{\bfB_{0} - \overline{\bfB}_{0}}_{\ell^{1}_{\calI} H^{s}} < \dlt$ with some choice of $\dlt = c(s,M,\mu,A,\eps)e^{-C(s,M,\nu,A)L} > 0$.
	
	To begin with, \eqref{eq:size-M} is trivial for $\dlt < M$ and \eqref{eq:nondegen-sgm} is immediate for $\dlt < \mu/2$ from the embedding $\ell^{1}_{\calI} H^{s}\subset L^{\infty}$. Similarly, \eqref{eq:asymp-unif-R} follows from taking $\dlt \le c(R)\eps$. From now on, we shall assume that $\dlt$ is sufficiently small with respect to $(M,\mu,R,\eps)$, so that in particular  \eqref{eq:size-M}, \eqref{eq:nondegen-sgm} and \eqref{eq:asymp-unif-R} holds.

	\medskip
	
	\noindent \textbf{Proof of \eqref{eq:nontrapping-L}}. We now proceed to prove \eqref{eq:nontrapping-L}, by taking $\dlt$ even smaller. We fix some $(x,\xi)$ with $\xi \ne 0$ and $-2R < x^{3} < 2R$. The latter can be assumed (by a translation of $t$ variable) without loss of generality since if either $x^{3} \ge 2R$ or $x^{3} \le -2R$, $\dot{X}^{3} \gtrsim 1$ as long as $X^{3}(t) \ge 2R$ or $X^{3}(t) \le -2R$ holds, respectively. Here $X(t),\Xi(t)$ is the bicharacteristic curve corresponding to $\overline{\bfB_{0}}$, and we shall take the time interval $\{ t_{0} < t < t_{1} \}$ where $X^{3}(t_{0}) = -2R, X^{3}(t_{1}) = 2R$. Note that $t_{0}<0<t_{1}$. 
	
	Let us introduce $\bfB_{0}^{(\sgm)} := \sgm \bfB_{0} + (1-\sgm) \br{\bfB}_{0}$ and denote $(X^{(\sgm)}, \Xi^{(\sgm)})(t; x, \xi)$ the bicharacteristics with initial conditions $(x, \xi)$ associated with $\bfB_{0}^{(\sgm)}$. For simplicity, let us write $(\bar{X},\bar{\Xi}) = (X^{(0)}, \Xi^{(0)})(t; x, \xi)$. Then $(X^{(\sgm) \prime}, \Xi^{(\sgm) \prime}) := \rd_{\sgm}(X^{(\sgm)}, \Xi^{(\sgm)})$ obeys the equation
	\begin{align*}
		\frac{\ud}{\ud t}
		\begin{pmatrix}
			X^{(\sgm) \prime} \\ \Xi^{(\sgm) \prime}
		\end{pmatrix}
		&= 
		\begin{pmatrix}
			\nb_{x}^{\top} \nb_{\xi} \dprin_{\bfB_{0}^{(\sgm)}} & \nb_{\xi}^{\top} \nb_{\xi} \dprin_{\bfB_{0}^{(\sgm)}}\\
			- \nb_{x}^{\top} \nb_{x} \dprin_{\bfB_{0}^{(\sgm)}}  & - \nb_{\xi}^{\top} \nb_{x} \dprin_{\bfB_{0}^{(\sgm)}} 
		\end{pmatrix}_{(x, \xi) = (X^{(\sgm)}, \Xi^{(\sgm)})}
		\begin{pmatrix}
			X^{(\sgm) \prime} \\ \Xi^{(\sgm) \prime}
		\end{pmatrix}  + 
		\begin{pmatrix}
			\nb_{\xi} \dprin_{\bfB_{0} - \br{\bfB}_{0}} \\
			- \nb_{x} \dprin_{\bfB_{0} - \br{\bfB}_{0}} 
		\end{pmatrix}_{(x, \xi) = (X^{(\sgm)}, \Xi^{(\sgm)})}
	\end{align*}
	with zero initial conditions. We are going to restrict the variable $t$ to $[t_{0},t_{1}]$ and estimate $(X^{(\sgm) \prime}, \Xi^{(\sgm) \prime})$, under the following \textbf{bootstrap assumptions}: \begin{equation}\label{eq:boot}
		\begin{split}
			\sup_{\sigma \in [0,1]} \sup_{t : |\bar{X}^3(t)| <2R } | X^{\sgm}(t) - \bar{X}(t) | \le \min\left\{ \frac{L}{100}, {\frac{1}{4}} \right\}, \qquad \sup_{\sigma \in [0,1]} \sup_{t : |\bar{X}^3(t)| <2R } \frac{|\Xi^{\sgm}(t)|}{|\bar{\Xi}(t)|} \le 10.  
		\end{split}
	\end{equation} From our assumption $R\ge1$ in Definition \ref{def:nontrapping-class}, \eqref{eq:boot} implies in particular that $| X^{\sgm}(t) - \bar{X}(t) | \le R/4$. 
	We now proceed as in the proof of Lemma \ref{lem:lin-bichar}. To begin with, we estimate the coefficient matrix, after rescaling the equation as in \eqref{eq:lin-h-flow-rescale}. Writing for simplicity  $\dprin =  \dprin_{\bfB_{0}^{(\sgm)}} (t,X^{(\sgm)},\Xi^{(\sgm)})$, we obtain under \eqref{eq:boot} that \begin{equation*}
		\begin{split}
			|\nb_{\xi}^{2}p| + |\xi|^{-2}|\nb^{2}_{x}p| + |\xi|^{-1}|\nb_{x}\nb_{\xi} p| \le 100C_{0} AM |\xi|^{-1}|\bar{\Xi}|. 
		\end{split}
	\end{equation*} Then, to estimate the forcing terms, we write for simplicity $\dprin_{\bfB_{0} - \br{\bfB}_{0}} =  \dprin_{\bfB_{0} - \br{\bfB}_{0}} (t,X^{(\sgm)},\Xi^{(\sgm)})$ and obtain that 
	\begin{equation*}
		\begin{split}
			|\nb_{\xi} \dprin_{\bfB_{0} - \br{\bfB}_{0}} | \le C_{1}\dlt|\Xi^{(\sgm)}| \le 10C_{1}\dlt|\bar{\Xi}| , \qquad 		\int_{0}^{t} |\nb_{\xi} \dprin_{\bfB_{0} - \br{\bfB}_{0}}|  \,\ud t' \le 20C_{1}\dlt\mu^{-1} \int_{0}^{t} |\dot{\bar{X}}| \, \ud t' .
		\end{split}
	\end{equation*} Similarly, using $|\Xi^{(\sgm)}| \le 10|\bar{\Xi}|$ and $|\xi|^{-1}|\bar{\Xi}|\le A$, we obtain \begin{equation*}
		\begin{split}
			\int_{0}^{t} |\nb_{x} \dprin_{\bfB_{0} - \br{\bfB}_{0}}|  \,\ud t' \le 100C_{2}\dlt A \mu^{-1} \int_{0}^{t} |\dot{\bar{X}}| \, \ud t' .
		\end{split}
	\end{equation*} With these bounds, following the proof of Lemma \ref{lem:lin-bichar} gives \begin{equation}\label{eq:dlt-star}
		\begin{split}
			|X^{(\sgm) \prime}| + |\xi|^{-1} 	|\Xi^{(\sgm) \prime}| \le C( 20C_{1} + 100C_{2} )A\dlt\mu^{-1} \int_{0}^{t} |\dot{\bar{X}}| \, \ud t' \, \exp\left( 100C C_{0} AM \mu^{-1}  \int_{0}^{t} |\dot{\bar{X}}| \, \ud t'  \right) =: \dlt^*.
		\end{split}
	\end{equation} By taking $\dlt$ small in the form \begin{equation*}
		\begin{split}
			\dlt \le c_{0} \mu(AL (A+L^{-1}))^{-1} \exp\left( - \frac{1}{c_{0}} AM\mu^{-1}L \right),
		\end{split}
	\end{equation*} we may obtain in particular that $|X^{(\sgm) \prime}| + |\xi|^{-1} 	|\Xi^{(\sgm) \prime}| \le \min\{ \frac{1}{100CA}, \frac{L}{200} \}$, where $C>0$ is the implicit constant from the estimate \eqref{eq:nontrapping-freq-bnd}. This gives for all $0<\sigma\le1$ that \begin{equation}\label{eq:variation-Xi}
		\begin{split}
			|\Xi^{(\sgm)} - \bar{\Xi}| \le |\xi|\int_{0}^{\sgm} |\xi|^{-1}|\Xi^{(\sgm') \prime}| \, \ud\sgm' \le \frac{|\xi|}{100CA} \le \frac{|\bar{\Xi}|}{2} \qquad \mbox{and} \qquad 		\frac{|\bar{\Xi}|}{2} \le |\Xi^{(\sgm)}| \le \frac{3|\bar{\Xi}|}{2}.
		\end{split}
	\end{equation} This in particular justifies the bootstrap assumption \eqref{eq:boot} for $\Xi^{(\sgm)}$. Similarly, we have \begin{equation}\label{eq:variation-X}
		\begin{split}
			|X^{(\sgm)} - \bar{X} | \le \int_{0}^{\sgm} |X^{(\sgm') \prime}| \, \ud\sgm' \le \min\left\{ \frac{L}{200}, \frac{1}{100CA} \right\} 
		\end{split}
	\end{equation} and this justifies \eqref{eq:boot} for $X^{(\sgm)} - \bar{X} $, since $A\ge \frac{1}{C}$ follows from \eqref{eq:nontrapping-freq-bnd}. The estimates \eqref{eq:variation-Xi}, \eqref{eq:variation-X} are valid as long as $t \in [t_{0},t_{1}]$.

	Now we split the integral \eqref{eq:nontrapping-L} for $\sigma = 1$ as 
	\begin{equation} \label{eq:nontrapping-L2}
		\left[ \int_{-\infty}^{t_{0}} + \int_{t_{0}}^{t_{1}}+ \int_{t_{1}}^{\infty} \right] \chf_{\set{-2 R < x^{3} < 2 R}}(X^{(1), 3}(t)) \abs{\dot{X}^{(1)}(t)} \, \ud t . 
	\end{equation} The last integral in $\{ t \ge t_{1} \}$ (which may be empty) can be easily estimated as follows. We know at $t = t_{1}$ that $X^{(1), 3} \ge 2R - \frac{R}{2} > R$ and therefore we have that for $t \ge t_{1}$, $  \dot{X}^{(1), 3} \ge 20|\dot{X}^{(1)}|$. But then \begin{equation*}
		\begin{split}
			\int_{t_{1}}^{\infty}  \chf_{\set{-2 R < x^{3} < 2 R}}(X^{(1)}(t)) \abs{\dot{X}^{(1), 3}(t)} \, \ud t \le 20\int_{ 2R - \frac{L}{200} < X^{3}(t) < 2R  } \dot{X}^{3}(t) \, \ud t  < \frac{L}{4}. 
		\end{split}
	\end{equation*} The integral in $t \in (-\infty, t_{0})$ can be estimated in the same manner. Finally, to estimate the second integral, we write \begin{equation}\label{eq:L-middle}
		\begin{split}
			\int_{t_{0}}^{t_{1}}  \chf_{\set{-2 R < x^{3} < 2 R}}(X^{(1), 3}(t)) \abs{\dot{X}^{(1)}(t)} \, \ud t \le 	 \int_{0}^{1} \left[ \int_{t_{0}}^{t_{1}}\abs{\dot{X}^{(\sigma) \prime }(t)} \, \ud t \right]  \ud \sigma . 
		\end{split}
	\end{equation} To estimate the {RHS} of \eqref{eq:L-middle}, we return to the equation for $\dot{X}^{(\sigma) \prime }$: \begin{equation*}
		\begin{split}
			\dot{X}^{(\sigma) \prime }=  
			\nb_{x}^{\top} \nb_{\xi} \dprin_{\bfB_{0}^{(\sgm)}}(X^{(\sgm)}, \Xi^{(\sgm)}) X^{(\sgm) \prime} +  \nb_{\xi}^{\top} \nb_{\xi} \dprin_{\bfB_{0}^{(\sgm)}}(X^{(\sgm)}, \Xi^{(\sgm)}) \Xi^{(\sgm) \prime} +  \nb_{\xi} \dprin_{\bfB_{0} - \br{\bfB}_{0}} (X^{(\sgm)}, \Xi^{(\sgm)}).
		\end{split}
	\end{equation*}
	We note that the last term is bounded by \begin{equation*}
		\begin{split}
			| \nb_{\xi} \dprin_{\bfB_{0} -\br{\bfB}_{0}} (X^{(\sgm)}, \Xi^{(\sgm)})| \le C\dlt |\Xi^{(\sgm)}| \le C\mu^{-1}\dlt |\dot{X}^{(\sgm)}| .
		\end{split}
	\end{equation*} For the other terms, using \eqref{eq:dlt-star} and \eqref{eq:variation-Xi} we obtain that  \begin{equation*}
		\begin{split}
			|\nb_{x}^{\top} \nb_{\xi} \dprin_{\bfB_{0}^{(\sgm)}}(X^{(\sgm)}, \Xi^{(\sgm)}) X^{(\sgm) \prime} | \le CM\mu^{-1}\dlt^* |\dot{X}^{(\sgm)}|  
		\end{split}
	\end{equation*}\begin{equation*}
		\begin{split}
			|\nb_{\xi}^{\top} \nb_{\xi} \dprin_{\bfB_{0}^{(\sgm)}}(X^{(\sgm)}, \Xi^{(\sgm)}) \Xi^{(\sgm) \prime}| \le CM |\xi| \dlt^* \le CMA\mu^{-1}\dlt^* |\dot{X}^{(\sgm)}|. 
		\end{split}
	\end{equation*} By taking $\dlt$ smaller depending on $\mu, M, A$ if necessary (this makes $\dlt^*$ smaller as well), we have \begin{equation*}
		\begin{split}
			|\dot{X}^{(\sigma) \prime }| \le \frac{1}{10} |\dot{X}^{(\sigma) }|.
		\end{split}
	\end{equation*} This allows us to write \begin{equation*}
		\begin{split}
			I(\sgm):= \int_{t_{0}}^{t_{1}}\abs{\dot{X}^{(\sigma) }(t)} \, \ud t  \le  \int_{t_{0}}^{t_{1}}\abs{\dot{X}^{(0) }(t)} \, \ud t  +  \frac{1}{10} \int_{t_{0}}^{t_{1}} \int_0^{\sgm} |\dot{X}^{(\sigma') }(t)|\, \ud\sgm'  \ud t < L +  \frac{1}{10}\int_0^{\sgm} I(\sgm') \, \ud \sgm'. 
		\end{split}
	\end{equation*} Note that $I(0) < L$. We may bootstrap the hypothesis $I(\sgm) < \frac54 L$; indeed, if we assume that it holds for some $\sgm \in [0,\sgm_0)$, then \begin{equation*}
		\begin{split}
			I(\sgm_{0}) < L +  \frac{1}{10} \sgm_{0} \frac54 L \le \frac98 L. 
		\end{split}
	\end{equation*} Therefore, we have that $I(1) < \frac54 L $, which allows us to conclude that \begin{equation*}
		\begin{split}
			\left[ \int_{-\infty}^{t_{0}} + \int_{t_{0}}^{t_{1}}+ \int_{t_{1}}^{\infty} \right] \chf_{\set{-2 R < x^{3} < 2 R}}(X^{(1), 3}(t)) \abs{\dot{X}^{(1)}(t)} \, \ud t <2L.
		\end{split}
	\end{equation*}

	\medskip
	
	\noindent \textbf{Proof of \eqref{eq:mizohata-A}}. We use the notation from above; fix some $(x,\xi)$ and denote $(X^{(\sgm)}(t),\Xi^{(\sgm)}(t))$ be the bicharacteristic curve corresponding to $\bfB_{0}^{(\sgm)}$. The goal is to prove \begin{equation}\label{eq:TM-goal}
		\begin{split}
			\int_{-\infty}^{\infty} |\nb\bfB_{0}^{(1)} (X^{(1)}(t))| |\Xi^{(1)}(t)| \, \ud t < 2 \max\{ A, M \}.
		\end{split}
	\end{equation} In the region  $\{ t_{1} < t \}$ (the region $\{ t < t_{0} \}$ can be handled similarly), we proceed as in the proof of Lemma \ref{lem:nontrapping-id}; using \eqref{eq:speed3-lower}, we can bound 
\begin{align}
	{\int_{-\infty}^{\infty} \chf_{\set{R < x^{3} }}(X^{(1)}(t))} \abs{\nb \bfB^{(1)}_{0}(X^{(1)}(t))} \abs{\Xi^{(1)}(t)} \, \ud t \aleq \sum_{I \in \calI: I \cap (R, \infty) \neq \0} \int_{I} \nrm{\nb \bfB^{(1)}_{0}(x^{1}, x^{2}, x^{3})}_{L^{\infty}_{x^{1}, x^{2}}} \, \ud x^{3} \aleq \eps.
\end{align}
Therefore, by taking a small absolute constant $\alp_{1}>0$ (recall $\eps \le \alp_{1}M$ from Definition \ref{def:nontrapping-class}), we can guarantee that \begin{equation}\label{eq:TM-1}
	\begin{split}
		\int_{ (-\infty,t_{0}] \cup [t_{1},\infty) } |\nb\bfB_{0}^{(\sgm)} (X^{(\sgm)}(t))| |\Xi^{(\sgm)}(t)| \, \ud t \le M. 
	\end{split}
\end{equation}
	To handle the integral in the region $\{ t_{0} \le t \le t_{1} \}$, we write \begin{equation*}
		\begin{split}
			\frac{\ud}{\ud\sgm} \int_{t_{0}}^{t_{1}} |\nb\bfB_{0}^{(\sgm)} (X^{(\sgm)}(t))| |\Xi^{(\sgm)}(t)| \, \ud t = I + II + III, 
		\end{split}
	\end{equation*} where \begin{equation*}
		\begin{split}
			I = \int_{t_{0}}^{t_{1}}  | \nb\bfB_{0}^{(\sgm)} (X^{(\sgm)}(t))| \frac{ \Xi^{(\sgm)} (t) \cdot \Xi^{(\sgm) \prime }(t) }{ |\Xi^{(\sgm)}(t)| } \, \ud t, 
		\end{split}
	\end{equation*}\begin{equation*}
		\begin{split}
			II = \int_{t_{0}}^{t_{1}} \frac{   \nb\bfB_{0}^{(\sgm)} (X^{(\sgm)}(t)) :  \nb(\bfB_{0}-\overline{\bfB}_{0})(X^{(\sgm)}(t)) }{| \nb\bfB_{0}^{(\sgm)} (X^{(\sgm)}(t))|} |\Xi^{(\sgm)}(t)|  \, \ud t,
		\end{split}
	\end{equation*} \begin{equation*}
		\begin{split}
			III = \int_{t_{0}}^{t_{1}} \frac12  \frac{ X^{(\sgm)\prime}(t) \cdot   \nb |\nb \bfB_{0}^{(\sgm)}(X^{(\sgm)}(t))|^{2} }{| \nb\bfB_{0}^{(\sgm)} (X^{(\sgm)}(t))|} |\Xi^{(\sgm)}(t)| \, \ud t.
		\end{split}
	\end{equation*} We may bound each term similarly as in the proof of \eqref{eq:mizohata-est} from Lemma \ref{lem:nontrapping-id}, using \eqref{eq:dlt-star}. We begin with \begin{equation*}
	\begin{split}
		|II| \lesssim \int_{t_{0}}^{t_{1}} | \nb(\bfB_{0}-\overline{\bfB}_{0})(X^{(\sgm)}(t))| |\Xi^{(\sgm)}(t)|  \, \ud t \lesssim \frac{\dlt}{\mu} \int_{t_{0}}^{t_{1}} | \bfB_{0}^{(\sgm)}(X^{(\sgm)}(t))| |\Xi^{(\sgm)}(t)|  \, \ud t  \lesssim \frac{\dlt}{\mu}L.
	\end{split}
\end{equation*} Next, from  \eqref{eq:dlt-star} we obtain $| \Xi^{(\sgm) \prime }(t) | \lesssim A \dlt^* |\Xi^{(\sgm) }(t)|$, which gives \begin{equation*}
\begin{split}
	|I| \lesssim \int_{t_{0}}^{t_{1}}  | \nb\bfB_{0}^{(\sgm)} (X^{(\sgm)}(t))|| \Xi^{(\sgm) \prime }(t) | \, \ud t \lesssim A\dlt^* \frac{M}{\mu}L .
\end{split}
\end{equation*} Lastly, using \eqref{eq:dlt-star} we bound \begin{equation*}
\begin{split}
	|III| \lesssim \int_{t_{0}}^{t_{1}}   | X^{(\sgm)\prime}(t)|  |\nb^{2} \bfB_{0}^{(\sgm)}(X^{(\sgm)}(t))| |\Xi^{(\sgm)}(t)| \, \ud t \lesssim  \frac{M}{\mu}\dlt^* \int_{t_{0}}^{t_{1}} |\bfB_{0}^{(\sgm)}(X^{(\sgm)}(t))| |\Xi^{(\sgm)}(t)| \, \ud t \lesssim \dlt^* \frac{M}{\mu}L .
\end{split}
\end{equation*} Therefore, by taking $\dlt$ smaller (which makes $\dlt^*$ smaller as well), \begin{equation*}
\begin{split}
	\left| \frac{\ud}{\ud\sgm} \int_{t_{0}}^{t_{1}} |\nb\bfB_{0}^{(\sgm)} (X^{(\sgm)}(t))| |\Xi^{(\sgm)}(t)| \, \ud t \right| \le \frac{A}{2},
\end{split}
\end{equation*} which together with \eqref{eq:TM-1} gives \eqref{eq:TM-goal}.
This finishes the proof. 
\end{proof} 

\section{Precise formulation and proof of the main theorem} \label{sec:lwp-pf}
\subsection{Precise formulation of the local wellposedness result}
We are now ready to give a precise formulation of the main local wellposedness result.
\begin{theorem}[Main theorem] \label{thm:main}
Given $s > \frac{7}{2}$, $M > 0$,  $\mu > 0$ and $A > 0$, there exist positive constants $\eps = \eps(s, M, \mu, A)$,  {$c(s, M, \mu, A)$ and $C(s, M, \mu, A)$} such that the following holds. Consider a vector field $\bfB_{0} : \bbR^{3} \to \bbR^{3}$ satisfying $\nb \cdot \bfB_{0} = 0$. Assume  furthermore that 
\begin{equation*}
\bfB_{0} \in \calB^{s}_{\eps}(M, \mu, A, R, L)
\end{equation*}
for some $R, L > 0$.  Then the Cauchy problem for \eqref{eq:e-mhd} with $\bfB(t=0) = \bfB_{0}$ is locally wellposed on $[0, T]$, where 
\begin{equation*}
T = c(s, M, \mu, A) \exp(-C(s, M, \mu, A) L),
\end{equation*}
in the sense that the following holds:
\begin{enumerate}
\item {\bf Existence and uniqueness.} There exists a unique solution $\bfB$ with $\bfB - \bfe_{3} \in \ell^{1}_{\calI} X^{s}[0, T]$.
\item {\bf Continuous dependence.} The solution map $\bfB_{0} \mapsto \bfB$ is continuous as a map $\ell^{1}_{\calI} H^{s} \to \ell^{1}_{\calI} X^{s}[0, T]$;
\item {\bf Weak Lipschitz dependence.} The solution map $\bfB_{0} \mapsto \bfB$ is Lipschitz continuous as a map $\ell^{1}_{\calI} L^{2} \to \ell^{1}_{\calI} X^{0}[0, T]$;
\item {\bf Persistence of regularity.} If $\bfB_{0} - \bfe_{3} \in \ell^{1}_{\calI} H^{s+m}$ for $m \in \bbZ_{\geq 0}$, then $\bfB - \bfe_{3} \in \ell^{1}_{\calI} X^{s+m}[0, T]$.
\item {\bf Frequency envelope bound.} Let $c_{k}$ be a \emph{$\dlta$-admissible frequency envelope} (i.e., $\set{c_{k}}_{k \in \bbZ_{\geq 0}}$ is a sequence of positive numbers obeying $\max \set*{c_{k} / c_{j}, c_{j} / c_{k} } \leq 2^{\dlta \abs{k-j}}$) with $0 < \dlta < s$. If $\nrm{P_{k}(\bfB_{0} - \bfe_{3})}_{\ell^{1}_{\calI} H^{s}} \leq c_{k}$ for $k \in \bbZ_{\geq 0}$, then $\nrm{P_{k}(\bfB - \bfe_{3})}_{\ell^{1}_{\calI} X^{s}} \aleq c_{k}$ for all $k \in \bbZ_{\geq 0}$.
\end{enumerate}
\end{theorem} 
Here, $\dlt_{1}$ is a universal small positive constant to be determined in the proof below.

As a direct consequence of the above, we have local wellposedness for small initial data. 
\begin{corollary}\label{cor:main}
	Given $s>\frac{7}{2}$, there exist positive constants $\varepsilon^* = \varepsilon^*(s)$ and $T = T(s)$ such that the following holds. For $\bfB_{0} : \bbR^{3} \to \bbR^{3}$ satisfying $\nb \cdot \bfB_{0} = 0$ and $\nrm{ \bfB_{0} - \bfe_{3} }_{\ell^{1}_{\calI}H^{s}} < \varepsilon^*$, the Cauchy problem for \eqref{eq:e-mhd} with $\bfB(t=0)=\bfB_{0}$ is locally wellposed on $[0,T]$, in the same sense of Theorem \ref{thm:main}. 
\end{corollary}
To obtain Corollary \ref{cor:main} from Theorem \ref{thm:main}, it suffices to note that for $\bfB_{0}$ satisfying $\nrm{ \bfB_{0} - \bfe_{3} }_{\ell^{1}_{\calI}H^{s}} < \varepsilon^*$, we have $\bfB_{0} \in \calB^{s}_{\eps}(M, \mu, A, R, L)$ with $M = 1, \mu = 1/2, A = 1, R = 1, L = 10$ as long as $\varepsilon^*$ is taken to be sufficiently small. Then $T>0$ is provided by Theorem \ref{thm:main}. 

\subsection{Paralinearization of \protect{\eqref{eq:e-mhd}} and the resulting error}
A basic ingredient of the proof is a wellposedness theorem for the \emph{paralinearization of \eqref{eq:e-mhd-b}}, which refers to the following system:
\begin{align} 
	\rd_{t} b - \nb \times (T_{\bfB} \times (\nb \times b)) + \nb \times (T_{\nb \times \bfB} \times b) &= g, \label{eq:paralin} \\
 	\nb \cdot b = \nb \cdot g &= 0, \label{eq:paralin-constraint}
\end{align}
where
\begin{align*}
	\nb \times (T_{\bfB} \times (\nb \times b)) &= \sum_{k} \nb \times (P_{<k-10} \bfB \times (\nb \times P_{k} b)), \\
	\nb \times (T_{\nb \times \bfB} \times b) &= \sum_{k} \nb \times ((\nb \times P_{<k-10} \bfB) \times P_{k} b).
\end{align*}
We also introduce the \emph{paralinearization error},
\begin{equation} \label{eq:paralin-err-def}
	G(b) = \nb \times (b \times (\nb \times b)) - \nb \times (T_{b} \times (\nb \times b)) + \nb \times (T_{\nb \times b} \times b).
\end{equation}
As its name suggests, $G(b)$ is precisely the error term that arises when writing \eqref{eq:e-mhd} (or \eqref{eq:e-mhd-b}) in the form \eqref{eq:paralin}--\eqref{eq:paralin-constraint} (we omit the straightforward proof):
\begin{lemma} \label{lem:e-mhd-parlin}
If $\bfB$ solves \eqref{eq:e-mhd} on $(0, T)$, then $b := \bfB - \bfe_{3}$ solves \eqref{eq:paralin}--\eqref{eq:paralin-constraint} with $\bfB = \bfe_{3} + b$ and $g = G(b)$ on $(0, T)$, and vice versa. 
\end{lemma}

The following two results will be proved in Section~\ref{sec:mag-lin}: 
\begin{proposition}[Wellposedness of the paralinearized system] \label{prop:paralin-full}
Let $s_{1} > \frac{7}{2}$ and $0 < T \leq 1$. Assume that, for some $M_{1}, M_{1}', \mu, A, \eps_{1}, R, L >0$, we have \begin{gather}
			\bfB(0) \in \calB^{s_{1}}_{\eps_{1}}(M_{1}, \mu, A, R, L),  \label{eq:paralin-hyp-B0} \\
					\nrm{\bfB-\bfe_{3}}_{\ell^{1}_{\calI} X^{s_{1}}[0,T]} \leq M_{1}, \label{eq:paralin-hyp} \\
			 {\nrm{\rd_{t} \bfB}_{\ell^{1}_{\calI} L^{\infty} H^{s_{1}-2}[0, T]} \leq M_{1}'.} \label{eq:paralin-hyp-dt}
\end{gather}  
Let $b$ and $g$ satisfy \eqref{eq:paralin}--\eqref{eq:paralin-constraint} on $(0, T)$. Then, for any $\sgm \geq 0$, the following holds: if
\begin{equation*}
\eps_{1} \leq \eps_{(1)}(\sgm, s_{1}, M_{1}, M_{1}', \mu, A), \quad T \leq c_{(1)}(\sgm, s_{1}, M_{1}, M_{1}', \mu, A) \exp(-C_{(1)}(\sgm, s_{1}, M_{1}, M_{1}', \mu, A) L)
\end{equation*}
then we have
\begin{equation} \label{eq:paralin-full}
	 \nrm{b}_{\ell^{1}_{\calI}  X^{\sgm}[0, T]} 
	\leq C_{(2)}(\sgm, s_{1}, M_{1}, M_{1}', \mu, A) \exp(C_{(2)}(\sgm, s_{1}, M_{1}, M_{1}', \mu, A) L) \left( \nrm{b_{0}}_{\ell^{1}_{\calI}  H^{\sgm}} + \nrm{g}_{\ell^{1}_{\calI}  Y^{\sgm}[0, T]}\right).
\end{equation}
\end{proposition}
\begin{proposition} [Bounds for the paralinearization error] \label{prop:paralin-err}
Let $s_{1} > \frac{7}{2}$ and $b, \br{b} \in \ell^{1}_{\calI} X^{s_{1}}[0, T]$ with $$\nrm{b}_{\ell^{1}_{\calI} X^{s_{1}}[0, T]} < M_{1}, \qquad \nrm{\br{b}}_{\ell^{1}_{\calI} X^{s_{1}}[0, T]} < M_{1}.$$ Then for any $\sgm \geq 0$, we have
\begin{align} 
	\nrm{G(b)}_{\ell^{1}_{\calI} Y^{\sgm}[0, T]} &\aleq_{\sgm, s_{1}} T M_{1} \nrm{b}_{\ell^{1}_{\calI} X^{\sgm}[0, T]},  \label{eq:paralin-err} \\
	\nrm{G(b) - G(\br{b})}_{\ell^{1}_{\calI} Y^{\sgm}[0, T]} & \aleq_{\sgm, s_{1}} T M_{1} \nrm{b - \br{b}}_{\ell^{1}_{\calI} X^{\sgm}[0, T]}.  \label{eq:paralin-err-diff}
\end{align}
Moreover, for $0 \leq \sgm < s_{1} - 1$, we have
\begin{align} 
\nrm{\nb \times (T_{b - \br{b}} \times (\nb \times b)) - \nb \times (T_{\nb \times (b -\br{b})} \times b)}_{\ell^{1}_{\calI} Y^{\sgm}} \aleq_{s_{1} - 1 - \sgm} T^{\frac{s_{1} - 1 - \sgm}{s_{1} - 1}} M_{1} \nrm{b - \br{b}}_{\ell^{1}_{\calI} X^{\sgm}[0, T]}. \label{eq:paralin-err-diff-low}
\end{align}
 {Finally, we also have the fixed time bounds
\begin{equation} \label{eq:paralin-err-fixed-t}
	\nrm{\nb \times (T_{b} \times (\nb \times b))(t)}_{\ell^{1}_{\calI} H^{s_{1}-2}} + \nrm{\nb \times (T_{\nb \times b} \times b)(t)}_{\ell^{1}_{\calI} H^{s_{1}-2}} + \nrm{G(b(t))}_{\ell^{1}_{\calI} H^{s_{1}-2}} \aleq_{s_{1}} M_{1}^{2}.
\end{equation}}
\end{proposition}

In the next subsection, we will utilize these propositions to prove Theorem~\ref{thm:main}.

\subsection{Proof of local wellposedness}
We now show how Theorem~\ref{thm:main} follows from Propositions~\ref{prop:paralin-full} and \ref{prop:paralin-err}. We follow the streamlined approach of \cite{MMT1, MMT2, MMT3} (see also \cite{IfTa}), which does not rely on Nash--Moser or viscosity method. 
\begin{proof}[Proof of Theorem~\ref{thm:main}]
Throughout the proof, we will assume that 
\begin{equation} \label{eq:main-pf:B0-T}
\bfB_{0} \in \calB^{s}_{\eps}(M, \mu, A, R, L), \quad T \leq c(s, M, \mu, A) \exp(-C(s, M, \mu, A) L), 
\end{equation}
 {where $c(s, M, \mu, A)$ and $C(s, M, \mu, A)$, as well as $\eps(s, M, \mu, A)$, are to be fixed at the end of Step~2 below. }

\smallskip
\noindent {\it Step~0: Setup for the proof of existence.} The proof of existence is based on an iteration argument. We begin with $b^{(-1)} = 0$. With $\bfB^{(n-1)} := b^{(n-1)} + \bfe_{3}$, we define $b^{(n)}$ by solving
\begin{equation} \label{eq:e-mhd-iter}
\begin{aligned}
\rd_{t} b^{(n)} - \nb \times (T_{\bfB^{(n-1)}} \times (\nb \times b^{(n)})) + \nb \times (T_{\nb \times \bfB^{(n-1)}} \times b^{(n)}) &= G(b^{(n-1)}) \\
 	\nb \cdot b^{(n)} &= 0,
\end{aligned}
\end{equation}
with the initial conditions $b^{(n)}(0) = \bfB_{0} - \bfe_{3}$.

Note that the difference between two consecutive iterates solves the following equation:
\begin{equation} \label{eq:e-mhd-iter-diff}
\begin{aligned}
& \rd_{t} (b^{(n+1)} - b^{(n)}) - \nb \times (T_{\bfB^{(n)}} \times (\nb \times (b^{(n+1)} - b^{(n)}))) + \nb \times (T_{\nb \times \bfB^{(n)}} \times (b^{(n+1)}- b^{(n)})) \\
&= G(b^{(n-1)}) - G(b^{(n)}) + \nb \times (T_{b^{(n)} - b^{(n-1)}} \times (\nb \times b^{(n)})) - \nb \times (T_{\nb \times (b^{(n)} -b^{(n-1)})} \times b^{(n)}).
\end{aligned}
\end{equation}

\smallskip
\noindent {\it Step~1: Iteration.} Fix a choice of $s_{1} = s_{1}(s)$ such that $\frac{7}{2} < s_{1} < s$ (e.g., $s_{1} = \frac{1}{2}(\frac{7}{2} + s)$), and construct a sequence of iterates as above. We choose $M_{1} = C_{1} M$, $\eps_{1} = C_{1} \eps$, and $M_{1}' = C_{1}^{2} M(1+M)$ with $C_{1} = C_{1}(s)$ a suitably large constant (to be fixed at the end of this step). In particular, by taking $C_{1}$ to be sufficiently large, we may ensure that \eqref{eq:paralin-hyp-B0} holds. We also take
\begin{equation} \label{eq:main-pf:c-C}
	c(s, M, \mu, A) \leq c_{(1)}(s, s_{1}, M_{1}, M_{1}', \mu, A), \quad
	C(s, M, \mu, A) {\ge} C_{(1)}(s, s_{1}, M_{1}, M_{1}',  \mu, A),
\end{equation}
 {where $c_{(1)}$ and $C_{(1)}$ are from Proposition~\ref{prop:paralin-full}.}

We will propagate the following bounds by induction on $n$:  
\begin{enumerate}
\item $\nrm{\bfB^{(n)} - \bfe_{3}}_{\ell^{1}_{\calI}  X^{s_{1}}[0, T]}
 \leq M_{1}$.
\item $\nrm{\bfB^{(n)} - \bfe_{3}}_{\ell^{1}_{\calI}  X^{s}[0, T]} \leq 2 C_{(2)}(s, s_{1}, M_{1}, M_{1}', \mu, A) e^{C_{(2)}(s, s_{1}, M_{1}, M_{1}', \mu, A) L} M$.
\item $\nrm{\rd_{t} \bfB^{(n)}}_{\ell^{1}_{\calI} L^{\infty} H^{s_{1}-2}[0, T]} \leq M_{1}'$.
\end{enumerate}

Indeed, these bounds trivially hold for $\bfB^{(-1)} = \bfB_{0}$. Now, assuming that (1), (2) and (3) hold for $\bfB^{(n)}$, we will show that $\bfB^{(n+1)}$ obeys the same bounds as well. In view of \eqref{eq:main-pf:B0-T} and \eqref{eq:main-pf:c-C}, we may apply Propositions~\ref{prop:paralin-full} and \ref{prop:paralin-err} to \eqref{eq:e-mhd-iter} for $b^{(n+1)} = \bfB^{(n+1)} - \bfe_{3}$. We obtain
\begin{equation*}
\nrm{\bfB^{(n+1)} - \bfe_{3}}_{\ell^{1}_{\calI}  X^{s}[0, T]}
 \leq C_{(2)}(s, s_{1}, M_{1}, M_{1}', \mu, A) e^{C_{(2)}(s, s_{1}, M_{1}, M_{1}', \mu, A) L} (M + C(s, s_{1}) T M_{1} M),
\end{equation*}
{where $C(s, s_{1})$ is a constant arising from applying Proposition~\ref{prop:paralin-err}.}
Taking $c$ to be even smaller (so that $T$ is small), we may ensure that (2) holds. To establish (1) and (3),  {we begin by noting that
\begin{equation*}
	\nrm{\rd_{t} \bfB^{(n+1)}}_{\ell^{1}_{\calI} L^{\infty} H^{s_{1}-2}[0, T]} \aleq (1+M_{1}) \nrm{\bfB^{(n+1)}-\bfe_{3}}_{\ell^{1}_{\calI} X^{s_{1}}[0, T]} + M_{1}^{2}
\end{equation*}
by \eqref{eq:paralin-err-fixed-t} (in Proposition~\ref{prop:paralin-err}), the induction hypothesis and the simple algebraic computation
\begin{align*}
&\nb \times (T_{\bfB^{(n)}} \times (\nb \times b^{(n+1)})) - \nb \times (T_{\nb \times \bfB^{(n)}} \times b^{(n+1)}) \\
&= \nb \times (\bfe_{3} \times (\nb \times b^{(n+1)})) + \nb \times (T_{b^{(n)}} \times (\nb \times b^{(n+1)})) - \nb \times (T_{\nb \times b^{(n)}} \times b^{(n+1)}).
\end{align*}}
Splitting
\begin{align*}
\bfB^{(n+1)} - \bfe_{3} 
&= P_{<k}(\bfB^{(n+1)} - \bfe_{3}) + P_{\geq k} (\bfB^{(n+1)} - \bfe_{3}) \\
&= P_{<k}(\bfB_{0} - \bfe_{3}) + \int_{0}^{t} P_{<k} \rd_{t} \bfB^{(n+1)} (t') \, \ud t' + P_{\geq k} (\bfB^{(n+1)} - \bfe_{3}),
\end{align*}
{
then applying the preceding bounds and \eqref{eq:Xs-Hs+1/2}, we obtain
\begin{align*}
	\nrm{\bfB^{(n+1)} - \bfe_{3}}_{\ell^{1}_{\calI} X^{s_{1}}[0, T]}
	&\leq \nrm*{P_{<k} (\bfB_{0} - \bfe_{0}) + \int_{0}^{t} P_{<k} \rd_{t} \bfB^{(n+1)} (t') \, \ud t'}_{\ell^{1}_{\calI} X^{s_{1}}[0, T]} \\
	&\peq + \nrm{P_{\geq k}(\bfB^{(n+1)} - \bfe_{3})}_{\ell^{1}_{\calI} X^{s_{1}}[0, T]} \\
	&\aleq (1+ T^{\frac{1}{2}} 2^{\frac{k}{2}})  \nrm*{P_{<k} (\bfB_{0} - \bfe_{0})}_{\ell^{1}_{\calI} H^{s_{1}}} \\
	&\peq + (1+ T^{\frac{1}{2}} 2^{\frac{k}{2}}) T 2^{2k} \left((1+M_{1}) \nrm{\bfB^{(n+1)}-\bfe_{3}}_{\ell^{1}_{\calI} X^{s_{1}}[0, T]} + M_{1}^{2}\right) \\
	&\peq + 2^{-(s-s_{1}) k} 2 C_{(2)}(s, s_{1}, M_{1}, M_{1}', \mu, A) e^{C_{(2)}(s, s_{1}, M_{1}, M_{1}', \mu, A) L} M.
\end{align*}
Note that $s - s_{1} > 0$ by our choice. Taking $k$ sufficiently large (depending on $s, M, \mu, A$, to handle the term on the last line) and $c(s, M, \mu, A)$ smaller (so that $T$ is small depending on $k$), and $C_{1}(s)$ sufficiently large (so that $M_{1}$ and $M_{1}'$ are large compared to $M$ and $M(1+M)$, respectively), we may ensure that (1) and (3) hold.
}

Next, we establish the convergence of $\bfB^{(n)} = b^{(n)} + \bfe_{3}$. Indeed, applying Propositions~\ref{prop:paralin-full} and \ref{prop:paralin-err} to the difference equation \eqref{eq:e-mhd-iter-diff} with $\sgm = 0$, we obtain
\begin{equation*}
\nrm{b^{(n+1)} - b^{(n)}}_{\ell^{1}_{\calI}  X^{0}[0, T]}
 < C_{(2)}(s, s_{1}, M_{1}, M_{1}', \mu, A) e^{C_{(2)}(s, s_{1}, M_{1}, M_{1}', \mu, A) L} C(s, s_{1}) T M_{1}  \nrm{b^{(n)} - b^{(n-1)}}_{\ell^{1}_{\calI} X^{0}[0, T]}.
\end{equation*}
Taking $c_{(1)}$ sufficiently small (so that $T$ is small), we may ensure that $\bfB^{(n)} = b^{(n)} + \bfe_{3}$ is a Cauchy sequence in $\ell^{1}_{\calI} X^{0}[0, T]$; let us denote the limit by $\bfB$. Observing that (1) and (2) hold for the limit $P_{<k} (\bfB - \bfe_{3})$ for every $k \in \bbZ_{\geq 0}$, it follows that $\bfB - \bfe_{3} \in \ell^{1}_{\calI} X^{s}[0, T]$ with the bounds (1) and (2). Moreover, $b = \bfB - \bfe_{3}$ solves
\begin{equation} \label{eq:e-mhd-paralin}
\begin{aligned}
\rd_{t} b - \nb \times (T_{\bfB} \times (\nb \times b)) + \nb \times (T_{\nb \times \bfB} \times b) &= G(b) \\
 	\nb \cdot b &= 0,
\end{aligned}
\end{equation}
with the initial condition $b(0) = \bfB_{0} - \bfe_{3}$, as desired.

{
We end this step by noting that, taking $c(s, M, \mu, A)$ sufficiently small (so that $T$ is small), we have
\begin{equation} \label{eq:main-pf:Bt}
	\bfB(t) \in \calM^{s_{1}}_{2 \eps_{1}}(2 M_{1}, \tfrac{1}{2} \mu, 2 \max\set{A, M_{1}}, R, 2 L) \quad \hbox{ for all } t \in [0, T].
\end{equation}
To see this, fix $t \in [0, T]$. From the splitting 
\begin{align*}
\bfB^{(n+1)} - \bfB_{0}
&= \int_{0}^{t} P_{<k} \rd_{t} \bfB^{(n+1)} (t') \, \ud t' + P_{\geq k} (\bfB^{(n+1)} - \bfe_{3}) - P_{\geq k} (\bfB_{0} - \bfe_{3}),
\end{align*}
we may apply (2) and (3) for $\bfB^{(n+1)}$, as well as \eqref{eq:main-pf:B0-T}, to bound
\begin{align*}
\nrm{\bfB^{(n+1)}(t) - \bfB_{0}}_{\ell^{1}_{\calI} H^{s_{1}}[0, T]} 
\aleq T 2^{2k} M_{1}' + 2^{-(s-s_{1}) k} \left( 2 C_{(2)}(s, s_{1}, M_{1}, M_{1}', \mu, A) e^{C_{(2)}(s, s_{1}, M_{1}, M_{1}', \mu, A) L} M + M \right).
\end{align*} 
Taking $k$ sufficiently large (depending on $s, M, \mu, A$) and $c(s, M, \mu, A)$ small enough (depending on $k$, so that $T$ is small), we may apply Proposition~\ref{prop:nontrapping-stable} to conclude \eqref{eq:main-pf:Bt}.

 }
 \smallskip
 
\noindent {\it Step~2: Uniqueness and weak Lipschitz dependence.}
Let $\bfB$ be the solution constructed in Step~1, and let $\br{\bfB} \in \ell^{1}_{\calI} X^{s}[0, T]$ be  a (possibly different) solution to the initial value problem. We will show that, in fact, $\bfB = \br{\bfB}$.

Introducing the shorthands $b = \bfB - \bfe_{3}$ and $\br{b} = \br{\bfB} - \bfe_{3}$, note that the difference $b - \br{b}$ solves
\begin{equation} \label{eq:e-mhd-paralin-diff}
\begin{aligned}
& \rd_{t} (\br{b} - b) - \nb \times (T_{\bfB} \times (\nb \times (\br{b} - b))) + \nb \times (T_{\nb \times \bfB} \times (\br{b}- b)) \\
&= G(\br{b}) - G(b) + \nb \times (T_{\br{b} - b} \times (\nb \times \br{b})) - \nb \times (T_{\nb \times (\br{b} -b)} \times \br{b}).
\end{aligned}
\end{equation}
{
Let $0 \leq t_{0} < t_{1} \leq T$. In view of \eqref{eq:main-pf:Bt} at $t = t_{0}$, Propositions~\ref{prop:paralin-full} and \ref{prop:paralin-err} are applicable to the difference equation \eqref{eq:e-mhd-paralin-diff} on $[t_{0}, t_{1}]$ with $\sgm = 0$ (after making $\eps(s, M, \mu, A)$, $c(s, M, \mu, A)$ and $C(s, M, \mu, A)^{-1}$ smaller if necessary). We obtain
\begin{align*}
\nrm{\br{b} - b}_{\ell^{1}_{\calI}  X^{0}[t_{0}, t_{1}]}
&\leq C_{(2)}(0, s_{1}, M_{1}, M_{1}', \mu, A) e^{C_{(2)}(0, s_{1}, M_{1}, M_{1}', \mu, A) L}  \\
&\pleq \times (\nrm{(\br{b} - b)(t_{0})}_{\ell^{1}_{\calI} L^{2}} + (t_{1} - t_{0}) C(0, s) (\nrm{b}_{\ell^{1}_{\calI} X^{s}[0, T]} + \nrm{\br{b}}_{\ell^{1}_{\calI} X^{s}[0, T]})  \nrm{\br{b} - b}_{\ell^{1}_{\calI} X^{0}[t_{0}, t_{1}]}).
\end{align*}
Provided that $b (t_{0}) = \br{b} (t_{0})$ and $t_{1} - t_{0}$ is sufficiently small (depending not only on $s, M, \mu, A$, but also on $\nrm{b}_{\ell^{1}_{\calI} X^{s}[0, T]} + \nrm{\br{b}}_{\ell^{1}_{\calI} X^{s}[0, T]}$), it follows that $\bfB = \br{\bfB}$ on $[t_{0}, t_{1}]$. Then from a straightforward continuous induction argument, it follows that $\bfB = \br{\bfB}$ on $[0, T]$.
}

By the uniqueness statement that we just established, it follows that any $\ell^{1}_{\calI} X^{s}[0, T]$ solution to the initial value problem coincides with the one constructed in Step~1, and thus obeys (1) and (2). Given two such solutions, the weak Lipschitz dependence assertion follows from the preceding difference bound on $[0, T]$.

{We point out that, at this point, our choices of $\eps(s, M, \mu, A)$, $c(s, M, \mu, A)$ and $C(s, M, \mu, A)$ are final.}

 \smallskip
\noindent {\it Step~3: Persistence of regularity.} 
{Applying Lemma~\ref{lem:nontrap-unif}, which is possibly in view of \eqref{eq:main-pf:Bt}, given any $\eps' > 0$ we may find $R', L' > 0$ such that 
\begin{equation*}
\bfB(t) \in \calB^{s_{1}}_{\eps_{1}'}(M_{2}, \mu_{2}, M_{2}, R', L') \quad \hbox{ for all } t \in [0, T],
\end{equation*}
where $\mu_{2} = \tfrac{1}{2} \mu$ and $M_{2} = 2 \max\set{A, M_{1}}$.

Let $0 \leq t_{0} < t_{1} \leq T$. By essentially\footnote{Pedantically, one must repeat Steps~1--2 to construct a solution in $\ell^{1}_{\calI} X^{\sgm}$ and conclude that it coincides with $\bfB$ by uniqueness. We omit the details.} applying Propositions~\ref{prop:paralin-full} and \ref{prop:paralin-err} to \eqref{eq:e-mhd-paralin} on $[t_{0}, t_{1}]$, which is possible if we take $\eps'$ and $t_{1} - t_{0}$ sufficiently small depending (in particular) on $\sgm$, we obtain
\begin{align*}
\nrm{b}_{\ell^{1}_{\calI}  X^{\sgm}[t_{0}, t_{1}]}
&\leq C_{(2)}(\sgm, s_{1}, M_{2}, M_{1}', \mu_{2}, M_{2}) e^{C_{(2)}(\sgm, s_{1}, M_{2}, M_{1}', \mu_{2}, M_{2}) L'}  \\
&\pleq \times (\nrm{(b(t_{0})}_{\ell^{1}_{\calI} H^{\sgm}} + (t_{1} - t_{0}) C(\sgm, s_{1}) M_{1} \nrm{b}_{\ell^{1}_{\calI} X^{\sgm}[t_{0}, t_{1}]}).
\end{align*}
Provided that $t_{1} - t_{0}$ is sufficiently small depending on $\sgm, s_{1}, M_{2}, M_{1}', \mu_{2}, M_{2}$ (or equivalently, $\sgm, s, M, \mu, A$), we may absorb the contribution of the last term in the parentheses and conclude that
\begin{align*}
\nrm{b}_{\ell^{1}_{\calI}  X^{\sgm}[t_{0}, t_{1}]}
\leq 2 C_{(2)}(\sgm, s_{1}, M_{2}, M_{1}', \mu_{2}, M_{2}) e^{C_{(2)}(\sgm, s_{1}, M_{2}, M_{1}', \mu_{2}, M_{2}) L'}  \nrm{(b(t_{0})}_{\ell^{1}_{\calI} H^{\sgm}}.
\end{align*}
Applying this bound repeatedly in time until we exhaust $[0, T]$, the persistence of regularity bound follows.}

 \smallskip
\noindent {\it Step~4: Frequency envelope bound.} The idea is to combine Steps 2 and 3. We remark that while we do not keep track of the dependence of constants on $s, M, \mu, A, \eps, R, L$ in this step and below, it may be easily recovered by a closer inspection of the argument.

Let $\dlt$ be any positive constant bigger than $\dlta$. For each $k \in \bbZ_{\geq 0}$, denote by $\bfB_{<k}$ the solution to \eqref{eq:e-mhd} with $\bfB_{<k}(0) = P_{<k} \bfB_{0}$, and write $b_{<k} = \bfB_{<k} - \bfe_{3}$. We also introduce the convention $b_{<-1} = \bfB_{<-1} - \bfe_{3} = 0$. By Steps~2 and 3, we have
\begin{align*}
	\nrm{b_{<k} - b_{<k-1}}_{\ell^{1}_{\calI} X^{0}} &\aleq \nrm{P_{k-1} (\bfB_{0} - \bfe_{3})}_{\ell^{1}_{\calI} L^{2}} \aleq 2^{-s k} c_{k}, \\
	\nrm{b_{<k} - b_{<k-1}}_{\ell^{1}_{\calI} X^{s+\dlt}} &\aleq \nrm{P_{< k} (\bfB_{0} - \bfe_{3})}_{\ell^{1}_{\calI} H^{s+\dlt}} \aleq \sum_{0 \leq k' < k} 2^{\dlt k'} c_{k'} \aleq 2^{\dlt k}c_{k}.
\end{align*}
for any $k \in \bbZ_{\geq 0}$, with the convention $P_{-1} = P_{<0}$. Here, we have used Step~2 for the first bound, while Step~3, the triangle inequality and the $\dlta$-admissibility condition have been used in the last inequality. Writing
\begin{align*}
	P_{k_{0}} (\bfB-\bfe_{3}) = \sum_{k \in \bbZ_{\geq 0}} P_{k_{0}} (b_{<k} - b_{<k-1})
\end{align*}
it follows that
\begin{align*}
	\nrm{P_{k_{0}} (\bfB-\bfe_{3})}_{\ell^{1}_{\calI} X^{s}}
	\aleq \sum_{k \in \bbZ_{\geq 0}} \min\set{2^{s(k_{0} - k)}, 2^{\dlt(k - k_{0})}} c_{k}
	\aleq c_{k_{0}},
\end{align*}
where we again used the $\dlta$-admissibility condition in the last inequality. 

 \smallskip
\noindent {\it Step~5: Continuous dependence.} The idea is to combine Steps~2 and 4. Given a convergent sequence $\bfB^{(n)}_{0} \to \bfB_{0}$ in $\ell^{1} H^{s}$, observe that there exist $\dlta$-admissible frequency envelopes $c_{k}^{(n)}$, $c_{k}$ for $\bfB^{(n)}_{0}, \bfB_{0}$ in $\ell^{1}_{\calI} H^{s}$ with the following common upper bounds:
\begin{equation*}
\sup_{n} \sum_{k} (c_{k}^{(n)})^{2}+ \sum_{k} c_{k}^{2} \aleq \nrm{\bfB_{0}}_{\ell^{1} H^{s}}^{2}, \quad,
\end{equation*}
and for every $\eps > 0$, there exists $K \in \bbZ_{\geq 0}$ such that 
\begin{equation*}
\sup_{n} \sum_{k \geq K} (c_{k}^{(n)})^{2}+ \sum_{k \geq K} c_{k}^{2} < \eps^{2}.
\end{equation*}
By the preceding arguments, these initial data sets give rise to solutions $\bfB^{(n)}$ to the initial value problem on a common interval $[0, T]$ (we may have to restrict to $n$ large enough if necessary). Moreover, given any $\eps > 0$, we use the frequency envelope bound in Step~4 to find $K \in \bbZ_{\geq 0}$ such that 
\begin{equation*}
\nrm{P_{\geq K} \bfB^{(n)}}_{\ell^{1} X^{s}[0, T]} < \frac{\eps}{3}, \quad \nrm{P_{\geq K} \bfB}_{\ell^{1} X^{s}[0, T]} < \frac{\eps}{3}
\end{equation*}
 uniformly in $n$. On the other hand, by Step~2, it follows that
\begin{align*}
\nrm{P_{< K} (\bfB^{(n)} - \bfB)}_{\ell^{1} X^{s}[0, T]} 
\aleq 2^{s K} \nrm{P_{< K} (\bfB^{(n)} - \bfB)}_{\ell^{1} X^{0}[0, T]}
\aleq 2^{s K} \nrm{P_{< K} (\bfB^{(n)}_{0} - \bfB_{0})}_{\ell^{1} L^{2}}.
\end{align*}
Hence, taking $n$ sufficiently large (depending on $k$), we may ensure that $\nrm{\bfB^{(n)} - \bfB}_{\ell^{1} X^{s}[0, T]} < \eps$. This implies the desired continuous dependence assertion. \qedhere
\end{proof}

\section{Linear and multilinear estimates} \label{sec:multi-ests}
In this section, we state and prove linear and multilinear estimates on the function spaces defined in Section~\ref{subsec:ftn-sp}, which will be useful in Section~\ref{sec:mag-lin} below.
\subsection{Localization property of Littlewood--Paley projections}
We begin with a basic property of the Littlewood--Paley projection $P_{k}$ with respect to localization (in physical space) on slabs of the form $\set{x^{3} \in I}$.
\begin{lemma} \label{lem:local-LP}
Let $j, k \in \bbZ_{\geq 0}$ and $1 \leq p, q_{1}^{i}, q^{i}, r \leq \infty$, where $j + k \geq 0$ and $q^{i} \leq q_{1}^{i}$ ($i=1,2,3$). Then for any $I_{0} \in \calI_{2j}$ and $\set{i_{1}, i_{2}, i_{3}} = \set{1, 2, 3}$, we have
\begin{align} 
	\nrm{\chi_{I_{0}}(x^{3}) P_{k} f}_{L^{p} L^{q_{1}^{i_{1}}}_{x^{i_{1}}} L^{q_{1}^{i_{2}}}_{x^{i_{2}}} L^{q_{1}^{i_{3}}}_{x^{i_{3}}}} 
	& \aleq 2^{k \sum_{i} (\frac{{1}}{q^{i}} - \frac{{1}}{q_{1}^{i}})}\sum_{I \in \calI_{j}} 
	\frac{1}{\brk{2^{k} \dist(I_{0}, I)}^{10}} \nrm{\chi_{I}(x^{3}) f}_{L^{p} L^{q^{i_{1}}}_{x^{i_{1}}} L^{q^{i_{2}}}_{x^{i_{2}}} L^{q^{i_{3}}}_{x^{i_{3}}}}. \label{eq:local-LP}
\end{align}
\end{lemma}
\begin{proof}
The kernel for $P_{k}$ obeys
\begin{equation*}
	\abs{K_{k}(x-y)} \aleq \frac{2^{3 k}}{\brk{2^{k} \abs{x - y}}^{100}}.
\end{equation*}
Thus
\begin{align*}
	\abs{\chi_{I_{0}}(x) K_{\ell}(x-y)} 
	& \aleq \sum_{I \in \calI_{j}} \chi_{I_{0}}(x) \frac{2^{3 k}}{\brk{2^{k} \abs{x - y}}^{100}} \chi_{I}(y) \\
	& \aleq \sum_{I \in \calI_{j}} \frac{2^{3 k}}{\brk{2^{k} \abs{x - y}}^{30}} \frac{1}{\brk{2^{k} \dist(I_{0}, I)}^{10}} \chi_{I}(y).
\end{align*}
Now the desired statement follows from Young's inequality.
\end{proof}

\begin{corollary} \label{cor:local-bernstein}
Let $j, k \in \bbZ_{\geq 0}$ and $1 \leq p, q_{1}^{i}, q^{i}, r \leq \infty$, where $j + k \geq 0$ and $q^{i} \leq q_{1}^{i}$ ($i=1,2,3$). Then for any $\set{i_{1}, i_{2}, i_{3}} = \set{1, 2, 3}$, we have
\begin{align} 
	\nrm{P_{k} f}_{\ell^{r}_{\calI} (L^{p} L^{q_{1}^{i_{1}}}_{x^{i_{1}}} L^{q_{1}^{i_{2}}}_{x^{i_{2}}} L^{q_{1}^{i_{3}}}_{x^{i_{3}}})_{j}} 
	& \aleq 2^{k \sum_{i} (\frac{{1}}{q^{i}} - \frac{{1}}{q_{1}^{i}})} \nrm{f}_{\ell^{r}_{\calI} (L^{p} L^{q^{i_{1}}}_{x^{i_{1}}} L^{q^{i_{2}}}_{x^{i_{2}}} L^{q^{i_{3}}}_{x^{i_{3}}})_{j}}. \label{eq:local-bernstein} 
\end{align}
\end{corollary}
\begin{proof}
This statement is an easy consequence of Lemma~\ref{lem:local-LP} and Schur's test applied to the $I$-summations; the conditions $j+k \geq 0$ ensures that these sums are uniformly bounded.
\end{proof}

\subsection{Pseudodifferential operator on $X$ and $Y$} \label{subsec:psdo-ests}
Order $0$ classical pseudodifferential operators obey the following boundedness properties in $X^{0}$ and $Y^{0}$, as well as their $\ell^{r}_{\calI}$-modifications.
\begin{proposition} \label{prop:psdo-ests}
Let $\frka \in S^{0}$. Then the following statements hold (where $1 \leq r \leq \infty$).
\begin{enumerate}
\item We have
\begin{align} 
	\nrm{\frka(x, D) b}_{X^{0}} &\leq C [\frka]_{S^{0}; 1000} \nrm{b}_{X^{0}}, \label{eq:psdo-ests-X} \\
	\nrm{\frka(x, D) b}_{\ell^{r}_{\calI} X^{0}} &\leq C [\frka]_{S^{0}; 1000} \nrm{b}_{\ell^{r}_{\calI} X^{0}}, \label{eq:psdo-ests-X-ell} \\
	\nrm{\frka(x, D) g}_{Y^{0}} &\leq C [\frka]_{S^{0}; 1000} \nrm{g}_{Y^{0}}, \label{eq:psdo-ests-Y} \\
	\nrm{\frka(x, D) g}_{\ell^{r}_{\calI} Y^{0}} &\leq C [\frka]_{S^{0}; 1000} \nrm{g}_{\ell^{r}_{\calI} Y^{0}}. \label{eq:psdo-ests-Y-ell}
\end{align}
\item For $k_{1}$ satisfying
\begin{equation} \label{eq:psdo-ests-k1}
	2^{k_{1}} \geq \frac{[\frka]_{S^{0}; 1000}}{\nrm{\frka}_{L^{\infty}}},
\end{equation}
we have
\begin{align} 
	\nrm{\frka(x, D) P_{\geq k_{1}} b}_{X^{0}} &\leq C \nrm{\frka}_{L^{\infty}} \nrm{b}_{X^{0}}, \label{eq:psdo-ests-X-high} \\
	\nrm{\frka(x, D) P_{\geq k_{1}} b}_{\ell^{r}_{\calI} X^{0}} &\leq C \nrm{\frka}_{L^{\infty}} \nrm{b}_{\ell^{r}_{\calI} X^{0}}, \label{eq:psdo-ests-X-ell-high} \\	
	\nrm{\frka(x, D) P_{\geq k_{1}} g}_{Y^{0}} &\leq C \nrm{\frka}_{L^{\infty}} \nrm{g}_{Y^{0}}, \label{eq:psdo-ests-Y-high} \\
	\nrm{\frka(x, D) P_{\geq k_{1}} g}_{\ell^{r}_{\calI} Y^{0}} &\leq C \nrm{\frka}_{L^{\infty}} \nrm{g}_{\ell^{r}_{\calI} Y^{0}}. \label{eq:psdo-ests-Y-ell-high} 
\end{align}
\end{enumerate}
\end{proposition}
\begin{proof}
\noindent{\it Step~1: Initial reduction.}
Let $\td{P}_{k}(\xi)$ be a rescaled bump function adapted to the annulus $\set{\abs{\xi} \aeq 2^{k}}$ for $k \geq 1$ (resp.~the ball $\set{\abs{\xi} \aeq 1}$ for $k = 0$) such that $\td{P}_{k}(\xi) P_{k}(\xi) = P_{k}(\xi)$, and let $\td{P}_{k} = \td{P}_{k}(D)$. We claim that \eqref{eq:psdo-ests-X} and \eqref{eq:psdo-ests-Y} follow from
\begin{align} 
	\nrm{P_{k_{0}} \frka(x, D) \td{P}_{k}}_{L^{2} \to L^{2}} & \aleq 2^{-50 \abs{k_{0} - k}} [\frka]_{S^{0}; 500}, \label{eq:psdo-ests-L2} \\
	\sup_{\ell \in \bbZ_{\geq 0}} \sup_{I, I_{0} \in \calI_{\ell}}\nrm{\chi_{I_{0}}(x^{3}) P_{k_{0}} \frka(x, D) \td{P}_{k} \chi_{I}(x^{3})}_{L^{2} \to L^{2}}
	&\aleq 2^{-50 \abs{k_{0} - k}} \frac{1}{\brk{2^{k_{\min}} \dist(I_{0}, I)}^{5}}[\frka]_{S^{0}; 500}.\label{eq:psdo-ests-L2-nonlocal}
\end{align}
Indeed, observe first that \eqref{eq:psdo-ests-L2} immediately implies, for all $k_{0}, k \in \bbZ_{\geq 0}$,
\begin{equation} \label{eq:psdo-ests-reduce-en} 
	\nrm{P_{k_{0}} \frka(x, D) P_{k} b}_{L^{\infty} L^{2}} \aleq 2^{-50 \abs{k_{0}-k}}[\frka]_{S^{0}; 500} \nrm{P_{k} b}_{L^{\infty} L^{2}}.
\end{equation}
Next, consider $k_{0}, k \in \bbZ_{\geq 0}$ and $I_{0} \in \calI_{\ell_{0}}$ with $\ell_{0} \in \bbZ_{\geq 0}$. We would have
\begin{align*}
2^{-\frac{\ell_{0}}{2}}\nrm{\chi_{I_{0}}(x^{3}) P_{k_{0}} \frka(x, D) P_{k} b}_{L^{2} L^{2}}  
&\aleq
\sum_{I \in \calI_{\ell_{0}}} 2^{-\frac{\ell_{0}}{2}}\nrm{\chi_{I_{0}}(x^{3}) P_{k_{0}} \frka(x, D) \td{P}_{k} (\chi_{I}(x^{3}) P_{k} b)}_{L^{2} L^{2}}   \\
& \aleq \sum_{I \in \calI_{\ell_{0}}} 2^{-\frac{\ell_{0}}{2}} 2^{-50 \abs{k_{0} -k}} \frac{[\frka]_{S^{0}; 500}}{\brk{2^{-\ell_{0}} \dist(I_{0}, I)}^{5}} \nrm{\chi_{I}(x^{3}) P_{k} b}_{L^{2} L^{2}} \\
&\aleq 2^{-50 \abs{k_{0} - k}} [\frka]_{S^{0}; 500} \nrm{P_{k} b}_{LE}.
\end{align*}
Note that, in the second inequality, we used \eqref{eq:psdo-ests-L2} when $\dist(I_{0}, I) = 0$ and \eqref{eq:psdo-ests-L2-nonlocal} when $\dist(I_{0}, I) \ageq 2^{\ell_{0}} $ (recall also that $2^{k_{\min}} \geq 1 \geq 2^{-\ell_{0}}$). Taking the supremum in $I_{0} \in \calI_{\ell_{0}}$ and $\ell_{0} \in \bbZ_{\geq 0}$ (and recalling the definition of $\nrm{\cdot}_{LE}$), we obtain
\begin{align*}
	2^{\frac{k_{0}}{2}}\nrm{P_{k_{0}} \frka(x, D) P_{k} b}_{LE} \aleq 2^{-50 \abs{k_{0} - k}} [\frka]_{S^{0}; 500} \nrm{P_{k} b}_{LE}.
\end{align*}
Then combined with \eqref{eq:psdo-ests-reduce-en} and Schur's test (for $k_{0}$ and $k$), \eqref{eq:psdo-ests-X} follows as desired. 

To prove \eqref{eq:psdo-ests-X-ell}, observe that, for any $I_{0}' \in \calI_{k}$, a minor variant of the above argument leads to
\begin{equation*}
	\nrm{\chi_{I_{0}'}(x^{3})P_{k_{0}} \frka(x, D) P_{k} b}_{X_{k}}
	\aleq \sum_{I' \in \calI_{k}} \frac{[\frka]_{S^{0}; 500}}{\brk{2^{-k}\dist(I', I_{0}')}^{2}} \nrm{\chi_{I'}(x^{3}) P_{k} b}_{X_{k}}.
\end{equation*}
Then by Lemma~\ref{lem:XY-slow-var}, we have
\begin{equation*}
	\nrm{P_{k_{0}} \frka(x, D) P_{k} b}_{\ell^{r}_{\calI} X_{k_{0}}}
	\aleq 2^{\frac{3}{2} \abs{k_{0} - k}} \nrm{P_{k_{0}} \frka(x, D) P_{k} b}_{\ell^{r}_{\calI} X_{k}}
	\aleq 2^{-(50-\frac{3}{2})\abs{k_{0} - k}} [\frka]_{S^{0}; 500} \nrm{P_{k} b}_{\ell^{r}_{\calI} X_{k}},
\end{equation*}
where we used bound and Schur's test (in $I_{0}', I' \in \calI_{k}$) in the second inequality. By Schur's test in $k, k_{0}$, \eqref{eq:psdo-ests-X-ell} follows. 

Finally, estimates \eqref{eq:psdo-ests-Y} and \eqref{eq:psdo-ests-Y-ell} follows from duality and Proposition~\ref{prop:psdo-adj}.

\smallskip
\noindent{\it Step~2: Proof of \eqref{eq:psdo-ests-L2} and \eqref{eq:psdo-ests-L2-nonlocal}.}
To complete the proof of \eqref{eq:psdo-ests-X}--\eqref{eq:psdo-ests-Y-ell}, it remains to establish \eqref{eq:psdo-ests-L2} and \eqref{eq:psdo-ests-L2-nonlocal}. By Proposition~\ref{prop:psdo-L2-bdd-garding}, we have
\begin{equation} \label{eq:psdo-ests-L2-bal}
	\nrm{P_{k_{0}} \frka(x, D) P_{k}}_{L^{2} \to L^{2}} \aleq [\frka]_{S^{0}; 30}.
\end{equation}
Moreover, recall that the integral kernel for $P_{k_{0}} \frka(x, D) P_{k}$ is given by
\begin{align*}
	K_{k_{0}, k}(x, y) = P_{k_{0}} \int \frka(x, \xi) \td{P}_{k}(\xi) e^{i \xi \cdot (x-y)} \, \frac{\ud \xi}{(2 \pi)^{d}}.
\end{align*}
Using $(x^{j} - y^{j}) e^{i \xi \cdot (x-y)} = - i \rd_{\xi_{j}} e^{i \xi \cdot (x-y)}$ and integrating by parts in $\xi$, we have, for $2^{k} \abs{x-y} > 1$,
\begin{equation*} 
	\abs{K_{k_{0}, k}(x, y)} \aleq 
	\frac{2^{3 k}}{\brk{2^{k} (x-y)}^{10}} [\frka]_{S^{0}; 100}.
\end{equation*}
It follows that
\begin{equation} \label{eq:psdo-ests-L2-nonlocal-bal}
	\sup_{\ell \in \bbZ_{\geq 0}} \sup_{I, I_{0} \in \calI_{\ell}}\nrm{\chi_{I_{0}}(x^{3}) P_{k_{0}} \frka(x, D) \td{P}_{k} \chi_{I}(x^{3})}_{L^{2} \to L^{2}}
	\aleq \frac{1}{\brk{2^{k} \dist(I_{0}, I)}^{5}}[\frka]_{S^{0}; 100},
\end{equation}
where we used \eqref{eq:psdo-ests-L2-bal} when $2^{k} \dist(I_{0}, I) \aleq 1$ and the kernel bound and Young's inequality when $2^{k} \dist(I_{0}, I) \ageq 1$. From these two bounds, \eqref{eq:psdo-ests-L2} and \eqref{eq:psdo-ests-L2-nonlocal} follow when $\abs{k - k_{0}} < 10$.

In the remaining case $\abs{k - k_{0}} \geq 10$, we need to obtain factors of $2^{-\abs{k_{0} - k}}$ by exploiting the frequency imbalance. Let $\doublewidetilde{P}_{k}(\xi)$ be a rescaled bump function adapted to the annulus $\set{\abs{\xi} \aeq 2^{k}}$ for $k \geq 1$ (resp.~ the ball $\set{\abs{\xi} \aeq 1}$ for $k = 0$) such that $\doublewidetilde{P}_{k}(\xi) \td{P}_{k}(\xi) = \td{P}_{k}(\xi)$, and let $\doublewidetilde{P}_{k} = \doublewidetilde{P}_{k}(D)$. For any $N \in \bbZ_{\geq 0}$, we may write (in operator notation)
\begin{align*}
&P_{k_{0}} \frka(x, D) \td{P}_{k} \\
&= 2^{\pm 2 N (k-k_{0})} P_{k} \left[ \left( 2^{\pm 2 N k_{0}} (1-\lap)^{\mp N}\td{P}_{k_{0}}\right) \left(  (1-\lap)^{\pm N} \frka(x, D) (1-\lap)^{\mp N} \right) \left( 2^{\mp 2 N k} (1-\lap)^{\pm N} \doublewidetilde{P}_{k} \right) \right] \td{P}_{k} \\
&=: 2^{\pm 2 N (k-k_{0})} P_{k} \frka_{k_{0}, k, N, \pm} \td{P}_{k} 
\end{align*}
where we choose $\pm = +$ when $k < k_{0}$ and $\pm = -$ when $k_{0} > k$. By Proposition~\ref{prop:psdo-comp}.(1), we have $\frka_{k_{0}, k, N, \pm} \in S^{0}$ with $[\frka_{k_{0}, k, N, \pm}]_{S^{0}; 100} \aleq [\frka]_{S^{0}; 500}$. Applying \eqref{eq:psdo-ests-L2-bal} and \eqref{eq:psdo-ests-L2-nonlocal-bal} to $\frka_{k_{0}, k, N, \pm}$ with a suitable $N$ (e.g., $N = 55$), the desired estimates \eqref{eq:psdo-ests-L2} and \eqref{eq:psdo-ests-L2-nonlocal} follow.

\smallskip
\noindent{\it Step~3: Proof of (2).} We now give a proof of (2), which is a small modification of the preceding argument. As in Step~1, the proof of \eqref{eq:psdo-ests-X-high}--\eqref{eq:psdo-ests-Y-ell-high} is reduced to the following bounds: for all $k_{0}, k \in \bbZ_{\geq 0}$ with $k \geq k_{1}-3$ (with $k_{1}$ satisfying \eqref{eq:psdo-ests-k1}),
\begin{align} 
	\nrm{P_{k_{0}} \frka(x, D) \td{P}_{k}}_{L^{2} \to L^{2}} & \aleq 2^{-40 \abs{k_{0} - k}} \nrm{\frka}_{L^{\infty}}, \label{eq:psdo-ests-L2-high} \\
	\sup_{\ell \in \bbZ_{\geq 0}} \sup_{I, I_{0} \in \calI_{\ell}}\nrm{\chi_{I_{0}}(x^{3}) P_{k_{0}} \frka(x, D) \td{P}_{k} \chi_{I}(x^{3})}_{L^{2} \to L^{2}}
	&\aleq 2^{-40 \abs{k_{0} - k}} \frac{1}{\brk{2^{k_{\min}} \dist(I_{0}, I)}^{4}} \nrm{\frka}_{L^{\infty}}.\label{eq:psdo-ests-L2-nonlocal-high}
\end{align}
Observe that the symbol $\frka_{k_{0}, k, N, \pm}$ from Step~2 obeys $\nrm{\frka_{k_{0}, k, N, \pm}}_{L^{\infty}} \aleq \nrm{\frka}_{L^{\infty}}$. Estimate \eqref{eq:psdo-ests-L2-high} is proved by applying Lemma~\ref{lem:hf-L2} instead of Proposition~\ref{prop:psdo-L2-bdd-garding} to $\frka$ and $\frka_{k_{0}, k, N, \pm}$ in the preceding proof of \eqref{eq:psdo-ests-L2}. To prove \eqref{eq:psdo-ests-L2-nonlocal-high}, note that we only need to consider $I_{0}, I \in \calI_{\ell}$ with $2^{k_{\min}} \dist(I_{0}, I) \ageq 1$ in view of \eqref{eq:psdo-ests-L2-high}. In this case, by \eqref{eq:psdo-ests-L2-nonlocal} and $\dist(I_{0}, I) \ageq 2^{\ell}$ (otherwise, $I_{0} \cap I \neq \0$ by our choice of $\calI_{\ell_{0}}$),
\begin{align*}
\nrm{\chi_{I_{0}}(x^{3}) P_{k_{0}} \frka(x, D) \td{P}_{k} \chi_{I}(x^{3})}_{L^{2} \to L^{2}}
	&\aleq 2^{-50 \abs{k_{0} - k}} \frac{1}{\brk{2^{k_{\min}} \dist(I_{0}, I)}^{5}} [\frka]_{S^{0}; 500} \\
	&\aleq 2^{-40 \abs{k_{0} - k}} \frac{1}{\brk{2^{k_{\min}} \dist(I_{0}, I)}^{4}} 2^{-k} 2^{-\ell}[\frka]_{S^{0}; 500}.
\end{align*}
Since $2^{\ell} \geq 1$ and $2^{k} \ageq \nrm{\frka}_{L^{\infty}}^{-1} [\frka]_{S^{0}; 500}$ by \eqref{eq:psdo-ests-k1} and $k \geq k_{1} - 3$, we obtain \eqref{eq:psdo-ests-L2-nonlocal-high} as desired. \qedhere 
\end{proof}

We also record here a localization property of the norm $\ell^{1}_{\calI} X^{s}$ that will be useful in Section~\ref{subsec:paralin-bdd} below, whose proof is a quick application of Propositions~\ref{prop:psdo-comp} and \ref{prop:psdo-ests}.
\begin{proposition} \label{prop:local-small}
For any $R \geq 1$ and $k \in \bbZ_{\geq 0}$, we have
\begin{equation*}
	\nrm{(1-\chi_{< 2R}(x^{3})) P_{< k} b}_{\ell^{1}_{\calI} X^{s}} \aleq \nrm{(1-\chi_{<R}(x^{3})) b}_{\ell^{1}_{\calI} X^{s}} + (2^{k} R)^{-1} \nrm{\chi_{< R} b}_{\ell^{1}_{\calI} X^{s}}
\end{equation*}
\end{proposition}
\begin{proof}
Splitting $b = (1-\chi_{<R}(x^{3})) b + \chi_{<R}(x^{3}) b$ and observing that $(1-\chi_{< 2R}(x^{3}))\chi_{<R}(x^{3}) = 0$, we see that it suffices to establish the following operator bounds
\begin{align*}
	\nrm{1-\chi_{< 2R}(x^{3})}_{\ell^{1}_{\calI} X^{s} \to \ell^{1}_{\calI} X^{s}} \aleq 1, \quad
	\nrm{P_{<k}}_{\ell^{1}_{\calI} X^{s} \to \ell^{1}_{\calI} X^{s}} \aleq_{s} 1, \quad
	\nrm{[P_{<k}, \chi_{<R}(x^{3})]}_{\ell^{1}_{\calI} X^{s} \to \ell^{1}_{\calI} X^{s}} \aleq_{s} 2^{-k} R^{-1}.
\end{align*}
These follow from Propositions~\ref{prop:psdo-comp}.(1) and \ref{prop:psdo-ests}.(1). \qedhere
\end{proof}

\subsection{Core product and commutator estimates} \label{subsec:product}
In this subsection, we collect the core product and commutator estimates in the function spaces defined in Section~\ref{subsec:ftn-sp} (cf.~\cite{MMT1}).
\begin{lemma} \label{lem:prod-core}
Let $T > 0$, and consider spaces of functions defined on $(0, T) \times \bbR^{3}$. For any $j, k, \ell \in \bbZ_{\geq 0}$, the following product bounds hold:
\begin{align} 
	\nrm{P_{\ell} (P_{j} a P_{k} b)}_{\ell^{1}_{\calI}  Y_{\ell}} &\aleq 2^{\frac{5}{2} k} 2^{-\frac{1}{2} \ell - \frac{1}{2} j} \nrm{P_{j} a}_{\ell^{\infty}_{\calI}  X_{j}} \nrm{P_{k} b}_{\ell^{1}_{\calI}  {L^{\infty} L^{2}_{k} }}, \label{eq:prod-core-hl} \\
	\nrm{P_{\ell} (P_{j} a P_{k} b)}_{\ell^{1}_{\calI}  Y_{\ell}} &\aleq T 2^{\frac{3}{2} \ell} \nrm{P_{j} a}_{X_{j}} \nrm{P_{k} b}_{X_{k}}. \label{eq:prod-core-hh}
\end{align}
Finally, provided that $j < k-8$, the following commutator bound holds:
\begin{equation} \label{eq:prod-core-comm}
	\nrm{[P_{k}, P_{j} a] b}_{\ell^{1}_{\calI}  Y_{k}} \aleq 2^{\frac{7}{2} j} 2^{- 2 k}  \nrm{P_{j} a}_{\ell^{1}_{\calI} L^{\infty}L^{2}_{j}} \nrm{P_{[k-3, k+3]} b}_{\ell^{\infty}_{\calI}  X_{k}}
\end{equation}
\end{lemma}
\begin{proof}
Estimate \eqref{eq:prod-core-hl} is proved as follows.
\begin{align*}
\nrm{P_{\ell} (P_{j} a P_{k} b)}_{\ell^{1}_{\calI}  Y_{\ell} }
&\aleq \sum_{I \in \calI_{\ell}} 2^{-\frac{1}{2} \ell} \nrm{\chi_{I}(x^{3}) P_{\ell} (P_{j} a P_{k} b)}_{ { LE^{\ast} } } & &\\
&\aleq 2^{-\frac{1}{2} \ell} 2^{\frac{1}{2} k} \sum_{I' \in \calI_{k}} \nrm{\chi_{I'}(x^{3}) P_{\ell} (P_{j} a P_{k} b)}_{L^{2} L^{2}} & & \hbox{(by \eqref{eq:LE*-def})} \\
&\aleq 2^{-\frac{1}{2} \ell} 2^{\frac{1}{2} k} \sup_{I' \in \calI_{k}} \nrm{\td{\chi}_{I'}(x^{3}) P_{j} a}_{L^{2} L^{2}} \sum_{I' \in \calI_{k}} \nrm{\chi_{I'}(x^{3}) P_{k} b}_{L^{\infty} L^{\infty}} & & \hbox{(H\"older)} \\
&\aleq 2^{- \frac{1}{2} \ell - \frac{1}{2} j} 2^{\frac{5}{2} k} \left( 2^{\frac{1}{2} j} \nrm{P_{j} a}_{{LE} } \right) \sum_{I' \in \calI_{k}} \nrm{\chi_{I'}(x^{3}) P_{k} b}_{L^{\infty} L^{2}} & & \hbox{(by \eqref{eq:LE-def} and Corollary~\ref{cor:local-bernstein})} \\
&\aleq 2^{- \frac{1}{2} \ell - \frac{1}{2} j} 2^{\frac{5}{2} k} \nrm{P_{j} a}_{\ell^{\infty}_{\calI} X_{j}} \nrm{P_{k} b}_{\ell^{1}_{\calI} L^{\infty} L^{2}_{k}}. & & 
\end{align*}
Next, estimate \eqref{eq:prod-core-hh} is proved as follows.
\begin{align*}
\nrm{P_{\ell} (P_{j} a P_{k} b)}_{\ell^{1}_{\calI}  Y_{\ell}}
&\aleq \sum_{I \in \calI_{\ell}} \nrm{\chi_{I}(x^{3}) P_{\ell} (P_{j} a P_{k} b)}_{L^{1} L^{2}} & &\\
&\aleq \sum_{I \in \calI_{\ell}} 2^{\frac{3}{2} \ell} \nrm{\chi_{I}(x^{3}) P_{\ell} (P_{j} a P_{k} b)}_{L^{1} L^{1}} & & \hbox{(Corollary~\ref{cor:local-bernstein})} \\
&\aleq 2^{\frac{3}{2} \ell} \nrm{P_{\ell} (P_{j} a P_{k} b)}_{L^{1} L^{1}} & & \hbox{(Fubini)}\\
&\aleq T 2^{\frac{3}{2} \ell} \nrm{P_{j} a}_{L^{\infty} L^{2}} \nrm{P_{k} b}_{L^{\infty} L^{2}} & & \hbox{(H\"older)} \\
&\aleq T 2^{\frac{3}{2} \ell} \nrm{P_{j} a}_{X_{j}} \nrm{P_{k} b}_{X_{k}}. & & 
\end{align*}

Finally, we prove the commutator bound \eqref{eq:prod-core-comm}. Since $j < k-8$, by considering the Fourier supports, it follows that $\nb [P_{k}, P_{j} a] b
= \nb [P_{k}, P_{j} a] P_{[k-3, k+3]} b$. Then by \cite[(3.21)]{MMT1}, we have the following representation for the commutator:
\begin{align*}
\nb [P_{k}, P_{j} a] b = P_{[k-5, k+5]} L(\nb P_{j} a, P_{[k-3, k+3]} b),
\end{align*}
where $L$ is a translation invariant bilinear operator of the form
\begin{equation*}
	L(a, b)(x) = \int a(x-y)  b(x-z) w(y, z) \, \ud y \ud z, 
\end{equation*}
with $\nrm{w}_{L^{1}} \aleq 1$ (independent of $k, j$). Since the norms $\ell^{1}_{\calI} Y_{k}$, $\ell^{1} X_{j}$ and $\ell^{\infty} X_{k}$ are translation invariant (in the sense that for each of these norms, its translates are equivalent norms), the desired bound follows from \eqref{eq:prod-core-hl} (with $(j, k, \ell)$ set to $(k', j, k'')$ with $k', k'' = k + O(1)$). \qedhere
\end{proof}

\section{Analysis of the paralinearized system and the paralinearization error} \label{sec:mag-lin}

The goal of this section is to establish Propositions~\ref{prop:paralin-full} and \ref{prop:paralin-err}. In Subsection \ref{subsec:phys-loc}, Proposition~\ref{prop:paralin-full} is reduced to a similar estimate without the $\ell^{1}_{\calI}$-summation structure. The rest of this section is devoted to the proof of this reduced statement, except for the last Subsection \ref{subsec:paralin-err}, where Proposition~\ref{prop:paralin-err} is proved. 

\subsection{Assumptions, parameters and conventions}

\subsubsection{Global assumptions}
Throughout this section, with the exception of Section~\ref{subsec:paralin-err} (where Proposition~\ref{prop:paralin-err} is proved), we assume that
\begin{equation*}
	\bfB_{0} := \bfB(0) \in \calB^{s_{1}}_{\eps}(M, \mu, A, R, L), \quad 
	\nrm{\bfB - \bfe_{3}}_{\ell^{1}_{\calI} X^{s_{1}}[0, T]} + \nrm{\rd_{t} \bfB}_{\ell^{1}_{\calI} L^{\infty} H^{s_{1}-2}[0, T]} \leq M,
\end{equation*}
which are the hypotheses for Proposition~\ref{prop:paralin-full}, except  {we introduced the shorthand $M := \max\set{M_{1}, M_{1}'}$}, and we removed the subscript $1$ from the other constants for notational convenience.  {We will assume that $\eps_{(1)}$, $c_{(1)}$, $C_{(1)}$, $c_{(2)}$, $C_{(2)}$ depend on $M = \set{M_{1}, M_{1}'}$ instead of $M_{1}$ and $M_{1}'$ separately.} For simplicity, we also assume that 
 \begin{equation*}
A \geq M, \quad L \geq \frac{1}{24} R  \geq \frac{1}{24}.
\end{equation*}
Observe, from the statement of Proposition~\ref{prop:paralin-full}, that making the assumption $A \geq M$ does not lose any generality. Moreover, $L \ageq R$ will be guaranteed by our assumptions on $\eps$ made below,  {while $R \geq 1$ is guaranteed by definition}.
 
We also fix $s_{0}$ such that $\frac{7}{2} < s_{0} < s_{1}$ (e.g., $s_{0} = \frac{1}{2} (s_{1} + \frac{7}{2})$) and let $\eps_{0}$ be the small positive constant from Lemma~\ref{lem:nontrapping-id}. Given $k_{(0)} \in \bbZ_{\geq 0}$, note that
\begin{equation*}
	\nrm{\bfB_{0} - (\bfB_{0})_{<k_{(0)}}}_{\ell^{1}_{\calI} H^{s_{0}}} \aleq 2^{-(s_{1}-s_{0}) k_{(0)}} M.
\end{equation*}
Therefore, in view of Lemma~\ref{lem:nontrapping-id} and Corollary~\ref{cor:nontrapping-id}, there exist $L_{0} \geq \frac{1}{24} R_{0} \geq \frac{1}{24}$ that depends only on $M$, $\mu$ and $A$, but \emph{not} on $\eps$ or $k_{(0)}$, such that
\begin{equation} \label{eq:nontrapping-B0<k0-fixed}
	(\bfB_{0})_{< k_{(0)}} \in \calB_{\eps_{0}}^{s_{0}}(2 M, \tfrac{1}{2} \mu, 2A, R_{0}, L_{0}),
\end{equation}
for $k_{(0)}$ sufficiently large depending on $s_{1}$, $M$, $\mu$, $A$ (but not on $\eps$ or $L$).
Similarly, we also have
\begin{equation} \label{eq:nontrapping-B0<k0}
	(\bfB_{0})_{<k_{(0)}} \in \calB^{s_{0}}_{2 \eps}(2M, \tfrac{1}{2}\mu, 2A, R, 2 L),
\end{equation}
but for $k_{(0)}$ sufficiently large depending on $s_{1}$, $M$, $\mu$, $A$ \emph{as well as} $\eps$ and $L$.

\subsubsection{Parameter choices}
We assume that 
\begin{equation*}
\eps \leq \eps_{(1)}, \quad T \leq c_{(1)} e^{-C_{(1)} L},
\end{equation*}
where $\eps_{(1)}$, $c_{(1)}$ and $C_{(1)}$ will be fixed towards the end of the proof -- it will be crucial to make these choice depend only on $\sgm$, $s_{1}$, $M$, $\mu$ and $A$, but \emph{not} on $R$ and $L$ (in particular, letting $\eps_{(1)}$ depend on $R$ or $L$ would lead to circular logic). Other important parameters are the frequency cutoff parameter $k_{(0)} \in \bbZ_{\geq 0}$, which was already introduced in \eqref{eq:nontrapping-B0<k0-fixed} and \eqref{eq:nontrapping-B0<k0}, as well as another frequency cutoff parameters $k_{(1)} \in \bbZ_{\geq 0}$, which will be used in the application of Proposition~\ref{prop:psdo-ests}.(2) (high frequency Calder\'on--Vaillancourt). 
 
The choices of $\eps_{(1)}$, $k_{(0)}$, $k_{(1)}$, $c_{(1)}$, $C_{(1)}$, $c_{(2)}$, $C_{(2)}$ will be finalized in Section~\ref{subsec:paralin-pf}, in the order
\begin{equation*}
	\eps_{(1)} \to k_{(0)} \to k_{(1)} \to c_{(1)}, C_{(1)} \to c_{(2)}, C_{(2)},
\end{equation*}
with the following restrictions (as well as others that will arise in the proof below):
\begin{itemize}
\item $\eps_{(1)} \leq \eps_{0}$, so that $\eps < \eps_{0}$ and Lemma~\ref{lem:nontrapping-cor} is applicable (in particular, this assumption assures $L \geq \frac{1}{24} R$ by Corollary~\ref{cor:nontrapping-id});
\item $k_{(0)}$ is sufficiently large so that \eqref{eq:nontrapping-B0<k0-fixed} and \eqref{eq:nontrapping-B0<k0} hold;
\item $k_{(1)} \geq k_{(0)}$.
\end{itemize}

\subsubsection{Notation and conventions}
Throughout this section, {\bf we suppress the dependence of implicit constants on $s_{1}$, $s_{0}$ and $\sgm$}.
Furthermore, given $g = (g_{+}, g_{-})^{\top}$ and a norm $Z$, we use the shorthand $\nrm{g}_{Z}^{2} := \sum_{\pm} \nrm{g_{\pm}}_{Z}^{2}$. Unless otherwise specified, {\bf all spacetime norms in this section are taken over $[0, T] \times \bbR^{3}$.}

\subsection{Physical space localization} \label{subsec:phys-loc}
The goal of this subsection is to reduce the proof of Proposition~\ref{prop:paralin-full} to establishing the following estimate for the $X^{\sgm}$ and $Y^{\sgm}$ spaces (without the $\ell^{1}_{\calI}$ structure):
\begin{proposition} \label{prop:paralin}
Let $s_{1} > \frac{7}{2}$ and $0 < T \leq 1$. Assume that $\bfB$ satisfies the assumptions in Proposition~\ref{prop:paralin-full}.
Let $b$ and $g$ solve the paralinearized system \eqref{eq:paralin}--\eqref{eq:paralin-constraint} on $(0, T)$. Then, for any $\sgm \geq 0$, the following holds: if
\begin{equation}\label{eq:paralin-core-para}
\eps \leq \eps_{(1)}(\sgm, s_{1}, M, \mu, A), \quad T \leq c_{(1)}(\sgm, s_{1}, M, \mu, A) \exp(-C_{(1)}(\sgm, s_{1}, M, \mu, A) L)
\end{equation}
then we have 
\begin{equation} \label{eq:paralin-core}
	\nrm{b}_{X^{\sgm}[0, T]} 
	\leq C(\sgm, s_{1}, M, \mu, A) e^{C(\sgm, s_{1}, M, \mu, A) L} \left(  \nrm{b(0)}_{H^{\sgm}} + \nrm{g}_{Y^{\sgm}[0, T]} \right).
\end{equation}
\end{proposition}
After showing that Proposition~\ref{prop:paralin} implies Proposition~\ref{prop:paralin-full} in this subsection, the remainder of this section will be devoted to the proof of Proposition~\ref{prop:paralin}.

\begin{proof}[Proof of Proposition~\ref{prop:paralin-full} assuming Proposition~\ref{prop:paralin}]
The proof proceeds in several steps. For simplicity, we shall not write out the dependence of constants in $\sgm,M,\mu,A,R,L,s_{1}$ and just write $\lesssim$. 

\medskip

\pfstep{Step~1}  We first relax the constraints $\nb \cdot b = 0$ and $\nb \cdot g = 0$. Let us denote by $\Pi_{\df}(D) := I - \Pi_{0}(D)$ the Leray projection to the space of divergence-free vector fields in $L^{2}$. To avoid confusion, let $b'$ and $g'$ satisfy \eqref{eq:paralin}, possibly without the divergence-free condition \eqref{eq:paralin-constraint}. Then we have \begin{equation}\label{eq:paralin-df}
	\begin{split}
			&\rd_{t} \Pi_{\df}(D) b' - \nb \times (T_{\bfB} \times (\nb \times \Pi_{\df}(D) b')) + \nb \times (T_{\nb \times \bfB} \times \Pi_{\df}(D) b') \\
		& = \Pi_{\df}(D) g' - \nb \times (T_{\nb \times \bfB} \times \Pi_{0}(D) b')  , 
	\end{split}
\end{equation}
and \begin{equation*}
	\begin{split}
		\rd_{t}  \Pi_{0}(D) b' =  \Pi_{0}(D) g'.
	\end{split}
\end{equation*} We simply propagate the usual Sobolev regularity for the equation for $\Pi_{0}(D) b'$: we have $ \nrm{\Pi_{0}(D) b'}_{L^{\infty} H^{\sgm+1}[0, T]} \lesssim  \nrm{\Pi_{0}(D) b'(0)}_{H^{\sgm+1}} + \nrm{\Pi_{0}(D) g'}_{L^{1} H^{\sgm+1}[0, T]} $ for any $\sgm$. From the product estimate in usual Sobolev spaces and $L^{1} H^{\sgm}[0, T] \subseteq Y^{\sgm}[0, T]$, we have \begin{equation*}
\begin{split}
	\nrm{\nb \times (T_{\nb \times \bfB} \times \Pi_{0}(D) b')}_{ Y^{\sgm}[0,T] } \lesssim T \nrm{ \Pi_{0}(D) b' }_{L^{\infty}H^{\sgm+1}[0,T]} . 
\end{split}
\end{equation*}  Using $T\le1$, controlling $X^{\sgm}$ of $\Pi_{0}(D) b'$ by $L^\infty H^{\sgm+1}$ and applying Proposition~\ref{prop:paralin} to \eqref{eq:paralin-df} gives 
\begin{equation}\label{eq:paralin-step1}
	\nrm{b'}_{X^{\sgm}[0, T]} 
\lesssim  \left(  \nrm{b'(0)}_{H^{\sgm}} +  \nrm{\Pi_{0}(D)b'(0)}_{H^{\sgm+1}} + \nrm{\Pi_{\df}(D) g'}_{Y^{\sgm}[0, T]}
	+ \nrm{\Pi_{0}(D) g'}_{L^{1} H^{\sgm+1}[0, T]} \right),
\end{equation}
for any $\sgm$. In what follows, we shall use the fact that $\Pi_{0}(D)$ and $\Pi_{\df}$ are bounded operators in $X^{\sgm}$ and $Y^{\sgm}$, which follows from Proposition \ref{prop:psdo-ests}. 

\medskip

\pfstep{Step~2} Again, assume that $b'$ and $g'$ satisfy \eqref{eq:paralin}. For any fixed $z_{0} \in \bbR$, note that
\begin{align*}
	&(x^{3} - z_{0}) \nb \times (T_{\bfB} \times (\nb \times b' ))
	-  \nb \times (T_{\bfB} \times (\nb \times ((x^{3} - z_{0}) b' ))) \\
	&= \sum_{k} (x^{3} - z_{0}) \nb \times (P_{<k-10} \bfB \times (\nb \times P_{k} b' ))
	- \nb \times (P_{<k-10} \bfB \times (\nb \times P_{k}((x^{3} - z_{0}) b' ))) \\
	&= \sum_{k} O(P_{<k-10} \bfB, \nb P_{k} b' ),
\end{align*}
and the $L^{1} H^{\sgm}[0, T]$ norm of the RHS can be bounded by $T \nrm{b' }_{L^{\infty} H^{\sgm+1}[0, T]}$. Writing the other term in the $b'$ equation in a similar way, we note that $(x^{3}-z_{0})b' $ satisfies \eqref{eq:paralin} with the RHS of the form \begin{equation*}
	\begin{split}
		(x^{3}-z_{0})g' + \sum_{k} \left( O(P_{<k-10} \bfB, \nb P_{k} b') + O(P_{<k-10} \nb \bfB,  P_{k} b')\right) . 
	\end{split}
\end{equation*}
Since $T \leq 1$, applying \eqref{eq:paralin-step1} to the equation for $(x^{3}-z_{0})b' $ and again to $b' $ with $\sgm$ replaced by $\sgm+1$, we have
\begin{align*}
	\nrm{ (x^{3} - z_{0}) b'  }_{X^{\sgm}[0, T]}
	&\aleq \nrm{  (x^{3} - z_{0})  b' (0)}_{H^{\sgm}} + \nrm{ \Pi_{0}(D)(  (x^{3} - z_{0})  b' )(0)}_{H^{\sgm+1}} +  \nrm{ b'(0) }_{H^{\sgm+1}} + \nrm{ \Pi_{0}(D) b'(0) }_{H^{\sgm+2}}  \\
	& \qquad + \nrm{\Pi_{\df}(D)( (x^{3} - z_{0}) g')}_{Y^{\sgm}[0, T]} + \nrm{\Pi_{0}(D)( (x^{3} - z_{0}) g')}_{L^{1}H^{\sgm+1}[0,T]} \\
	&\qquad  +  \nrm{\Pi_{\df}(D) g' }_{Y^{\sgm+1}[0, T]} + \nrm{\Pi_{0}(D)g'}_{L^{1}H^{\sgm+2}[0, T]}.
\end{align*}
At this point, we may replace $({x^{3} - z_{0}})$ with $\brk{{x^{3} - z_{0}}}$. Similarly, commuting with $(x^{3} - z_{0})^{N}$ and using induction on $N$, we obtain that\begin{equation}\label{eq:paralin-loc-pre}
	\begin{split}
		\nrm{\brk{x^{3} - z_{0}}^{N} b'}_{X^{\sgm}[0, T]} 
		&\aleq \sum_{ m = 0}^{N} \left(\nrm{ \brk{x^{3} - z_{0}}^{m}  b' (0)}_{H^{\sgm+N-m}} + \nrm{ \Pi_{0}(D)(   \brk{x^{3} - z_{0}}^{m}  b' )(0)}_{H^{\sgm+N-m+1}} \right)\\
		& \qquad + \sum_{ m = 0}^{N} \left(  \nrm{\Pi_{\df}(D)( \brk{x^{3} - z_{0}}^{m} g')}_{Y^{\sgm+N-m}[0, T]} + \nrm{\Pi_{0}(D)( \brk{x^{3} - z_{0}}^{m} g')}_{L^{1}H^{\sgm+N-m+1}[0,T]} \right) 
	\end{split}
\end{equation}  holds for any integer $N\ge 1$. 

\medskip 

\pfstep{Step~3}  Now, let $b$ and $g$ satisfy \eqref{eq:paralin} together with the divergence-free condition \eqref{eq:paralin-constraint}. Take $I' \in \calI_{k'}$ and $I \in \calI_{k}$. Then, define $b_{k,I}$ as the solution of \eqref{eq:paralin} with initial data $b_{k,I}(0) = \chi_{I} P_{k}b(0)$ and the  {RHS}  of \eqref{eq:paralin}  given by $\chi_{I}P_{k}g$. In particular, we have that $b = \sum_{k \ge 0} \sum_{I \in \calI_{k}} b_{k, I}$. 

We claim that for any $N, N' \ge 1$, 
\begin{align}\label{eq:paralin-loc}
	\nrm{\chi_{I'} P_{k'} b_{k, I}}_{X^{\sgm}[0, T]}
	\aleq (2^{-k} \mathrm{dist}(I, I'))^{-N} 2^{-N' \abs{k-k'}} \sum_{I''} \left( \nrm{\chi_{I''} P_{k} b(0)}_{H^{\sgm}}
	+ \nrm{\chi_{I''} P_{k} g}_{Y^{\sgm}[0, T]} \right),
\end{align} where the sum in the  {RHS} is over the intervals $I''$ which intersect with the support of $\chi_{I}$. To show \eqref{eq:paralin-loc}, we first apply \eqref{eq:paralin-loc-pre} with $(b',g',\sgm)$ replaced by $(b_{k,I}, \chi_{I} P_{k}g , \sgm + N')$, where $z_{0}$ is taken to be the midpoint of $I$. Furthermore, we assume that $\dist(I,I') \ge 100\cdot 2^{k}$ and $k' \ge k$. (When $k' < k$, we put $\sgm - N'$ in place of $\sgm +N'$.) The point is that, since $P_{k}b(0)$ and $P_{k}g$ are both divergence-free, we gain one derivative in the estimate of the terms involving $\Pi_{0}(D) = \lap^{-1}\nb \nb \cdot $ in the  {RHS} of \eqref{eq:paralin-loc-pre}. This gives: 
\begin{equation*}
	\begin{split}
		\nrm{ \brk{x^{3} - z_{0} }^{N} b_{k,I} }_{X^{\sgm+N'}[0,T]} \lesssim 2^{k(N+N')}\sum_{I''}( \nrm{ \chi_{I''} P_{k} b(0)}_{H^{\sgm}} + \nrm{ \chi_{I''} P_{k} g }_{Y^{\sgm}[0,T]} ). 
	\end{split}
\end{equation*} Comparing this with \eqref{eq:paralin-loc}, it suffices to show  \begin{equation*}
\begin{split}
	\nrm{ \chi_{I'} P_{k'} b_{k,I} }_{X^{\sgm}[0,T]} &\lesssim 2^{-k(N+N')}  (2^{-k} \mathrm{dist}(I, I'))^{-N} 2^{N' ({k-k'}) }  \nrm{ \brk{x^{3} - z_{0} }^{N}b_{k,I} }_{X^{\sgm+N'}[0,T]} \\
	&=  \mathrm{dist}(I, I')^{-N} 2^{-N'k'}\nrm{  \brk{x^{3} - z_{0} }^{N} b_{k,I} }_{X^{\sgm+N'}[0,T]} .
\end{split}
\end{equation*} To estimate the LHS, we need to bound the $LE$ and $L^\infty L^2$ norms. We focus on the $LE$ norm bound since the other one is only easier. To this end we need to estimate \begin{equation}\label{eq:parlin-NEED}
\begin{split}
	\sum_{k_0\ge0} 2^{2k_0\sgm} 2^{2N'k'} 2^{k_{0}} \dist(I,I')^{2N} \nrm{ P_{k_{0}} \chi_{I'} P_{k'} b_{k,I} }_{LE}^{2} \lesssim \nrm{  \brk{x^{3} - z_{0} }^{N} b_{k,I} }_{X^{\sgm+N'}[0,T]}^{2} .
\end{split}
\end{equation} 
We first note the off-diagonal decay in $|k'-k_{0}|$: when $k_{0} > k'$, we may write \begin{equation*}
\begin{split}
	P_{k_{0}} \chi_{I'} P_{k'} = 2^{-n k_{0}} \widetilde{P}_{k_{0}} \nb^{n}(\chi_{I'} P_{k'}) 
\end{split}
\end{equation*} and the  {RHS} is a finite sum of operators of the form \begin{equation*}
\begin{split}
	2^{-n( k_{0} - k')} \widetilde{P}_{k_{0}} \widetilde{\chi}_{I'} \widetilde{P}_{k'}, 
\end{split}
\end{equation*} where we write $ \widetilde{P}_{k}$ to be an order 0 Fourier multiplier localized at frequencies $\sim 2^{k}$. It will be convenient to allow the precise form of $\widetilde{P}_{k}$ to vary from a line to another. Similarly, $\widetilde{\chi}_{I'}$ denotes some smooth function which retains the spatial localization property of $\chi_{I'}$. In the opposite case when $k' > k_{0}$, we can commute derivatives outside of $P_{k'}$ to gain decay of the form $2^{-n|k_{0} - k'| }$ again. 

Now we focus on the case $ k_{0} = k'$. The general case can be handled similarly, using the decay $2^{-n|k_{0} - k'| }$. The corresponding term from \eqref{eq:parlin-NEED} is simply \begin{equation*}
	\begin{split}
		2^{2k'(\sgm+N')} 2^{k'} \dist(I,I')^{2N} \nrm{ P_{k'} \chi_{I'} P_{k'} b_{k,I} }_{LE}^{2}. 
	\end{split}
\end{equation*} We recall that $P_{k'}$ is bounded in $LE$ and note the pointwise bound of $\chi_{I'} P_{k'} b_{k,I}$: \begin{equation*}
\begin{split}
	\dist(I,I')^{N}|\chi_{I'} P_{k'} b_{k,I}(x) | \lesssim_{K} \chi_{I'}(x) \int \frac{2^{3k'}}{ \brk{ 2^{k'} |x-y| }^{K} } \frac{\dist(I,I')^{N}}{ \brk{y^{3} - z_{0} }^{N}} ( \brk{y^{3} - z_{0} }^{N} |b_{k,I}(y)|  ) \, \ud y .
\end{split}
\end{equation*} In the integral, consider separately the regions $y^{3} \in 2I'$ and $y^{3} \notin 2I'$, where $2I'$ is the interval having the same center with $I'$ but twice thicker. In the first case, we simply have ${\dist(I,I')^{N}} \lesssim {|y^{3} - z_{0}|^{N}} $ and the corresponding integral in $y$ is bounded by $\nrm{ \widetilde{P}_{k'} ( |x^{3}- z_{0}|^{N} b_{k,I} ) }_{LE}$. In the other case when  $y^{3} \notin 2I'$, notice that whenever $x \in I'$, we have $2^{k'}|x-y| + |y^{3}-z_{0}| \gtrsim \dist(I,I')$. Therefore, this time we can take $K = N + 10$ and arrive at the same bound. 

Estimating similarly the off diagonal terms in \eqref{eq:parlin-NEED} and summing up the estimates, we arrive at \eqref{eq:paralin-loc}. 

\medskip 

\pfstep{Step~4} Given \eqref{eq:paralin-loc}, we have \begin{equation*}
	\begin{split}
		2^{k'\sgm} \nrm{P_{k'} b}_{\ell^{1}_{\calI}  X_{k'}[0, T]}
		\aleq \sum_{I'} \sum_{k } (2^{-k} \mathrm{dist}(I, I'))^{-N} 2^{-N' \abs{k-k'}} \sum_{I} \sum_{I''}\left( \nrm{\chi_{I''} P_{k} b(0)}_{H^{\sgm}}
		+ \nrm{\chi_{I''} P_{k} g}_{Y^{\sgm}[0, T]} \right),
	\end{split}
\end{equation*} and 
by Schur's test (the loss is $O(2^{\abs{k-k'}})$ if $N$ is sufficiently large), we deduce that 
\begin{align*}
		2^{k'\sgm} \nrm{P_{k'} b}_{\ell^{1}_{\calI}  X_{k'}[0, T]}
	\aleq \sum_{k} 2^{-(N'-1) \abs{k-k'}}  \left( \nrm{P_{k} b(0)}_{\ell^{1}_{\calI} H^{\sgm}}
	+ \nrm{P_{k} g}_{\ell^{1}_{\calI} Y^{\sgm}[0, T]} \right).
\end{align*}
Choosing $N'$ large enough, we obtain the desired bound. \end{proof}

The remainder of this section (with the exception of Section~\ref{subsec:paralin-err}) is devoted to the proof of Proposition~\ref{prop:paralin}.

\subsection{Principal diagonalization and conjugation by Bessel potential}
In this subsection, we rewrite \eqref{eq:paralin} as a system of equations for $\tilde{b}_{\pm} = \brk{D}^{\sgm} \Pi_{\pm}(D) b$, which has a diagonalized paradifferential principal part. 

\begin{proposition} \label{prop:para-diag-tdb}
Let $b$ be a solution to the system \eqref{eq:paralin}--\eqref{eq:paralin-constraint} on $(0, T)$. Define 
\begin{equation*}
\tb_{\pm} = \brk{D}^{\sgm} \Pi_{\pm}(D) b, \quad \tilde{g}_{\pm} = \brk{D}^{\sgm} \Pi_{\pm}(D) g.
\end{equation*}
Then $\brk{D}^{\sgm} b = \tb_{+} + \tb_{-}$, $\brk{D}^{\sgm} g = \tilde{g}_{+} + \tilde{g}_{-}$. Moreover, these solve
\begin{equation} \label{eq:para-diag-tb}
\begin{alignedat}{2}
	 & \rd_{t} \begin{pmatrix} \tb_{+} \\ \tb_{-} \end{pmatrix} 
	 + \begin{pmatrix} \diag_{\bfB}^{(2) \sharp} & 0 \\ 0 & - \diag_{\bfB}^{(2) \sharp} \end{pmatrix} \begin{pmatrix} \tb_{+} \\ \tb_{-} \end{pmatrix} & \\
	&
	 + \begin{pmatrix} \comm^{\sgm (1)}_{(\bfB_{0})_{<k_{(0)}}} & 0 \\ 0 & - \comm^{\sgm (1)}_{(\bfB_{0})_{<k_{(0)}}} \end{pmatrix} \begin{pmatrix} \tb_{+} \\ \tb_{-} \end{pmatrix} 
	  + \begin{pmatrix} 0 &  \symm^{(1)}_{(\bfB_{0})_{< k_{(0)}}}   \\ - \symm^{(1)}_{(\bfB_{0})_{< k_{(0)}}}  & 0 \end{pmatrix} \begin{pmatrix} \tb_{+} \\ \tb_{-} \end{pmatrix}  & \\
	& +  \begin{pmatrix} \asymm^{(1)}_{(\bfB_{0})_{< k_{(0)}}}  & 0 \\ 0 & - \asymm^{(1)}_{(\bfB_{0})_{< k_{(0)}}}  \end{pmatrix} \begin{pmatrix} \tb_{+} \\ \tb_{-} \end{pmatrix} 
	+ (\nb \times (\bfB_{0})_{<k_{(0)}})\cdot \nb \begin{pmatrix}  \tb_{+} \\ \tb_{-} \end{pmatrix}  & \\
	& + \nb \begin{pmatrix}  \covec^{(0)}_{(\bfB_{0})_{< k_{(0)}}} & \covec^{(0)}_{(\bfB_{0})_{<k_{(0)}}}  \\ - \covec^{(0)}_{(\bfB_{0})_{<k_{(0)}}}  & - \covec^{(0)}_{(\bfB_{0})_{< k_{(0)}}} \end{pmatrix}  \begin{pmatrix}  \tb_{+} \\ \tb_{-}  \end{pmatrix}
	&= \begin{pmatrix}  \tilde{h}_{+} \\ \tilde{h}_{-}  \end{pmatrix}  
\end{alignedat}
\end{equation}
where $\diag_{\bfB}^{(2) \sharp}$ represents the anti-symmetric 2nd order paradifferential operator
\begin{equation}\label{eq:principal-def}
	\diag_{\bfB}^{(2) \sharp} b = \frac{1}{2} T_{\bfB^{\alp}} \rd_{\alp} \abs{\nb} b + \frac{1}{2} \abs{\nb} \rd_{\alp} T_{\bfB^{\alp}}^{\ast} = \sum_{k} \left[ \frac{1}{2} \bfB^{\alp}_{<k-10} P_{k} \rd_{\alp} \abs{\nb} b + \frac{1}{2} \abs{\nb} \rd_{\alp} P_{k} \bfB^{\alp}_{<k-10}\right],
\end{equation}
$\comm^{\sgm (1)}_{(\bfB_{0})_{<k_{(0)}}}$ represents the commutator term
\begin{equation} \label{eq:comm-def}
\comm^{\sgm (1)}_{(\bfB_{0})_{<k_{(0)}}} b = [\brk{D}^{\sgm}, \diag_{(\bfB_{0})_{<k_{(0)}}}^{(2)}] \brk{D}^{-\sgm}  b,
\end{equation}
and
\begin{align}\label{eq:tilde-h-def}
\begin{pmatrix}  \tilde{h}_{+} \\ \tilde{h}_{-}  \end{pmatrix}
= \begin{pmatrix}  \tilde{g}_{+} \\ \tilde{g}_{-}  \end{pmatrix} + \begin{pmatrix} \tilde{E}_{0; +} \\ \tilde{E}_{0; -} \end{pmatrix},
\end{align}
with $\tilde{E}_{0}$ referring to error terms that obey the following bounds:
 \begin{align}
	\nrm{\tilde{E}_{0; +}}_{Y^{0}}
	+ \nrm{\tilde{E}_{0; -}}_{Y^{0}} &\aleq (2^{-(s_{0} - \frac{7}{2}) k_{(0)}} + T 2^{2 k_{(0)}}) M \sum_{\pm} \nrm{\tb_{\pm}}_{X^{0}}. \label{eq:para-diag-tb-err}
\end{align}
Finally, $\comm^{\sgm (1)}_{(\bfB_{0})_{<k_{(0)}}}  {= 0}$ if $\sgm = 0$.
\end{proposition}

\begin{remark}
	Notice that except for the second order operator \eqref{eq:principal-def}, we have replaced $\bfB$ in the coefficient by $(\bfB_{0})_{<k_{(0)}}$ in all of the other terms. 
\end{remark}

We collect here the key estimates for the proof of Proposition~\ref{prop:para-diag-tdb}.
\begin{lemma} \label{lem:paralin-rem}
The following statements hold.
\begin{enumerate}
\item {\bf Commutator term, $[P_{k}, T_{\bfB^{\alp}} \rd_{\alp}]$.}
\begin{align} 
	\nrm{\sum_{k} \abs{\nb} [P_{k},  {P_{< k-10} \bfB} \cdot \nb] b}_{Y^{0}} 
	&\aleq  {\bb( \sum_{j} 2^{\frac{7}{2} j} \nrm{P_{j} \bfB}_{\ell^{1}_{\calI} L^{\infty} L^{2}_{j}} \bb)} \nrm{ {b}}_{X^{0}}. \label{eq:rem-diag}
\end{align}
\item {\bf Commutator term, $\comm^{\sgm (1)}_{\bfB}$.} We have  {$\sum_{k} \comm^{\sgm (1)}_{P_{< k-10} \bfB} P_{k} = \dot{\comm}^{\sgm (1)  \sharp}_{\bfB} + \rem^{(0)}[\comm^{\sgm}_{\bfB}]$}, where
\begin{align} 
	\nrm{\dot{\comm}^{\sgm (1)  \sharp}_{ {\bfB}} b}_{Y^{0}} 
	&\aleq  {\bb( \sum_{j} 2^{\frac{7}{2} j} \nrm{P_{j} \bfB}_{\ell^{1}_{\calI} L^{\infty} L^{2}_{j}} \bb)} \nrm{ {b}}_{X^{0}},	\label{eq:rem-comm} \\
	\nrm{\rem^{(0)}[\comm^{\sgm}_{ {\bfB}}] b}_{Y^{0}} 
	& \aleq T  {\nrm{\nb \bfB}_{W^{1, \infty}}} \nrm{b}_{X^{0}}. \label{eq:rem-comm-rem}
\end{align}

\item {\bf Off-diagonal term, $\symm^{(1)}_{\bfB}$.} We have  {$\sum_{k} \symm^{(1)}_{P_{< k-10} \bfB} P_{k} = \dot{\symm}^{(1) \sharp}_{\bfB} + \rem^{(0)}[\symm_{\bfB}]$}, where
\begin{align} 
	\nrm{\dot{\symm}^{ (1)  \sharp}_{ {\bfB}} b}_{Y^{0}} 
	&\aleq  {\bb( \sum_{j} 2^{\frac{7}{2} j} \nrm{P_{j} \bfB}_{\ell^{1}_{\calI} L^{\infty} L^{2}_{j}} \bb)} \nrm{ {b}}_{X^{0}},	\label{eq:rem-symm} \\
	\nrm{\rem^{(0)}[\symm^{\sgm}_{ {\bfB}}] b}_{Y^{0}} 
	& \aleq T  {\nrm{\nb \bfB}_{W^{1, \infty}}} \nrm{b}_{X^{0}}. \label{eq:rem-symm-rem}
\end{align}
Moreover, we have the commutator bound
\begin{align} 
	\nrm{[\brk{D}^{\sgm}, \sum_{k} \symm^{(1)}_{ {P_{< k-10} \bfB}} P_{k}] \brk{D}^{-\sgm} b}_{Y^{0}} \aleq T  {\nrm{\nb \bfB}_{W^{1, \infty}}} \nrm{b}_{X^{0}}. \label{eq:rem-symm-comm}
\end{align}
\item {\bf Diagonal antisymmetric terms, $\asymm^{(1)}_{\bfB}$ and $(\nb \times \bfB) \cdot \nb$.} We have
\begin{align} 
	\nrm{\sum_{k} \asymm^{(1)}_{ {P_{< k-10}\bfB}} P_{k} b}_{Y^{0}} 
	&\aleq  {\bb( \sum_{j} 2^{\frac{7}{2} j} \nrm{P_{j} \bfB}_{\ell^{1}_{\calI} L^{\infty} L^{2}_{j}} \bb)} \nrm{ {b}}_{X^{0}},	\label{eq:rem-asymm} \\
	\nrm{\sum_{k} (\nb \times  {P_{< k-10} \bfB}) \cdot \nb P_{k} b}_{Y^{0}} 
	&\aleq  {\bb( \sum_{j} 2^{\frac{7}{2} j} \nrm{P_{j} \bfB}_{\ell^{1}_{\calI} L^{\infty} L^{2}_{j}} \bb)} \nrm{ {b}}_{X^{0}},	\label{eq:rem-curlB} \\
	\nrm{[\brk{D}^{\sgm}, \sum_{k} \asymm^{(1)}_{ {P_{< k-10} \bfB}} P_{k}] \brk{D}^{-\sgm} b}_{Y^{0}} &\aleq T  {\nrm{\nb \bfB}_{W^{1, \infty}}} \nrm{b}_{X^{0}}, \label{eq:rem-asymm-comm} \\
	\nrm{[\brk{D}^{\sgm}, \sum_{k} (\nb \times  {P_{< k-10}\bfB}) \cdot \nb P_{k}] \brk{D}^{-\sgm} b}_{Y^{0}} & \aleq T  {\nrm{\nb \bfB}_{W^{1, \infty}}} \nrm{b}_{X^{0}}. \label{eq:rem-curlB-comm}
\end{align}
\item {\bf Divergence term, $\covec^{(0)}_{\bfB}$.} We have
\begin{align} 
	\nrm{\sum_{k} \nb \covec^{(0)}_{ {P_{< k-10} \bfB}} P_{k} b}_{Y^{0}} 
	&\aleq  {\bb( \sum_{j} 2^{\frac{7}{2} j} \nrm{P_{j} \bfB}_{\ell^{1}_{\calI} L^{\infty} L^{2}_{j}} \bb)} \nrm{P_{k} b}_{X^{0}},	\label{eq:rem-covec} \\
	\nrm{[\brk{D}^{\sgm}, \sum_{k} \nb \covec^{(0)}_{ {P_{< k-10}\bfB}} P_{k}] \brk{D}^{-\sgm} b}_{Y^{0}} & \aleq T  {\nrm{\nb \bfB}_{W^{1, \infty}}} \nrm{b}_{X^{0}}. \label{eq:rem-covec-comm}
\end{align}

\item {\bf Zeroth order remainder term, $\rem^{(0)}_{\nb \times \bfB; \pm}$.} We have
\begin{align} 
\nrm{\brk{D}^{\sgm} \sum_{k} (\rem^{(0)}_{\nb \times \bfB_{<k-10}; \pm}) P_{k} \brk{D}^{-\sgm} b}_{Y^{0}} \aleq T  {\nrm{\nb \bfB}_{W^{1, \infty}}} \nrm{b}_{X^{0}}. \label{eq:rem-0}
\end{align}
\end{enumerate}
\end{lemma}

\begin{proof}
All estimates in this lemma are proved by combining symbolic calculus in Proposition~\ref{prop:psdo-comp} with the estimates in Section~\ref{sec:multi-ests}. 

\pfstep{Step~1: Proof of \eqref{eq:rem-diag}}
This estimate is an immediate consequence of \eqref{eq:prod-core-comm}.

\pfstep{Step~2: Proof of \eqref{eq:rem-comm}--\eqref{eq:rem-comm-rem}}
By the definitions of $\comm^{\sgm (1)}_{\bfB}$ and $\diag^{(2)}_{\bfB}$, and the fact that $P_{k}$ commutes with $\brk{D}^{\sgm}$ and $\brk{D}^{-\sgm}$, we have
\begin{align*}
\sum_{k} \comm^{\sgm (1)}_{\bfB_{< k-10}} P_{k}
&= \bb[ \brk{D}^{\sgm}, \sum_{k} \diag^{(2)}_{P_{<k-10} \bfB} \bb] \brk{D}^{-\sgm} P_{k} \\
&=  {\frac{1}{2} \bb[ \brk{D}^{\sgm}, \sum_{k} P_{<k-10} \bfB \cdot \nb \abs{\nb} P_{k}\bb] \brk{D}^{-\sgm} } 
+  {\frac{1}{2} \bb[ \brk{D}^{\sgm}, \abs{\nb} \sum_{k} P_{<k-10} \bfB \cdot \nb P_{k} \bb] \brk{D}^{-\sgm}}.
\end{align*}
By Proposition~\ref{prop:psdo-comp}.(2)--(3) and Corollary~\ref{cor:psdo-comp}, it follows that
\begin{align*}
\frac{1}{2} \bb[ \brk{D}^{\sgm}, \abs{\nb} \sum_{k} P_{<k-10} \bfB \cdot \nb   {P_{k}}\bb]   \brk{D}^{-\sgm}  
+ \frac{1}{2} \Op\left( \sum_{k} P_{<k-10} \rd_{\alp} \bfB^{\bt} i \xi_{\bt} \abs{\xi}  {P_{k}(\xi)} \right) \Op\left(i ^{-1} \rd_{\xi_{\alp}} \brk{\xi}^{\sgm} \right)  \brk{D}^{-\sgm}  
\end{align*}
is bounded in $L^{2}$ with the operator norm $O( {\nrm{\nb \bfB}_{W^{1, \infty}}})$. Proceeding similarly for the other commutator, and using \eqref{eq:psdo-op-product}, we arrive at the following: defining
\begin{equation*}
	\dot{\comm}^{\sgm (1) \sharp}_{\bfB} := \sum_{k} \left( P_{<k-10} (\rd_{\alp} \bfB^{\bt}) \right) \Op\left( {\xi_{\bt}} \abs{\xi} \frac{\rd_{\xi_{\alp}} \brk{\xi}^{\sgm}}{\brk{\xi}^{\sgm}}  \right) P_{k},
\end{equation*}
the remainder $\rem^{(0)}[\comm^{\sgm}_{\bfB}] := \sum_{k} \comm^{\sgm (1)}_{P_{< k-10} \bfB} P_{k} - \dot{\comm}^{\sgm (1)  \sharp}_{\bfB}$ obeys \eqref{eq:rem-comm-rem} by $L^{1} L^{2} \hookrightarrow Y^{0}$, H\"older's inequality (in $t$) and $X^{0} \hookrightarrow L^{\infty} L^{2}$. On the other hand, \eqref{eq:rem-comm} follows from \eqref{eq:prod-core-hl}. 

\pfstep{Step~3: Proof of \eqref{eq:rem-0}}
We first recall (from the definition of $\rem^{(0)}_{\nb \times \bfB; \pm}$ in Proposition~\ref{prop:lin-diag}) that
 \begin{align*}
\sum_{k} \tensor{(\rem^{(0)}_{\nb \times \bfB_{<k-10}; \pm})}{^{\alp}_{\bt}} P_{k}  = \sum_{k} [ \tensor{\Pi_{\pm}(D)}{^{\alp}_{\bt}}, (\nb \times \bfB_{<k-10}) \cdot \nb]  P_{k} - \sum_{k} \tensor{\Pi_{\pm}(D)}{^{\alp}_{\gmm}}  (\rd_{\bt} (\nb \times \bfB_{<k-10})^{\gmm}) P_{k}.
\end{align*}
The desired estimate is clearly true for the second term on the  {RHS}, so it only remains to treat the first term. As in Step~2, consideration of Fourier support properties leads to
\begin{align*}
\sum_{k} [ \tensor{\Pi_{\pm}(D)}{^{\alp}_{\bt}}, (\nb \times \bfB_{<k-10}) \cdot \nb]  P_{k} b
&=   \bb[ \tensor{\Pi_{\pm}(D)}{^{\alp}_{\bt}}, \sum_{k} (\nb \times \bfB_{<k-10}) P_{[k-5, k+5]}\cdot \nb \bb] P_{k} b.
\end{align*}
We may apply Proposition~\ref{prop:psdo-comp}.(2) and (3) to the commutator and conclude that it is bounded on $H^{\sgm}$ for any $\sgm \in \bbR$ with operator norm $O(C_{\sgm}  {\nrm{\nb \bfB}_{W^{1, \infty}}})$. Using $L^{1} H^{\sgm} \hookrightarrow Y^{\sgm}$, H\"older's inequality, $X^{\sgm} \hookrightarrow L^{\infty} H^{\sgm}$, $X^{\sgm} = \brk{D}^{-\sgm} X^{0}$ and $Y^{\sgm} = \brk{D}^{-\sgm} Y^{0}$, we obtain \eqref{eq:rem-0}.

\pfstep{Step~4: Proof of the remaining estimates}
The remaining estimates are proved using the techniques we already used in Steps~1, 2 and 3. Since there are essentially no new ideas, we will simply give a summary of these proofs.

Bounds \eqref{eq:rem-asymm}, \eqref{eq:rem-curlB} and \eqref{eq:rem-covec} -- which do not involve any commutators -- follow rather immediately from \eqref{eq:prod-core-hl}. To prove \eqref{eq:rem-symm}--\eqref{eq:rem-symm-rem}, we begin by recalling (from the definition of $\symm^{(1)}_{\bfB}$ in Proposition~\ref{prop:lin-diag}) that
\begin{align*}
\sum_{k} \symm^{(1)}_{P_{< k-10} \bfB} P_{k} &= 
\frac{1}{2} \sum_{k} \left(\abs{\nb} (\rd^{\alp} (P_{<k-10} \bfB)_{\bt}) + (\rd^{\alp} (P_{<k-10} \bfB)_{\bt}) \abs{\nb} \right) P_{k} 
+ \frac{1}{2} \sum_{k} [ \abs{\nb}, P_{< k-10} \bfB ] P_{k}.
\end{align*}
The contribution of $\frac{1}{2} [\abs{\nb}, P_{< k-10} \bfB] P_{k}$ is treated as in Step~2 using Proposition~\ref{prop:psdo-comp}, Proposition~\ref{prop:psdo-ests} and \eqref{eq:prod-core-hl} (in particular, this term gives rise to $\frkR^{(0)}[\symm_{\bfB}]$), whereas the contribution of the first line is handled directly using \eqref{eq:prod-core-hl} (like, say, \eqref{eq:rem-asymm}). Finally, the commutator bound \eqref{eq:rem-symm-comm}, \eqref{eq:rem-asymm-comm} and \eqref{eq:rem-covec-comm} are proved by an argument similar to Step~3. \qedhere
\end{proof}

\begin{proof}[Proof of Proposition~\ref{prop:para-diag-tdb}]
We begin by using Proposition~\ref{prop:lin-diag} to rewrite \eqref{eq:paralin} as:
\begin{equation} \label{eq:paralin-diag}
\begin{alignedat}{2}
	 & \rd_{t} \begin{pmatrix} b_{+} \\ b_{-} \end{pmatrix} 
	 + \begin{pmatrix} \sum_{k} \diag_{\bfB_{<k-10}}^{(2)} P_{k} & 0 \\ 0 & - \sum_{k} \diag_{\bfB_{<k-10}}^{(2)} P_{k} \end{pmatrix} \begin{pmatrix} b_{+} \\ b_{-} \end{pmatrix} & \\
	& + \begin{pmatrix} 0 & \sum_{k} \symm^{(1)}_{\bfB_{<k-10}} P_{k}  \\ - \sum_{k}  \symm^{(1)}_{\bfB_{<k-10}} P_{k}  & 0 \end{pmatrix} \begin{pmatrix} b_{+} \\ b_{-} \end{pmatrix}  \\
	& +  \begin{pmatrix} \sum_{k} \asymm^{(1)}_{\bfB_{<k-10}} P_{k} & 0 \\ 0 & - \sum_{k} \asymm^{(1)}_{\bfB_{<k-10}} P_{k} \end{pmatrix} \begin{pmatrix}  b_{+} \\ b_{-} \end{pmatrix} 
	 + \sum_{k} (\nb \times \bfB_{<k-10}) P_{k} \cdot \nb \begin{pmatrix}  b_{+} \\ b_{-} \end{pmatrix}  & \\
	& + \nb \begin{pmatrix}  \sum_{k} \covec^{(0)}_{\bfB_{<k-10}} P_{k} & \sum_{k} \covec^{(0)}_{\bfB_{<k-10}} P_{k} \\ - \sum_{k} \covec^{(0)}_{\bfB_{<k-10}} P_{k} & - \sum_{k} \covec^{(0)}_{\bfB_{<k-10}} P_{k} \end{pmatrix}  \begin{pmatrix}  b_{+} \\ b_{-}  \end{pmatrix} 
	- E_{0(i)} &
	 = \begin{pmatrix}  g_{+} \\ g_{-}  \end{pmatrix},
\end{alignedat}
\end{equation}
with $b_{\pm} = \Pi_{\pm}(D) b$ and
\begin{align*}
E_{0(i)} := - \begin{pmatrix}  \sum_{k} \rem^{(0)}_{\nb \times \bfB_{<k-10}; +} P_{k} & \sum_{k} \rem^{(0)}_{\nb \times \bfB_{<k-10}; +} P_{k} \\ \sum_{k} \rem^{(0)}_{\nb \times \bfB_{<k-10}; -} P_{k} & \sum_{k} \rem^{(0)}_{\nb \times \bfB_{<k-10}; -} P_{k}  \end{pmatrix} \begin{pmatrix}  b_{+} \\ b_{-}  \end{pmatrix}.
\end{align*}

Next, introducing the conjugation of $b_{\pm}$ by the Bessel potential $\brk{D}^{\sgm}$,
\begin{equation*}
\tb_{\pm} = \brk{D}^{\sgm} b_{\pm}, \quad \tilde{g}_{\pm} = \brk{D}^{\sgm} g_{\pm},
\end{equation*}
we have
\begin{equation*}
\begin{alignedat}{2}
	 & \rd_{t} \begin{pmatrix} \tb_{+} \\ \tb_{-} \end{pmatrix} 
	 + \begin{pmatrix} \sum_{k} \diag_{\bfB_{<k-10}}^{(2)} P_{k} & 0 \\ 0 & - \sum_{k} \diag_{\bfB_{<k-10}}^{(2)} P_{k} \end{pmatrix} \begin{pmatrix} \tb_{+} \\ \tb_{-} \end{pmatrix} & \\
	& + \begin{pmatrix} \sum_{k} \comm^{\sgm (1)}_{\bfB_{<k-10}} P_{k} & 0 \\ 0 & - \sum_{k} \comm^{\sgm (1)}_{\bfB_{<k-10}} P_{k} \end{pmatrix} \begin{pmatrix} \tb_{+} \\ \tb_{-} \end{pmatrix} 
	 + \begin{pmatrix} 0 & \sum_{k} \symm^{(1)}_{\bfB_{<k-10}} P_{k}  \\ - \sum_{k}  \symm^{(1)}_{\bfB_{<k-10}} P_{k}  & 0 \end{pmatrix} \begin{pmatrix} \tb_{+} \\ \tb_{-} \end{pmatrix}  & \\
	& +  \begin{pmatrix} \sum_{k} \asymm^{(1)}_{\bfB_{<k-10}} P_{k} & 0 \\ 0 & - \sum_{k} \asymm^{(1)}_{\bfB_{<k-10}} P_{k} \end{pmatrix} \begin{pmatrix} \tb_{+} \\ \tb_{-} \end{pmatrix}  	 + \sum_{k} (\nb \times \bfB_{<k-10}) P_{k} \cdot \nb \begin{pmatrix}  \tb_{+} \\ \tb_{-} \end{pmatrix}  & \\
	& + \nb \begin{pmatrix}  \sum_{k} \covec^{(0)}_{\bfB_{<k-10}} P_{k} & \sum_{k} \covec^{(0)}_{\bfB_{<k-10}} P_{k} \\ - \sum_{k} \covec^{(0)}_{\bfB_{<k-10}} P_{k} & - \sum_{k} \covec^{(0)}_{\bfB_{<k-10}} P_{k} \end{pmatrix}  \begin{pmatrix}  \tb_{+} \\ \tb_{-}  \end{pmatrix} - \brk{D}^{\sgm} E_{0 (i)} - \tilde{E}_{0 (ii)}
	&= \begin{pmatrix}  \tilde{g}_{+} \\ \tilde{g}_{-}  \end{pmatrix}  
\end{alignedat}
\end{equation*}
where $\comm^{\sgm (1)}_{\bfB_{<k-10}} = [\brk{D}^{\sgm}, \diag_{\bfB_{<k-10}}^{(2)}] \brk{D}^{-\sgm}$ is the commutator of $\brk{D}^{\sgm}$ and $\diag_{\bfB_{<k-10}}^{(2)}$ (conjugated by $\brk{D}^{-\sgm}$), while $\tilde{E}_{0 (ii)}$ represents all commutator terms except those involving $\diag_{\bfB_{<k-10}}^{(2)}$ (clearly, both are zero if $\sgm = 0$). By \eqref{eq:rem-0}, we have
\begin{equation} \label{eq:paralin-err-1}
	\nrm{\brk{D}^{\sgm} E_{0 (i)}}_{Y^{0}} 
	\aleq T M \sum_{\pm} \nrm{\tilde{b}_{\pm}}_{X^{0}}.
\end{equation}
Moreover, by  {\eqref{eq:rem-symm-comm}, \eqref{eq:rem-asymm-comm}, \eqref{eq:rem-covec-comm}, and \eqref{eq:rem-curlB-comm}}, we have
\begin{equation} \label{eq:paralin-err-2}
	\nrm{\tilde{E}_{0 (ii)}}_{Y^{0}} 
	 {\aleq T M \sum_{\pm} \nrm{\tilde{b}_{\pm}}_{X^{0}}.}
\end{equation}
These two estimates are parts of the proof of the error bound \eqref{eq:paralin-err}, which is completed below.

To proceed, we replace $\sum_{k} \diag_{\bfB_{<k-10}}^{(2)} P_{k}$ by its anti-symmetrization $\diag_{\bfB}^{(2) \sharp}$. Moreover, for all first order terms (i.e., lines 2--4), we split $\bfB_{<k-10} = (\bfB_{0})_{<k_{(0)}} + (\bfB_{<k-10} - (\bfB_{0})_{<k_{(0)}})$. Then we arrive at
\begin{equation*}
\begin{alignedat}{2}
	 & \rd_{t} \begin{pmatrix} \tb_{+} \\ \tb_{-} \end{pmatrix} 
	 + \begin{pmatrix} \diag_{\bfB}^{(2) \sharp} & 0 \\ 0 & - \diag_{\bfB}^{(2) \sharp} \end{pmatrix} \begin{pmatrix} \tb_{+} \\ \tb_{-} \end{pmatrix} & \\
	& + \begin{pmatrix} \sum_{k} \comm^{\sgm (1)}_{(\bfB_0)_{<k_{(0)}}} P_{k} & 0 \\ 0 & - \sum_{k} \comm^{\sgm (1)}_{(\bfB_0)_{<k_{(0)}}} P_{k} \end{pmatrix} \begin{pmatrix} \tb_{+} \\ \tb_{-} \end{pmatrix} 
	 + \begin{pmatrix} 0 & \sum_{k} \symm^{(1)}_{(\bfB_0)_{<k_{(0)}}} P_{k}  \\ - \sum_{k}  \symm^{(1)}_{(\bfB_0)_{<k_{(0)}}} P_{k}  & 0 \end{pmatrix} \begin{pmatrix} \tb_{+} \\ \tb_{-} \end{pmatrix}  & \\
	& +  \begin{pmatrix} \sum_{k} \asymm^{(1)}_{(\bfB_0)_{<k_{(0)}}} P_{k} & 0 \\ 0 & - \sum_{k} \asymm^{(1)}_{(\bfB_0)_{<k_{(0)}}} P_{k} \end{pmatrix} \begin{pmatrix} \tb_{+} \\ \tb_{-} \end{pmatrix} 
	+ \sum_{k} (\nb \times (\bfB_0)_{<k_{(0)}}) P_{k} \cdot \nb \begin{pmatrix}  \tb_{+} \\ \tb_{-} \end{pmatrix}  & \\
	& + \nb \begin{pmatrix}  \sum_{k} \covec^{(0)}_{(\bfB_0)_{<k_{(0)}}} P_{k} & \sum_{k} \covec^{(0)}_{(\bfB_0)_{<k_{(0)}}} P_{k} \\ - \sum_{k} \covec^{(0)}_{(\bfB_0)_{<k_{(0)}}} P_{k} & - \sum_{k} \covec^{(0)}_{(\bfB_0)_{<k_{(0)}}} P_{k} \end{pmatrix}  \begin{pmatrix}  \tb_{+} \\ \tb_{-}  \end{pmatrix} - \brk{D}^{\sgm} E_{0 (i)} - \tilde{E}_{0 (ii)} - \tilde{E}_{0 (iii)} - \tilde{E}_{0 (iv)}
	& = \begin{pmatrix}  \tilde{g}_{+} \\ \tilde{g}_{-}  \end{pmatrix},
\end{alignedat}
\end{equation*}
where $\tilde{E}_{0 (iii)}$ represents the contribution of $\bfB_{<k-10} - (\bfB_{0})_{<k_{(0)}}$ in the first order terms and $\tilde{E}_{0 (iv)}$ is the error generated by replacing $\sum_{k} \diag_{\bfB_{<k-10}}^{(2)} P_{k}$ by $\diag_{\bfB}^{(2) \sharp}$, i.e.,
\begin{equation*}
	\tilde{E}_{0 (iv)} := \begin{pmatrix} (\sum_{k} \diag_{\bfB_{<k-10}}^{(2)} P_{k} - \diag_{\bfB}^{(2) \sharp}) \tb_{+}  \\  - (\sum_{k} \diag_{\bfB_{<k-10}}^{(2)} P_{k} -  \diag_{\bfB}^{(2) \sharp}) \tb_{-}\end{pmatrix}.
\end{equation*}
Defining
\begin{equation*}
	\tilde{E}_{0} := \brk{D}^{\sgm} E_{0 (i)} + \tilde{E}_{0 (ii)} + \tilde{E}_{0 (iii)} + \tilde{E}_{0 (iv)},
\end{equation*}
and summing up in $k$ on lines 2--4, note that this equation is the same as \eqref{eq:para-diag-tb}. 

We are left to prove \eqref{eq:paralin-err}, and in view of \eqref{eq:paralin-err-1} and \eqref{eq:paralin-err-2}, it remains to only treat $\tilde{E}_{0 (iii)}$ and $\tilde{E}_{0 (iv)}$. To estimate $\tilde{E}_{0 (iii)}$, we further split
\begin{equation*}
	\bfB_{<k-10} - (\bfB_{0})_{<k_{(0)}} = (\bfB_{<k-10} - \bfB_{<k_{(0)}}) + (\bfB_{<k_{(0)}} - (\bfB_{0})_{<k_{(0)}}).
\end{equation*}
For the latter, observe that 
\begin{align*}
\nrm{\bfB_{<k_{(0)}} - (\bfB_{0})_{<k_{(0)}}}_{\ell^{1}_{\calI}  L^{\infty} H^{s_{0}}  }
& \aleq T 2^{2 k_{(0)}} \nrm{\rd_{t} \bfB_{<k_{(0)}}}_{\ell^{1}_{\calI}   L^{\infty} H^{s_{0}-2}  }.
\end{align*}
Hence, an application of  {\eqref{eq:rem-comm}, \eqref{eq:rem-comm-rem}, \eqref{eq:rem-symm} \eqref{eq:rem-symm-rem}, \eqref{eq:rem-asymm}, \eqref{eq:rem-curlB}, and \eqref{eq:rem-covec}} gives 
\begin{align*}
	\nrm{\tilde{E}_{0 (iii)}}_{Y^{0}} & \aleq \sum_{\pm} \bb[ \sum_{k} \bb[ \bb( \sum_{j \in [k_{(0)}, k-10)} 2^{\frac{7}{2} j} \nrm{P_{j} \bfB}_{\ell^{1}_{\calI} L^{\infty} L^{2}_{j}} + \sum_{j < k_{(0)}} 2^{\frac{7}{2} j} \nrm{P_{j} (\bfB - \bfB_{0})}_{\ell^{1}_{\calI} L^{\infty} L^{2}_{j}} \bb) \nrm{P_{k} \tb_{\pm}}_{X_{k}} \bb]^{2} \bb]^{\frac{1}{2}} \\
	& \aleq \left( 2^{-(s_{0} - \frac{7}{2}) k_{(0)}} \nrm{\bfB}_{\ell^{1}_{\calI} X^{s_{0}}} 
	+ T 2^{2 k_{(0)}} \nrm{\rd_{t} \bfB_{<k_{(0)} } }_{\ell^{1}_{\calI}  L^{\infty} H^{\frac{7}{2} - 2 }  } \right) \sum_{\pm} \nrm{\tb_{\pm}}_{X^{0}} \\
	& \aleq (2^{-(s_{0} - \frac{7}{2}) k_{(0)}} + T 2^{2 k_{(0)}}) M \sum_{\pm} \nrm{\tb_{\pm}}_{X^{0}}, 
\end{align*}
 {which is sufficient for proving \eqref{eq:para-diag-tb-err}.}

Finally, to estimate $\tilde{E}_{0 (iv)}$, we write
\begin{align*}
\pm (\sum_{k} \diag_{\bfB_{<k-10}}^{(2)} P_{k} - \diag_{\bfB}^{(2) \sharp}) \tb_{\pm} 
&= \pm \frac{1}{2} \sum_{k} \abs{\nb} [P_{k}, \bfB_{<k-10} \cdot \nb]  \tb_{\pm} 
= \pm \frac{1}{2} \sum_{k} \abs{\nb} [P_{k}, P_{[k_{(0)}, k-10)} \bfB \cdot \nb]  \tb_{\pm},
\end{align*}
where we used the fact that $\sum_{k} [P_{k}, P_{<k_{(0)}} \bfB \cdot \nb] = [1, P_{<k_{(0)}} \bfB \cdot \nb] = 0$. Using \eqref{eq:rem-diag}, we obtain  
\begin{align*}
\nrm{(\sum_{k} \diag_{\bfB_{<k-10}}^{(2)} P_{k} - \diag_{\bfB}^{(2) \sharp}) \tb_{\pm} }_{Y^{0}}
\aleq 2^{- (s_{0} - \frac{7}{2}) k_{(0)}} M \nrm{\tb_{\pm}}_{X^{0}},
\end{align*}
which is again acceptable for the proof of \eqref{eq:para-diag-tb-err}. \qedhere
\qedhere
\end{proof}

\subsection{Local smoothing estimate assuming boundedness of energy} \label{subsec:paralin-led}
Our remaining task is to control the $X^{0}$ norm of the solution $\tb_{\pm}$ to \eqref{eq:para-diag-tb} in terms of $\nrm{\tilde{h}_{\pm}}_{Y^{0}}$ and the initial data. In this subsection, we reduce the problem to estimating the $L^{\infty} L^{2}$ norm of $\tb_{\pm}$.
\begin{proposition} \label{prop:paralin-led}
Let $\tb_{\pm}$ solve \eqref{eq:para-diag-tb}. Then we have for all sufficiently large $k_{(1)}$ satisfying $k_{(1)} \ge k_{(0)}$,
\begin{equation} \label{eq:paralin-led}
\begin{aligned}
	\nrm{\tb}_{X^{0}[0, T]}^{2} 
	& \leq C_{M, \mu, A} e^{C_{M, \mu, A} L_{0}} \nrm{\tb}_{L^{\infty} L^{2}[0, T]}^{2} 
	+ C_{M, \mu, A} e^{C(k_{(0)} + k_{(1)} ) } e^{C_{M, \mu, A} L_{0}} \left( \nrm{b}_{L^{\infty} L^{2}[0, T]}^{2} + \nrm{\tilde{h}}_{Y^{0}[0, T]}^{2} \right) \\
	& \peq + C_{M, \mu, A} e^{C_{M, \mu, A} L_{0}} \left( 2^{-(s_{0} - \frac{7}{2}) k_{(0)}} + e^{C(k_{(0)} + k_{(1)})} T \right) \nrm{\tb}_{X^{0}[0, T]}^{2}.
\end{aligned}
\end{equation}
\end{proposition}
Here, it is important that the constant in front of $\nrm{\tb}_{L^{\infty} L^{2}[0, T]}^{2}$ is \emph{independent} of $L$ (also $R$), $k_{(0)}$ and $k_{(1)}$. Note also that the factor in front of $\nrm{\tb}_{X^{0}}^{2}$ can be made arbitrarily small by first taking $k_{(0)}$ large, then taking $T$ small. Furthermore, while the constants depend also on $R_{0}$ (increasing functions of $R_{0}$), we omit writing out the dependence since we always have $L_{0} \ge \frac{1}{24} R_{0}$. 

\begin{proof}
The overall strategy is to multiply \eqref{eq:para-diag-tb} (on the left) by a classical pseudodifferential operator $(\frkF \tb_{+}, - \frkF \tb_{-})$ and integrate over $\bbR^{3}$, where $\frkF(x, D)$ is constructed so that the commutator $[\diag_{\bfB}^{(2) \sharp}, \frkF(x, D)]$ has a good positivity property. 

\pfstep{Step~1: Construction of positive commutator}
In what follows, we will use the shorthand $\bgB := (\bfB_{0})_{<k_{(0)}}$. Let $\eps_{0}$ be as in Lemma~\ref{lem:nontrapping-id}, and let $R_{0}, L_{0} > 0$ be so that $\bfB_{0}$ lies in $\calB^{s_{0}}_{\frac{1}{2} \eps_{0}}(M, \mu, A, R_{0}, L_{0})$. We will assume that $k_{(0)}$ is large enough so that $\bgB \in \calB^{s_{0}}_{\eps_{0}}(2 M, \frac{1}{2} \mu, 2 A, R_{0}, 2 L_{0})$. 

We begin by introducing a (positive) envelope
\begin{equation*}
	F(z) := \int_{-\infty}^{\infty}  \brk{2^{k_{(0)}}(z-z')}^{-100} 2^{k_{(0)}} \sup_{\set{x^{3} = z'}}\abs{\nb \bgB}\, \ud z'.
\end{equation*}
Since $\bgB = (\bfB_{0})_{<k_{(0)}} = P_{<k_{(0)}+3} (\bfB_{0})_{<k_{(0)}}$, we have
\begin{align}
\abs{\nb \bgB(x^{1}, x^{2}, x^{3})} &= \abs{P_{<k_{(0)}+3} \nb \bgB(x^{1}, x^{2}, x^{3})} 
\aleq \int 2^{3k_{(0)}} \brk{2^{k_{(0)}}\abs{x - x'}}^{-200} \abs{\nb \bgB(x')} \, \ud x' 
\aleq F(x^{3}). \label{eq:env-dominate}
\end{align}
On the other hand, by Young's inequality, we have
\begin{align} \label{eq:env-minimal}
	\int F(z) \, \ud z \aleq \nrm{\nb \bgB}_{L^{1}_{x^{3}} L^{\infty}_{x^{1}, x^{2}}} \aleq M,
\end{align}
while for higher derivatives, we have
\begin{align} \label{eq:env-slowvar}
	\abs{\rd_{z}^{N} F(z)} \aleq_{N} 2^{N k_{(0)}} M.
\end{align}
Using $F$ and two positive constants $C_{f}$ and $C_{\med}$ to be determined below, we define
\begin{align*}
	f_{\out}(x^{3}) := 12 C_{f} \int_{-\infty}^{x^{3}} (1 - \chi_{< R_{0}}(z))F(z) \, \ud z
	+ 12 C_{\med} \int_{-\infty}^{x^{3}} R_{0}^{-1} (\chi_{< 8 R_{0}} - \chi_{< R_{0}})(z) \, \ud z.
\end{align*}
Recall that, by \eqref{eq:speed3-lower}, we have $\abs{\rd_{\xi_{3}} p_{\bgB}(x, \xi)} \geq \frac{1}{12} \abs{\xi}$ if $\abs{\bgB - \bfe_{3}} < \frac{1}{2}$, which holds on $\supp (1-\chi_{< R_{0}}(x^{3}))$. Hence, $f_{\out}$ satisfies the positive commutator (more precisely, Poisson bracket) property
\begin{align*}
	\{ \dprin_{\bgB}, f_{\out}  \} & \ge  \abs{\xi} \left[ C_{f} (1 - \chi_{< R_{0}})(x^{3}) F(x^{3}) + C_{\med} R_{0}^{-1} (\chi_{< 8 R_{0}} - \chi_{< R_{0}})(x^{3}) \right] > 0,
\end{align*}
while the following symbol bounds also hold:
\begin{align*}
	\abs{f_{\out}} &\aleq C_{f} M + C_{\med}, \\
	\abs{\rd_{x^{3}}^{N} f_{\out}} &\aleq_{N} C_{f} 2^{(N-1)_{+} k_{(0)}} M + C_{\med} R_{0}^{-N}.
\end{align*}
Next, we consider a Doi-type multiplier of the form
\begin{align*}
	\tilde{f}_{\intr}(x, \xi) := \int_{-\infty}^{0} \chi_{< 2 R_{0}}(X^{3}(t; x, \xi)) \abs{\Xi(t; x, \xi)}\, \ud t.
\end{align*}
where $(X, \Xi)(t; x, \xi)$ is the bicharacteristic associated with $\bgB$. Note that
\begin{equation*}
	\{\dprin_{\bgB}, \tilde{f}_{\intr}(x, \xi)\} =\chi_{< 2 R_{0}}(x^{3}) \abs{\xi},
\end{equation*}
while by \eqref{eq:nontrapping-length-bnd} and Proposition~\ref{prop:d-bichar}, the following symbol bounds hold:
\begin{align*}
	\abs{\tilde{f}_{\intr}(x, \xi)} &\aleq \mu^{-1} L_{0}, \\
	\abs{\rd_{x} \tilde{f}_{\intr}(x, \xi)} + \abs{\abs{\xi} \rd_{\xi}\tilde{f}_{\intr}} &\aleq_{M, \mu, A} e^{C_{M, \mu, A} L_{0}}, \\
	\abs{\xi}^{\abs{\bfbt}} \abs{\rd_{x}^{\bfalp} \rd_{\xi}^{\bfbt} \tilde{f}_{\intr}}  &\aleq_{M, \mu, A, \bfalp, \bfbt} 2^{(\abs{\bfalp} - 1)_{+} k_{(0)}} e^{C_{M, \mu, A, \bfalp, \bfbt} L_{0}}.  \label{eq:f-intr-symb-high}
\end{align*}
Using $f_{\out}$ and $\tilde{f}_{\intr}$, we define
\begin{equation} \label{eq:f-symb-def}
	f := \chi_{>1}(\xi) \left( f_{\out}(x^{3}) + C_{f} M \chi_{< 4 R_{0}}(x^{3}) \tilde{f}_{\intr}(x, \xi)\right).
\end{equation}
Then
\begin{equation} \label{eq:f-symb-comm}
\begin{aligned}
	\{\dprin_{\bgB}, f\}
	&\geq \abs{\xi} \chi_{>1}(\xi) C_{f} \left( (1- \chi_{< R_{0}}(x^{3})) F(x^{3}) + M \chi_{< 2 R_{0}}(x^{3}) \right) \\
	&\pgeq + \abs{\xi} \chi_{>1}(\xi) 
	\left( C_{\med} R_{0}^{-1} (\chi_{< 8 R_{0}} - \chi_{< R_{0}})(x^{3}) - C_{f} M (4 R_{0})^{-1} \chi_{<1}'(\tfrac{x^{3}}{4 R_{0}}) \sup_{x, \xi} \abs{\tilde{f}_{\intr}} \right) \\
	&\pgeq - \rem[\{p_{\bgB}, f\}],
\end{aligned}
\end{equation}
where the remainder $\rem[\{p_{\bgB}, f\}]$ is given by
\begin{align*}
\rem[\{p_{\bgB}, f\}] := \rd_{x^{\alp}} \dprin_{\bgB} \frac{\xi_{\alp}}{\abs{\xi}} \chi'_{>1}(\xi) \left( f_{\out}(x^{3}) + C_{f} M \chi_{< 4 R_{0}}(x^{3}) \tilde{f}_{\intr}(x, \xi)\right).
\end{align*}
Given $C_{f}$, to be fixed in Step~2 below, we choose $C_{\med} = C C_{f} M \mu^{-1} L_{0}$, where $C$ is a sufficiently large absolute constant that makes the second line in \eqref{eq:f-symb-comm} nonnegative. 
Observe that we have the following symbol bounds for $f$:
\begin{align}
	\abs{f} &\aleq C_{f} M (1+ \mu^{-1} L_{0}), \label{eq:f-symb-0} \\
	\abs{\rd_{x} f} + \abs{\brk{\xi} \rd_{\xi} f} &\aleq_{M, \mu, A} C_{f} e^{C_{M, \mu, A} L_{0}},  \label{eq:f-symb-1} \\
	\brk{\xi}^{\abs{\bfbt}} \abs{\rd_{x}^{\bfalp} \rd_{\xi}^{\bfbt} f}  &\aleq_{M, \mu, A, \bfalp, \bfbt} C_{f} 2^{(\abs{\bfalp} - 1)_{+} k_{(0)}} e^{C_{M, \mu, A, \bfalp, \bfbt} L_{0}},  \label{eq:f-symb-high}
\end{align}
where we used $R_{0} \aleq L_{0}$. Note also that the remainder term obeys the symbol bound
\begin{align*}
[\rem[\{p_{\bgB}, f\}]]_{S^{0}; N} \aleq_{M, \mu, A, N} e^{C k_{(0)}} e^{C_{M, \mu, A} L_{0}}.
\end{align*}
(In fact, it is infinitely smoothing, but we need not use this stronger property.)

The positive commutator property of the symbol $f$ will be used to bound the contribution of the first order terms in \eqref{eq:para-diag-tb}. To control the $LE$ norm, we also need the following additional construction. For $I \in \calI_{\ell}$, $I \cap (-10 R_{0}, 10 R_{0}) = \0$, define
\begin{equation} \label{eq:f-I-symb-def}
f_{I} = \int_{-\infty}^{x^{3}} 2^{-\ell} \chi_{I}^{2}(z) \, \ud z.
\end{equation}
Because of $I \cap (-10 R_{0}, 10 R_{0}) = \0$, we have
\begin{equation} \label{eq:f-I-symb-comm}
\left\{ f_{I}, \dprin_{\bgB} \right\} \ageq 2^{-\ell} \chi_{I}^{2}(x^{3}) \abs{\xi}.
\end{equation}
Moreover, since $\ell \geq 0$, note that $f_{I}$ obeys the bounds
\begin{equation} \label{eq:f-I-symb}
\abs{\rd_{x}^{\bfalp} f_{I}} \aleq_{\bfalp} 2^{-\abs{\bfalp} \ell} \aleq_{\bfalp} 1.
\end{equation}

\pfstep{Step~2: Positive commutator identity}
Form a symmetric operator 
\begin{equation*}
	\frkF_{I} := e^{f}(D, x) e^{f_{I}}(x^{3}) e^{f_{I}}(x^{3}) e^{f}(x, D).
\end{equation*}
We multiply \eqref{eq:para-diag-tb} by $(\frkF_{I} \tb_{+}, - \frkF_{I} \tb_{-})$ and integrate over $\bbR^{3}$. From the principal part $\pm \diag_{\bfB}^{(2) \sharp} \tb_{\pm}$, we obtain
\begin{align*}
	\brk*{\rd_{t} \tilde{b}_{\pm} \pm \diag_{\bfB}^{(2) \sharp} \tilde{b}_{\pm}, \pm \frkF_{I} \tilde{b}_{\pm}}
	&= \pm \frac{1}{2} \rd_{t} \brk{\tilde{b}_{\pm}, \frkF_{I} \tilde{b}_{\pm}} + \frac{1}{2}\brk*{\tilde{b}_{\pm}, [\frkF_{I}, \diag_{\bfB}^{(2) \sharp}] \tilde{b}_{\pm}} \\
	&= \pm \frac{1}{2} \rd_{t} \brk{e^{f_{I}}(x) e^{f}(x, D) \tilde{b}_{\pm}, e^{f_{I}}(x) e^{f}(x, D) \tilde{b}_{\pm}} 
	+ \frac{1}{2}\brk*{\tilde{b}_{\pm}, [\frkF_{I},  \diag_{\bgB}^{(2)}] \tilde{b}_{\pm}} + \brk{\tilde{b}_{\pm}, \tilde{E}_{I; \pm}},
\end{align*}
where
\begin{align}\label{eq:tilde-E-I-def}
	\tilde{E}_{I; \pm} = \frac{1}{2} [\frkF_{I}, \diag_{\bfB}^{(2) \sharp} - \diag_{\bgB}^{(2)}]  \tilde{b}_{\pm}.
\end{align}
For the ensuing argument, it will be convenient to introduce the shorthands
\begin{equation*}
\tilde{B}_{I; \pm} := e^{f_{I}}(x^{3}) \tilde{B}_{\pm}, \quad 
\tilde{B}_{\pm} := e^{f}(x, D) \tilde{b}_{\pm}, \quad
\tilde{H}_{I; \pm} := e^{f_{I}}(x^{3}) \tilde{H}_{\pm}, \quad 
\tilde{H}_{\pm} := e^{f}(x, D) \tilde{h}_{\pm}.
\end{equation*}

The main term in the preceding multiplier identity is the commutator term $\frac{1}{2}\brk{\tilde{b}_{\pm}, [\frkF_{I},  \diag_{\bgB}^{(2)}] \tilde{b}_{\pm}}$, which we move to the  {LHS}. Observe that both $\frkF_{I}$ and $\diag^{(2)}_{\bgB}$ are pseudodifferential operators with classical symbols $\frkF_{I}(x, \xi)$ and $\diag^{(2)}_{\bgB}(x, \xi)$ with bounds 
\begin{equation*}
	[\frkF_{I}(x, \xi)]_{S^{0}; N} \aleq_{M, \mu, A, N} e^{C_{f}} e^{C k_{(0)}} e^{C_{M, \mu, A} L_{0}}, \quad 
	[\diag^{(2)}_{\bgB}(x, \xi)]_{S^{2}; N} \aleq_{M, N} e^{C k_{(0)}},
\end{equation*}
where we used Proposition~\ref{prop:psdo-comp}.(1), Proposition~\ref{prop:psdo-adj}.(1) and the symbol bounds \eqref{eq:f-symb-0}--\eqref{eq:f-symb-high} and \eqref{eq:f-I-symb} for $\frkF_{I}$, and $\bgB = (\bfB_{0})_{<k_{(0)}}$ for $\diag^{(2)}_{\bgB}$. By Proposition~\ref{prop:psdo-L2-bdd-garding}, Proposition~\ref{prop:psdo-comp}.(1), Proposition~\ref{prop:psdo-adj}.(1), and the positivity bounds \eqref{eq:f-symb-comm} and \eqref{eq:f-I-symb-comm}, we have
\begin{align*}
- \frac{1}{2}\brk*{\tilde{b}_{\pm}, [\frkF_{I},  \diag_{\bgB}^{(2)}] \tilde{b}_{\pm}} 
&= - \frac{1}{2}\brk*{\tilde{b}_{\pm}, \{e^{2(f+f_{I})}, \dprin_{\bgB} \}(x, D) \tilde{b}_{\pm}} + O(C_{M, \mu, A} e^{C k_{(0)}} e^{C_{M, \mu, A} L} \nrm{\tilde{b}}_{L^{2}}^{2}) \\
&= \brk*{e^{f_{I}} (x^{3}) e^{f} (x, D) \tilde{b}_{\pm}, \{\dprin_{\bgB}, f + f_{I} \} (x, D) e^{f_{I}}(x^{3}) e^{f} (x, D) \tilde{b}_{\pm}} \\
&\peq + O(e^{C_{f}} C_{M, \mu, A} e^{C k_{(0)}} e^{C_{M, \mu, A} L_{0}} \nrm{\tilde{b}}_{L^{2}}^{2}) \\
&\geq \brk*{\tilde{B}_{I; \pm}, \left( C_{f} (1- \chi_{< R_{0}}(x^{3})) F(x^{3}) + C_{f} M \chi_{< 2 R_{0}}(x^{3}) + 2^{-\ell} \chi_{I}^{2}(x^{3}) \right) \brk{D} \tilde{B}_{I; \pm}} \\
&\peq - C_{M, \mu, A} e^{C_{f}} e^{C k_{(0)}} e^{C_{M, \mu, A} L_{0}} \nrm{\tilde{b}}_{L^{2}}^{2},
\end{align*} 

Next, we use equation \eqref{eq:para-diag-tb} to rewrite $(\rd_{t} \pm \diag^{(2) \sharp}_{\bfB}) \tilde{b}_{\pm}$ on the  {LHS} in terms of lower order terms and $\tilde{h}_{\pm}$. By Proposition~\ref{prop:psdo-comp}.(1) and Proposition~\ref{prop:psdo-adj}.(1),  we see that $\comm^{\sgm(1)}_{\bgB}$, $\symm^{(1)}_{\bgB}$, $\asymm^{(1)}_{\bgB}$, $(\nb \times \bgB) \cdot \nb$ and $\covec^{(0)}_{\bgB}$ are pseudodifferential operators with classical symbols satisfying
\begin{align*}
	[\comm^{\sgm (1)}_{\bgB}(x, \xi)]_{S^{1}; N}
	+ [\symm^{(1)}_{\bgB}(x, \xi)]_{S^{1}; N}
	+ [\asymm^{(1)}_{\bgB}(x, \xi)]_{S^{1}; N}
	+ [(\nb \times \bgB)(x) \cdot (i \xi)]_{S^{1}; N}
	+ [\covec^{(0)}_{\bgB}(x, \xi)]_{{S^{0}}; N} \aleq_{M, N} e^{C k_{(0)}}.
\end{align*}
Hence, by Proposition~\ref{prop:psdo-L2-bdd-garding}, Proposition~\ref{prop:psdo-comp}.(1) and Proposition~\ref{prop:psdo-adj}.(1), we compute
\begin{align*}
\brk*{\pm \comm^{\sgm(1)}_{\bgB} \tb_{\pm}, \pm \frkF_{I} \tb_{\pm}}
&= \brk*{e^{f+f_{I}}(x, D) \comm^{\sgm(1)}_{\bgB} \tb_{\pm}, e^{f+f_{I}}(x, D) \tb_{\pm}} \\
&= \brk*{\tilde{B}_{I; \pm}, \comm^{\sgm(1)}_{\bgB} \tilde{B}_{I; \pm}} 
+ O(C_{M, \mu, A} e^{C_{f} }e^{C k_{(0)}} e^{C_{M, \mu, A} L_{0}} \nrm{\tb}_{L^{2}}^{2}).
\end{align*}
Similarly, we derive
\begin{align*}
\brk*{\pm \symm^{(1)}_{\bgB} \tb_{\mp}, \pm \frkF_{I} \tb_{\pm}}
&= \brk*{\tilde{B}_{I; \mp}, \symm^{(1)}_{\bgB} \tilde{B}_{I; \pm}} 
+ O(C_{M, \mu, A} e^{C_{f}} e^{C k_{(0)}} e^{C_{M, \mu, A} L_{0}} \nrm{\tb}_{L^{2}}^{2}).
\end{align*}
In view of \eqref{eq:env-dominate} and $\nrm{\nb \bgB}_{L^{\infty}} \aleq M$, choosing $C_{f}$ sufficiently large (as a universal constant) ensures that
\begin{align*}
	C_{f}(1-\chi_{<R_{0}}(x^{3})) F(x^{3}) + C_{f} M \chi_{< 2 R_{0}}(x^{3}) \geq \abs{\comm^{\sgm (1)}_{\bgB}(x, \xi)} + \abs{\symm^{(1)}_{\bgB}(x, \xi)},
\end{align*}
which implies, by Proposition~\ref{prop:psdo-L2-bdd-garding}, Cauchy--Schwarz and Proposition~\ref{prop:psdo-comp}, that
\begin{align*}
&\sum_{\pm} \brk*{\tilde{B}_{I; \pm}, \left( C_{f} (1- \chi_{< R_{0}}(x^{3})) F(x^{3}) + C_{f} M \chi_{< 2 R_{0}}(x^{3}) + 2^{-\ell} \chi_{I}^{2}(x^{3}) \right) \brk{D} \tilde{B}_{I; \pm}} \\
&+ \sum_{\pm} \brk*{\tilde{B}_{I; \pm}, \comm^{\sgm(1)}_{\bgB} \tilde{B}_{I; \pm}}
+ \sum_{\pm} \brk*{\tilde{B}_{I; \mp}, \symm^{(1)}_{\bgB} \tilde{B}_{I; \pm}} \\
&\geq \sum_{\pm} \brk*{\tilde{B}_{I; \pm}, 2^{-\ell} \brk{D}^{\frac{1}{2}} \chi_{I}^{2}(x^{3}) \brk{D}^{\frac{1}{2}} \tilde{B}_{I; \pm}} 
- C_{M, \mu, A} e^{C_{f}} e^{C k_{(0)}} e^{C_{M, \mu, A} L_{0}} \nrm{\tilde{b}}_{L^{2}}^{2} \\
&\geq \sum_{\pm} \nrm{2^{-\frac{1}{2} \ell} \chi_{I} \brk{D}^{\frac{1}{2}} (e^{f_{I}}(x^{3}) \tilde{B}_{\pm})}_{L^{2}}^{2}
- C_{M, \mu, A} e^{C_{f}} e^{C k_{(0)}} e^{C_{M, \mu, A} L_{0}} \nrm{\tilde{b}}_{L^{2}}^{2}.
\end{align*}
Now that we fixed $C_{f}$, we will stop keeping track of the dependence of constants on $C_{f}$. For the anti-symmetric operator $\asymm^{(1)}_{\bgB}$, we have
\begin{align*}
\brk*{\pm \asymm^{(1)}_{\bgB} \tb_{\pm}, \pm \frkF_{I} \tb_{\pm}} 
&= \frac{1}{2}\brk*{\tb_{\pm}, [\frkF_{I}, \asymm^{(1)}_{\bgB}] \tb_{\pm}} 
= O(C_{M, \mu, A} e^{C k_{(0)}} e^{C_{M, \mu, A} L_{0}} \nrm{\tb}_{L^{2}}^{2}).
\end{align*}
and similarly,
\begin{align*}
\brk*{(\nb \times \bgB) \cdot \nb \tb_{\pm}, \pm \frkF_{I} \tb_{\pm}}
&= O(C_{M, \mu, A} e^{C k_{(0)}} e^{C_{M, \mu, A} L_{0}} \nrm{\tb}_{L^{2}}^{2}).
\end{align*}
Finally, since $\nb \cdot \tb_{\pm} = 0$, we also have
\begin{align*}
\brk*{\pm \nb \covec^{(0)}_{\bgB} (\tb_{+}+\tb_{-}), \pm \frkF_{I} \tb_{\pm}} 
&= - \brk*{\covec^{(0)}_{\bgB} (\tb_{+}+\tb_{-}), [\nb \cdot, \frkF_{I}] \tb_{\pm}} 
= O(C_{M, \mu, A} e^{C k_{(0)}} e^{C_{M, \mu, A} L_{0}} \nrm{\tb}_{L^{2}}^{2}).
\end{align*}
In conclusion, we have
\begin{align*}
\sum_{\pm} \nrm{2^{-\frac{1}{2} \ell} \chi_{I} \brk{D}^{\frac{1}{2}} (e^{f_{I}} \tilde{B}_{\pm})}_{L^{2}}^{2}
&\leq \rd_{t} \left( \brk{\tilde{B}_{+}, e^{2f_{I}}(x^{3}) \tilde{B}_{+}} - \brk{\tilde{B}_{-}, e^{2f_{I}}(x^{3}) \tilde{B}_{-}} \right) + \sum_{\pm} \brk{\tb_{\pm}, \tilde{E}_{I; \pm}} \\
&\peq + { \brk{\tilde{B}_{\pm}, \tilde{H}_{I; \pm} } } + C_{M, \mu, A} e^{C k_{(0)}} e^{C_{M, \mu, A} L_{0}} \nrm{\tb}_{L^{2}}^{2}.
\end{align*}
Integrating this inequality on $t \in (0, T)$ and using the duality between $X^{0}$ and $Y^{0}$, we obtain
\begin{equation} \label{eq:led-I-key}
\begin{aligned}
\sum_{\pm} \nrm{2^{-\frac{1}{2} \ell} \chi_{I} \brk{D}^{\frac{1}{2}} (e^{f_{I}} \tilde{B}_{\pm})}_{L^{2} L^{2}}^{2}
&\aleq \nrm{e^{2f_{I}}\tilde{B}}_{L^{\infty} L^{2}}^{2}
+ \nrm{e^{f_{I}}\tilde{B}}_{X^{0}}\nrm{e^{f_{I}}\tilde{H}}_{Y^{0}} \\
&\peq + \nrm{\tilde{b}}_{X^{0}} \nrm{\tilde{E}_{I}}_{Y^{0}} 
+ C_{M, \mu, A} e^{C k_{(0)}} e^{C_{M, \mu, A} L_{0}} T \nrm{\tb}_{L^{\infty} L^{2}}^{2}.
\end{aligned}\end{equation}
where, here and below, all spacetime norms are restricted to $(0, T) \times \bbR^{3}$. Recall, from \eqref{eq:f-I-symb}, that $\abs{f_{I}} \aleq 1$. Taking the supremum over $I \in \calI_{\ell}$ and $\ell \in \bbZ_{\geq 0}$ and adding a suitable multiple of $\nrm{\tilde{B}}_{L^{\infty} L^{2}}$ on both sides, we arrive at the key inequality
\begin{equation} \label{eq:led-pos-comm}
\begin{aligned}
\nrm{\tilde{B}}_{X^{0}}^{2} \aleq \nrm{\tilde{B}}_{L^{\infty} L^{2}}^{2} + \nrm{\tilde{H}}_{Y^{0}}^{2} + \sup_{\ell \geq \bbZ_{\geq 0}} \sup_{I \in \calI_{\ell}} \nrm{\tilde{b}}_{X^{0}} \nrm{\tilde{E}_{I}}_{Y^{0}} + C_{M, \mu, A} e^{C k_{(0)}} e^{C_{M, \mu, A} L_{0}} T \nrm{\tb}_{L^{\infty} L^{2}}^{2}.
\end{aligned}\end{equation}

\pfstep{Step~3: Error bound}
In this step, our goal is to prove the following bound for the error term in \eqref{eq:led-I-key}
\begin{align}
	\nrm{\tilde{E}_{I; \pm}}_{Y^{0}} \leq C_{M, \mu, A} e^{C_{M, \mu, A} L_{0}} \left( 2^{-(s_{0} - \frac{7}{2}) k_{(0)}}  + e^{C(k_{(0)} + k_{(1)})} T \right) \nrm{\tilde{b}}_{X^{0}}. \label{eq:led-err}
\end{align}
where an important point is that $C_{M, \mu, A} e^{C_{M, \mu, A} L_{0}}$ is independent of $k_{(0)}$. 

We begin by rewriting $\tilde{E}_{I; \pm}$ as
\begin{align*}
\tilde{E}_{I; \pm} &=  \frac{1}{2} \left([\frkF_{I}, \diag_{\bfB}^{(2) \sharp} - \diag_{\bfB_{<k_{(0)}}}] + [\frkF_{I},  \diag_{\bfB_{<k_{(0)}} - (\bfB_{0})_{<k_{(0)}}}^{(2)}] \right) \tilde{b}_{\pm} \\
&= \frac{1}{2} [\frkF_{I}, \diag_{\bfB}^{(2) \sharp} - \diag_{\bfB_{< k_{(0)}}}^{(2) \sharp}]  \tilde{b}_{\pm}
+ \frac{1}{2} [\frkF_{I},  \diag_{\bfB_{< k_{(0)}}}^{(2) \sharp} - \diag_{\bfB_{<k_{(0)}}}^{(2)}]  \tilde{b}_{\pm}
+ \frac{1}{2} [\frkF_{I},  \diag_{\bfB_{<k_{(0)}} - (\bfB_{0})_{<k_{(0)}}}^{(2)}]  \tilde{b}_{\pm} \\
&=: \tilde{E}_{I, 1; \pm} + \tilde{E}_{I, 2; \pm} + \tilde{E}_{I, 3; \pm}.
\end{align*}

To estimate $\tilde{E}_{I, 1; \pm}$, we write out $\diag_{\bfB}^{(2) \sharp} - \diag_{\bfB_{< k_{(0)}}}^{(2) \sharp} = \frac{1}{2} T_{\bfB_{\geq k_{(0)}}^{\alp}} \rd_{\alp} \abs{\nb} + \frac{1}{2} \abs{\nb} \rd_{\alp} T_{\bfB_{\geq k_{(0)}}^{\alp}}$ and treat each term separately. For the contribution of the first term, we apply Proposition~\ref{prop:psdo-comp} to exploit the commutator structure and derive
\begin{align*}
[\frkF_{I}, T_{\bfB_{\geq k_{(0)}}^{\alp}} \rd_{\alp} \abs{\nb}] \tilde{b}_{\pm}
&= \Op \left(i^{-1} \rd_{\xi_{\bt}} \frkF_{I} \right) T_{\rd_{x^{\bt}} \bfB_{\geq k_{(0)}}^{\alp}} \rd_{\alp} \abs{\nb} \tilde{b}_{\pm} \\
&\peq - \Op \left(\rd_{x^{\bt}} \frkF_{I} \right) \Op \left(i^{-1} \rd_{\xi_{\bt}} \left( \sum_{k} P_{<k-10} \bfB_{\geq k_{(0)}}^{\alp}(x) P_{k}(\xi)  i \xi_{\alp} \abs{\xi} \right)  \right) \tilde{b}_{\pm} \\
&\peq + O_{L^{1} L^{2}}(C_{M, \mu, A} e^{C k_{(0)}} e^{C_{M, \mu, A} L_{0}} T \nrm{\tilde{b}}_{L^{\infty} L^{2}} ) \\
&= \Op \left(i^{-1} \rd_{\xi_{\bt}} \frkF_{I} \right) P_{\geq k_{(1)}} T_{\rd_{x^{\bt}} \bfB_{\geq k_{(0)}}^{\alp}} \rd_{\alp} \abs{\nb} \tilde{b}_{\pm} \\
&\peq - \Op \left(\rd_{x^{\bt}} \frkF_{I} \right) P_{\geq k_{(1)}} \Op \left(i^{-1} \rd_{\xi_{\bt}} \left( \sum_{k} P_{<k-10} \bfB_{\geq k_{(0)}}^{\alp}(x) P_{k}(\xi)  i \xi_{\alp} \abs{\xi} \right)  \right) \tilde{b}_{\pm} \\
&\peq + O_{L^{1} L^{2}}(C_{M, \mu, A} e^{C (k_{(0)}+k_{(1)})} e^{C_{M, \mu, A} L_{0}} T \nrm{\tilde{b}}_{L^{\infty} L^{2}} ),
\end{align*}
where we used Proposition~\ref{prop:psdo-L2-bdd-garding}, $\nrm{2^{2 k_{(1)}} P_{\geq k_{(1)}} \brk{D}^{-2}}_{L^{2} \to L^{2}} \aleq 1$ and H\"older's inequality for the last identity. Using \eqref{eq:prod-core-hl}, we have
\begin{align*}
\nrm{\brk{D}^{-1} (T_{\rd_{x^{\bt}} \bfB_{\geq k_{(0)}}^{\alp}} \rd_{\alp} \abs{\nb} \tilde{b}_{\pm})}_{Y^{0}}
+ \nrm{\Op \big(i^{-1} \rd_{\xi_{\bt}} \big( \sum_{k} P_{<k-10} \bfB_{\geq k_{(0)}}^{\alp}(x) P_{k}(\xi)  i \xi_{\alp} \abs{\xi} \big) \big) \tilde{b}_{\pm}}_{Y^{0}} \aleq M 2^{-(s_{0} - \frac{7}{2}) k_{(0)}} \nrm{\tilde{b}}_{X^{0}},
\end{align*}
whereas by Proposition~\ref{prop:psdo-ests}.(2) (i.e., the high frequency Calder\'on--Vaillancourt trick) and the symbol bounds \eqref{eq:f-symb-0}--\eqref{eq:f-symb-high} and \eqref{eq:f-I-symb}, we have
\begin{align*}
\nrm{\Op \left(i^{-1} \rd_{\xi_{\bt}} \frkF_{I} \right) P_{\geq k_{(1)}} \brk{D} }_{Y^{0} \to Y^{0}}
+ \nrm{\Op \left(\rd_{x^{\bt}} \frkF_{I} \right) P_{\geq k_{(1)}}}_{Y^{0} \to Y^{0}} \aleq_{M, \mu, A} e^{C_{M, \mu, A} L_{0}},
\end{align*}
provided that $k_{(1)}$ is sufficiently large depending on $M$, $\mu$, $A$, $L_{0}$ and $k_{(0)}$.  The contribution of $\frac{1}{2} \abs{\nb} \rd_{\alp} T_{\bfB_{\geq k_{(0)}}^{\alp}}$ is treated similarly. In conclusion, we have
\begin{align*}
	\nrm{\tilde{E}_{I; 1, \pm}}_{Y^{0}} \leq C_{M, \mu, A} e^{C_{M, \mu, A} L_{0}} 2^{-(s_{0} - \frac{7}{2}) k_{(0)}} \nrm{\tilde{b}}_{X^{0}}
+ C_{M, \mu, A} e^{C(k_{(0)} + k_{(1)})} e^{C_{M, \mu, A} L_{0}} T \nrm{\tilde{b}}_{X^{0}}.
\end{align*}

To estimate $\tilde{E}_{I, 2; \pm}$, we simply note that 
\begin{equation*}
\diag_{\bfB_{< k_{(0)}}}^{(2) \sharp} - \diag_{\bfB_{<k_{(0)}}}^{(2)} = \frac{1}{2} \sum_{k} (P_{\geq k-10} \bfB_{< k_{(0)}}^{\alp}) \rd_{\alp} \abs{\nb} P_{k} + \frac{1}{2} \sum_{k} P_{k} \abs{\nb} \rd_{\alp} (P_{\geq k-10} \bfB_{< k_{(0)}}^{\alp})
\end{equation*}
where both summands are trivial if $k > k_{(0)} + 20$ (since $P_{\geq k-10} \bfB_{<k_{(0)}} = 0$). Then without using the commutator structure, we may bound
\begin{align*}
\nrm{\tilde{E}_{I; 2, \pm}}_{Y^{0}}
&\aleq \nrm{\frkF_{I}(x, D) (\diag_{\bfB_{< k_{(0)}}}^{(2) \sharp} - \diag_{\bfB_{<k_{(0)}}}^{(2)}) \td{b}_{\pm}}_{L^{1} L^{2}}
+ \nrm{(\diag_{\bfB_{< k_{(0)}}}^{(2) \sharp} - \diag_{\bfB_{<k_{(0)}}}^{(2)}) \frkF_{I}(x, D)  \td{b}_{\pm}}_{L^{1} L^{2}} \\
&\aleq_{M, \mu, A} T 2^{2 k_{(0)}} M \nrm{\tilde{b}}_{L^{\infty} L^{2}},
\end{align*}
where we used Proposition~\ref{prop:psdo-L2-bdd-garding}, the symbol bounds \eqref{eq:f-symb-0}--\eqref{eq:f-symb-high} and \eqref{eq:f-I-symb} and $\nrm{P_{k} \abs{\nb} \rd_{\alp}}_{L^{2} \to L^{2}} \aleq 2^{2k}$.

Finally, to estimate $\tilde{E}_{I, 3; \pm}$, we again start by applying Proposition~\ref{prop:psdo-comp} to write
\begin{align*}
[\frkF_{I}, \diag_{\bfB_{<k_{(0)}} - (\bfB_{0})_{<k_{(0)}}}^{(2)}] \tilde{b}_{\pm}
&= \Op \left(i^{-1} \rd_{\xi_{\bt}} \frkF_{I} \right) \diag_{\rd_{x^{\bt}} (\bfB_{<k_{(0)}} - (\bfB_{0})_{<k_{(0)}})}^{(2)} \tilde{b}_{\pm} \\
&\peq - \Op \left(\rd_{x^{\bt}} \frkF_{I} \right) \Op\left(i^{-1} \rd_{\xi_{\bt}}((\bfB_{<k_{(0)}}^{\alp} - (\bfB_{0}^{\alp})_{<k_{(0)}}) i \xi_{\alp} \abs{\xi}) \right) \tilde{b}_{\pm} \\
&\peq + O_{L^{1} L^{2}}(C_{M, \mu, A} e^{C k_{(0)}} e^{C_{M, \mu, A} L_{0}} T \nrm{\tilde{b}}_{L^{\infty} L^{2}} ).
\end{align*}
By \eqref{eq:prod-core-hl}, we have
\begin{align*}
&\nrm{\brk{D}^{-1} (\diag_{\rd_{x^{\bt}} (\bfB_{<k_{(0)}} - (\bfB_{0})_{<k_{(0)}})}^{(2)} \tilde{b}_{\pm})}_{Y^{0}}
+ \nrm{\Op\left(i^{-1} \rd_{\xi_{\bt}}((\bfB_{<k_{(0)}}^{\alp} - (\bfB_{0}^{\alp})_{<k_{(0)}}) i \xi_{\alp} \abs{\xi}) \right) \tilde{b}_{\pm}}_{Y^{0}} \\
& \aleq \sum_{j < k_{(0)}} 2^{\frac{7}{2} j} \nrm{P_{j} (\bfB - \bfB_{0})}_{\ell^{1}_{\calI} L^{\infty} L^{2} } \nrm{\tilde{b}}_{X^{0}}  \\
&\aleq T 2^{2 k_{(0)}} \nrm{\rd_{t} \bfB}_{\ell^{1}_{\calI} L^{\infty} H^{\frac{7}{2}-2}} \nrm{\tilde{b}}_{X^{0}},
\end{align*}
which is bounded by $T 2^{2 k_{(0)}} M \nrm{\tilde{b}}_{X^{0}}$. Moreover, by Proposition~\ref{prop:psdo-ests}.(1) and the symbol bounds \eqref{eq:f-symb-0}--\eqref{eq:f-symb-high} and \eqref{eq:f-I-symb}, we have
\begin{align*}
\nrm{\Op \left(i^{-1} \rd_{\xi_{\bt}} \frkF_{I} \right) \brk{D}}_{Y^{0} \to Y^{0}}
+ \nrm{\Op \left(\rd_{x^{\bt}} \frkF_{I} \right)}_{Y^{0} \to Y^{0}} \aleq_{M, \mu, A} e^{C k_{(0)}} e^{C_{M, \mu, A} L_{0}}.
\end{align*}
It follows that
\begin{align*}
	\nrm{\tilde{E}_{I; 3, \pm}}_{Y^{0}} \leq C_{M, \mu, A} e^{C k_{(0)}} e^{C_{M, \mu, A} L_{0}} T \nrm{\tilde{b}}_{X^{0}},
\end{align*}
which concludes the proof of \eqref{eq:led-err}.

\pfstep{Step~4: Returning to $\tilde{b}$}
We are now ready to conclude the proof. We claim that
\begin{align}
\nrm{\tilde{b}}_{X^{0}}^{2} &\aleq C_{M, \mu, A} e^{C_{M, \mu, A} L_{0}} \nrm{\tilde{B}}_{X^{0}}^{2} + C_{M, \mu, A} e^{C (k_{(0)} + k_{(1)})} e^{C_{M, \mu, A} L_{0}} \nrm{b}_{L^{\infty} L^{2}}^{2}, \label{eq:led-tb} \\
\nrm{\tilde{B}}_{L^{\infty} L^{2}}^{2} &\aleq C_{M, \mu, A} e^{C k_{(0)}} e^{C_{M, \mu, A} L_{0}} \nrm{\tilde{b}}_{L^{\infty} L^{2}}^{2}
+ C_{M, \mu, A} e^{C (k_{(0)} + k_{(1)})} e^{C_{M, \mu, A} L_{0}} \nrm{b}_{L^{\infty} L^{2}}^{2}, \label{eq:led-tB} \\
\nrm{\tilde{H}}_{Y^{0}}^{2} &\aleq C_{M, \mu, A} e^{C k_{(0)}} e^{C_{M, \mu, A} L_{0}} \nrm{\tilde{h}}_{Y^{0}}^{2}. \label{eq:led-tH}
\end{align}
Indeed, combining \eqref{eq:led-pos-comm} and \eqref{eq:led-err} with \eqref{eq:led-tb}--\eqref{eq:led-tH}, we obtain the desired estimate \eqref{eq:paralin-led} for $\tilde{b}$.

It remains to prove \eqref{eq:led-tb}--\eqref{eq:led-tH}. In what follows, we shall freely use the classical symbolic calculus (i.e, Proposition~\ref{prop:psdo-comp}.(1)), boundedness properties of $\Op(\frka)$ for $\frka \in S^{0}$ (i.e., Proposition~\ref{prop:psdo-L2-bdd-garding} and Proposition~\ref{prop:psdo-ests}), and the symbol bounds \eqref{eq:f-symb-0}--\eqref{eq:f-symb-high} for $f$. We shall also freely use the simple bound
\begin{equation} \label{eq:LE-L2}
	\nrm{b}_{LE} \aleq T^{\frac{1}{2}} \nrm{\brk{D}^{-\frac{1}{2}} b}_{L^{\infty} L^{2}},
\end{equation}
which follows from Littlewood--Paley theory and H\"older in time  {(see also \eqref{eq:Xs-Hs+1/2})}.

To prove \eqref{eq:led-tb}, we begin by noting that
\begin{align*}
	\nrm{(1 - e^{-f} (x, D) e^{f} (x, D)) P_{\geq k_{(1)}}}_{X^{0} \to X^{0}} \aleq_{M, \mu, A} e^{C k_{(0)}} e^{C_{M, \mu, A} L_{0}} 2^{-k_{(1)}},
\end{align*}
which follows from Proposition~\ref{prop:psdo-comp}.(1), the symbol bounds and Proposition~\ref{prop:psdo-ests}.
Hence, using $\tilde{B}_{\pm}  = e^{f}(x, D) \tilde{b}_{\pm}$ and $\tilde{b}_{\pm} = \brk{D}^{\sgm} b_{\pm}$, we have
\begin{align*}
	\tilde{b}_{\pm} 
	&= e^{-f}(x, D) \tilde{B}_{\pm} + (1 - e^{-f}(x, D) e^{f}(x, D)) P_{<k_{(1)}} \tilde{b}_{\pm} + O_{X^{0}}(C e^{C k_{(0)}} e^{C_{M, \mu, A} L_{0}} 2^{-k_{(1)}} \nrm{\tilde{b}}_{X^{0}}) \\
	&= e^{-f}(x, D) P_{\geq k_{(1)}} \tilde{B}_{\pm} + e^{-f}(x, D) P_{< k_{(1)}} e^{f}(x, D) \brk{D}^{\sgm} b_{\pm}\\
	&\peq + (1 - e^{-f}(x, D) e^{f}(x, D)) P_{<k_{(1)}} \brk{D}^{\sgm} b_{\pm} + O_{X^{0}}(C e^{C k_{(0)}} e^{C_{M, \mu, A} L_{0}} 2^{-k_{(1)}} \nrm{\tilde{b}}_{X^{0}}) \\
	&= e^{-f}(x, D) P_{\geq k_{(1)}} \tilde{B}_{\pm} + O_{X^{0}}(C e^{C (k_{(0)}+k_{(1)})} e^{C_{M, \mu, A} L_{0}} (1+T^{\frac{1}{2}}) \nrm{b}_{L^{\infty} L^{2}})
	+ O_{X^{0}}(C e^{C k_{(0)}} e^{C_{M, \mu, A} L_{0}} 2^{-k_{(1)}} \nrm{\tilde{b}}_{X^{0}})
\end{align*}
Taking $k_{(1)}$ sufficiently large depending on $M$, $\mu$, $A$, $L_{0}$ and $k_{(0)}$, we may apply Proposition~\ref{prop:psdo-ests}.(2) to the first term on the last line, and also absorb the last term into the  {LHS}. Thus \eqref{eq:led-tb} follows.

To prove \eqref{eq:led-tB}, we begin by noting that
\begin{align*}
	\tilde{B}_{\pm} 
	&= e^{f}(x, D) P_{\geq k_{(1)}} \tilde{b}_{\pm} + e^{f}(x, D) P_{< k_{(1)}} \tilde{b}_{\pm} \\
	&= e^{f}(x, D) P_{\geq k_{(1)}} \tilde{b}_{\pm} + e^{f}(x, D) P_{< k_{(1)}} \brk{D}^{\sgm} b_{\pm} \\
	&= e^{f}(x, D) P_{\geq k_{(1)}} \tilde{b}_{\pm} + O_{L^{\infty} L^{2}}(C e^{C (k_{(0)}+k_{(1)})} e^{C_{M, \mu, A} L_{0}} \nrm{b}_{L^{\infty} L^{2}}).
\end{align*}
Taking $k_{(1)}$ sufficiently large depending on $M$, $\mu$, $A$, $L_{0}$ and $k_{(0)}$, we may apply Lemma~\ref{lem:hf-L2} to the first term on the last line, which proves \eqref{eq:led-tB}.

Finally, \eqref{eq:led-tH} is a quick consequence of Proposition~\ref{prop:psdo-ests}.(1). \qedhere
\end{proof}

\subsection{Renormalization and boundedness of energy} \label{subsec:paralin-bdd}
Next, we establish the $L^{\infty} L^{2}$ bound for $\tb$ that was assumed in the previous subsection. We begin with the $L^{\infty} L^{2}$ bound for $b$, which is a straightforward consequence of the conservative structure of \eqref{eq:para-diag-tb} without $\comm^{\sgm(1)}_{(\bfB_{0})_{<k_{(0)}}}$.
\begin{proposition} \label{prop:paralin-bdd-0} 
Let $b$ be a solution to the system \eqref{eq:paralin}--\eqref{eq:paralin-constraint} on $(0, T)$. Then $b_{\pm} = \Pi_{\pm}(D) b$ obeys
\begin{equation}\label{eq:paralin-bdd-0}
	\sum_{\pm} \nrm{b_{\pm}}_{L^{\infty} L^{2}[0, T]}^{2} \aleq \sum_{\pm} \left(\nrm{b_{\pm}(0)}_{L^{2}}^{2} + \nrm{h_{\pm}}_{Y^{0}[0, T]} \nrm{b_{\pm}}_{X^{0}[0, T]} \right),
\end{equation}
where 
\begin{equation*}
\begin{pmatrix} h_{+} \\ h_{-} \end{pmatrix} = \begin{pmatrix} \Pi_{+}(D) g \\ \Pi_{-}(D) g \end{pmatrix} + O_{Y^{0}}(2^{-(s_{0}-\frac{7}{2}) k_{(0)}} + T 2^{2 k_{(0)}}) M \sum_{\pm} \nrm{b_{\pm}}_{X^{0}}.
\end{equation*}
\end{proposition}
\begin{proof}
By Proposition~\ref{prop:para-diag-tdb}, $b_{\pm}$ obey \eqref{eq:para-diag-tb} with $\sgm = 0$, i.e., $\comm^{\sgm (1)}_{(\bfB_{0})_{<k_{(0)}}} = 0$ and $E_{1}^{\sgm} = 0$. It also follows that $h_{\pm}$ are of the form above. To prove the $L^{\infty} L^{2}$ estimate, we multiply \eqref{eq:para-diag-tb} to the left by $(b_{+}, b_{-})$ and integrate over $\bbR^{3}$, or in other words, take the inner product with $(b_{+}, b_{-})^{\top}$ with respect to $L^{2}(\bbR^{3}; \bbR^{3} \oplus \bbR^{3})$. Observe that the contribution of the second, third, fourth, fifth and sixth terms on the  {LHS} of \eqref{eq:para-diag-tb} vanish, since the operator in front of $(b_{+}, b_{-})^{\top}$ is anti-symmetric with respect to $L^{2}(\bbR^{3}; \bbR^{3} \oplus \bbR^{3})$. Moreover, the contribution of the last term on the  {LHS} of \eqref{eq:para-diag-tb} also vanishes since $b_{\pm}$ are divergence-free (i.e., $\brk{b_{\pm} \cdot \nb(\cdots)} = - \brk{\nb \cdot b_{\pm}, (\cdots)} = 0$). We end up with the following identity:
\begin{align*}
	\frac{1}{2} \rd_{t} \left(\nrm{b_{+}}_{L^{2}}^{2} +  \nrm{b_{+}}_{L^{2}}^{2} \right)
	&= \brk{h_{+}, b_{+}} + \brk{h_{-}, b_{-}}.
\end{align*}
Now the desired conclusion follows by integrating in $t$ on $(0, T)$. \qedhere
\end{proof}

Next, we prove the $L^{\infty} L^{2}$ bound for $\tb_{\pm} = \brk{D}^{\sgm} b_{\pm}$, which is the main result of this subsection.
\begin{proposition} \label{prop:paralin-bdd}
Let $\tb_{\pm}$ solve \eqref{eq:para-diag-tb}. Then we have for all sufficiently large $k_{(1)}$ satisfying $k_{(1)} \ge k_{(0)}$ that
\begin{equation}\label{eq:paralin-bdd}
\begin{aligned}
	\nrm{\tb}_{L^{\infty} L^{2}[0, T]}^{2} &\aleq C_{M, \mu, A} e^{C k_{(0)}} e^{C_{M, \mu, A} L} \left( \nrm{\tb(0)}_{L^{2}}^{2}
	+ \nrm{\tilde{h}}_{Y^{0}[0, T]} \nrm{\tb}_{X^{0}[0, T]}\right) \\
	&\peq + C e^{\sgm C A} \eps \nrm{\tb}_{X^{0}}^{2}
	+ C_{M, \mu, A} e^{C_{M, \mu, A} L} 2^{-(s_{0} - \frac{7}{2}) k_{(0)}}  \nrm{\tb}^{2}_{X^{0}[0, T]} \\
	&\peq + C_{M, \mu, A} e^{C(k_{(0)} + k_{(1)})} e^{C_{M, \mu, A} L} T \nrm{\tb}^{2}_{X^{0}[0, T]} \\
	&\peq + C_{M, \mu, A} e^{C k_{(0)}} e^{C_{M, \mu, A} L} 2^{\sgm k_{(1)}} \nrm{\brk{D}^{-\sgm} \tb}_{L^{\infty} L^{2}[0, T]}^{2}.
\end{aligned}\end{equation}
\end{proposition}
The most important point is that the constant in front of $\eps \nrm{\tb}_{X^{0}}^{2}$ is \emph{independent} of $L$ (and $R$). Coupled with Propositions~\ref{prop:paralin-led} and \ref{prop:paralin-bdd-0}, we will thus be able to take $\eps$ small enough to absorb this term into the  {LHS}; see Section~\ref{subsec:paralin-pf} below.

\begin{proof}
As in the proof of Proposition~\ref{prop:paralin-led}, we will use the shorthand $\bgB := (\bfB_{0})_{<k_{(0)}}$. We divide the proof into several steps.

\pfstep{Step~1: Conjugation by $O_{\pm}$} We start from \eqref{eq:para-diag-tb}, but now with $\sgm > 0$. In order to remove $\pm \comm^{\sgm (1)}_{\bgB}$, we conjugate the equation for $\tb_{\pm}$ by a time-independent pseudodifferential operator $O_{\pm}(x, D)$ (to be constructed below). We first record the equation obtained by commuting $O_{+}$ with the principal operator $\pm \diag_{\bfB}^{(2) \sharp}$:
\begin{equation*}
\begin{alignedat}{2}
	 & \rd_{t} \begin{pmatrix} O_{+} \tb_{+} \\ O_{-} \tb_{-} \end{pmatrix} 
	 + \begin{pmatrix} \diag_{\bfB}^{(2) \sharp} & 0 \\ 0 & - \diag_{\bfB}^{(2) \sharp} \end{pmatrix} \begin{pmatrix} O_{+} \tb_{+} \\ O_{-} \tb_{-} \end{pmatrix} & \\
	&	 + \begin{pmatrix} [O_{+}, \diag_{\bfB}^{(2) \sharp}] & 0 \\ 0 & - [O_{-}, \diag_{\bfB}^{(2) \sharp}] \end{pmatrix} \begin{pmatrix} \tb_{+} \\ \tb_{-} \end{pmatrix} 
	 + \begin{pmatrix} O_{+} \comm^{\sgm (1)}_{\bgB} & 0 \\ 0 & - O_{-} \comm^{\sgm (1)}_{\bgB} \end{pmatrix} \begin{pmatrix} \tb_{+} \\ \tb_{-} \end{pmatrix} & \\
	& + \begin{pmatrix} 0 &  O_{+} \symm^{(1)}_{\bgB}   \\ - O_{-} \symm^{(1)}_{\bgB}  & 0 \end{pmatrix} \begin{pmatrix} \tb_{+} \\ \tb_{-} \end{pmatrix}  
	+  \begin{pmatrix} O_{+} \asymm^{(1)}_{\bgB}  & 0 \\ 0 & - O_{-} \asymm^{(1)}_{\bgB}  \end{pmatrix} \begin{pmatrix} \tb_{+} \\ \tb_{-} \end{pmatrix}\\
	& + \begin{pmatrix} O_{+}(\nb \times \bgB)\cdot \nb & 0 \\ 0 & O_{-}(\nb \times \bgB)\cdot \nb \end{pmatrix} \begin{pmatrix}  \tb_{+} \\ \tb_{-} \end{pmatrix}  & \\
	& +\begin{pmatrix} O_{+} & 0 \\ 0 & O_{-} \end{pmatrix} \nb \begin{pmatrix}  \covec^{(0)}_{\bgB} & \covec^{(0)}_{\bgB}  \\ - \covec^{(0)}_{\bgB}  & - \covec^{(0)}_{\bgB} \end{pmatrix}  \begin{pmatrix}  \tb_{+} \\ \tb_{-}  \end{pmatrix} 
	= \begin{pmatrix}  O_{+} \tilde{h}_{+} \\ O_{-} \tilde{h}_{-}  \end{pmatrix}.  &
\end{alignedat}
\end{equation*}
For the commutator of $O_{\pm}$ and $\diag_{\bfB}^{(2) \sharp}$, let us write (in the operator notation)
\begin{align*}
\pm [O_{\pm}, \diag_{\bfB}^{(2) \sharp}] \tb_{\pm}
&= 
\pm \frac{1}{2} \sum_{k} [O_{\pm},  \bfB_{<k-10}^{\alp} P_{k} \rd_{\alp} \abs{\nb}] \tb_{\pm} \pm \frac{1}{2} \sum_{k} [O_{\pm}, \abs{\nb} \rd_{\alp} P_{k} \bfB_{<k-10}^{\alp}] \tb_{\pm}\\
&= 
\pm \frac{1}{2} \sum_{k} [O_{\pm},  (\bfB_{<k-10} - \bfB_{<k_{(0)}})^{\alp} P_{k} \rd_{\alp} \abs{\nb}] \tb_{\pm} \pm \frac{1}{2} \sum_{k} [O_{\pm}, \abs{\nb} \rd_{\alp} P_{k} (\bfB_{<k-10} - \bfB_{<k_{(0)}})^{\alp}] \tb_{\pm} \\
&\peq + [O_{\pm}, \diag^{(2)}_{\bfB_{<k_{(0)}} - \bgB}] \tb_{\pm}
+ [O_{\pm}, \diag^{(2)}_{\bgB}] \tb_{\pm}.
\end{align*}
Commuting $O_{+}$, $O_{-}$ with all the other terms in the equation, we arrive at
\begin{equation} \label{eq:para-diag-Otb}
\begin{alignedat}{2}
	 & \rd_{t} \begin{pmatrix} O_{+} \tb_{+} \\ O_{-} \tb_{-} \end{pmatrix} 
	 + \begin{pmatrix} \diag_{\bfB}^{(2) \sharp} & 0 \\ 0 & - \diag_{\bfB}^{(2) \sharp} \end{pmatrix} \begin{pmatrix} O_{+} \tb_{+} \\ O_{-} \tb_{-} \end{pmatrix} \\
	&
	 + \begin{pmatrix} \left( [O_{+}, \diag_{\bgB}^{(2)}] + O_{+} \comm^{\sgm (1)}_{\bgB}\right) \tb_{+} \\ - \left( [O_{-}, \diag_{\bgB}^{(2)}] + O_{-} \comm^{\sgm (1)}_{\bgB} \right) \tb_{-} \end{pmatrix} 
	+ \begin{pmatrix} \symm^{(1)}_{\bgB} O_{+} \tb_{-} \\ - \symm^{(1)}_{\bgB} O_{-} \tb_{+} \end{pmatrix}  \\
	& +  \begin{pmatrix} \asymm^{(1)}_{\bgB}  & 0 \\ 0 & - \asymm^{(1)}_{\bgB}  \end{pmatrix} \begin{pmatrix} O_{+} \tb_{+} \\ O_{-} \tb_{-} \end{pmatrix}
	+ \begin{pmatrix} (\nb \times \bgB)\cdot \nb & 0 \\ 0 & (\nb \times \bgB)\cdot \nb \end{pmatrix} \begin{pmatrix}  O_{+} \tb_{+} \\ O_{-} \tb_{-} \end{pmatrix}  \\
	& +\nb \begin{pmatrix}  \covec^{(0)}_{\bgB} O_{+} & \covec^{(0)}_{\bgB}  O_{+} \\ - \covec^{(0)}_{\bgB} O_{-}  & - \covec^{(0)}_{\bgB} O_{-} \end{pmatrix}  \begin{pmatrix} \tb_{+} \\ \tb_{-}  \end{pmatrix} - \tilde{E}_{O; 1} - \tilde{E}_{O; 2}
	= \begin{pmatrix}  O_{+} \tilde{h}_{+} \\ O_{-} \tilde{h}_{-}  \end{pmatrix},  &
\end{alignedat}
\end{equation}
where 
\begin{align*}
\tilde{E}_{O; 1}
&= - \begin{pmatrix}
[O_{+}, \diag_{\bfB}^{(2) \sharp} - \diag_{\bgB}^{(2)}] \tb_{+} \\
- [O_{-}, \diag_{\bfB}^{(2) \sharp} - \diag_{\bgB}^{(2)}] \tb_{+} 
\end{pmatrix}, \\
\tilde{E}_{O; 2} &= - \begin{pmatrix} 0 &  [O_{+}, \symm^{(1)}_{\bgB}]   \\ - [O_{-}, \symm^{(1)}_{\bgB}]  & 0 \end{pmatrix} \begin{pmatrix} \tb_{+} \\ \tb_{-} \end{pmatrix} 
	-  \begin{pmatrix} [O_{+}, \asymm^{(1)}_{\bgB}]  & 0 \\ 0 & - [O_{-}, \asymm^{(1)}_{\bgB}]  \end{pmatrix} \begin{pmatrix} \tb_{+} \\ \tb_{-} \end{pmatrix} \\
	& \peq - \begin{pmatrix} [O_{+}, (\nb \times \bgB)\cdot \nb] & 0 \\ 0 & [O_{-}, (\nb \times \bgB)\cdot \nb] \end{pmatrix} \begin{pmatrix}  \tb_{+} \\ \tb_{-} \end{pmatrix}  \\
	& \peq - \left[ \begin{pmatrix} O_{+} & 0 \\ 0 & O_{-} \end{pmatrix},  \nb \right] \begin{pmatrix}  \covec^{(0)}_{\bgB} & \covec^{(0)}_{\bgB}  \\ - \covec^{(0)}_{\bgB}  & - \covec^{(0)}_{\bgB} \end{pmatrix} \begin{pmatrix}  \tb_{+} \\ \tb_{-}  \end{pmatrix} 
	- \nb \begin{pmatrix}  [O_{+}, \covec^{(0)}_{\bgB}] & [O_{+}, \covec^{(0)}_{\bgB}]  \\ - [O_{-}, \covec^{(0)}_{\bgB}]  & - [O_{-}, \covec^{(0)}_{\bgB}] \end{pmatrix} \begin{pmatrix}  \tb_{+} \\ \tb_{-}  \end{pmatrix} .
\end{align*}

\pfstep{Step~2: Construction of $O_{\pm} = e^{\psi_{\pm}}$}
The energy estimate will be proved by multiplying \eqref{eq:para-diag-Otb} by $(O_{+} \tb_{+}, O_{-} \tb_{-})$ (to the left) and integrating over $\bbR^{3}$ as in Proposition~\ref{prop:paralin-bdd-0} (see also Step~4 below). We will list some properties that $O_{\pm}$ need to satisfy for this argument to work, and show that such $O_{\pm}$ exist.

We assume that $O_{\pm} (x, \xi) = e^{\psi_{\pm}(x, \xi)}$, where $\psi_{\pm}$ is real-valued and
\begin{equation} \label{eq:renrm-real}
	\psi_{\pm}(x, \xi) = \psi_{\pm}(x, -\xi).
\end{equation}
This ensures that given a real-valued $b$, $O_{\pm} b$ is also real-valued.

A key requirement for $\psi_{\pm}$ is that the contribution of the third term on the  {LHS} of \eqref{eq:para-diag-Otb} is acceptable for the energy estimate. We require the following bound to hold for any $0 < \tau \leq T$:
\begin{equation} \label{eq:renrm-core}
\begin{aligned}
	& \int_{0}^{\tau} \left( \brk*{\left( [O_{+}, \diag_{\bgB}^{(2)}] + O_{+} \comm^{\sgm (1)}_{\bgB} \right) \tb_{+}, O_{+} \tb_{+}} 
	 - \brk*{\left( [O_{-}, \diag_{\bgB}^{(2)}] + O_{-} \comm^{\sgm (1)}_{\bgB} \right) \tb_{-}, O_{-} \tb_{-}} \right) \, \ud t\\
	& \geq - \left(C_{M, \mu, A} \eps + C_{M, \mu, A} e^{C(k_{(0)}+k_{(1)})} e^{C_{M, \mu, A} L} T \right) \sum_{\pm} \nrm{\tb}_{X^{0}}^{2},
\end{aligned}
\end{equation}
where it is crucial that $C_{M, \mu, A}$ does \emph{not} depend on $R$, $k_{(0)}$ and $k_{(1)}$. 

Another important requirement for $\psi_{\pm}$ is that the fourth term on the  {LHS} of \eqref{eq:para-diag-Otb} is acceptable in the energy estimate. In the region $\set{-R < x^{3} < R}$, we require that
\begin{equation} \label{eq:renrm-offdiag-bulk}
	\psi_{+}(x, \xi) = \psi_{-}(x, \xi) \quad \hbox{ for } -4R < x^{3} < 4R.
\end{equation}
This condition ensures that the off-diagonal terms $\pm \symm^{(1)}_{\bgB} O_{\pm} \tb_{\mp}$ are cancelled in this region, where $\bgB$ may potentially be large. Globally, we require that the following bound holds for any $0 < \tau \leq T$:
\begin{equation} \label{eq:renrm-offdiag}
\abs*{	\int_{0}^{\tau} \left( \brk{\symm^{(1)}_{\bgB} O_{+} \tb_{-}, O_{+} \tb_{+}}
	- \brk{\symm^{(1)}_{\bgB} O_{-} \tb_{+}, O_{-} \tb_{-}}\right) \, \ud t}
\aleq_{M, \mu, A} \eps \nrm{\tb}_{X^{0}}^{2} + e^{C(k_{(0)} + k_{(1)})} e^{C_{M, \mu, A} L} T \nrm{\tilde{b}}_{L^{\infty} L^{2}}^{2}.
\end{equation}
Finally, we require the symbols $O_{\pm}$ to satisfy
\begin{align}
	\abs{O_{\pm}} &\leq e^{C \sgm A}, \label{eq:renrm-symb-0} \\
	\abs{\rd_{x} O_{\pm}} + \abs{\brk{\xi} \rd_{\xi} O_{\pm}} &\aleq_{M, \mu, A} e^{C_{M, \mu, A} L} e^{C \sgm A}, \label{eq:renrm-symb-1} \\
	\brk{\xi}^{\abs{\bfbt}} \abs{\rd_{x}^{\bfalp} \rd_{\xi}^{\bfbt} O_{\pm}} &\aleq_{M, \mu, A, \bfalp, \bfbt} e^{C_{M, \mu, A, \bfalp, \bfbt} L}2^{(\abs{\bfalp} - 1)_{+} k_{(0)}} e^{C \sgm A}, \label{eq:renrm-symb-high}
\end{align}
where $C$ is a \emph{universal} constant. That the  {RHS} of \eqref{eq:renrm-symb-0} is independent of $\eps$, $R$, $L$ and $k_{(0)}$ will be key to proving \eqref{eq:renrm-core} and \eqref{eq:renrm-offdiag} (in combination with Lemma~\ref{lem:hf-L2} and Proposition~\ref{prop:psdo-ests}). Moreover, that the  {RHS} of \eqref{eq:renrm-symb-1} is independent of $k_{(0)}$ will be crucial to the estimate of the commutator term $E_{3}$ (see Step~3 below). 

We now construct $O_{\pm}(x, \xi) = e^{\psi_{\pm}(x, \xi)}$ with the properties listed above. Let us write $\frkc^{\sgm (1)}$ for the principal symbol of $-i \comm^{\sgm(1)}_{\bgB}$, i.e.,
\begin{align*}
	\frkc^{\sgm (1)}_{\bgB} 
	&:= - \{ \brk{\xi}^{\sgm}, \bgB(x) \cdot \xi \abs{\xi} \} \brk{\xi}^{-\sgm} 
	= - \sgm \rd_{\alp} \bgB^{\bt}(x) \frac{\xi_{\alp} \xi_{\bt}}{\brk{\xi}^{2}} \abs{\xi}.
\end{align*}
We define
\begin{align*}
	\td{\psi}(x, \xi) &:= \frac{1}{2} \int_{0}^{\infty} \chi_{< 16R}(X^{3}(- t; x, \xi)) \frkc^{\sgm (1)}_{\bgB}(X(- t; x, \xi), \Xi(- t; x, \xi)) \, \ud t \\
	&\phantom{:=} - \frac{1}{2} \int_{0}^{\infty} \chi_{< 16R}(X^{3}(t; x, \xi)) \frkc^{\sgm (1)}_{\bgB}(X(t; x, \xi), \Xi(t; x, \xi)) \, \ud t,
\end{align*}
so that $\td{\psi}$ is real-valued and
\begin{equation*}
	\{ \dprin_{\bgB}, \td{\psi} \} (x, \xi) = \chi_{< 16R}(x^{3}) \frkc^{\sgm (1)}_{\bgB}(x, \xi).
\end{equation*}
Moreover, note that $(X, \Xi) (t; x, -\xi) = (X, - \Xi)(t; x, \xi)$ and $\frkc^{\sgm (1)}_{\bgB}(x, \xi) = \frkc^{\sgm (1)}_{\bgB}(x, -\xi)$. Therefore, 
\begin{equation*}
\td{\psi}(x, - \xi) = \td{\psi}(x, \xi).
\end{equation*}
Hence, \eqref{eq:renrm-real} holds. Next, recall from \eqref{eq:nontrapping-B0<k0} that $\bgB = (\bfB_{0})_{<k_{(0)}} \in \calB^{s_{0}}_{\eps}(2 M, \frac{1}{2} \mu, 2 A, R, 2 L)$. In view of \eqref{eq:mizohata-A}, we have the symbol bound
\begin{equation*}
	\abs*{\td{\psi}(x, \xi)} \leq 2 \sgm A.
\end{equation*}
By Lemma~\ref{lem:cone-dir}, \eqref{eq:nontrapping-L}, Lemma~\ref{lem:nontrapping-cor}, Proposition~\ref{prop:d-bichar}.(1) and the bound $\nrm{\nb \bgB}_{L^{\infty}} + \nrm{\nb^{2} \bgB}_{L^{\infty}} \aleq M$, we have
\begin{equation*}
	\abs*{\rd_{x} \td{\psi}(x, \xi)} + \abs*{\abs{\xi} \rd_{\xi} \td{\psi}(x, \xi)} \aleq_{M, \mu, A} e^{C_{M, \mu, A} L} \quad \hbox{ for } - 32 R < x^{3} < 32 R,
\end{equation*}
where the implicit constant is \emph{independent of $k_{(0)}$}.
For higher derivatives, we also use Proposition~\ref{prop:d-bichar}.(2) and we obtain
\begin{align*}
	\abs{\xi}^{\abs{\bfbt}} \abs*{\rd_{x}^{\bfalp} \rd_{\xi}^{\bfbt} \td{\psi}}
	& \aleq_{M, \mu, A, \bfalp, \bfbt} e^{C_{M, \mu, A, \bfalp, \bfbt} L} 2^{(\abs{\alp} -1)_{+} k_{(0)}} \quad \hbox{ for } - 32 R < x^{3} < 32 R.
\end{align*}
The drawback of $\td{\psi}$ is that the symbol bounds are only good in $\set{-32 R < x^{3} < 32 R, \, \abs{\xi} > 1}$. To finally produce $\psi_{\pm}$, we localize $\td{\psi}$ and add a $\pm$-dependent correction in the following way:
\begin{equation*}
	\psi_{\pm}(x, \xi) := \chi_{>1}(\xi) \left( \chi_{< 16 R} (x^{3}) \td{\psi}(x, \xi) \pm \sgm C_{0} A q(x^{3}) \right),
\end{equation*}
where $C_{0}$ is a positive \emph{universal} constant that will be chosen below, and
\begin{align*}
	q(z) &= \int_{0}^{z} \left( R^{-1} \tilde{\chi}_{(8R, 16R)}(- z') + R^{-1} \tilde{\chi}_{(8R, 16R)}(z') \right) \, \ud z', \\
	\tilde{\chi}_{(8R, 16R)}(z) &= \tilde{\chi}_{(1, 2)}((8 R)^{-1} z).
\end{align*}
and $\tilde{\chi}_{(1, 2)} \in C^{\infty}_{c}(\bbR)$ satisfies $\tilde{\chi}_{(1, 2)}(z) = 1$ for $z \in (1, 2)$ and $\supp \tilde{\chi}_{(1, 2)} \subseteq (\frac{1}{2}, 4)$. Observe that $q(z) = 0$ for $z \in (-4R, 4R)$, so \eqref{eq:renrm-offdiag-bulk} holds.

From the construction (and since $R \geq 1$), it follows that
\begin{align*}
	\abs{\psi_{\pm}} &\aleq (1+C_{0}) \sgm A, \\
	\abs*{\rd_{x} \psi_{\pm}(x, \xi)} + \abs*{\brk{\xi} \rd_{\xi} \psi_{\pm}(x, \xi)} &\aleq_{M, \mu, A} e^{C_{M, \mu, A} L}, \\
	\brk{\xi}^{\abs{\bfbt}} \abs*{\rd_{x}^{\bfalp} \rd_{\xi}^{\bfbt} \psi_{\pm}}
	& \aleq_{M, \mu, A, \bfalp, \bfbt} e^{C_{M, \mu, A} L} 2^{(\abs{\bfalp} -1)_{+} k_{(0)}},
\end{align*}
which implies the desired symbol bounds \eqref{eq:renrm-symb-0}--\eqref{eq:renrm-symb-high} for $O_{\pm} = e^{\psi_{\pm}}$. To complete the construction, it remains to prove \eqref{eq:renrm-core} and \eqref{eq:renrm-offdiag}, and to fix the positive universal constant $C_{0}$. 

Let us first fix the choice of $C_{0}$ and prove \eqref{eq:renrm-core}. Consider the following decomposition:
\begin{align*}
	\pm \left( \{ \dprin_{\bgB}, \psi_{\pm} \} - \frkc^{\sgm (1)}_{\bgB} \right)
	&= \mp \chi_{>1}(\xi) (1-\chi_{< 16R}) \frkc^{\sgm (1)}_{\bgB} \\
	&\peq +  \chi_{>1}(\xi) \{\dprin_{\bgB}, x^{3} \} \left( \sgm C_{0} A q'(x^{3}) \pm \chi'_{< 16R}(x^{3}) \td{\psi}(x, \xi) \right) \\
	&\peq + \chi_{>1}'(\xi) \{ \dprin_{\bgB}, \abs{\xi} \} \left( \chi_{< 16R} (x^{3}) \td{\psi}(x, \xi) \pm \sgm C_{0} A q(x^{3}) \right) \\
	&=: \mp (1-\chi_{< 16R}) \frkc^{\sgm (1)}_{\bgB} + \frkq_{\pm}^{(1)} + \frke_{\pm}.
\end{align*}
Writing $\chi'_{< 16R}(x^{3}) = (16R)^{-1} \chi'_{<1}(\frac{x^{3}}{16R})$, it is evident that $\chi'_{< 16R}(x^{3})$ is supported in $(-16R, -8R) \cup (8R, 16R)$ and is of size $O(R^{-1})$. Therefore, by fixing $C_{0}$ to be a sufficiently large \emph{universal} constant, we may ensure that
\begin{align*}
\frkq_{\pm}^{(1)} = \chi_{>1}(\xi) \{\dprin_{\bgB}, x^{3} \} \left( \sgm C_{0} A q'(x^{3}) \pm (16R)^{-1} \chi'_{< 1}(\tfrac{x^{3}}{16R}) \td{\psi}(x, \xi) \right) > 0.
\end{align*}
Moreover, by the symbol bounds for $\td{\psi}$, it follows that $\frkq_{\pm}^{(1)} \in S^{1}$ with $[\frkq_{\pm}^{(1)}]_{S^{1}; N} \aleq_{M, \mu, A, N} e^{C_{M, \mu, A} L} 2^{(N-1)_{+} k_{(0)}}$.
Next, in view of its $\xi$-support property, we see that $\frke_{\pm} \in S^{0}$ (in fact, $\frke_{\pm} \in S^{-\infty}$) with $[\frke_{\pm}]_{S^{0}; N} \aleq_{M, \mu, A, N} e^{C_{M, \mu, A} L} 2^{(N-1)_{+} k_{(0)}}$. By Propositions~\ref{prop:psdo-L2-bdd-garding}, \ref{prop:psdo-comp} and \ref{prop:psdo-adj}, it follows that
\begin{align*}
&\pm \int_{0}^{\tau} \brk*{\left( [O_{\pm}, \diag_{\bgB}^{(2)}] + O_{\pm} \comm^{\sgm (1)}_{(\bfB_{0})_{<k_{(0)}}} \right) \tb_{\pm}, O_{\pm} \tb_{\pm}} \, \ud t \\
&= \pm \int_{0}^{\tau} \brk*{\Op { e^{\psi_{\pm}} } \left( \{ \dprin_{\bgB}, \psi_{\pm} \} - \frkc^{\sgm (1)}_{\bgB}  \right) \tb_{\pm}, O_{\pm} \tb_{\pm}} \, \ud t 
+ O (C_{M, \mu, A} e^{C_{M, \mu, A} L} e^{C k_{(0)}} T\nrm{\tilde{b}}_{L^{\infty} L^{2}}^{2})\\
&= \mp \int_{0}^{\tau} \brk*{\Op \left( { e^{\psi_{\pm}} } (1 - \chi_{< 16R}(x^{3})) \frkc^{\sgm (1)}_{\bgB} \right) \tb_{\pm}, O_{\pm} \tb_{\pm}} \, \ud t
+ \int_{0}^{\tau} \brk*{\Op (\frkq_{\pm}^{(1)}) \tilde{b}_{\pm}, O_{\pm} \tilde{b}_{\pm}} \, \ud t\\
&\peq + O (C_{M, \mu, A} e^{C_{M, \mu, A} L} e^{C k_{(0)}} T\nrm{\tilde{b}}_{L^{\infty} L^{2}}^{2}) \\
&= \mp \int_{0}^{\tau} \brk*{\Op \left( { e^{2\psi_{\pm}} } (1 - \chi_{< 16R}(x^{3})) \frkc^{\sgm (1)}_{\bgB} \right) \tb_{\pm}, \tb_{\pm}} \, \ud t
+ \int_{0}^{\tau} \brk*{\Op ( { e^{2\psi_{\pm}} } \frkq_{\pm}^{(1)}) \tilde{b}_{\pm}, \tilde{b}_{\pm}} \, \ud t\\
&\peq + O (C_{M, \mu, A} e^{C_{M, \mu, A} L} e^{C k_{(0)}} T\nrm{\tilde{b}}_{L^{\infty} L^{2}}^{2})
\end{align*}
By the positivity property, the symbol bounds and Proposition~\ref{prop:psdo-L2-bdd-garding}, we have
\begin{align*}
\int_{0}^{\tau} \brk*{\Op ( { e^{2\psi_{\pm}} } \frkq_{\pm}^{(1)}) \tilde{b}_{\pm}, \tilde{b}_{\pm}} \, \ud t
\geq  - C_{M, \mu, A} e^{C_{M, \mu, A} L} e^{C k_{(0)}} T\nrm{\tilde{b}}_{L^{\infty} L^{2}}^{2}.
\end{align*}
To deal with the first term, we first peel off ${ e^{2\psi_{\pm}} }$ using the high frequency Calder\`on--Vaillancourt trick. More precisely, we estimate
\begin{align*}
& \abs*{\mp \int_{0}^{\tau} \brk*{\Op \left( { e^{2\psi_{\pm}} } (1 - \chi_{< 16R}(x^{3})) \frkc^{\sgm (1)}_{\bgB} \right) \tb_{\pm}, \tb_{\pm}} \, \ud t} \\
&\leq \nrm*{\Op \left( { e^{2\psi_{\pm}} }\right) {P_{\geq k_{(1)}}} \Op\left((1 - \chi_{< 16R}(x^{3})) \frkc^{\sgm (1)}_{\bgB} \right) \tb_{\pm}}_{Y^{0}} \nrm{\tb_{\pm}}_{X^{0}} \\
&\peq { + T\nrm*{\Op \left( e^{2\psi_{\pm}} \right)  P_{< k_{(1)}} \Op\left((1 - \chi_{< 16R}(x^{3})) \frkc^{\sgm (1)}_{\bgB} \right) \tb_{\pm}}_{L^{\infty} L^{2}} \nrm{\tb_{\pm}}_{L^{\infty} L^{2}} }\\
&\peq 
+ C_{M, \mu, A} e^{C k_{(0)}} e^{C_{M, \mu, A} L} T \nrm{\tb_{\pm}}_{L^{\infty} L^{2}}^{2} \\
&\leq e^{C \sgm A} \nrm*{\Op\left((1 - \chi_{< 16R}(x^{3})) \frkc^{\sgm (1)}_{\bgB} \right) \tb_{\pm}}_{Y^{0}} \nrm{\tb_{\pm}}_{X^{0}} 
+ C_{M, \mu, A} e^{C(k_{(0)} + k_{(1)})} e^{C_{M, \mu, A} L} T \nrm{\tb_{\pm}}_{L^{\infty} L^{2}}^{2}.
\end{align*}
In the first inequality, we used Proposition~\ref{prop:psdo-comp}. In the last inequality, we used Proposition~\ref{prop:psdo-ests}.(2) (high frequency Calder\`on--Vaillancourt trick) and the symbol bound \eqref{eq:renrm-symb-0} for the contribution of $P_{\geq k_{(1)}}$, and that $\nrm{2^{-k_{(1)}} P_{<k_{(1)}} \brk{D}}_{L^{2} \to L^{2}} \aleq 1$, $\brk{D}^{-1} \in S^{-1}$ and $(1 - \chi_{< 16R}(x^{3})) \frkc^{\sgm (1)}_{\bgB} \in S^{1}$ for the contribution of $P_{< k_{(1)}}$. Next, observe that
\begin{align*}
\Op\left((1 - \chi_{< 16R}(x^{3})) \frkc^{\sgm (1)}_{\bgB} \right) \tb_{\pm}
&= \sgm (1 - \chi_{< 16R}(x^{3})) (\rd_{\alp} \bgB^{\bt})(x) \frac{\rd_{\alp} \rd_{\bt}}{\brk{D}^{2}} \abs{D} \tb_{\pm} \\
&= \sgm (1 - \chi_{< 16R}(x^{3})) \left( \rd_{\alp} \left( (1 - \chi_{< 8R}(x^{3})) \bgB^{\bt}(x) \right)\right) \frac{\rd_{\alp} \rd_{\bt}}{\brk{D}^{2}} \abs{D} \tb_{\pm} \\
\end{align*}
and that $\nrm{(1-\chi_{< 8R}(x^{3})) \bgB}_{\ell^{1}_{\calI} X^{s_{0}}} \aleq \eps$ by Proposition~\ref{prop:local-small} (here, we have taken $k_{(0)}$ to be sufficiently small depending on $\eps$). By \eqref{eq:prod-core-hl}, it follows that
\begin{align*}
\nrm*{\Op\left((1 - \chi_{< 16R}(x^{3})) \frkc^{\sgm (1)} \right) \tb_{\pm}}_{Y^{0}}
\aleq \eps \nrm{\tb_{\pm}}_{X^{0}},
\end{align*}
which implies \eqref{eq:renrm-core}.

Finally, we prove \eqref{eq:renrm-offdiag}. Using Propositions~\ref{prop:psdo-L2-bdd-garding}, \ref{prop:psdo-comp} and \ref{prop:psdo-adj}, we have
\begin{align*}
&\int_{0}^{\tau} \left( \brk{\symm^{(1)}_{\bgB} O_{+} \tb_{-}, O_{+} \tb_{+}}
	- \brk{\symm^{(1)}_{\bgB} O_{-} \tb_{+}, O_{-} \tb_{-}}\right) \, \ud t \\
&= \int_{0}^{\tau} \brk*{\tb_{-}, \left(O_{+}^{\ast} \symm^{(1)}_{\bgB} O_{+} - O_{-}^{\ast} \symm^{(1)}_{\bgB} O_{-} \right)\tb_{+}} \, \ud t \\
&= \int_{0}^{\tau} \brk*{\tb_{-}, \Op \left((e^{2 \psi_{+}} - e^{2 \psi_{-}}) \frks^{(1)}_{\bgB}(x, \xi)\right) \tb_{+}} \, \ud t + O (C_{M, \mu, A} e^{C_{M, \mu, A} L} e^{C k_{(0)}} T\nrm{\tilde{b}}_{L^{\infty} L^{2}}^{2}) ,
\end{align*}
where $\frks^{(1)}_{\bgB}(x, \xi)$ is the principal symbol of $\symm^{(1)}_{\bgB}$ given by
\begin{equation*}
\tensor{(\frks^{(1)}_{\bgB})}{^{\alp}_{\bt}} = \frac{1}{2} \abs{\xi} \left((\rd_{\alp} \bgB^{\bt}) + (\rd_{\bt} \bgB^{\alp})  \right) 
+ \frac{1}{2} \{\abs{\xi}, \bgB \cdot \xi\} \dlt^{\alp}_{\bt}.
\end{equation*}
Observe that, by \eqref{eq:renrm-offdiag-bulk},
\begin{equation*}
(e^{2 \psi_{-}} - e^{2 \psi_{+}}) \frks^{(1)}_{\bgB}(x, \xi) = (e^{2 \psi_{-}} - e^{2 \psi_{+}}) \frks^{(1)}_{(1-\chi_{< 2 R}(x^{3}))\bgB}(x, \xi).
\end{equation*}
As in the proof of \eqref{eq:renrm-core}, we now peel off $(e^{2 \psi_{-}} - e^{2 \psi_{+}})(x, \xi)$ using Proposition~\ref{prop:psdo-comp}, \eqref{eq:renrm-symb-0} and the high frequency Calder\`on--Vaillancourt trick (i.e., Proposition~\ref{prop:psdo-ests}.(2)):
\begin{align*}
\abs*{\int_{0}^{\tau} \brk*{\tb_{-}, \Op \left((e^{2 \psi_{+}} - e^{2 \psi_{-}}) \frks^{(1)}_{\bgB}\right) \tb_{+}} \, \ud t}
&\leq e^{C \sgm A} \nrm*{\Op \left( \frks^{(1)}_{(1-\chi_{< 2 R}(x^{3}))\bgB} \right) \tb_{+}}_{Y^{0}} \nrm{\tb_{-}}_{X^{0}} \\
&\peq +C_{M, \mu, A} e^{C(k_{(0)} + k_{(1)})} e^{C_{M, \mu, A} L} T \nrm{\tb}_{L^{\infty} L^{2}}^{2} 
\end{align*}
Again as before, note that $\nrm{(1-\chi_{< 2R}(x^{3})) \bgB}_{\ell^{1}_{\calI} X^{s_{0}}} \aleq \eps$ by Proposition~\ref{prop:local-small}. By \eqref{eq:prod-core-hl}, it follows that
\begin{align*}
\nrm*{\Op \left( \frks^{(1)}_{(1-\chi_{< 2 R}(x^{3}))\bgB} \right) \tb_{+}}_{Y^{0}} \aleq \eps \nrm{\tb_{+}}_{X^{0}},
\end{align*}
which completes the proof of \eqref{eq:renrm-offdiag}.

\pfstep{Step~3: Bounds for $\tilde{E}_{O; 1}$ and $\tilde{E}_{O; 2}$}
For $\tilde{E}_{O; 1}$ and $\tilde{E}_{O; 2}$, we shall prove
\begin{align} 
	\nrm{\tilde{E}_{O; 1}}_{Y^{0}} &\leq C_{M, \mu, A} e^{C_{M, \mu, A} L} e^{C \sgm A} \left( 2^{-(s_{0} - \frac{7}{2}) k_{(0)}} + e^{C(k_{(0)} + k_{(1)})} T \right) \nrm{\tb}_{X^{0}}, \label{eq:paralin-bdd-E3} \\
	\nrm{\tilde{E}_{O; 2}}_{Y^{0}} & \leq C_{M, \mu, A} e^{C k_{(0)}} e^{C_{M, \mu, A} L} T \nrm{\tb}_{X^{0}} \label{eq:paralin-bdd-E4}.
\end{align}

Estimate \eqref{eq:paralin-bdd-E3} is proved in exactly the same manner as \eqref{eq:led-err}, using the symbol bounds \eqref{eq:renrm-symb-0}--\eqref{eq:renrm-symb-high} (as before, it is crucial that the bound \eqref{eq:renrm-symb-1} for $\rd_{x} O_{\pm}$ and $\brk{\xi} \rd_{\xi} O_{\pm}$ do not involve $k_{(0)}$). 

To prove \eqref{eq:paralin-bdd-E4}, observe that all the operators in $\tilde{E}_{O; 2}$ has order $0$ by the standard pseudodifferential calculus (i.e., Proposition~\ref{prop:psdo-comp}.(1)). In view of the symbol bounds \eqref{eq:renrm-symb-0}--\eqref{eq:renrm-symb-high}, as well as those for $\symm^{(1)}$, $\asymm^{(1)}$ and $\covec^{(0)}$, which involve $M$ and $k_{(0)}$, we have
\begin{align*}
	\nrm{\tilde{E}_{O; 2}}_{L^{1} L^{2}} \aleq_{M, \mu, A} e^{C k_{(0)}} e^{C_{M, \mu, A} L } T \sum_{\pm} \nrm{\tb}_{L^{\infty} L^{2}},
\end{align*}
which implies \eqref{eq:paralin-bdd-E4}.

\pfstep{Step~4: Completion of the proof}
We multiply \eqref{eq:para-diag-Otb} on the left by $(O_{+} \tb_{+}, O_{-}\tb_{-})$ and integrating on $(0, \tau) \times \bbR^{3}$, where $0 < \tau \leq T$. Using \eqref{eq:renrm-core}, \eqref{eq:renrm-offdiag}, \eqref{eq:paralin-bdd-E3} and \eqref{eq:paralin-bdd-E4}, as well as the antisymmetry of the operators $\diag_{\bfB}^{(2) \sharp}$, $\asymm^{(1)}_{\bgB}$ and $(\nb \times \bgB)\cdot \nb$, we obtain
\begin{align*}
	\sum_{\pm} \nrm{O_{\pm} \tb_{\pm}}_{L^{\infty} L^{2}}^{2}
	&\aleq \sum_{\pm} \nrm{O_{\pm} \tb_{\pm}(0)}_{L^{2}}^{2}
	+ \sum_{\pm} \nrm{O_{\pm}^{\ast} O_{\pm} \tilde{h}_{\pm}}_{Y^{0}} \nrm{\tb_{\pm}}_{X^{0}} \\
	&\peq + T \sum_{\pm} \nrm{\nb \cdot (O_{\pm} \tb_{\pm})}_{L^{\infty} L^{2}}\nrm{\covec_{\bgB}^{(0)} O_{\pm} \tb_{\pm}}_{L^{\infty} L^{2}} \\
	&\peq + C e^{\sgm C A} \eps \nrm{\tb}_{X^{0}}^{2}
	+ C_{M, \mu, A} e^{C_{M, \mu, A} L} 2^{-(s_{0} - \frac{7}{2}) k_{(0)}} \nrm{\tb}^{2}_{X^{0}} \\
	&\peq + C_{M, \mu, A} e^{C(k_{(0)} + k_{(1)})} e^{C_{M, \mu, A} L} T \nrm{\tb}^{2}_{X^{0}}.
	\end{align*}
Using \eqref{eq:renrm-symb-0}--\eqref{eq:renrm-symb-high} and Propositions~\ref{prop:psdo-L2-bdd-garding} and \ref{prop:psdo-ests}, we may peel off $O_{\pm}$ and $O_{\pm}^{\ast}$. Moreover, $\nrm{\covec_{\bgB}^{(0)}}_{L^{2} \to L^{2}} \aleq M$ and, by $\nb \cdot b_{\pm} = 0$ and Proposition~\ref{prop:psdo-comp} (as well as \eqref{eq:renrm-symb-0}--\eqref{eq:renrm-symb-high} and Proposition~\ref{prop:psdo-L2-bdd-garding}),
\begin{equation*}
	\nrm{\nb \cdot O_{\pm} b_{\pm}}_{L^{\infty} L^{2}} = \nrm{[\nb \cdot, O_{\pm}] b_{\pm}}_{L^{\infty} L^{2}} \aleq_{M, \mu, A} e^{C k_{(0)}} e^{C_{M, \mu, A} L} T \nrm{b_{\pm}}_{L^{\infty} L^{2}}.
\end{equation*}
Putting all these together, we finally arrive at:
\begin{equation} \label{eq:paralin-en-key}
\begin{aligned}
	\sum_{\pm} \nrm{O_{\pm} \tb_{\pm}}_{L^{\infty} L^{2}}^{2}
	&\aleq C_{M, \mu, A} e^{C k_{(0)}} e^{C_{M, \mu, A} L} \left( \nrm{\tb(0)}_{L^{2}}^{2}
	+ \nrm{\tilde{h}}_{Y^{0}} \nrm{\tb}_{X^{0}} \right) \\
	&\peq + C_{M, \mu, A} e^{C_{M, \mu, A} L} 2^{-(s_{0} - \frac{7}{2}) k_{(0)}} \nrm{\tb}^{2}_{X^{0}} 
	+ C_{M, \mu, A} e^{C(k_{(0)} + k_{(1)})} e^{C_{M, \mu, A} L} T \nrm{\tb}^{2}_{X^{0}}.
\end{aligned}
\end{equation}

To return to $\tb_{\pm}$, we note as in Step~4 in the proof of Proposition~\ref{prop:paralin-led} that (recall that $O_{\pm} = e^{\psi_{\pm}}$)
\begin{align*}
	\nrm{(1 - e^{-\psi_{\pm}} (x, D) e^{\psi_{\pm}} (x, D)) P_{\geq k_{(1)}}}_{L^{2} \to L^{2}} \aleq_{M, \mu, A} e^{C k_{(0)}} e^{C_{M, \mu, A} L} 2^{-k_{(1)}},
\end{align*}
which follows from Proposition~\ref{prop:psdo-comp}.(1), the symbol bounds \eqref{eq:renrm-symb-0}--\eqref{eq:renrm-symb-high} and Proposition~\ref{prop:psdo-L2-bdd-garding}. Similarly, but only using the fact that $1 - e^{-\psi_{\pm}}(x, D) e^{\psi_{\pm}}(x, D)$ is bounded on $L^{2}$, we have
\begin{align*}
	\nrm{(1 - e^{-\psi_{\pm}} (x, D) e^{\psi_{\pm}} (x, D)) P_{< k_{(1)}} \brk{D}^{\sgm}}_{L^{2} \to L^{2}} \aleq_{M, \mu, A} e^{C k_{(0)}} e^{C_{M, \mu, A} L} 2^{\sgm k_{(1)}}.
\end{align*}
Hence,
\begin{align*}
	\nrm{\tb_{\pm}}_{L^{\infty} L^{2}}
	&\leq \nrm{e^{-\psi_{\pm}}(x, D) O_{\pm}(x, D) \tb_{\pm}}_{L^{\infty} L^{2}} 
	+  C_{M, \mu, A} e^{C k_{(0)}} e^{C_{M, \mu, A} L} 2^{-k_{(1)}} \nrm{\tb_{\pm}}_{L^{\infty} L^{2}} \\
&\peq + C_{M, \mu, A} e^{C k_{(0)}} e^{C_{M, \mu, A} L} 2^{\sgm k_{(1)}} \nrm{\brk{D}^{-\sgm} \tb_{\pm}}_{L^{\infty} L^{2}} \\
	&\leq \nrm{e^{-\psi_{\pm}}(x, D) P_{\geq k_{(1)}} O_{\pm}(x, D) \tb_{\pm}}_{L^{\infty} L^{2}} 
	+ \nrm{e^{-\psi_{\pm}}(x, D) P_{< k_{(1)}} O_{\pm}(x, D) \tb_{\pm}}_{L^{\infty} L^{2}} \\
&\peq +  C_{M, \mu, A} e^{C k_{(0)}} e^{C_{M, \mu, A} L} 2^{-k_{(1)}} \nrm{\tb_{\pm}}_{L^{\infty} L^{2}} + C_{M, \mu, A} e^{C k_{(0)}} e^{C_{M, \mu, A} L} 2^{\sgm k_{(1)}} \nrm{\brk{D}^{-\sgm} \tb_{\pm}}_{L^{\infty} L^{2}} \\
	&\leq e^{C \sgm A} \nrm{O_{\pm} \tb_{\pm}}_{L^{\infty} L^{2}} 
	+  C_{M, \mu, A} e^{C k_{(0)}} e^{C_{M, \mu, A} L} 2^{-k_{(1)}} \nrm{\tb_{\pm}}_{L^{\infty} L^{2}} \\
	&\peq + C_{M, \mu, A} e^{C k_{(0)}} e^{C_{M, \mu, A} L} 2^{\sgm k_{(1)}} \nrm{\brk{D}^{-\sgm} \tb_{\pm}}_{L^{\infty} L^{2}},
\end{align*}
where we used Lemma~\ref{lem:hf-L2} (high frequency Calder\`on--Vaillancourt) to peel off $e^{-\psi_{\pm}}(x, D) P_{\geq k_{(1)}}$, and that $\nrm{2^{- \sgm k_{(1)}} \brk{D}^{\sgm} P_{< k_{(1)}}}_{L^{2} \to L^{2}} \aleq 1$, as well as Propositions~\ref{prop:psdo-L2-bdd-garding} and \ref{prop:psdo-comp}, to estimate $e^{-\psi_{\pm}}(x, D) P_{< k_{(1)}} O_{\pm}(x, D) \tb_{\pm}$ in terms of $\brk{D}^{-\sgm} \tb_{\pm}$. Using \eqref{eq:paralin-en-key} to bound $\nrm{O_{\pm} \tb_{\pm}}_{L^{\infty} L^{2}}$, the proof is complete. \qedhere
\end{proof}

\subsection{Proof of Proposition~\ref{prop:paralin}} \label{subsec:paralin-pf}
In this subsection, we complete the proof of Proposition~\ref{prop:paralin}, simply by combining Propositions~\ref{prop:paralin-led}, \ref{prop:paralin-bdd-0} and \ref{prop:paralin-bdd}. We need to prove the estimate \eqref{eq:paralin-core}, under a choice of $\eps_{(1)}$ and $T$ satisfying \eqref{eq:paralin-core-para}.

To begin with, applying the $L^\infty L^2$ estimate on $b$ \eqref{eq:paralin-bdd-0} to the  {RHS} of \eqref{eq:paralin-led} gives 
\begin{equation}
	\begin{aligned}
		\nrm{\tb}_{X^{0}[0, T]}^{2} 
		& \leq C_{M, \mu, A} e^{C_{M, \mu, A} L_{0}} \nrm{\tb}_{L^{\infty} L^{2}[0, T]}^{2} \\
		& \peq + C_{M, \mu, A} e^{C(k_{(0)} + k_{(1)} ) } e^{C_{M, \mu, A} L_{0}} \left( \nrm{b(0)}_{L^{2}}^{2} + \nrm{h}_{Y^{0}[0, T]} \nrm{b}_{X^{0}[0, T]} +  \nrm{\tilde{h}}_{Y^{0}[0, T]}^{2} \right) \\
		& \peq + C_{M, \mu, A} e^{C_{M, \mu, A} L_{0}} \left( 2^{-(s_{0} - \frac{7}{2}) k_{(0)}} + e^{C(k_{(0)} + k_{(1)})} T \right) \nrm{\tb}_{X^{0}[0, T]}^{2}.
	\end{aligned}
\end{equation} Then, note that first by taking $k_{(0)} = k_{(0)}(M, \mu, A, L)$ sufficiently large and $T = c_{(1)} e^{-C_{(1)}L} $ sufficiently small so that $c_{(1)} = c_{(1)}(\sgm,s_{1},M,\mu,A,k_{(0)}, k_{(1)})$ and $C_{(1)}= C_{(1)}(\sgm,s_{1},M,\mu,A)$, we can absorb the last term on the RHS to the LHS. Moreover, we have $\nrm{\tb}_{X^{0}[0, T]} \ge \nrm{b}_{X^{0}[0, T]}$ and $\nrm{\tilde{h}}_{Y^{0}[0, T]} \ge \nrm{h}_{Y^{0}[0, T]}$ simply because $\sgm\ge0$, and therefore we can absorb the $\nrm{h}_{Y^{0}[0, T]} \nrm{b}_{X^{0}[0, T]}$ term using $\eps$-Young inequality as well. Next, recall from the definition of $\tilde{h}$ \eqref{eq:tilde-h-def} that we have \begin{equation*}
	\begin{split}
		\nrm{\tilde{h}}_{Y^{0}[0,T]} \lesssim \nrm{\tilde{g}}_{Y^{0}[0,T]} +  (2^{-(s_{0} - \frac{7}{2}) k_{(0)}} + T 2^{2 k_{(0)}}) M   \nrm{\tb}_{X^{0}},
	\end{split}
\end{equation*} which is \eqref{eq:para-diag-tb-err}. Notice that by taking $k_{(0)}$ larger and then $T$ smaller if necessary, we can absorb  {the contribution of $\nrm{\tb}_{X^{0}[0,T]}$} to the LHS as well; this gives \begin{equation}\label{eq:core-combined1} 
\begin{aligned}
	\nrm{\tb}_{X^{0}[0, T]}^{2} 
	& \leq C_{M, \mu, A} e^{C_{M, \mu, A} L_{0}} \nrm{\tb}_{L^{\infty} L^{2}[0, T]}^{2}  + C_{M, \mu, A} e^{C(k_{(0)} + k_{(1)} ) } e^{C_{M, \mu, A} L_{0}} \left( \nrm{b(0)}_{L^{2}}^{2} +  \nrm{\tilde{g}}_{Y^{0}[0, T]}^{2} \right) . 
\end{aligned}
\end{equation} 
We now apply \eqref{eq:paralin-bdd-0} to the RHS of \eqref{eq:paralin-bdd}, which gives 
\begin{equation*}
	\begin{aligned}
		\nrm{\tb}_{L^{\infty} L^{2}[0, T]}^{2} &\le  C_{M, \mu, A} e^{C k_{(0)}} e^{C_{M, \mu, A} L} \left( \nrm{\tb(0)}_{L^{2}}^{2}
		+ \nrm{\tilde{h}}_{Y^{0}[0, T]} \nrm{\tb}_{X^{0}[0, T]}\right) \\
		&\peq + C e^{\sgm C A} \eps \nrm{\tb}_{X^{0}}^{2}
		+ C_{M, \mu, A} e^{C_{M, \mu, A} L} 2^{-(s_{0} - \frac{7}{2}) k_{(0)}}  \nrm{\tb}^{2}_{X^{0}[0, T]} \\
		&\peq + C_{M, \mu, A} e^{C(k_{(0)} + k_{(1)})} e^{C_{M, \mu, A} L} T \nrm{\tb}^{2}_{X^{0}[0, T]} \\
		&\peq + C_{M, \mu, A} e^{C k_{(0)}} e^{C_{M, \mu, A} L} 2^{\sgm k_{(1)}} \left(\nrm{b(0)}_{L^{2}}^{2} + \nrm{h}_{Y^{0}[0, T]} \nrm{b}_{X^{0}[0, T]} \right).
\end{aligned}\end{equation*} We now apply the previous estimate to the RHS of \eqref{eq:core-combined1}, and take $\eps_{(1)} = \eps_{(1)}(M, \mu, A)$ sufficiently small to absorb $\eps \nrm{\tb}_{X^{0}}^{2}$-term to the LHS. The other terms involving $\nrm{\tb}_{X^{0}[0,T]}$ can be handled similarly as in the above. This finishes the proof. \hfill \qedsymbol

\subsection{Proof of Proposition~\ref{prop:paralin-err}} \label{subsec:paralin-err}
In this subsection, we will establish Proposition~\ref{prop:paralin-err}. In what follows, we will omit writing the time interval $[0, T]$.
\begin{proof}[Proof of \eqref{eq:paralin-err}]
We begin by proving the identity
\begin{equation} \label{eq:paralin-err-decomp}
	G(b) 
	= \nb \times \sum_{k} P_{[k-10, k+10]} b \times (\nb \times P_{k} b).
\end{equation}
By definition, we have
\begin{align*}
G(b) = \nb \times \left[ b \times (\nb \times b) - T_{b} \times (\nb \times b) + T_{\nb \times b} \times b \right].
\end{align*}
Starting with the first two terms in the brackets, we compute
\begin{align*}
	b \times (\nb \times b) - T_{b} \times (\nb \times b) 
	&= \sum_{k} P_{> k+10} b \times (\nb \times P_{k} b)  + \sum_{k} P_{[k-10, k+10]} b \times (\nb \times P_{k} b) \\
&= \sum_{j} P_{j} b \times (\nb \times P_{<j-10} b) 
	+ \sum_{k} P_{[k-10, k+10]} b \times (\nb \times P_{k} b) ,
\end{align*}
and observe that the first term on the last line is precisely $T_{\nb \times b} \times b$.

We use \eqref{eq:prod-core-hh} to estimate the $\nb \times (\hbox{RHS of \eqref{eq:paralin-err-decomp}})$ as follows:
\begin{align*}
& \nrm{\nb \times \sum_{k} P_{[k-10, k+10]} b \times (\nb \times P_{k} b)}_{\ell^{1}_{\calI} Y^{\sgm}}  \\
&\aleq \sum_{\ell} \sum_{k : k > \ell - 20} T 2^{(\frac{5}{2}+\sgm)\ell} \nrm{P_{[k-10, k+10]} b}_{\ell^{1}_{\calI} X_{k}} \nrm{P_{k} b}_{\ell^{1}_{\calI} X_{k}} \\
&\aleq \nrm{b}_{\ell^{1}_{\calI} X^{s_{0}}} \nrm{b}_{\ell^{1}_{\calI} X^{\sgm}} \sum_{\ell} \sum_{k : k > \ell-20} T 2^{(\frac{5}{2}+\sgm)\ell} 2^{-(s_{0}+\sgm - 1) k}  \\
&\aleq T \nrm{b}_{\ell^{1}_{\calI} X^{s_{0}}} \nrm{b}_{\ell^{1}_{\calI} X^{\sgm}},
\end{align*}
which proves the desired estimate. \qedhere
\end{proof}

\begin{proof}[Proof of \eqref{eq:paralin-err-diff}]
By \eqref{eq:paralin-err-decomp}, we have
\begin{align*}
	G(b) - G(\br{b})
	= \nb \times \sum_{k} \left( P_{k - 5 \leq \cdot \leq k + 5} b \times (\nb \times P_{k} (b - \br{b})) +P_{k - 5 \leq \cdot \leq k + 5} (b - \br{b}) \times (\nb \times P_{k} \br{b}) \right).
\end{align*}
Then the desired estimate follows from \eqref{eq:prod-core-hh} as in the proof of \eqref{eq:paralin-err}.
\end{proof}

\begin{proof}[Proof of \eqref{eq:paralin-err-diff-low}]
We begin with $\nb \times (T_{b - \br{b}} \times (\nb \times b))$. Introducing a parameter $m$ to be determined later, we write
\begin{align*}
\nb \times (T_{b - \br{b}} \times (\nb \times b)) 
&= \nb \times \sum_{k} P_{<k-m} (b - \br{b}) \times (\nb \times P_{k} b)
+ \nb \times \sum_{k} P_{[k-m, k-10)} (b - \br{b}) \times (\nb \times P_{k} b).
\end{align*}
By our conventions (see Section~\ref{subsec:notation}), $P_{<k-m}(b-\br{b}) = 0$ and $P_{[k-m, k-10)} = P_{<k-10}$ unless $k > m$. For the first term, we use \eqref{eq:prod-core-hl} (with $(a, b) = (b, b-\br{b})$) and estimate
\begin{align*}
& \nrm{\nb \times \sum_{k} P_{<k-m} (b - \br{b}) \times (\nb \times P_{k} b)}_{\ell^{1}_{\calI} Y^{\sgm}} \\
&\aleq \bb(\sum_{k} \bb(2^{(\sgm+1) k} \sum_{j < k-m} 2^{\frac{5}{2} j} \nrm{P_{j} (b - \br{b})}_{\ell^{1}_{\calI} X_{j}} \nrm{P_{k} b}_{\ell^{1}_{\calI} X_{k}} \bb)^{2} \bb)^{\frac{1}{2}} \\
&\aleq \bb(\sum_{k} \bb(2^{(\sgm+1) k} \sum_{j < k-m} 2^{\frac{5}{2} j} \nrm{P_{j} (b - \br{b})}_{\ell^{1}_{\calI} X_{j}} \nrm{P_{k} b}_{\ell^{1}_{\calI} X_{k}} \bb)^{2} \bb)^{\frac{1}{2}} \\
&\aleq 2^{- (s_{0} - 1 - \sgm) m} \nrm{b - \br{b}}_{\ell^{1}_{\calI} X^{\sgm}} \bb(\sum_{k} \bb( 2^{s_{0} k} \nrm{P_{k} b}_{\ell^{1}_{\calI} X_{k}} \sum_{j < k-m} 2^{-(s_{0} - \frac{7}{2}) j}  \bb)^{2} \bb)^{\frac{1}{2}} \\
&\aleq 2^{- (s_{0} - 1 - \sgm) m} M \nrm{b - \br{b}}_{\ell^{1}_{\calI} X^{\sgm}}.
\end{align*}
For the second term, we use \eqref{eq:prod-core-hh} to estimate
\begin{align*}
& \nrm{\nb \times \sum_{k} P_{[k-m, k-10)} (b - \br{b}) \times (\nb \times P_{k} b)}_{\ell^{1}_{\calI} Y^{\sgm}} \\
& \aleq T  \bb( \sum_{k} \bb( 2^{(\sgm + \frac{7}{2}) k} \sum_{j \in [k- m, k-10)} \nrm{P_{j} (b - \br{b})}_{\ell^{1}_{\calI} X_{j}} \nrm{P_{k} b}_{\ell^{1}_{\calI} X_{k}} \bb)^{2} \bb)^{\frac{1}{2}} \\
& \aleq T 2^{\sgm m} \nrm{b - \br{b}}_{\ell^{1}_{\calI} X^{\sgm}} \bb( \sum_{k} \bb( 2^{-(s_{0} - \frac{7}{2}) k}  2^{s_{0} k} \nrm{P_{k} b}_{\ell^{1}_{\calI} X_{k}} \bb)^{2} \bb)^{\frac{1}{2}} \\
& \aleq T 2^{\sgm m} M_{0} \nrm{b - \br{b}}_{\ell^{1}_{\calI} X^{\sgm}}.
\end{align*}
Choosing $m$ so that $T = 2^{-(s_{0} - 1) m}$, the desired estimate follows for $\nb \times (T_{b - \br{b}} \times (\nb \times b))$. 

The remaining term $\nb \times (T_{\nb \times (b - \br{b})} \times b)$ is handled similarly (in fact, the numerology is more favorable compared to the previous case since $\nb \times$ now falls on $b - \br{b}$) and will be omitted. \qedhere
\end{proof}

\begin{proof}[Proof of \eqref{eq:paralin-err-fixed-t}]
The desired bounds for the paraproduct terms $\nb \times (T_{b} \times (\nb \times b))$ and $\nb \times (T_{\nb \times b} \times b)$ follow by writing out these expressions, then estimating the first (i.e., low frequency) factor of $b$ in $L^{\infty}$ (to which $\ell^{1}_{\calI} H^{s_{1}}$ embeds) and the second (i.e., high frequency) factor of $b$ in $\ell^{1}_{\calI} H^{s_{1}}$. For $G(b)$, we use \eqref{eq:paralin-err-decomp} and bound (using Littlewood--Paley theory, Corollary~\ref{cor:local-bernstein} and H\"older),
\begin{align*}
\nrm{\nb \times \sum_{k} P_{[k-10, k+10]} b \times (\nb \times P_{k} b)}_{\ell^{1}_{\calI} H^{s_{1} - 2}}
&\aleq \sum_{\ell} \sum_{k: k > \ell-20} 2^{(s_{1}-1) \ell} \nrm{P_{\ell}(P_{[k-10, k+10]} b \times (\nb \times P_{k} b))}_{\ell^{1}_{\calI} L^{2}} \\
&\aleq \sum_{\ell} \sum_{k: k > \ell-20} 2^{(s_{1}-1) \ell} 2^{\frac{3}{2} \ell} 2^{k} \nrm{P_{[k-10, k+10]} b}_{\ell^{1}_{\calI} L^{2}} \nrm{P_{k} b}_{\ell^{1}_{\calI} L^{2}} \\
&\aleq \nrm{b}_{\ell^{1}_{\calI} X^{s_{1}}}^{2} \sum_{\ell} \sum_{k: k > \ell-20} 2^{(s_{1}+\frac{1}{2}) \ell} 2^{-(2s_{1}-1) k} \aleq \nrm{b}_{\ell^{1}_{\calI} X^{s_{1}}}^{2},
\end{align*}
where we used $s_{1} > \frac{1}{2}$ in the last inequality.
\end{proof}

\bibliographystyle{amsplain}
\bibliography{hallmhd}

\providecommand{\bysame}{\leavevmode\hbox to3em{\hrulefill}\thinspace}
\providecommand{\MR}{\relax\ifhmode\unskip\space\fi MR }
\providecommand{\MRhref}[2]{%
  \href{http://www.ams.org/mathscinet-getitem?mr=#1}{#2}
}
\providecommand{\href}[2]{#2}
\begin{thebibliography}{10}

\bibitem{ADF}
Marion Acheritogaray, Pierre Degond, Amic Frouvelle, and Jian-Guo Liu,
  \emph{Kinetic formulation and global existence for the
  {H}all-{M}agneto-hydrodynamics system}, Kinet. Relat. Models \textbf{4}
  (2011), no.~4, 901--918. \MR{2861579}

\bibitem{Ak}
Timur Akhunov, \emph{A sharp condition for the well-posedness of the linear
  {K}d{V}-type equation}, Proc. Amer. Math. Soc. \textbf{142} (2014), no.~12,
  4207--4220. \MR{3266990}

\bibitem{ASWY}
David~M. Ambrose, Gideon Simpson, J.~Douglas Wright, and Dennis~G. Yang,
  \emph{Ill-posedness of degenerate dispersive equations}, Nonlinearity
  \textbf{25} (2012), no.~9, 2655--2680. \MR{2967120}

\bibitem{AmWr}
David~M. Ambrose and J.~Douglas Wright, \emph{Dispersion vs. anti-diffusion:
  well-posedness in variable coefficient and quasilinear equations of {K}d{V}
  type}, Indiana Univ. Math. J. \textbf{62} (2013), no.~4, 1237--1281.
  \MR{3179690}

\bibitem{BaeKang}
Hantaek Bae and Kyungkeun Kang, \emph{On the existence and temporal asymptotics
  of solutions for the two and half dimensional {H}all {MHD}}, J. Math. Fluid
  Mech. \textbf{25} (2023), no.~2, Paper No. 24, 30. \MR{4547410}

\bibitem{BSD}
D.~Biskamp, E.~Schwarz, and J.~F. Drake, \emph{{Two-fluid theory of
  collisionless magnetic reconnection}}, Physics of Plasmas \textbf{4} (1997),
  no.~4, 1002--1009.

\bibitem{CDL}
Dongho Chae, Pierre Degond, and Jian-Guo Liu, \emph{Well-posedness for
  {H}all-magnetohydrodynamics}, Ann. Inst. H. Poincar\'{e} Anal. Non
  Lin\'{e}aire \textbf{31} (2014), no.~3, 555--565. \MR{3208454}

\bibitem{CJO}
Dongho Chae, In-Jee Jeong, and Sung-Jin Oh, \emph{Illposedness via degenerate
  dispersion for generalized surface quasi-geostrophic equations with singular
  velocities}, 2023.

\bibitem{CL}
Dongho Chae and Jihoon Lee, \emph{On the blow-up criterion and small data
  global existence for the {H}all-magnetohydrodynamics}, J. Differential
  Equations \textbf{256} (2014), no.~11, 3835--3858. \MR{3186849}

\bibitem{CSch}
Dongho Chae and Maria Schonbek, \emph{On the temporal decay for the
  {H}all-magnetohydrodynamic equations}, J. Differential Equations \textbf{255}
  (2013), no.~11, 3971--3982. \MR{3097244}

\bibitem{CWe}
Dongho Chae and Shangkun Weng, \emph{Singularity formation for the
  incompressible {H}all-{MHD} equations without resistivity}, Ann. Inst. H.
  Poincar\'{e} Anal. Non Lin\'{e}aire \textbf{33} (2016), no.~4, 1009--1022.
  \MR{3519529}

\bibitem{CWo1}
Dongho Chae and J\"{o}rg Wolf, \emph{On partial regularity for the 3{D}
  nonstationary {H}all magnetohydrodynamics equations on the plane}, SIAM J.
  Math. Anal. \textbf{48} (2016), no.~1, 443--469. \MR{3455137}

\bibitem{CWo2}
\bysame, \emph{Regularity of the {$3D$} stationary hall magnetohydrodynamic
  equations on the plane}, Comm. Math. Phys. \textbf{354} (2017), no.~1,
  213--230. \MR{3656516}

\bibitem{Chi}
Hiroyuki Chihara, \emph{Local existence for semilinear {S}chr\"{o}dinger
  equations}, Math. Japon. \textbf{42} (1995), no.~1, 35--51. \MR{1344627}

\bibitem{ChLa}
Jungyeon Cho, Jungyeon Cho, and Alex Lazarian, \emph{Simulations of electron
  magnetohydrodynamic turbulence}, The Astrophysical Journal \textbf{701}
  (2009), 236 -- 252.

\bibitem{CrKaSt}
Walter Craig, Thomas Kappeler, and Walter Strauss, \emph{Microlocal dispersive
  smoothing for the {S}chr\"{o}dinger equation}, Comm. Pure Appl. Math.
  \textbf{48} (1995), no.~8, 769--860. \MR{1361016}

\bibitem{Dai20}
Mimi Dai, \emph{Local well-posedness of the {H}all-{MHD} system in {$H^s(\Bbb
  R^n)$} with {$s>\frac n2$}}, Math. Nachr. \textbf{293} (2020), no.~1, 67--78.
  \MR{4060362}

\bibitem{Dai21}
\bysame, \emph{Local well-posedness for the {H}all-{MHD} system in optimal
  {S}obolev spaces}, J. Differential Equations \textbf{289} (2021), 159--181.
  \MR{4248458}

\bibitem{Dai212}
\bysame, \emph{Nonunique weak solutions in {L}eray-{H}opf class for the
  three-dimensional {H}all-{MHD} system}, SIAM J. Math. Anal. \textbf{53}
  (2021), no.~5, 5979--6016. \MR{4328510}

\bibitem{Dai-IMRN}
\bysame, \emph{{Blowup of Dyadic MHD Models with Forward Energy Cascade}},
  International Mathematics Research Notices \textbf{2023} (2022), no.~24,
  21805--21837.

\bibitem{Dai3}
\bysame, \emph{Global existence of 2{D} electron {M}{H}{D} near a steady
  state}, preprint (2023).

\bibitem{DaiFr}
Mimi Dai and Susan Friedlander, \emph{{U}niqueness and {N}on-{U}niqueness
  {R}esults for {F}orced {D}yadic {M}{H}{D} {M}odels}, J Nonlinear Sci
  \textbf{33} (2023), no.~10.

\bibitem{DKL}
Mimi Dai, Jacob Krol, and Han Liu, \emph{On uniqueness and helicity
  conservation of weak solutions to the electron-{MHD} system}, J. Math. Fluid
  Mech. \textbf{24} (2022), no.~3, Paper No. 69, 17. \MR{4434207}

\bibitem{DaiLiu}
Mimi Dai and Han Liu, \emph{Long time behavior of solutions to the 3d
  hall-magneto-hydrodynamics system with one diffusion}, preprint (2017).

\bibitem{Doi}
Shin-ichi Doi, \emph{Remarks on the {C}auchy problem for {S}chr\"{o}dinger-type
  equations}, Comm. Partial Differential Equations \textbf{21} (1996), no.~1-2,
  163--178. \MR{1373768}

\bibitem{DGCT}
J.~C. Dorelli, Alex Glocer, Glyn Collinson, and G{\'a}bor T{\'o}th, \emph{The
  role of the hall effect in the global structure and dynamics of planetary
  magnetospheres: Ganymede as a case study}, Journal of Geophysical Research:
  Space Physics \textbf{120} (2015), no.~7, 5377--5392.

\bibitem{DKM}
James Drake, Robert~G. Kleva, and M.~E. Mandt, \emph{Structure of thin current
  layers: Implications for magnetic reconnection.}, Physical review letters
  \textbf{73 9} (1994), 1251--1254.

\bibitem{DRG}
J.~Dreher, V.~Runban, and R.~Grauer, \emph{Axisymmetric flows in
  {H}all-{M}{H}{D}: a tendency towards finite-time singularity formation},
  Phys. {S}cr. \textbf{72} (2005).

\bibitem{Du1}
Baoying Du, \emph{Global regularity for the {$2\frac12$}d incompressible
  {H}all-{MHD} system with partial dissipation}, J. Math. Anal. Appl.
  \textbf{484} (2020), no.~1, 123701, 27. \MR{4042622}

\bibitem{FHN}
Jishan Fan, Shuxiang Huang, and Gen Nakamura, \emph{Well-posedness for the
  axisymmetric incompressible viscous {H}all-magnetohydrodynamic equations},
  Appl. Math. Lett. \textbf{26} (2013), no.~9, 963--967. \MR{3069946}

\bibitem{GaSm}
S.~Peter Gary and Charles~W. Smith, \emph{Short‐wavelength turbulence in the
  solar wind: Linear theory of whistler and kinetic alfv{\'e}n fluctuations},
  Journal of Geophysical Research \textbf{114} (2009).

\bibitem{GHGM}
Pierre Germain, Benjamin Harrop-Griffiths, and Jeremy~L. Marzuola,
  \emph{Existence and uniqueness of solutions for a quasilinear {K}d{V}
  equation with degenerate dispersion}, Comm. Pure Appl. Math. \textbf{72}
  (2019), no.~11, 2449--2484. \MR{4011864}

\bibitem{GoRe}
Peter {Goldreich} and Andreas {Reisenegger}, \emph{{Magnetic Field Decay in
  Isolated Neutron Stars}}, Astrophysical Journal \textbf{395} (1992), 250.

\bibitem{GoCu}
K.~N. Gourgouliatos and A.~Cumming, \emph{{Hall effect in neutron star crusts:
  evolution, endpoint and dependence on initial conditions}}, Monthly Notices
  of the Royal Astronomical Society \textbf{438} (2013), no.~2, 1618--1629.

\bibitem{Hall}
E.~H. Hall, \emph{On a {N}ew {A}ction of the {M}agnet on {E}lectric
  {C}urrents}, Amer. J. Math. \textbf{2} (1879), no.~3, 287--292. \MR{1505227}

\bibitem{HGM}
Benjamin Harrop-Griffiths and Jeremy~L. Marzuola, \emph{Local well-posedness
  for a quasilinear {S}chr\"{o}dinger equation with degenerate dispersion},
  Indiana Univ. Math. J. \textbf{71} (2022), no.~4, 1585--1626. \MR{4481094}

\bibitem{HaOz}
Nakao Hayashi and Tohru Ozawa, \emph{Remarks on nonlinear {S}chr\"{o}dinger
  equations in one space dimension}, Differential Integral Equations \textbf{7}
  (1994), no.~2, 453--461. \MR{1255899}

\bibitem{HoGr}
Holger Homann and Rainer Grauer, \emph{Bifurcation analysis of magnetic
  reconnection in {H}all-{MHD}-systems}, Phys. D \textbf{208} (2005), no.~1-2,
  59--72. \MR{2167907}

\bibitem{IfTa}
Mihaela Ifrim and Daniel Tataru, \emph{Global bounds for the cubic nonlinear
  {S}chr\"{o}dinger equation ({NLS}) in one space dimension}, Nonlinearity
  \textbf{28} (2015), no.~8, 2661--2675. \MR{3382579}

\bibitem{JM}
Juhi Jang and Nader Masmoudi, \emph{Derivation of {O}hm's law from the kinetic
  equations}, SIAM J. Math. Anal. \textbf{44} (2012), no.~5, 3649--3669.
  \MR{3023426}

\bibitem{JKL}
Eunji Jeong, Junha Kim, and Jihoon Lee, \emph{Local well-posedness and blow-up
  for the solutions to the axisymmetric inviscid {H}all-{MHD} equations}, Adv.
  Math. Phys. (2018), Art. ID 5343824, 16. \MR{3875724}

\bibitem{JO3}
In-Jee Jeong and Sung-Jin Oh, \emph{On illposedness of the hall and electron
  magnetohydrodynamic equations without resistivity on the whole space}.

\bibitem{JO2}
\bysame, \emph{On the {C}auchy problem for the {H}all and electron
  magnetohydrodynamic equations without resistivity {I}{I}: geometric
  conditions ensuring wellposedness}.

\bibitem{JO1}
\bysame, \emph{On the {C}auchy problem for the {H}all and electron
  magnetohydrodynamic equations without resistivity {I}: illposedness near
  degenerate stationary solutions}, Annals of PDE \textbf{8} (2019).

\bibitem{JO4}
In-Jee Jeong and Sung-Jin Oh, \emph{Illposedness for dispersive equations:
  {D}egenerate dispersion and {T}akeuchi--{M}izohata condition}, 2023.

\bibitem{Jones}
P.~B. {Jones}, \emph{{Neutron star magnetic field decay - Hall drift and Ohmic
  diffusion}}, Monthly Notices of the Royal Astronomical Society \textbf{233}
  (1988), 875--885.

\bibitem{KPRV1}
C.~E. Kenig, G.~Ponce, C.~Rolvung, and L.~Vega, \emph{Variable coefficient
  {S}chr\"{o}dinger flows for ultrahyperbolic operators}, Adv. Math.
  \textbf{196} (2005), no.~2, 373--486. \MR{2166312}

\bibitem{KPRV2}
Carlos~E. Kenig, Gustavo Ponce, Christian Rolvung, and Luis Vega, \emph{The
  general quasilinear ultrahyperbolic {S}chr\"{o}dinger equation}, Adv. Math.
  \textbf{206} (2006), no.~2, 402--433. \MR{2263709}

\bibitem{KPV1}
Carlos~E. Kenig, Gustavo Ponce, and Luis Vega, \emph{Small solutions to
  nonlinear {S}chr\"{o}dinger equations}, Ann. Inst. H. Poincar\'{e} C Anal.
  Non Lin\'{e}aire \textbf{10} (1993), no.~3, 255--288. \MR{1230709}

\bibitem{KPV2}
\bysame, \emph{Smoothing effects and local existence theory for the generalized
  nonlinear {S}chr\"{o}dinger equations}, Invent. Math. \textbf{134} (1998),
  no.~3, 489--545. \MR{1660933}

\bibitem{KPV3}
\bysame, \emph{The {C}auchy problem for quasi-linear {S}chr\"{o}dinger
  equations}, Invent. Math. \textbf{158} (2004), no.~2, 343--388. \MR{2096797}

\bibitem{Light}
M.~J. Lighthill, \emph{Studies on magneto-hydrodynamic waves and other
  anisotropic wave motions}, Philos. Trans. Roy. Soc. London Ser. A
  \textbf{252} (1960), 397--430. \MR{0148337}

\bibitem{LP}
Wee~Keong Lim and Gustavo Ponce, \emph{On the initial value problem for the one
  dimensional quasi-linear {S}chr\"{o}dinger equations}, SIAM J. Math. Anal.
  \textbf{34} (2002), no.~2, 435--459. \MR{1951782}

\bibitem{MMT1}
Jeremy~L. Marzuola, Jason Metcalfe, and Daniel Tataru, \emph{Quasilinear
  {S}chr\"{o}dinger equations {I}: {S}mall data and quadratic interactions},
  Adv. Math. \textbf{231} (2012), no.~2, 1151--1172. \MR{2955206}

\bibitem{MMT2}
\bysame, \emph{Quasilinear {S}chr\"{o}dinger equations, {II}: {S}mall data and
  cubic nonlinearities}, Kyoto J. Math. \textbf{54} (2014), no.~3, 529--546.
  \MR{3263550}

\bibitem{MMT3}
\bysame, \emph{Quasilinear {S}chr\"{o}dinger equations {III}: {L}arge data and
  short time}, Arch. Ration. Mech. Anal. \textbf{242} (2021), no.~2,
  1119--1175. \MR{4331023}

\bibitem{LSDP}
Laura~F. Morales, Sergio Dasso, Daniel~O. G{\'o}mez, and Pablo Mininni,
  \emph{Hall effect on magnetic reconnection at the earth's magnetopause},
  Journal of Atmospheric and Solar-Terrestrial Physics \textbf{67} (2005),
  no.~17, 1821--1826, Space Geophysics.

\bibitem{Pecseli}
Hans P{\'e}cseli, \emph{Waves and oscillations in plasmas}, Series in Plasma
  Physics (2012).

\bibitem{PWX}
Yi~Peng, Huaqiao Wang, and Qiuju Xu, \emph{Derivation of the {H}all-{MHD}
  equations from the {N}avier-{S}tokes-{M}axwell equations}, J. Nonlinear Sci.
  \textbf{32} (2022), no.~6, Paper No. 90, 27. \MR{4491095}

\bibitem{PiTa}
Ben Pineau and Mitchell~A. Taylor, \emph{{L}ow regularity solutions for the
  general {Q}uasilinear ultrahyperbolic {S}chr\"{o}dinger equation}, preprint
  (2023).

\bibitem{RaYa}
Mohammad~Mahabubur Rahman and Kazuo Yamazaki, \emph{Remarks on the global
  regularity issue of the two-and-a-half-dimensional
  {H}all-magnetohydrodynamics system}, Z. Angew. Math. Phys. \textbf{73}
  (2022), no.~5, Paper No. 217, 29. \MR{4490634}

\bibitem{RaFi}
Y.~Raitses and N.~J. Fisch, \emph{{Parametric investigations of a
  nonconventional Hall thruster}}, Physics of Plasmas \textbf{8} (2001), no.~5,
  2579--2586.

\bibitem{Ros05}
Philip Rosenau, \emph{What is{$\dots$}a compacton?}, Notices Amer. Math. Soc.
  \textbf{52} (2005), no.~7, 738--739. \MR{2159688}

\bibitem{RoHy}
Philip Rosenau and James~M. Hyman, \emph{Compactons: Solitons with finite
  wavelength}, Physical Review Letters \textbf{70} (1993), no.~5, 564--567.

\bibitem{RoSc}
Philip Rosenau and Steven Schochet, \emph{Compact and almost compact breathers:
  a bridge between an anharmonic lattice and its continuum limit}, Chaos
  \textbf{15} (2005), no.~1, 015111, 18. \MR{2133462}

\bibitem{RoZi15}
Philip Rosenau and Alon Zilburg, \emph{On a strictly compact discrete breathers
  in a {K}lein-{G}ordon model}, Phys. Lett. A \textbf{379} (2015), no.~43-44,
  2811--2816. \MR{3402749}

\bibitem{SDRD}
M.~A. Shay, J.~F. Drake, B.~N. Rogers, and R.~E. Denton, \emph{Alfv{\'e}nic
  collisionless magnetic reconnection and the hall term}, Journal of
  Geophysical Research: Space Physics \textbf{106} (2001), no.~A3, 3759--3772.

\bibitem{SRF}
Artem Smirnov, Yegeny Raitses, and Nathaniel~J. Fisch, \emph{{Experimental and
  theoretical studies of cylindrical Hall thrustersa)}}, Physics of Plasmas
  \textbf{14} (2007), no.~5, 057106.

\bibitem{Stein}
Elias~M. Stein, \emph{Harmonic analysis: real-variable methods, orthogonality,
  and oscillatory integrals}, Princeton Mathematical Series, vol.~43, Princeton
  University Press, Princeton, NJ, 1993, With the assistance of Timothy S.
  Murphy, Monographs in Harmonic Analysis, III. \MR{1232192}

\bibitem{Tay}
Michael~E. Taylor, \emph{Pseudodifferential operators and nonlinear pde},
  Progress in Mathematics, Birkh\"{a}user Boston, MA, 1991.

\bibitem{WaZh2}
Renhui Wan and Yong Zhou, \emph{On global existence, energy decay and blow-up
  criteria for the {H}all-{MHD} system}, J. Differential Equations \textbf{259}
  (2015), no.~11, 5982--6008. \MR{3397315}

\bibitem{WaZh1}
\bysame, \emph{Global well-posedness, {BKM} blow-up criteria and zero {$h$}
  limit for the 3{D} incompressible {H}all-{MHD} equations}, J. Differential
  Equations \textbf{267} (2019), no.~6, 3724--3747. \MR{3955610}

\bibitem{WaHo}
C~J Wareing and Rainer Hollerbach, \emph{Forward and inverse cascades in
  decaying two-dimensional electron magnetohydrodynamic turbulence}, Physics of
  Plasmas \textbf{16} (2009), 042307.

\bibitem{WoHoLy}
Toby~S. Wood, Rainer Hollerbach, and Maxim Lyutikov, \emph{Density-shear
  instability in electron magneto-hydrodynamics}, Physics of Plasmas
  \textbf{21} (2014), 052110.

\bibitem{Ya}
Kazuo Yamazaki, \emph{Stochastic {H}all-magneto-hydrodynamics system in three
  and two and a half dimensions}, J. Stat. Phys. \textbf{166} (2017), no.~2,
  368--397. \MR{3596853}

\bibitem{YBMH}
Rtimi Youness, Maxime Bonnet, Fr\'{e}d\'{e}ric Messine, and Carole H\'{e}naux,
  \emph{Topology shape and parametric design optimization of {H}all effect
  thrusters}, Scientific computing in electrical engineering, Math. Ind.,
  vol.~32, Springer, Cham, [2020] \copyright 2020, pp.~255--264. \MR{4214205}

\bibitem{Ze}
Yong Zeng, \emph{Steady states of {H}all-{MHD} system}, J. Math. Anal. Appl.
  \textbf{451} (2017), no.~2, 757--793. \MR{3624766}

\end{thebibliography}

\end{document}